\documentclass{article}
\usepackage{latexsym, amscd, amsfonts, eucal, mathrsfs, amsmath, amssymb, amsthm, xypic,xr, makeidx, stmaryrd}
\def\Z{\mathbf{Z}}
\def\C{\mathbf{C}}

\DeclareMathOperator{\Spectra}{Sp}

\DeclareMathOperator{\an}{an}

\DeclareMathOperator{\smooth}{sm}

\DeclareMathOperator{\alg}{alg}
\DeclareMathOperator{\der}{der}

\DeclareMathOperator{\Stein}{Stein}

\DeclareMathOperator{\Comm}{CRing}

\DeclareMathOperator{\Top}{Top}
\DeclareMathOperator{\sob}{sob}
\DeclareMathOperator{\loc}{loc}

\DeclareMathOperator{\RGeom}{\mathcal{R}{\mathcal T}op}
\DeclareMathOperator{\RingSpace}{RingSpace}
\DeclareMathOperator{\LRingSpace}{RingSpace^{loc}}
\DeclareMathOperator{\proadm}{pro-adm}

\DeclareMathOperator{\USpec}{Spec}

\DeclareMathOperator{\SSpec}{Spec} 
\DeclareMathOperator{\Open}{Open}
\DeclareMathOperator{\Poly}{Poly}
\DeclareMathOperator{\SCR}{SCR}

\DeclareMathOperator{\lex}{lex}

\DeclareMathOperator{\CompAn}{{\mathcal A}n_{\C}}

\newcommand{\Andre}{Andr\'{e}}

\DeclareMathOperator{\Diff}{Diff}

\DeclareMathOperator{\Sch}{Sch}

\newcommand{\EInfty}{ {\mathfrak E}_{\infty}}
\DeclareMathOperator{\LGeom}{{\mathcal L}{\mathcal T}op}
\DeclareMathOperator{\LGeo}{\LGeom}
\DeclareMathOperator{\Struct}{Str}

\DeclareMathOperator{\adm}{ad}

\DeclareMathOperator{\disc}{disc}

\DeclareMathOperator{\conn}{conn}

\DeclareMathOperator{\Sing}{Sing}

\DeclareMathOperator{\CAlg}{CAlg}

\DeclareMathOperator{\LFun}{\Fun^{L}}

\DeclareMathOperator{\RFun}{\Fun^{R}}

\newcommand{\toposref}[1]{T.\ref{HTT-#1}}
\newcommand{\stableref}[1]{S.\ref{STA-#1}}
\newcommand{\monoidref}[1]{M.\ref{MON-#1}}
\newcommand{\symmetricref}[1]{C.\ref{SYM-#1}}
\newcommand{\deformationref}[1]{D.\ref{DEF-#1}}

\newcommand{\degree}{\text{o}}
\newcommand{\bfA}{{\mathbf A}}

\DeclareMathOperator{\Kan}{\mathcal{K}an}
\DeclareMathOperator{\cDelta}{{\bf \Delta}}

\DeclareMathOperator{\Set}{\mathcal{S}et}
\DeclareMathOperator{\sSet}{\mathcal{S}et_{\Delta}}
\DeclareMathOperator{\sCoNerve}{\mathfrak{C}}
\DeclareMathOperator{\Nerve}{N}
\DeclareMathOperator{\Cat}{\mathcal{C}at}
\newcommand{\h}[1]{\rm{h} \! #1}

\newcommand{\Adjoint}[4]{\xymatrix@1{#2 \ar@<.4ex>[r]^-{#1} & #3 \ar@<.4ex>[l]^-{#4}}}

\newcommand{\etale}{{\'{e}tale}\,\,}

\newcommand{\et}{\'{e}t}
\newcommand{\mathet}{\text{\et}}
\newcommand{\Etale}{\'{E}tale\,}
\newcommand{\bigdot}{\bullet}

\newcommand{\fin}{\text{fin}}
\DeclareMathOperator{\bHom}{Map}
\DeclareMathOperator{\Sh}{\mathcal{S}hv}
\DeclareMathOperator{\Fin}{\mathcal{F}in}
\DeclareMathOperator{\Shv}{\Sh}

\newcommand{\Cech}{\v{C}ech\,}

\DeclareMathOperator{\Aff}{Aff}
\DeclareMathOperator{\Sym}{Sym}

\DeclareMathOperator{\Mod}{\mathcal{M}od}

\DeclareMathOperator{\Pro}{Pro}
\DeclareMathOperator{\p}{\mathfrak{p}}
\DeclareMathOperator{\calM}{\mathcal{M}}
 
\DeclareMathOperator{\calL}{\mathcal{L}}
 \DeclareMathOperator{\G}{\mathbf{G}}
\DeclareMathOperator{\calZ}{\mathcal{Z}}

\DeclareMathOperator{\Zar}{Zar} 
\DeclareMathOperator{\DerRing}{\mathcal{SCR}}

\DeclareMathOperator{\Q}{\mathbf{Q}}
\DeclareMathOperator{\Tor}{Tor} 
\DeclareMathOperator{\colim}{colim}
\DeclareMathOperator{\A}{\mathbf A}
\DeclareMathOperator{\calA}{\mathcal{A}}
\DeclareMathOperator{\bd}{\partial}

\DeclareMathOperator{\calU}{\mathcal{U}}

\DeclareMathOperator{\calV}{\mathcal{V}}
\DeclareMathOperator{\calE}{\mathcal{E}}
\DeclareMathOperator{\R}{\mathbf{R}}
\DeclareMathOperator{\calO}{\mathcal{O}}
\DeclareMathOperator{\calT}{\mathcal{T}}

\DeclareMathOperator{\calB}{\mathcal{B}}
\DeclareMathOperator{\calK}{\mathcal{K}}
\DeclareMathOperator{\Spec}{{\bf Spec}}

\DeclareMathOperator{\calF}{\mathcal{F}}
\DeclareMathOperator{\calG}{\mathcal{G}}
\DeclareMathOperator{\Hom}{Hom} 
\DeclareMathOperator{\HH}{H} 
\DeclareMathOperator{\id}{id} \DeclareMathOperator{\Fun}{Fun}
\DeclareMathOperator{\calC}{\mathcal{C}}
\DeclareMathOperator{\calI}{\mathcal{I}}

\DeclareMathOperator{\calJ}{\mathcal{J}}

\DeclareMathOperator{\SSet}{\mathcal{S}}

\DeclareMathOperator{\calX}{\mathcal{X}}
\DeclareMathOperator{\Idem}{Idem}

\DeclareMathOperator{\calY}{\mathcal{Y}}
\DeclareMathOperator{\rex}{rex}
\DeclareMathOperator{\sk}{sk}
\DeclareMathOperator{\op}{op}
\DeclareMathOperator{\calD}{\mathcal{D}}
\DeclareMathOperator{\Ind}{Ind} 
\DeclareMathOperator{\calP}{\mathcal{P}} \topmargin=0in

\newtheorem{theorem}{Theorem}[subsection]
\newtheorem{lemma}[theorem]{Lemma}
\newtheorem{proposition}[theorem]{Proposition}

\newtheorem{corollary}[theorem]{Corollary}

\theoremstyle{definition}
\newtheorem{definition}[theorem]{Definition}

\newtheorem{example}[theorem]{Example}

\newtheorem{notation}[theorem]{Notation}

\newtheorem{warning}[theorem]{Warning}
\newtheorem{remark}[theorem]{Remark}

\makeindex

\begin{document}

\title{Derived Algebraic Geometry V: Structured Spaces}

\maketitle
\tableofcontents

\section*{Introduction: Bezout's Theorem}

Let $C,C' \subseteq \mathbf{P}^2$ be two smooth algebraic curves
of degrees $m$ and $n$ in the complex projective plane
$\mathbf{P}^2$. If $C$ and $C'$ meet transversely, then the
classical theorem of Bezout (see for example \cite{bezout})
asserts that $C \cap C'$ has precisely $mn$ points.

We may reformulate the above statement using the language of
cohomology. The curves $C$ and $C'$ have fundamental classes $[C],
[C'] \in \HH^2(\mathbf{P}^2; \Z)$. If $C$ and $C'$ meet
transversely, then we have the formula
$$[C] \cup [C'] = [C \cap C'],$$ where the fundamental class $[C
\cap C'] \in \HH^4(\mathbf{P}^2; \Z) \simeq \Z$ of the
intersection $C \cap C'$ simply counts the number of points where $C$ and $C'$ meet. Of course, this should not be surprising: the
cup product on cohomology classes is defined so as to encode the
operation of intersection. However, it would be a mistake to
regard the equation $[C] \cup [C'] = [C \cap C']$ as obvious,
because it is not always true. For example, if the curves $C$ and
$C'$ meet nontransversely (but still in a finite number of
points), then we always have a strict inequality
$$[C] \cup [C'] > [C \cap C']$$
if the right hand side is again interpreted as counting the number
of points in the set-theoretic intersection of $C$ and $C'$.

If we want a formula which is valid for non-transverse
intersections, then we must alter the definition of $[C \cap C']$
so that it reflects the appropriate intersection multiplicities.
Determination of these intersection multiplicities requires
knowledge of the intersection $C \cap C'$ as a scheme, rather than
simply as a set. This is one of the classic arguments that
nonreduced scheme structures carry useful information: the
intersection number $[C] \cup [C'] \in \Z$, which is defined {\it
a priori} by perturbing the curves so that they meet
transversally, can also be computed directly (without
perturbation) if one is willing to contemplate a potentially
non-reduced scheme structure on the intersection.

In more complicated situations, the appropriate intersection
multiplicities cannot always be determined from the
scheme-theoretic intersection alone. Suppose that $C$ and $C'$ are
smooth subvarieties of $\mathbf{P}^n$ of
complementary dimension, having a zero-dimensional
intersection. In this case, the appropriate intersection number
associated to a point $p \in C \cap C'$ is not always given by the
complex dimension of the local ring
$$ \calO_{C \cap C',p} = \calO_{C,p} \otimes_{\calO_{\mathbf{P}^n,p}}
\calO_{C',p}.$$ The reason for this is easy to understand from the
point of view of homological algebra. Since the tensor product
functor $\otimes_{\calO_{\mathbf{P}^n,p}}$ is not exact, it does
not have good properties when considered alone. According to
Serre's intersection formula, the correct intersection
multiplicity is instead the Euler characteristic
$$ \sum (-1)^i \text{dim} \Tor^{\calO_{\mathbf{P}^n,p}}_i ( \calO_{C,p},
\calO_{C',p}).$$ This Euler characteristic contains the dimension
of the local ring of the scheme-theoretic intersection as its
leading term, but also higher-order corrections. We refer the
reader to \cite{intersection} for further discussion of this
formula for the intersection multiplicity.

If we would like the equation $[C] \cup [C'] = [C \cap C']$ to
remain valid in the more complicated situations described above,
then we will need to interpret the intersection $C \cap C'$ in
some way which remembers not only the tensor product $\calO_{C,p}
\otimes_{\calO_{\mathbf{P}^n,p}} \calO_{C',p}$, but the higher
$\Tor$ groups. Moreover, we should not interpret these
$\Tor$-groups separately, but rather should think of the total
derived functor $\calO_{C,p} \otimes^L_{\calO_{\mathbf{P}^n,p}}
\calO_{C',p}$ as a kind of generalized ring. 

These considerations lead us naturally to the subject of {\em
derived algebraic geometry}. Using an appropriate notion of
``generalized ring'', we will mimic the constructions of classical
scheme theory to obtain a theory of {\em derived schemes} in which
a version of the formula $[C] \cup [C'] = [C \cap C']$ can be
shown to hold with (essentially) {\em no} hypotheses on $C$ and
$C'$. Here, we must interpret the intersection $C \cap C'$ in the
sense of derived schemes, and we must take great care to give the
proper definition for the fundamental classes (the so-called {\it
virtual fundamental classes} of \cite{virtual}).

What sort of objects should our generalized rings be? To answer this question, we begin by
considering the simplest case of Bezout's theorem, in which $C$
and $C'$ are {\em lines} in the projective plane $\mathbf{P}^2$. In this
case, we know that $[C] \cup [C']$ is the cohomology class of a
point, and that $C$ intersects $C'$ transversely in one point so
long as $C$ and $C'$ are distinct. However, when the equality $C =
C'$ holds, the scheme-theoretic intersection $C \cap C'$ does not
even have the correct dimension.

Let us now try to formulate a theory of intersections which will handle the degenerate situation where $C = C'$. To simplify the discussion, we will consider only lines in the affine plane $\A^2
\subseteq \mathbf{P}^2$, with coordinate ring $\C[x,y]$. Two
distinct lines in $\A^2$ may be given (without loss of generality) by the equations $x=0$ and
$y=0$. The scheme-theoretic intersection of these two lines is the
spectrum of the ring $\C[x,y]/ (x,y) \simeq \C$, obtained from
$\C[x,y]$ by enforcing the defining equations of both lines. This ring
has Krull dimension $0$ because $\C[x,y]$ has Krull dimension $2$ and we have imposed
$2$ independent conditions.

Now suppose that instead of $C$ and $C'$ being two distinct lines,
they are actually two {\em identical} lines, both of which are defined by 
the equation $x=0$. In this case, the affine ring of the scheme
theoretic intersection is given by $\C[x,y]/(x,x) \simeq \C[y]$.
This ring has Krull dimension $1$, rather than the expected dimension
$0$, because the two equations are not independent: setting $x$ equal to zero twice 
has the same effect as setting $x$ equal to zero once. To obtain the theory we are looking for, we need a notion of generalized ring which remembers not only whether or not $x$ is equal to $0$, but how many {\em different ways} $x$ is equal to $0$.

One way to obtain such a formalism is by {\em categorifying} the
notion of a commutative ring. That is, in place of ordinary
commutative rings, we consider {\em categories} equipped
with addition and multiplication operations (which are encoded by
functors, rather than ordinary functions). For purposes of the
present discussion, let us call such an object a {\it categorical
ring}. We will not give a precise axiomatization of this notion;
this turns out to be somewhat complicated (see \cite{categoricalring}, for example).

\begin{example}\label{Dfin}
Let $\Z_{\geq 0}$ denote the semiring of nonnegative integers. We
note that $\Z_{\geq 0}$ arises in nature through the process of
{\em decategorification}. The nonnegative integers were originally introduced in order to
{\em count}: in other words, in order to measure the size of finite sets. To make this statement more precise, let us denote by $\Fin$ the category whose objects are finite sets, and whose morphisms
are isomorphisms of finite sets. Then we can identify $\Z_{\geq 0}$ with the set of isomorphism classes of objects in $\Fin$. The addition and multiplication operations on $\Z_{\geq 0}$ are induced by {\em functors} $\Fin \times \Fin \rightarrow \Fin$, given by the formation of disjoint union and Cartesian product. Moreover, all of the axioms for a commutative semiring have analogues that hold at the categorical level: for example, the distributive law $xy+xz = x(y+z)$ translates into the existence of a canonical isomorphism
$$ (X \times Y) \coprod (X \times Z) \simeq X \times (Y \coprod Z)$$
for every triple of objects $X,Y,Z \in \Fin$. (In order to complete the analogy
with the above discussion, we should ``complete'' the category
$\Fin$ by formally adjoining inverses, to obtain a categorical
ring rather than a categorical semiring, but we will ignore this
point for the time being.)
\end{example}

To simplify the discussion, we will consider only categorical
rings which are groupoids: that is, every morphism in the
underlying category is an isomorphism. If $\calC$ is a categorical
ring, then the collection of isomorphism classes of objects in $\calC$ forms an ordinary commutative ring,
which we will denote by $\pi_0 \calC$. Every commutative ring
$R$ arises in this way: for example, we may take $\calC$ to be a
category whose objects are the elements of $R$ and which contains
only identity maps for morphisms. The categorical rings which
arise in this way are very special: their objects have no
nontrivial automorphisms. For a given commutative ring $R$, there
are usually many other ways to realize an isomorphism of $R$ with the collection of isomorphism classes of objects in a
categorical ring $\calC$. Although
$\calC$ is not uniquely determined by $R$, there is often
a natural choice for $\calC$ which is dictated by the manner in
which $R$ is constructed.

As an example, let us suppose that the commutative ring $R$ is
given as a quotient $R'/(x-y)$, where $R'$ is some other
commutative ring and $x,y \in R'$ are two elements. Suppose that
the ring $R'$ has already been ``categorified'' in the sense that
we have selected some categorical ring $\calC'$ and an
identification of $R'$ with $\pi_0 \calC' $. To this data, we wish
to associate some ``categorification'' $\calC$ of $R$. Roughly,
the idea is to think of $x$ and $y$ as objects of $\calC'$,
and to impose the relation $x=y$ at the {\em categorical} level.
However, it is extremely unnatural to ask that two objects in a
category be {\em equal}; instead one should ask that they be {\em
isomorphic}. In other words, the quotient category $\calC$ should
not be obtained from $\calC'$ by identifying the objects $x$ and
$y$. Instead we should construct $\calC$ by {\it enlarging}
$\calC'$ so that it includes an isomorphism $\alpha: x \simeq y$.
Since we want $\calC$ to be a categorical ring, the formation of
this enlargement is a somewhat complicated business: in addition
to the new isomorphism $\alpha$, we must also adjoin other
isomorphisms which can be obtained from $\alpha$ through addition,
multiplication, and composition (and new relations, which may
cause distinct isomorphisms in $\calC'$ to have the same image in
$\calC$).

To make the connection with our previous discussion, let us note
that the construction of $\calC$ from $\calC'$ described in the
preceding paragraph is interesting even in the situation
where $x = y$. In this case, $x$ and $y$ are already isomorphic
when thought of as objects of $\calC'$. However, in $\calC$ we get
a {\em new} isomorphism $\alpha$ between $x$ and $y$, which
usually does not lie in the image of the map
$\Hom_{\calC'}(x,y) \rightarrow \Hom_{\calC}(x,y)$. Consequently,
even though the quotient map $R' \rightarrow R$ is an
isomorphism, the underlying functor
$\calC' \rightarrow \calC$ need not be an equivalence of
categories. Imposing the new relation $x = y$ does not change the
collection of isomorphism classes of objects, but often {\em
does} change the automorphism groups of the objects. Consequently,
if we begin with {\em any} objects $x$ and $y$, we can iterate the
above construction two or more times to obtain a categorical ring
$\calC$ equipped with multiple isomorphisms $x \simeq y$. These
isomorphisms are (in general) distinct from one another, so that
the categorical ring $\calC$ ``knows'' how many times $x$ and $y$
have been identified.

We have now succeeded in finding a formalism which is sensitive to
redundant information: we just need to replace ordinary
commutative rings with categorical rings. The next question to ask is whether or not this formalism is of any use. For example, can we carry out computations in this kind of formalism?
Suppose that $\calC'$ is a categorical ring containing a pair of objects $x,y \in \calC'$
as above, and we form a new categorical ring $\calC$ by adjoining an isomorphism $\alpha: x \rightarrow y$. For simplicity, we will suppose that $\calC'$ is a {\em discrete} category (having no nonidentity morphisms), which we may identify with an ordinary commutative ring $R' = \pi_0 \calC'$. The commutative ring $R = \pi_0 \calC$ is easy to compute: it can be identified with the
cokernel of the map $$\phi: R'
\stackrel{x-y}{\rightarrow} R'.$$ It turns out that the
automorphism groups of objects of $\calC$ are also readily computable: for every object $x \in \calC$, there is a canonical isomorphism $\Hom_{\calC}(x,x) \simeq \ker(\phi)$.

Let us return to geometry for a moment, and suppose that $R'$ is
the affine ring of a curve (possibly nonreduced) in $\Aff^2 =
\Spec \C[x,y]$. Let $R'' = \C[x,y]/(x-y)$ denote the affine ring
of the diagonal. Then the cokernel and kernel of $\phi$ may be
naturally identified with $\Tor_0^{\C[x,y]}(R',R'')$ and
$\Tor_1^{\C[x,y]}(R',R'')$. In other words, just as the leading
term in Serre's intersection formula has an interpretation in terms of tensor constructions with ordinary commutative rings, the second term has a geometric interpretation in terms of categorical rings.

Unfortunately, this is far as categorical rings will take us. In
order to interpret the next term in Serre's intersection formula,
we need to take categorification one step further and
consider ring structures on {\it $2$-groupoids}. To understand the entire formula, 
we need to consider commutative ring structures on $\infty$-groupoids, which we may think of as topological spaces or as simplicial sets. There are (at least) two essentially different ways to make this idea precise. The first is to work with topological spaces (or better yet, simplicial sets) which have a commutative ring structure, where the addition and multiplication are given by continuous maps; here we require the axioms of a commutative ring to be satisfied ``on the nose''. The collection of
all such topological rings can be organized into an $\infty$-category $\SCR$ which we will study in \S \ref{escr}. There is also a more sophisticated theory of commutative ring structures on $\infty$-groupoids, in which the commutative ring axioms are only required to hold up to (coherent) homotopy.
This leads to the theory of (connective) {\it $E_{\infty}$-rings} which we studied in \cite{symmetric} (see also \cite{EKMM}).

\begin{remark}
We should emphasize that when studying a topological ring $A$, we are much more interested
in the {\em homotopy type} of $A$ than we are in the {\em topology} of $A$. In other words, we regard the topology on $A$ as a mechanism
which allows us to discuss paths, homotopies between paths, and so
forth. Consequently, most of the topological rings which arise naturally in mathematics are quite
uninteresting from our point of view. For example, any ring which
is a topological vector space over $\R$ is contractible, and thus
equivalent to the zero ring. On the other hand, any $\p$-adically
topologized ring has no nontrivial paths, and is thus equivalent
to a discrete ring from our point of view. The topological rings
which {\em do} arise in derived algebraic geometry are generally
obtained from discrete rings by applying various categorical
constructions, and are difficult to describe directly.
\end{remark}

We now have two reasonable candidates for our theory of
``generalized rings'': $E_{\infty}$-ring spectra and simplicial
commutative rings. Which is the better notion? The answer depends,
of course, on what we want to do. Roughly speaking, the theory of simplicial commutative
rings can be regarded as a mechanism for applying ideas from algebraic topology
to the study of commutative algebra. If we take the
point of view that our ultimate interest is in {\em ordinary}
commutative rings, then simplicial commutative rings arise
naturally because certain constructions (such as left derived
tensor products) force us to consider more general objects (as in our discussion
of Bezout's theorem above). By contrast, the theory of $E_{\infty}$-rings can be thought of as a mechanism
for applying ideas from commutative algebra to algebraic topology (more specifically, to stable homotopy theory). For example, we might observe that for every compact topological space
$X$, the complex $K$-theory $K(X)$ has the structure of a commutative ring; we would then like to summarize this fact by saying that, in some sense, $K$-theory itself is a
commutative ring. The theory of $E_{\infty}$-ring spectra provides
the correct language for describing the situation: $K$-theory and
many other generalized cohomology theories of interest may be
endowed with $E_{\infty}$-structures.

We are now in a position to describe (at least informally) the theory of {\em derived schemes}
which we will introduce in this paper. Just as an ordinary scheme is defined to be ``something which
looks locally like $\Spec A$ where $A$ is a commutative ring'', a
derived scheme can be described as ``something which looks
locally like $\Spec A$ where $A$ is a simplicial commutative ring''.
Of course, many variations on this basic idea are possible.
In this paper, we are concerned with laying the foundations for the theory of derived algebraic geometry, rather than any particular application (such the generalization of Bezout's theorem described above). Many of the ideas involved can be fruitfully exported to other contexts (such as complex analytic geometry), so it seems worthwhile to establish the foundations of the theory in a very general form. To this end, we will introduce the notion of a {\it geometry}. Given a geometry $\calG$ and a topological space $X$, there is an associated theory of $\calG$-structures on $X$.
Roughly speaking, a $\calG$-structure on a topological space $X$ is a sheaf $\calO_X$ endowed with some operations, whose exact nature depends on $\calG$. 
By choosing $\calG$ appropriately, we can recover the classical theory of ringed and locally ringed spaces. However, we can also obtain variations on this classical theory, where the structure sheaf
$\calO_{X}$ takes values not in the ordinary category of commutative rings, but in the $\infty$-category
$\SCR$ of simplicial commutative rings (see \S \ref{derzar}). Another possibility is to work with the $\infty$-category of $E_{\infty}$-rings. This leads naturally to the subject of {\it spectral algebraic geometry}, which we will take up in \cite{spectral}.

\subsection*{Overview}

Let us now outline the contents of this paper. First it is convenient to recall a few definitions from
classical scheme theory. A {\it ringed space} is a pair $(X, \calO_{X})$, where $X$ is a topological space and $\calO_{X}$ is a sheaf of commutative rings on $X$. We say that $(X, \calO_{X})$ is
{\it locally ringed} if, for every point $x \in X$, the stalk $\calO_{X,x}$ is a local ring.
The collection of locally ringed spaces can be organized into a category, where
a map of locally ringed spaces $(X, \calO_{X}) \rightarrow (Y, \calO_{Y})$
is given by a continuous map $f: X \rightarrow Y$, together with a map
of sheaves $f^{\ast} \calO_{Y} \rightarrow \calO_X$ satisfying the following
locality condition: for every point $x \in X$, the induced ring homomorphism
$\calO_{Y, f(x)} \rightarrow \calO_{X,x}$ is {\em local}. We say that a locally ringed
space $(X, \calO_X)$ is a {\it scheme} if it is locally isomorphic to a locally
ringed space of the form $( \Spec A, \calO_{\Spec A})$, where $A$ is a commutative ring and
$\Spec A$ is the usual Zariski spectrum of prime ideals in the ring $A$.
A map of schemes is simply a map of the underlying locally ringed spaces, so we
can regard the category of schemes as a full subcategory of the category of locally ringed spaces.

We wish to build a theory of derived algebraic geometry following the above outline.
However, our theory will be different in several respects:

\begin{itemize}
\item[$(a)$] We are ultimately interested in studying moduli problems in derived algebraic geometry. We will therefore need not only a theory of schemes, but also a theory of Deligne-Mumford stacks.
For this, we will need to work with the \etale topology in addition to the Zariski topology.

\item[$(b)$] Let $A$ be a commutative ring. There exists a good theory of \etale sheaves on $\Spec A$, but the category of such sheaves is not equivalent to the category of sheaves on any topological space. For this reason, the theory of ringed spaces $(X, \calO_X)$ needs to be replaced by the more sophisticated theory of {\it ringed topoi} $( \calX, \calO_{\calX} )$. Here $\calX$ denotes
a (Grothendieck) topos, and $\calO_{\calX}$ a commutative ring object of $\calX$. Actually,
it will be convenient to go even further: we will allow $\calX$ to be an arbitrary
$\infty$-topos, as defined in \cite{topoi}. Though the additional generality makes little
difference in practice (see Theorem \ref{top4}), it is quite convenient in setting up
the foundations of the theory.

\item[$(c)$] The structure sheaves $\calO_{\calX}$ that we consider will not take values in the
category of ordinary commutative rings, but instead the $\infty$-category $\SCR$ of {\it simplicial commutative rings} (see Definition \ref{hutip}) or some other variation, such as the $\infty$-category of $E_{\infty}$-rings.
\end{itemize}

We will begin in \S \ref{geo} by introducing the definition of a {\it geometry}. Given a geometry $\calG$ and an $\infty$-topos $\calX$, there is an associated theory of {\it $\calG$-structures on $\calX$}. We can think of a $\calG$-structure on $\calX$ as a sheaf on $\calX$ with some additional structures, whose exact nature depends on the choice of geometry $\calG$. If $\calX$ is an $\infty$-topos and $\calO_{\calX}$ a $\calG$-structure on $\calX$, then we will refer to the pair $(\calX, \calO_{\calX})$ as a {\it $\calG$-structured $\infty$-topos}: these will play the role of locally ringed spaces in our formalism.

To any $\calG$-structured $\infty$-topos $(\calX, \calO_{\calX})$, we can associate its
{\it global sections} $\Gamma( \calX, \calO_{\calX}) \in \Ind( \calG^{op} )$. In \S \ref{secscheme}, we will prove that the functor
$$ (\calX, \calO_{\calX}) \mapsto \Gamma( \calX, \calO_{\calX})$$
admits a right adjoint, which we denote by $\Spec^{\calG}$. We will say that a $\calG$-structured $\infty$-topos $(\calX, \calO_{\calX})$ is an {\it affine $\calG$-scheme} if it belongs to the essential image of $\Spec^{\calG}$. More generally, we say that $(\calX, \calO_{\calX})$ is a {\it $\calG$-scheme} if it is equivalent to an affine $\calG$-scheme, locally on $\calX$.

By construction, our theory of $\calG$-schemes bears a formal analogy to the usual theory of schemes. In \S \ref{exzar} and \S \ref{exet} will show that this analogy can be extended to a dictionary. Namely, for an appropriate choice for the geometry $\calG$ (and after restricting the class of $\infty$-topoi that we consider), our theory of $\calG$-schemes reduces to the usual theory of schemes. By varying $\calG$, we can recover other classical notions as well, such as the theory of Deligne-Mumford stacks.

Of course, our primary goal is to develop a language for describing the {\it derived algebraic geometry} sketched in the introduction. To describe the passage from classical to derived algebraic geometry, 
we will introduce in \S \ref{appus} the notion of a {\it pregeometry} $\calT$. To every pregeometry
$\calT$ and every $0 \leq n \leq \infty$, we can associate a geometry $\calG$ which we call the
{\it $n$-truncated geometric envelope} of $\calT$. We then refer to $\calG$-schemes as
{\it $n$-truncated $\calT$-schemes}. We will consider some examples in \S \ref{app6}; in particular, we will describe a pregeometry $\calT_{\Zar}$ whose scheme theory
interpolates between classical algebraic geometry (associated to $n=0$) and
derived algebraic geometry (associated to $n=\infty$).

The theory of derived algebraic geometry presented here is not new. For another very general foundational approach (of a rather different flavor from ours), we refer the reader to \cite{toen2} and \cite{toen3}.

\subsection*{Notation and Terminology}

For an introduction to the language of higher category theory (from the point of view taken in this paper), we refer the reader to \cite{topoi}. For convenience, we will adopt the following conventions concerning references to \cite{topoi} and to the other papers in this series:

\begin{itemize}
\item[$(T)$] We will indicate references to \cite{topoi} using the letter T.
\item[$(S)$] We will indicate references to \cite{DAGStable} using the letter S.
\item[$(M)$] We will indicate references to \cite{monoidal} using the letter M.
\item[$(C)$] We will indicate references to \cite{symmetric} using the letter C.
\item[$(D)$] We will indicate references to \cite{deformation} using the letter D.
\end{itemize}

For example, Theorem \toposref{mainchar} refers to Theorem \ref{HTT-mainchar} of \cite{topoi}.

If $\calC$ and $\calD$ are $\infty$-categories which admit finite limits, we let $\Fun^{\lex}(\calC, \calD)$ denote the full subcategory of $\Fun(\calC, \calD)$ spanned by those functor which are {\em left exact}: that is, those functors which preserve finite limits. If instead
$\calC$ and $\calD$ admits finite colimits, we let $\Fun^{\rex}(\calC, \calD) = \Fun^{\lex}( \calC^{op}, \calD^{op})^{op}$ denote the full subcategory of $\Fun(\calC, \calD)$ spanned by the {\em right exact} functors: that is, those functors which preserve finite colimits.

If $\calC$ is a small $\infty$-category, we let $\Pro(\calC)$ denote the $\infty$-category
$\Ind(\calC^{op})^{op}$. We will refer to $\Pro(\calC)$ as the {\it $\infty$-category of pro-objects of $\calC$}. We will view $\Pro(\calC)$ as a full subcategory of $\Fun( \calC, \SSet)^{op}$. 
If $\calC$ admits finite limits, then this identification reduces to an equality
$\Pro(\calC) = \Fun^{\lex}( \calC, \SSet)^{op}$. Note that the Yoneda embedding
$j: \calC \rightarrow \calP(\calC)$ determines functors
$$ \Ind(\calC) \leftarrow \calC \rightarrow \Pro(\calC);$$
we will abuse terminology by referring to either of these functors also as the Yoneda embedding.

For every small $\infty$-category $\calC$, the $\infty$-category $\Ind(\calC)$ admits small filtered colimits, and the $\infty$-category $\Pro(\calC)$ admits small filtered limits.
According to Proposition \toposref{intprop}, the $\infty$-categories $\Ind(\calC)$ and $\Pro(\calC)$ can be characterized by the following universal properties:
\begin{itemize}
\item[$(a)$] Let $\calD$ be an $\infty$-category which admits small filtered colimits, and let
$\Fun'(\Ind(\calC), \calD)$ denote the full subcategory of $\Fun(\Ind(\calC), \calD)$ spanned by those functors which preserve small filtered colimits. Then composition with the Yoneda embedding
induces an equivalence
$$ \Fun'( \Ind(\calC), \calD) \rightarrow \Fun(\calC, \calD).$$
\item[$(b)$] Let $\calD$ be an $\infty$-category which admits small filtered limits, and let
$\Fun'(\Pro(\calC), \calD)$ denote the full subcategory of $\Fun(\Pro(\calC), \calD)$ spanned by those functors which preserve small filtered limits. Then composition with the Yoneda embedding
induces an equivalence
$$ \Fun'( \Pro(\calC), \calD) \rightarrow \Fun(\calC, \calD).$$
\end{itemize}
These properties characterize $\Ind(\calC)$ and $\Pro(\calC)$ up to equivalence.
If we assume only that $\calC$ is essentially small, then there still exists a maps
$$ \calC' \leftarrow \calC \rightarrow \calC''$$
satisfying the analogues $(a)$ and $(b)$ (where $\calC'$ admits small filtered colimits and
$\calC''$ admits small filtered limits). We will then abuse notation by writing
$\calC' = \Ind(\calC)$ and $\calC'' = \Pro(\calC)$ (so that $\Ind$-objects and $\Pro$-objects
are defined for all essentially small $\infty$-categories). (It is not difficult to extend this definition to $\infty$-categories which are not assumed to be essentially small, but we will refrain from doing so.)

\section{Structure Sheaves}\label{geo}

Let $X$ be a topological space. Our goal in this section is to develop a general theory of ``sheaves with structure'' $\calF$ on $X$ (or, more generally, on any $\infty$-topos $\calX$). With an eye towards future applications, we would like to be as open-minded as possible regarding the exact nature of this structure. For example, we want to include as a possibility each of the following examples:

\begin{itemize}
\item[$(a)$] The space $X$ is the underlying topological space of a scheme and $\calF = \calO_X$ is the structure sheaf of $X$ (taking values in the category of commutative rings).
\item[$(b)$] The space $X$ is the underlying topological space of a scheme and $\calF$ is a quasi-coherent sheaf of $\calO_X$-modules on $X$.
\item[$(c)$] The space $X$ is the underlying topological space of a scheme and $\calF$ is an object of the derived category of quasi-coherent sheaves on $X$: this can be regarded as a sheaf on $X$ taking values in a suitable $\infty$-category of module spectra.
\item[$(d)$] The ``space'' $X$ is the underlying \etale topos of a Deligne-Mumford stack, and 
the sheaf $\calF$ is of a nature described by $(a)$, $(b)$, or $(c)$.
\item[$(e)$] The space $X$ is a smooth manifold, and $\calF$ is the sheaf of smooth real-valued functions on $X$. Then $\calF$ is a sheaf of commutative rings, but also has additional structure:
for example, any smooth map $f: \R \rightarrow \R$ induces, by composition, a map from $\calF$ to itself.
\item[$(f)$] The space $X$ is the underlying topological space of a derived scheme, and
$\calF = \calO_X$ is the structure sheaf of $X$ (taking values in the $\infty$-category of simplicial commutative rings).
\end{itemize}

We begin in \S \ref{scurrel} by studying the $\infty$-category $\Shv_{\calC}(X)$, where
$\calC$ is an arbitrary $\infty$-category. We can define $\Shv_{\calC}(X)$ as a full subcategory
of the $\infty$-category of $\calC$-valued presheaves on $X$: namely, the full subcategory spanned by those objects which satisfy a suitable descent condition. Here it becomes extremely convenient to work with $\infty$-topoi, rather than with topological spaces: if $X$ is an $\infty$-topos, then we can formulate the descent condition simply by saying that we have a functor $X^{op} \rightarrow \calC$ which preserves small limits. 

The theory of \S \ref{scurrel} can be used to provide a perfectly adequate theory of {\em ringed spaces} (or, more generally, {\it ringed $\infty$-topoi}). However, in classical algebraic geometry, a more prominent role is played by the theory of {\em locally ringed spaces}: that is, ringed spaces $(X, \calO_{X})$ 
for which the stalk $\calO_{X,x}$ is a local ring, for every point $x \in X$. To formulate an analogous locality condition on the $\infty$-category of $\calC$-valued sheaves, we need some additional structure on $\calC$. In \S \ref{scurgeo} we introduce a formalism for describing this additional structure, using the language of {\em geometries}. Roughly speaking, a geometry $\calG$ is a small
$\infty$-category with some additional data, which will enable us to develop a good theory of
{\em local} $\Ind(\calG^{op})$-valued sheaves on an arbitrary space (or $\infty$-topos) $\calX$.
We will refer to these local sheaves as {\it $\calG$-structures on $\calX$}; they can be organized into an $\infty$-category which we denote by $\Struct_{\calG}(\calX)$. 

The category $\LRingSpace$ of locally ringed spaces is not a full subcategory of the $\RingSpace$ category of ringed spaces: a morphism $f: (X, \calO_X) \rightarrow (Y, \calO_Y)$ in
$\RingSpace$ between objects of $\LRingSpace$ is a morphism of $\LRingSpace$ only if,
for every point $x$, the induced map on stalks $\calO_{Y,f(x)} \rightarrow \calO_{X,x}$ is
a {\em local} homomorphism. To describe the situation in more detail, it is convenient to introduce a bit of terminology. Let us say that a map of commutative rings $\beta: B \rightarrow C$ is {\it local}
if $\alpha$ carries noninvertible elements of $B$ to noninvertible elements of $C$. At the other extreme, we can consider the class of {\em localizing} homomorphisms $\alpha: A \rightarrow B$: that
is, homomorphisms which induce an isomorphism $A[S^{-1}] \simeq B$, where
$S$ is some collection of elements of $A$. An arbitrary ring homomorphism
$\gamma: A \rightarrow C$ admits an essentially unique factorization
$$ A \stackrel{\alpha}{\rightarrow} B \stackrel{\beta}{\rightarrow} C,$$
where $\alpha$ is localizing and $\beta$ is local: namely, we can take
$B = A[S^{-1}]$, where $S$ is the collection of all elements $a \in A$ such that $\gamma(a)$ is
invertible in $C$. We can summarize the situation by saying that the collections of
local and localizing morphisms form a factorization system on the category of commutative rings.
In fact, this is a general phenomenon: in \S \ref{factspe}, we will see that if
$\calG$ is a geometry and $\calX$ is an $\infty$-topos, then there is a canonical factorization
system on the $\infty$-category $\Struct_{\calG}(\calX)$ of $\calG$-structures on $\calX$, which depends functorially on $\calX$.

For every geometry $\calG$, there exists a {\em universal} example of a $\calG$-structure.
More precisely, in \S \ref{geo3} we will prove that there exists an $\infty$-topos $\calK$ and a $\calG$-structure on $\calK$ with the following universal property: for every $\infty$-topos $\calX$, the $\infty$-category $\Struct_{\calG}(\calX)$ of $\calG$-structures on $\calX$ is (canonically) equivalent to the
$\infty$-category $\Fun^{\ast}( \calK, \calX)$ of geometric morphisms from $\calX$ to $\calK$.
It follows that the entire theory of $\calG$-structures can be reformulated in terms of $\calK$, without ever making direct reference to $\calG$; in particular, the factorization systems on
$\Struct_{\calG}(\calX)$ determine some additional structure on $\calK$, which we refer to
as a {\it geometric structure}. This suggests the possibility of developing a still more general theory of ``structure sheaves'', based on $\infty$-topoi with geometric structure rather than on geometries. 
This additional generality costs us little so long as we confine our study to very formal aspects of the theory of $\calG$-structures, which we consider in \S \ref{geo4}. However, it does not seem to interact well with the scheme theory of \S \ref{secscheme}, so our attention in this paper will remain primarily focused on the theory of geometries.

\subsection{$\calC$-Valued Sheaves}\label{scurrel}

Let us begin by reviewing the classical notion of a sheaf of commutative rings
on a topological space $X$. Let $\Comm$ denote the category of commutative rings, and let
$\Shv_{\Set}(X)$ denote the (ordinary) category of sheaves of sets on $X$. A sheaf of commutative
rings $\calO_X$ on $X$ can be defined in many different ways:

\begin{itemize}
\item[$(a)$] We can view $\calO_{X}$ as a sheaf on $X$ taking values in the
category of commutative rings. From this point of view, $\calO_{X}$ is a functor
$\calU(X)^{op} \rightarrow \Comm$, which satisfies the usual sheaf axioms; here $\calU(X)$ denotes the partially ordered set of open subsets of $X$. 

\item[$(b)$] A sheaf of commutative rings $\calO_{X}$ on $X$ can be evaluated not only on open subsets of $X$, but on arbitrary sheaves of sets $\calF$ on $X$: namely, we can define $\calO_{X}(\calF)$ to be the commutative ring $\Hom_{\Shv_{\Set}(X)}( \calF, \calO_{X} )$. From this point of view, we can view
$\calO_{X}$ as representing a functor $\Shv_{\Set}(X)^{op} \rightarrow \Comm$. The advantage of
this point of view, when compared with $(a)$, is that the sheaf axiom is easier to state:
it merely asserts that the functor $\calO_{X}$ carries colimits in $\Shv_{\Set}(X)$ to limits in $\Comm$.

\item[$(c)$] Another point of view is to consider $\calO_{X}$ as a single object in the category
$\Shv_{\Set}(X)$ of sheaves of sets on $X$, equipped with some additional structure:
namely, addition and multiplication maps $\calO_{X} \times \calO_{X} \rightarrow \calO_{X}$
that satisfy the usual axioms defining the notion of a commutative ring. In other words,
$\calO_{X}$ is a {\em commutative ring object} in the category $\Shv_{\Set}(X)$.

\item[$(d)$] Given a commutative ring object $\calO_{X}$ in $\Shv_{\Set}(X)$, we can define
a functor $F: \Comm^{op} \rightarrow \Shv_{\Set}(X)$ as follows. For every commutative ring
$R$, we let $F(R)$ denote the ``sheaf of maps from $R$ to $\calO_{X}$''. More precisely,
for every open subset $U \subseteq X$, we let $F(R)(U)$ denote the set of ring homomorphisms
$\Hom_{\Comm}( R, \calO_{X}(U) )$. It is clear that $F(R)(U)$ depends functorially on $R$ and $U$,
and defines a functor $F: \Comm^{op} \rightarrow \Shv_{\Set}(X)$ as indicated above. By construction,
$F$ carries colimits of commutative rings to limits in the category $\Shv_{\Set}(X)$.

The functor $F$ determines the sheaf $\calO_{X}$, together with its ring structure. For example, we have
a canonical isomorphism (as sheaves of sets) $\calO_{X} \simeq F( \Z[x] )$. More generally,
$F( \Z[x_1, \ldots, x_n ])$ can be identified with the $n$th power $\calO_{X}^{n}$. 
We recover the commutative ring structure $\calO_{X}$ using the fact that
$F$ is a functor; for example, the multiplication map 
$\calO_{X} \times \calO_{X} \rightarrow \calO_{X}$
is obtained by applying $F$ to the ring homomorphism
$$ \Z[x] \mapsto \Z[x_1, x_2]$$
$$ x \mapsto x_1 x_2.$$
Consequently, we can {\em define} a sheaf of commutative rings on $X$ to be a 
limit-preserving functor $\Comm^{op} \rightarrow \Shv_{\Set}(X)$.

\item[$(e)$] In the preceding discussion, we did not need to use the entire category
of commutative rings; we can recover the ring structure on $\calO_{X}$ knowing only the
restriction of the functor $F$ to the category of {\em finitely generated} commutative rings
(in fact, it is sufficient to use polynomial rings; we will exploit this observation in \S \ref{app6}). Let $\Comm^{\fin}$ denote the category of finitely generated commutative rings (these are the same as finitely presented commutative rings, since the ring $\Z$ is Noetherian). We can then define a sheaf of commutative rings on $X$ to be a functor $F: (\Comm^{\fin})^{op} \rightarrow \Shv_{\Set}(X)$ which preserves finite limits. The equivalence of this definition with $(b)$ follows from the equivalence of categories $\Comm \simeq \Ind( \Comm^{\fin})$. 

\item[$(f)$] We can identify $\Comm^{\fin}$ with the
{\em opposite} of the category $\Aff$ of affine schemes of finite type over $\Z$. By abstract nonsense, any functor $F: \Aff \rightarrow \Shv_{\Set}(X)$ can be extended uniquely (up to unique isomorphism)
to a colimit-preserving functor $\pi^{\ast}: \calP(\Aff) \rightarrow \Shv_{\Set}(X)$, where $\calP(\Aff)$ denotes the category of presheaves of sets on $\Aff$. Moreover, $\pi^{\ast}$ preserves finite limits if and only if
$F$ preserves finite limits. In view of $(e)$, we obtain yet another definition of a sheaf
of commutative rings on $X$: namely, a functor $\calP(\Aff) \rightarrow \Shv_{\Set}(X)$ which preserves
small colimits and finite limits.

\item[$(g)$] In the situation of $(f)$, the functor $\pi^{\ast}$ admits a right adjoint $\pi_{\ast}$, and the
adjunction
$$ \Adjoint{ \pi^{\ast} }{ \calP(\Aff) }{ \Shv_{\Set}(X) }{ \pi_{\ast} }$$
is a {\it geometric morphism} of topoi from $\Shv_{\Set}(X)$ to $\calP(\Aff)$. We can summarize the situation as follows: the topos $\calP(\Aff)$ is a {\it classifying topos} for sheaves of commutative rings. For any topological space $X$, sheaves of commutative rings on $X$ can be identified with geometric morphisms from $\Shv_{\Set}(X)$ to $\calP(\Aff)$.
\end{itemize}

These definitions are all equivalent to one another, but are not always equally useful in practice. 
Our goal in this section is to adapt some of the above picture to an $\infty$-categorical setting: we will replace the topological space $X$ by an $\infty$-topos, and the category $\Comm$ of commutative rings with an arbitrary compactly generated $\infty$-category. We will focus on the analogues of the equivalences of $(a)$ through $(e)$, reserving the discussion of definitions $(f)$ and $(g)$ for \S \ref{geo3}. We begin with a review of some definitions from \cite{topoi}.

\begin{notation}\label{ubba}
Let $\calC$ and $\calD$ be $\infty$-categories. We let
$\LFun(\calC, \calD)$
denote the full subcategory of $\Fun(\calC, \calD)$ spanned by those functors which admit
right adjoints, and $\RFun(\calC, \calD) \subseteq \Fun(\calC, \calD)$ the full subcategory spanned by those functors which admit left adjoints. In view of Proposition \toposref{switcheroo}, the formation of adjoint functors gives rise to an equivalence
$$ \RFun( \calC, \calD) \simeq \LFun( \calD, \calC)^{op},$$
which is well-defined up to homotopy.
\end{notation}

\begin{remark}\label{hupp}
Let $\calC$ and $\calD$ be $\infty$-categories. Using the evident isomorphism
$$\RFun( \calC^{op}, \calD) \simeq \LFun( \calC, \calD^{op} )^{op},$$ we can formulate
the equivalence of Notation \ref{ubba} in the following more symmetric form:
$$ \RFun( \calC^{op}, \calD) \simeq \RFun( \calD^{op}, \calC).$$
\end{remark}

\begin{definition}\label{swinging}
Let $\calC$ be an arbitrary $\infty$-category, and let $\calX$ be an $\infty$-topos.
A {\it $\calC$-valued sheaf on $\calX$} is a functor
$\calX^{op} \rightarrow \calC$ which preserves small limits. We let
$\Shv_{\calC}(\calX)$ denote the full subcategory of $\Fun( \calX^{op}, \calC)$ spanned by
the $\calC$-valued sheaves on $\calX$. 
\end{definition}

\begin{remark}\label{hutt}
Let $\calC$ and $\calD$ be presentable $\infty$-categories. Using Corollary \toposref{adjointfunctor} and
Remark \toposref{afi}, we deduce that a functor $\calD^{op} \rightarrow \calC$ admits a left adjoint if and only if it preserves small limits. Consequently, for every $\infty$-topos $\calX$, we have $\Shv_{\calC}(\calX) = \RFun( \calX^{op}, \calC)$. 
\end{remark}

\begin{remark}\label{switchera}
Let $\calC$ be a presentable $\infty$-category and $\calX$ an $\infty$-topos. Then
$\Shv_{\calC}(\calX)$ can be identified with the tensor product $\calC \otimes \calX$
constructed in \S \monoidref{jurmit} (see Remark \monoidref{huger}). In particular, $\calC \otimes \calX$ is a presentable $\infty$-category.
\end{remark}


\begin{remark}\label{switcheru}
Let $\calG$ be a small $\infty$-category which admits finite limits, and 
let $j: \calG \rightarrow \Pro(\calG)$ denote the Yoneda embedding.

Let $\calX$ be an $\infty$-topos. Using
Remark \ref{hutt}, Proposition \toposref{sumatch}, and Proposition \toposref{intprop}, we
deduce that composition with $j$ induces an equivalence of $\infty$-categories
$$ \RFun( \Pro(\calG), \calX ) \rightarrow \Fun^{\lex}(\calG, \calX).$$
Combining this observation with Remark \ref{hupp}, we obtain a canonical equivalence
$$ \Fun^{\lex}(\calG, \calX) \simeq \RFun( \calX^{op}, \Ind(\calG^{op}) ) = \Shv_{\Ind(\calG^{op})}(\calX).$$
\end{remark}

\begin{remark}\label{tunner}
Let $\calG$ be a small $\infty$-category which admits finite limits and let $\calX$ an $\infty$-topos. The
$\infty$-categories $\Shv_{\Ind(\calG^{op})}(\calX)$ and $\Fun^{\lex}(\calG, \calX)$ are canonically equivalent. Objects of either $\infty$-category can be viewed as describing sheaves $\calF$ on $\calX$ with values in $\Ind(\calG^{op})$, but from different points of view. If we regard $\calF$ as an object of $\Shv_{\Ind(\calG^{op})}(\calX)$, then we are emphasizing the idea that $\calF$ can be evaluated on the ``opens'' $U \in \calX$, to obtain objects of $\Ind(\calG^{op})$. On the other hand, if we view $\calF$ as an object of $\Fun^{\lex}(\calG,\calX)$, then we are emphasizing the idea that $\calF$ can be viewed as an object (or several objects) of $\calX$, perhaps equipped with some additional structures.
\end{remark}

\begin{remark}\label{qwise}
Let $\calG$ be a small $\infty$-category which admits finite limits and let $\calC = \Ind( \calG^{op} )$. Remark \ref{tunner} implies that for a {\em fixed} $\infty$-topos $\calX$, the $\infty$-category
$\Fun^{\lex}(\calG, \calX)$ is equivalent to $\Shv_{\calC}(\calX)$. However, the first description is more evidently functorial in $\calX$.
To see this, let us suppose that we are given a geometric morphism
$$ \Adjoint{ \pi^{\ast} }{ \calX}{\calY}{\pi_{\ast}}$$
of $\infty$-topoi. Composition with $\pi^{\ast}$ induces a map
$\Shv_{\calC}(\calY) \rightarrow \Shv_{\calC}(\calX)$, and composition with
$\pi_{\ast}$ induces a functor $\Fun^{\lex}(\calG, \calY) \rightarrow \Fun^{\lex}(\calG, \calX)$.
We can view either of these operations as encoding the {\em pushforward} of
$\calC$-valued sheaves. We observe that the diagram
$$ \xymatrix{ \Shv_{\calC}( \calY) \ar[r]^{ \circ \pi^{\ast} } \ar[d]^{\sim} & \Shv_{\calC}( \calX) \ar[d]^{\sim} \\
\Fun^{\lex}(\calG, \calY) \ar[r]^{\pi_{\ast} \circ } & \Fun^{\lex}(\calG,\calX) }$$
commutes up to (canonical) homotopy. The bottom horizontal map admits a left adjoint $\Fun^{\lex}(\calG, \calX) \rightarrow \Fun^{\lex}(\calG,\calY)$, given by composition with the functor $\pi^{\ast}$. We will generally abuse notation by denoting this functor by $\pi^{\ast}$; we will refer to it informally as given by
{\it pullback of $\calG$-structures}. It is not so easy to describe this left adjoint directly in terms of $\Shv_{\calC}(\calX)$ and $\Shv_{\calC}(\calY)$. For example, the pushforward operation
$\Shv_{\calC}(\calY) \rightarrow \Shv_{\calC}(\calX)$ can be defined for {\em any} $\infty$-category $\calC$ (that is, we need not assume that $\calC$ is compactly generated), but generally does not admit a left adjoint.
\end{remark}

We can regard Remark \ref{tunner} as an $\infty$-categorical analogue of the equivalence between
the definitions $(b)$ and $(e)$ appearing earlier in this section (and the proof shows that both are equivalent to an $\infty$-categorical analogue of $(d)$). We close this section by discussing the relationship between $(a)$ and $(b)$. For this, we need to introduce a bit of notation.

\begin{definition}\label{ury}
Let $\calT$ be an essentially small $\infty$-category equipped with a Grothendieck topology, and $\calC$ another $\infty$-category. We will say that a functor $\calO: \calT^{op} \rightarrow \calC$ is a {\it $\calC$-valued sheaf on $\calT$} if the following condition is satisfied: for every object $X \in \calT$ and every covering sieve $\calT^{0}_{/X} \subseteq \calT_{/X}$, the composite map
$$ (\calT^{0}_{/X})^{\triangleleft} \subseteq (\calT_{/X})^{\triangleleft} \rightarrow
\calT \stackrel{\calO^{op}}{\rightarrow} \calC^{op}$$
is a colimit diagram in $\calC^{op}$.
We let $\Shv_{\calC}(\calT)$ denote the full subcategory of $\Fun(\calT^{op},\calC)$ spanned by the
$\calC$-valued sheaves on $\calT$.
\end{definition}

\begin{notation}
If $X$ is a topological space, then we let $\Shv_{\calC}(X)$ denote the $\infty$-category
$\Shv_{\calC}( \calU(X) )$, where $\calU(X)$ is the nerve of the partially ordered set of open subsets of $X$, endowed with its usual Grothendieck topology.
\end{notation}

\begin{example}
Let $\calT$ be a small $\infty$-category equipped with a Grothendieck topology, and let
$\SSet$ denote the $\infty$-category of spaces. Then the $\infty$-categories
$\Shv(\calT)$ (Definition \toposref{defsheaff}) and $\Shv_{\SSet}(\calT)$
(Definition \ref{ury}) coincide (as full subcategories of $\calP(\calT)$).
\end{example}

The notation and terminology of Definition \ref{ury} are potentially in conflict with the notation and terminology of Definition \ref{swinging}. However, little confusion should arise in view of the following compatibility result:

\begin{proposition}\label{selwus}
Let $\calT$ be a small $\infty$-category equipped with a Grothendieck topology. Let
$j: \calT \rightarrow \calP(\calT)$ denote the Yoneda embedding and $L: \calP(\calT) \rightarrow \Shv(\calT)$ a left adjoint to the inclusion. Let $\calC$ be an arbitrary $\infty$-category which admits small limits. Then composition with $L \circ j$ induces an equivalence of $\infty$-categories
$$ \Shv_{\calC}( \Shv(\calT) ) \rightarrow \Shv_{\calC}( \calT).$$
\end{proposition}

\begin{proof}
According to Theorem \toposref{charpresheaf}, composition with $j$ induces an equivalence of $\infty$-categories
$$ \Fun_{0}( \calP(\calT)^{op}, \calC) \rightarrow \Fun( \calT^{op}, \calC),$$
where $\Fun_{0}( \calP(\calT)^{op}, \calC)$ denotes the full subcategory of
$\Fun( \calP(\calT)^{op}, \calC)$ spanned by those functors which preserve small limits.
According to Proposition \toposref{unichar}, composition with $L$ induces a fully faithful embedding
$\Shv_{\calC}(\Shv(\calT)) \rightarrow \Fun_{0}(\calP(\calT)^{op}, \calC)$. The essential image of this embedding consists of those limit-preserving functors $F: \calP(\calT)^{op} \rightarrow \calC$
such that, for every $X \in \calT$ and every covering sieve $\calT^{0}_{/X} \subseteq \calT_{/X}$, the
induced map $F(jX) \rightarrow F(Y)$ is an equivalence in $\calC$, where $Y$ is the subobject of $jX$ corresponding to the sieve $\calT^{0}_{/X}$. Unwinding the definitions, this translates into the condition that the composition
$$ (\calT^{0}_{/X})^{\triangleright} \subseteq (\calT_{/X})^{\triangleright}
\rightarrow \calT \stackrel{j}{\rightarrow} \calP(\calT) \stackrel{F}{\rightarrow} \calC^{op}$$
is a colimit diagram. It follows that the composition
$$ \Shv_{\calC}(\Shv(\calT)) \rightarrow \Fun_0( \calP(\calT)^{op}, \calC)
\rightarrow \Fun(\calT^{op}, \calC)$$
is fully faithful, and its essential image is the full subcategory $\Shv_{\calC}(\calT)$.
\end{proof}

\subsection{Geometries}\label{scurgeo}

Let $X$ be a topological space. If $\calO_X$ is a sheaf of commutative rings on $X$, then it makes sense to ask if $\calO_X$ is {\it local}: that is, if each stalk $\calO_{X,x}$ is a local commutative ring.
Suppose instead that $\calO_X$ takes values in some other category (or $\infty$-category) $\calC$: is there an analogous condition of locality that we can impose? Of course, the answer to this question depends on $\calC$. We might formulate the question better as follows: what features of the category
$\Comm$ are required for introducing the subcategory of local $\Comm$-valued sheaves on $X$?

We first observe that $\Comm$ is {\em compactly generated}; as we saw in \S \ref{scurrel}, this allows
us to think of a $\Comm$-valued sheaf on $X$ as a functor $\Aff \rightarrow \Shv_{\Set}(X)$, where
$\Aff = (\Comm^{\fin})^{op}$ is the category of affine schemes of finite type over $\Z$. We can now reformulate our question once again: what features of the category $\Aff$ are required to define the notion of a local $\Comm$-valued sheaf? Our answer to this question is that $\Aff$
(or rather its nerve) has the structure of a {\em geometry}, in the sense of Definition \ref{psyob}
(see Example \ref{summ}). The goal of this section is to introduce the definition of a geometry $\calG$
and describe the associated theory of (local) $\Ind(\calG^{op})$-valued sheaves. We begin with some preliminaries.

\begin{definition}\label{stubb}
Let $\calG$ be an $\infty$-category. An {\it admissibility structure} on $\calG$ consists of the following data:
\begin{itemize}
\item[$(1)$] A subcategory $\calG^{\adm} \subseteq \calG$, containing every object of $\calG$.
Morphisms of $\calG$ which belong to $\calG^{\adm}$ will be called {\it admissible} morphisms in $\calG$.
\item[$(2)$] A Grothendieck topology on $\calG^{\adm}$.
\end{itemize}
These data are required to satisfy the following conditions:
\begin{itemize}

\item[$(i)$] Let $f: U \rightarrow X$ be an admissible morphism in $\calG$, and
$g: X' \rightarrow X$ any morphism. Then there exists a pullback diagram
$$ \xymatrix{ U' \ar[d]^{f'} \ar[r] & U \ar[d]^{f} \\
X' \ar[r]^{g} & X, }$$
where $f'$ is admissible.

\item[$(ii)$] Suppose given a commutative triangle
$$ \xymatrix{ & Y \ar[dr]^{g} \\
X \ar[ur]^{f} \ar[rr]^{h} & & Z }$$
in $\calG$, where $g$ and $h$ are admissible. Then $f$ is admissible.

\item[$(iii)$] Every retract of an admissible morphism of $\calG$ is admissible.
\end{itemize}
\end{definition}

\begin{remark}\label{switchtop}
Let $\calG$ be an $\infty$-category endowed with an admissibility structure $\calG^{\adm} \subseteq \calG$. In view of part $(ii)$ of Definition \ref{stubb}, for every object $X \in \calG$, we can identify
$\calG^{\adm}_{/X}$ with the full subcategory of $\calG_{/X}$ spanned by the admissible morphisms
$U \rightarrow X$.
\end{remark}

\begin{remark}\label{sagewise}
Let $\calG$ be an $\infty$-category equipped with an admissibility structure.
It follows from condition $(ii)$ of Definition \ref{stubb} that any section of an admissible morphism
is again admissible. In particular, if $U \rightarrow X$ is admissible, then the
diagonal map $\delta: U \rightarrow U \times_{X} U$ is a section of the projection onto
the first factor, so that $\delta$ is again admissible.
\end{remark}

\begin{remark}
Let $\calG$ be an $\infty$-category. Every admissibility structure $\calG^{\adm} \subseteq \calG$ determines a Grothendieck topology on $\calG$: the Grothendieck topology generated by
the topology on $\calG^{\adm}$. In other words, we consider a sieve $\calG^{0}_{/X} \subseteq \calG_{/X}$ on an object $X \in \calG$ to be covering if intersection $\calG^{\adm}_{/X} \cap \calG^{0}_{/X}
\subseteq \calG^{\adm}_{/X}$ is a covering sieve on $X \in \calG^{\adm}$. This Grothendieck topology on $\calG$ determines the original Grothendieck topology on $\calG^{\adm}$: a sieve
$\calG^{\adm,0}_{/X} \subseteq \calG^{\adm}_{/X}$ is covering if and only if it generates a covering sieve on $X$ in $\calG$. In other words, in Definition \ref{stubb}, we can think of the Grothendieck topology as a topology on $\calG$, rather than $\calG^{\adm}$. However, if we adopt this point of view, then an additional axiom is required: the Grothendieck topology on $\calG$ must be {\it generated} by admissible morphisms. That is, every covering sieve $\calG^{0}_{/X} \subseteq \calG_{/X}$ contains a 
collection of admissible morphisms $U_{\alpha} \rightarrow X$ which cover $X$.
\end{remark}

\begin{definition}\label{psyob}
A {\it geometry} consists of the following data:
\begin{itemize}
\item[$(1)$] An essentially small $\infty$-category $\calG$ which admits finite limits and is idempotent complete.
\item[$(2)$] An admissibility structure on $\calG$.
\end{itemize}
\end{definition}

We will generally abuse terminology by identifying a geometry with its underlying $\infty$-category $\calG$.

\begin{definition}
Let $\calG$ and $\calG'$ be geometries. We will say that a functor $f: \calG \rightarrow \calG'$
is a {\it transformation of geometries} if the following conditions are satisfied:
\begin{itemize}
\item[$(1)$] The functor $f$ preserves finite limits.
\item[$(2)$] The functor $f$ carries admissible morphisms of $\calG$ to admissible morphisms of $\calG'$.
\item[$(3)$] For every admissible cover $\{ U_{\alpha} \rightarrow X\}$ of an object $X \in \calG$, the
collection of morphisms $\{ f(U_{\alpha}) \rightarrow f(X) \}$ is an admissible cover of $f(X) \in \calG'$.
\end{itemize}
\end{definition}

\begin{remark}\label{gener}
Let $\calG$ be an idempotent complete $\infty$-category which admits finite limits.
We will say that an admissibility structure $\calG^{\adm}$ on $\calG$ is a {\it refinement} of
another admissibility structure $\calG^{\adm'}$ on $\calG$ if the identity functor $\id_{\calG}$
is a transformation of geometries $( \calG, \calG^{\adm'}) \rightarrow (\calG, \calG^{\adm})$. In this case, we will also say that $\calG^{\adm}$ is {\it finer than} $\calG^{\adm'}$ or that
$\calG^{\adm'}$ is {\it coarser than} $\calG^{\adm}$. 

Given any collection $S$ of morphisms of $\calG$, and any collection $T$ of
sets of morphisms $\{ f_{\alpha}: U_{\alpha} \rightarrow X \}$ belonging to $S$, there
is a coarsest admissibility structure on $\calG$ such that every element of $S$ is admissible, and
every element of $T$ generates a covering sieve. We will refer to this admissibility structure as the
{\it admissibility structure generated by $S$ and $T$}.

As a special case, suppose that $\calG$ is a geometry, $\calG'$ another idempotent complete $\infty$-category which admits finite limits, and that $f: \calG \rightarrow \calG'$ is a functor which preserves finite limits. Then there exists a coarsest admissibility structure on $\calG'$ such that $f$ is a transformation of geometries. We will refer to this admissibility structure on $\calG'$ as {\it the admissibility structure generated by $f$}.
\end{remark}

\begin{definition}\label{psyab}
Let $\calG$ be a geometry and $\calX$ an $\infty$-topos. A {\it $\calG$-structure} on $\calX$
is a left exact functor $\calO: \calG \rightarrow \calX$ with the following property: for every collection of admissible morphisms $\{ U_{\alpha} \rightarrow X \}$ in $\calG$ which generates a covering sieve on $X$, the induced map $\coprod_{\alpha} \calO(U_{\alpha}) \rightarrow \calO(X)$ is an effective epimorphism in $\calX$. We let $\Struct_{\calG}(\calX)$ denote the full subcategory of
$\Fun(\calG, \calX)$ spanned by the $\calG$-structures on $\calX$. 

Given a pair of $\calG$-structures $\calO, \calO': \calG \rightarrow \calX$, we will say that a natural transformation $\alpha: \calO \rightarrow \calO'$ is a {\it local transformation $\calG$-structures} if,
for every admissible morphism $U \rightarrow X$ in $\calG$, the induced diagram
$$ \xymatrix{ \calO(U) \ar[r] \ar[d] & \calO'(U) \ar[d] \\
\calO(X) \ar[r] & \calO'(X) }$$
is a pullback square in $\calX$. We let $\Struct_{\calG}^{\loc}(\calX)$ denote the subcategory of
$\Struct_{\calG}(\calX)$ spanned by the local transformations of $\calG$-structures.
\end{definition}


\begin{remark}\label{henhun}
Let $\calG$ be a geometry, $\calX$ an $\infty$-topos, and $\calO: \calG \rightarrow \calX$ a functor. The condition that $\calO$ define a $\calG$-structure on $\calX$ can be tested ``stalkwise'', in the following sense. Suppose that $\calX$ has enough points (see Remark \toposref{notenough}). Then
$\calO$ is a $\calG$-structure on $\calX$ if and only if, for every point $x^{\ast}: \calX \rightarrow \SSet$, the ``stalk'' $\calO_{x} = x^{\ast} \circ \calO$ is a $\calG$-structure on $\SSet$.

Similarly, if $\calX$ has enough points, then a morphism $\alpha: \calO \rightarrow \calO'$ in
$\Struct_{\calG}(\calX)$ belongs to $\Struct^{\loc}_{\calG}(\calX)$ if and only if, for every point
$x^{\ast}: \calX \rightarrow \SSet$, the induced map on stalks $\alpha_{x}: \calO_{x} \rightarrow \calO'_{x}$ belongs to $\Struct_{\calG}^{\loc}(\SSet)$. 
\end{remark}

\begin{definition}\label{disker}
We will say that a geometry $\calG$ is {\it discrete} if the following conditions are satisfied:
\begin{itemize}
\item[$(1)$] The admissible morphisms in $\calG$ are precisely the equivalences.
\item[$(2)$] The Grothendieck topology on $\calG$ is trivial: that is, a sieve $\calG^{0}_{/X} \subseteq \calG_{/X}$ on an object $X \in \calG$ is a covering sieve if and only if $\calG^{0}_{/X} = \calG_{/X}$.
\end{itemize}
\end{definition}

\begin{remark}
Let $\calG$ be an essentially small $\infty$-category which admits finite limits. Then requirements $(1)$ and $(2)$ of Definition \ref{disker} endow $\calG$ with the structure of a (discrete) geometry.
\end{remark}

\begin{remark}
If $\calG$ is a discrete geometry and $\calX$ is an $\infty$-topos, then we have equivalences
$$ \Struct_{\calG}^{\loc}(\calX) = \Struct_{\calG}(\calX) = \Fun^{\lex}( \calG, \calX)
\simeq \Shv_{ \Ind(\calG^{op})}(\calX).$$
In other words, the theory of $\calG$-structures is equivalent to the theory of
$\Ind(\calG^{op})$-valued sheaves studied in \S \ref{scurrel}.
\end{remark}

\begin{example}\label{summ}
Let $\Comm^{\fin}$ denote the category of all finitely generated commutative rings. 
Let $\calG_{\Zar} = \Nerve(\Comm^{\fin})^{op}$ be the opposite of the nerve of $\Comm^{\fin}$; we can identify $\calG_{\Zar}$ with the ($\infty$)-category of affine schemes of finite type over $\Z$. To emphasize this identification, for every finitely generated commutative ring $A$, we let $\SSpec A$ denote the associated object of $\calG_{\Zar}$. We regard $\calG_{\Zar}$ as a geometry via the following prescription:
\begin{itemize}
\item[$(1)$] A morphism $\SSpec A \rightarrow \SSpec B$ is admissible if and only if it induces an isomorphism $B[ \frac{1}{b} ] \rightarrow A$, for some element $b \in B$.
\item[$(2)$] A collection of admissible morphisms $\{ \SSpec A[ \frac{1}{a_{\alpha}}] \rightarrow \SSpec A \}$ generates a covering sieve on $\SSpec A$ if and only if it is a covering in the sense of classical algebraic geometry: in other words, if and only if the set $\{ a_{\alpha} \} \subseteq A$ generates the unit ideal in $A$.
\end{itemize}

If $X$ is a topological space and $\calX = \Shv(X)$ is the $\infty$-category of sheaves (of spaces) on $\calX$, then we can identify $\calG_{\Zar}$-structures on $\calX$ with
sheaves of commutative rings $\calO$ on the topological space $X$ which are {\em local} in the sense that for every point $x \in X$, the stalk $\calO_{x}$ is a local commutative ring (Remark \ref{testlo}).
More generally, we can think of a $\calG_{\Zar}$-structure on an arbitrary $\infty$-topos $\calX$ as a {\it sheaf of local commutative rings on $\calX$}. We will study this example in more detail in
\S \ref{exzar}.
\end{example}

Example \ref{summ} is the prototype which motivates the theory of geometries that we will develop in this section. Although we will meet many other examples of geometries in this and subsequent papers, we will primarily be interested in mild variations on Example \ref{summ}. 

We will devote the remainder of this section to the proof of a rather technical result concerning
admissibility structures which will be needed later in this paper.

\begin{proposition}\label{prespaz}
Let $\calT \rightarrow \Delta^1$ be an $($essentially small$)$ correspondence between $\infty$-categories
$\calT_0 = \calT \times_{ \Delta^1} \{0\}$ to $\calT_1 = \calT \times_{ \Delta^1} \{1\}$. Assume
that $\calT_0$ and $\calT_1$ are equipped with admissibility structures, and that $\calT$ satisfies the following conditions:
\begin{itemize}
\item[$(i)$] For every admissible morphism $u_1: U_1 \rightarrow X_1$ in $\calT_1$, every
object $X_0 \in \calT_0$, and every morphism $X_0 \rightarrow X_1$ in $\calT$, there
exists a pullback diagram
$$ \xymatrix{ U_0 \ar[r] \ar[d]^{u_0} & U_1 \ar[d]^{u_1} \\
X_0 \ar[r] & X_1, }$$
in $\calM$, where $u_0$ is an admissible morphism in $\calT_0$.

\item[$(ii)$] Let $\{ U_{\alpha} \rightarrow X_1 \}$ be a collection of admissible morphisms
in $\calT_1$ which generates a covering sieve on $X_1$, and let $X_0 \rightarrow X_1$ be
an arbitrary morphism in $\calT$, where $X_0 \in \calT_0$. Then the induced maps
$\{ U_{\alpha} \times_{X_1} X_0 \rightarrow X_0 \}$ generate a covering sieve on $X_0$.
\end{itemize}

Let $\calX$ be an $\infty$-topos. We will say that a functor $\calO: \calT_{i} \rightarrow \calX$
is {\it local} if the following conditions are satisfied:
\begin{itemize}
\item[$(a)$] For every pullback diagram
$$ \xymatrix{ U' \ar[r] \ar[d] & U \ar[d] \\
X' \ar[r] & X }$$
in $\calT_i$ such that the vertical morphisms are admissible, the induced diagram
$$ \xymatrix{ \calO(U') \ar[r] \ar[d] & \calO(U) \ar[d] \\
\calO(X') \ar[r] & \calO(X) }$$
is a pullback square in $\calX$.
\item[$(b)$] For every collection of morphisms $\{ U_{\alpha} \rightarrow X \}$ in $\calT_i$ which generates a covering sieve on $X$, the induced map
$$ \coprod_{\alpha} \calO(U_{\alpha}) \rightarrow \calO(X)$$
is an effective epimorphism in $\calX$.
\end{itemize}

We will say that a natural transformation $\alpha: \calO \rightarrow \calO'$ of
local functors $\calO, \calO': \calT_i \rightarrow \calX$ is {\it local} if the following condition is satisfied:

\begin{itemize}
\item[$(c)$] For every admissible morphism $U \rightarrow X$ in $\calT_i$, the induced diagram
$$ \xymatrix{ \calO(U) \ar[r] \ar[d] & \calO'(U) \ar[d] \\
\calO(X) \ar[r] & \calO'(X) }$$
is a pullback square in $\calX$.
\end{itemize}

Then: 

\begin{itemize}
\item[$(1)$] Let $F: \Fun( \calT_0, \calX) \rightarrow \Fun( \calT_1, \calX)$ be the functor
defined by left Kan extension along $\calT$. Then $F$ carries local objects of
$\Fun(\calT_0, \calX)$ to local objects of $\Fun( \calT_1, \calX)$.

\item[$(2)$] The functor $F$ carries local morphisms between local objects of $\Fun( \calT_0, \calX)$ to local morphisms of $\Fun( \calT_1, \calX)$.

\item[$(3)$] Let $\calO: \calT \rightarrow \calX$ be a functor such that the restrictions
$\calO_0 = \calO | \calT_0$, $\calO_1 = \calO| \calT_1$ are local. Then $\calO$ induces a local transformation
$F(\calO_0) \rightarrow \calO_1$ if and only if the following condition is satisfied:
\begin{itemize}
\item[$(\star)$] For every pullback diagram
$$ \xymatrix{ U_0 \ar[r] \ar[d] & U_1 \ar[d] \\
X_0 \ar[r] & X_1 }$$
in $\calT$ of the type appearing in $(i)$, the induced diagram
$$ \xymatrix{ \calO(U_0) \ar[r] \ar[d] & \calO(U_1) \ar[d] \\
\calO(X_0) \ar[r] & \calO(X_1)} $$
is a pullback square in $\calX$.
\end{itemize}
\end{itemize}
\end{proposition}

The proof of Proposition \ref{prespaz} is elementary but somewhat tedious. We first establish
the following lemma:

\begin{lemma}\label{gooduse}
Let $\calX$ be an $\infty$-topos, $\calC$ a small $\infty$-category, and 
$\alpha: \calF \rightarrow \calF'$ a natural transformation between functors $\calF, \calF': \calC \rightarrow \calX$. Suppose that, for every morphism $C \rightarrow D$ in $\calC$, the induced diagram
$$ \xymatrix{ \calF(C) \ar[r]^{\alpha(C)} \ar[d] & \calF'(C) \ar[d] \\
\calF(D) \ar[r]^{ \alpha(D) } & \calF'(D) }$$
is a pullback square. Then:
\begin{itemize}
\item[$(1)$] For every object $C \in \calC$, the diagram
$$ \xymatrix{ \calF(C) \ar[r]^{\alpha(C)} \ar[d] & \calF'(C) \ar[d] \\
\colim \calF \ar[r] & \colim \calF' }$$
is a pullback diagram in $\calX$.

\item[$(2)$] Suppose given a diagram $\sigma:$
$$ \xymatrix{ \colim \calF \ar[r] \ar[d] & \colim \calF' \ar[d] \\
X \ar[r] & X' }$$
in $\calX$ such that for each $C \in \calC$, the induced diagram
$$ \xymatrix{ \calF(C) \ar[r] \ar[d] & \calF'(C) \ar[d] \\
X \ar[r] & X' }$$
is a pullback square. Then $\sigma$ is a pullback square.
\end{itemize}
\end{lemma}

\begin{proof}
Assertion $(1)$ follows immediately from the characterization of $\infty$-topoi given in
Theorem \toposref{charleschar}. To prove $(2)$, we observe that for each $C \in \calC$ we
have a commutative diagram
$$ \xymatrix{ \calF(C) \ar[r] \ar[d] & \calF'(C) \ar[d] \\
(\colim \calF') \times_{X'} X \ar[r] \ar[d] & \colim \calF' \ar[d] \\
X \ar[r] & X'. }$$
The lower and outer squares are pullback diagrams, so the upper square is a pullback diagram
as well. Passing to the colimit over $C$ (and using the fact that colimits in $\calX$ are universal),
we conclude that that the canonical map $\colim \calF \rightarrow ( \colim \calF') \times_{X'} X$
is an equivalence, so that $\sigma$ is a pullback diagram as desired.
\end{proof}

\begin{proof}[Proof of Proposition \ref{prespaz}]
We first prove $(1)$. Given a functor $\calO_0: \calT_0 \rightarrow \calX$, the functor
$F(\calO_0)$ can be obtained by first choosing a functor
$\calO: \calT \rightarrow \calX$ which is a left Kan extension of
$\calO_0 = \calO | \calT_0$, and then setting $F(\calO_0) = \calO | \calT_1$. Let
us therefore suppose that $\calO$ satisfies the following condition:

\begin{itemize}
\item[$(\ast)$] The functor $\calO: \calT \rightarrow \calX$ is a left Kan extension of
$\calO_0 = \calO | \calT_0$, and $\calO_0$ is local on $\calX$.
\end{itemize}

We wish to show that, in this case, $\calO_1 = \calO | \calT_1$ is a $\calT_1$-structure on $\calX$.
We must verify that two conditions are met:
\begin{itemize}
\item The functor $\calO_1$ preserves pullbacks by admissible morphisms. As a first
step, we establish that $\calO$ has the following property:
\begin{itemize}
\item[$(\ast')$] Suppose given a pullback diagram
$$ \xymatrix{ U_0 \ar[r] \ar[d] & U_1 \ar[d] \\
X_0 \ar[r] & X_1 }$$ in $\calT$ as in $(i)$. Then 
the associated diagram
$$ \xymatrix{ \calO(U_0) \ar[r] \ar[d] & \calO( U_1) \ar[d] \\
\calO(X_0) \ar[r] & \calO(X_1) }$$
is a pullback square in $\calX$.
\end{itemize}
To prove $(\ast')$, we let
$\phi: \calT_0^{/X_1} \rightarrow \calT_0^{/ U_1}$ be the functor given by pullback along $u$.
Let $f$ denote the composition 
$$ \calT_0^{/X_1} \rightarrow \calT_0 \stackrel{\calO_0}{\rightarrow} \calX$$
and $f'$ the composition
$$ \calT_0^{/U_1} \rightarrow \calT_0 \stackrel{\calO_0}{\rightarrow} \calX.$$
The functor $\calO_0$ determines a Cartesian transformation
$f' \circ \phi \rightarrow f$ in the $\infty$-category $\Fun( \calT_0^{/X_1}, \calX)$. 
We have a commutative diagram
$$ \xymatrix{ \calO(U_0) \ar[r] \ar[d] & \colim(f' \circ \phi) \ar[d] \ar[r]^{\beta} & 
\colim(f') \ar[r]^{\beta'} & \calO(U_1) \ar[d] \\
\calO(X_0) \ar[r] & \colim(f) \ar[rr]^{\beta''} & & \calO(X_1). }$$
Lemma \ref{gooduse} implies that the left square is a pullback. It will therefore suffice to show
that $\beta$, $\beta'$, and $\beta''$ are equivalences. For $\beta'$ and $\beta''$, this follows immediately from assumption $(\ast)$. To show that $\beta$ is an equivalence, it will suffice to show
that $\phi$ is cofinal. This follows from the observation that $\phi$ admits a left adjoint, 
given by composition with $u$.

Now consider a pullback diagram
$$ \xymatrix{ U_1 \ar[r] \ar[d] & U'_1 \ar[d] \\
X_1 \ar[r] & X'_1 }$$
in $\calT_1$, where the vertical maps are admissible. Let $f$, $f'$, and $\phi$ be as above,
so that we have a commutative diagram
$$ \xymatrix{ \colim(f' \circ \phi) \ar[r] \ar[d] & \calO_1(U_1) \ar[r] \ar[d] & \calO_1( U'_1) \ar[d] \\
\colim(f) \ar[r] & \calO_1(X_1) \ar[r] & \calO_1( X'_1). }$$
We wish to show that the right square is a pullback. Since the left horizontal arrows are equivalences, it will suffice to show that the outer square is a pullback. By Lemma \ref{gooduse}, it will suffice to show that
for every $X_0 \in \calT_0$ and every map $X_0 \rightarrow X_1$ in $\calT$, 
the induced diagram
$$ \xymatrix{ \calO( X_0 \times_{ X_1} U_1 ) \ar[r] \ar[d] & \calO( U'_1) \ar[d] \\
\calO( X_0) \ar[r] & \calO( X'_1) }$$
is a pullback square in $\calX$, which follows immediately from $(\ast')$.

\item Let $\{ U_{\alpha} \rightarrow X_1 \}$ be an admissible covering of an object $X_1 \in \calT_1$. We
wish to show that the induced map $\coprod_{\alpha} \calO(U_{\alpha}) \rightarrow \calO(X_1)$ is
an effective epimorphism in $\calX$. Let $f: \calT_0^{/X_1} \rightarrow \calX$ be as above, so that
$\calO(X_1) \simeq \colim(f)$. It will therefore suffice to show that, for every object $X_0 \in \calT_0$
and every map $X_0 \rightarrow X_1$ in $\calT$, the induced map
$$ u: \coprod_{ \alpha} ( \calO(U_{\alpha}) \times_{ \calO(X_1) } \calO(X_0) ) \rightarrow \calO(X_0)$$
is an effective epimorphism. For each $\alpha$, set $V_{\alpha} = U_{\alpha} \times_{X_1} X_0$.
In view of $(\ast')$, we can identify the left hand side with
$\coprod_{\alpha} \calO_0( V_{\alpha})$. Since $\calO_0$ is local and
and the maps $\{ V_{\alpha} \rightarrow X_1 \}$ form an admissible covering (by $(ii)$), 
the map $u$ is an effective epimorphism as desired.
\end{itemize}

We now prove $(3)$. Let $\calO: \calT \rightarrow \calX$ satisfy the hypotheses of $(3)$.
Set $\calO_0 = \calO | \calT_0$, and let $\calO': \calT \rightarrow \calX$ be a left Kan extension
of $\calO_0$, so that $\calO'_1 = \calO' | \calT_1 \simeq F(\calO_0)$ is local by virtue of
$(1)$. The identity transformation from $\calO_0$ to itself extends, in an essentially unique fashion, to a natural transformation $\beta: \calO' \rightarrow \calO$. We wish to show:
\begin{itemize}
\item[$(3')$] If $\beta$ induces a local map $\beta_1: \calO'_{1} \rightarrow \calO_1$, then for every pullback diagram
$$ \xymatrix{ U_0 \ar[r] \ar[d] & U_1 \ar[d] \\
X_0 \ar[r] & X_1 }$$
as in $(i)$, the outer square in the diagram
$$ \xymatrix{ \calO_0(U_0) \ar[r] \ar[d] & \calO'_1(U_1) \ar[r] \ar[d] & \calO_1(U_1) \ar[d] \\
\calO_0(X_0) \ar[r] & \calO'_1(X_1) \ar[r] & \calO_1(X_1) }$$
is a pullback diagram. To prove this, we observe that the left square is a pullback by $(\ast')$, and
the right square by our hypothesis that $\beta_1$ is local.

\item[$(3'')$] Suppose that, for every pullback diagram
$$ \xymatrix{ U_0 \ar[r] \ar[d] & U_1 \ar[d] \\
X_0 \ar[r] & X_1 }$$
as in $(i)$, the induced diagram
$$ \xymatrix{ \calO(U_0) \ar[r] \ar[d] & \calO_{1}( U_1) \ar[d] \\
\calO(X_0) \ar[r] & \calO(X_1) }$$
is a pullback square in $\calX$. We wish to prove that $\beta_1$ is local. Fix an admissible
morphism $U_1 \rightarrow X_1$ in $\calT_1$, and let
$f: \calT_0^{/X_1} \rightarrow \calX$, $f': \calT_0^{/U_1} \rightarrow \calX$, and
$\phi: \calT_0^{/X_1} \rightarrow \calT_0^{/U_1}$ be defined as above, so that we have a commutative diagram
$$ \xymatrix{ \colim(f' \circ \phi) \ar[r] \ar[d] & \calO'(U_1) \ar[r] \ar[d] & \calO(U_1) \ar[d] \\
\colim(f) \ar[r] & \calO'(X_1) \ar[r] & \calO(X_1). }$$
Since the left horizontal arrows are equivalences, it suffices to show that the outer square is a pullback.
This follows immediately from Lemma \ref{gooduse}.
\end{itemize}

It remains to prove $(2)$. Let $\alpha: \calO_0 \rightarrow \calO'_0$ be a local morphism between local objects of $\Fun( \calT_0, \calX)$. We wish to prove that $F(\alpha)$ is a local morphism in
$\Fun( \calT_1, \calX)$. Let $\calO, \calO': \calT_0 \rightarrow \calX$ be left Kan extensions of
$\calO_0$ and $\calO'_0$, respectively, and let $\overline{\alpha}: \calO \rightarrow \calO'$
be an extension of $\alpha$ (which is uniquely determined up to equivalence).
We wish to show that $\overline{\alpha}$ induces a local map from $\calO| \calT_1$
to $\calO' | \calT_1$. By virtue of $(3)$, it will suffice to show that the outer square in the diagram
$$ \xymatrix{ \calO_0(U_0) \ar[r] \ar[d] & \calO'_0( U_0) \ar[r] \ar[d] & \calO'(U_1) \ar[d] \\
\calO_0(X_0) \ar[r] & \calO'_0(X_0) \ar[r] & \calO'_0(X_1) }$$
is a pullback square in $\calX$, whenever
$$ \xymatrix{ U_0 \ar[r] \ar[d] & U_1 \ar[d] \\
X_0 \ar[r] & X_1 }$$
is as in the statement of $(i)$. The left square is a pullback diagram in virtue of the assumption 
that $\alpha$ is local, and the right square is a pullback diagram 
by $(\ast')$.
\end{proof}

\subsection{The Factorization System on $\Struct_{\calG}(\calX)$}\label{factspe}

Let $f: A \rightarrow B$ be a homomorphism of commutative rings. Then $f$ factors as a composition
$$ A \stackrel{f'}{\rightarrow} A[S^{-1}] \stackrel{f''}{\rightarrow} B$$
where $S$ is the collection of elements $a \in A$ such that $f(a) \in B$ is invertible, and
$f''$ is {\em local} in the sense that it carries noninvertible elements of $A[S^{-1}]$ to noninvertible elements of $B$.
This factorization is unique up to (unique) isomorphism and depends functorially on $f$, so that
we obtain a factorization system on the category $\Comm$ of commutative rings. More generally, if $f: \calA \rightarrow \calB$ is a map of {\em sheaves} of commutative rings
on a space $X$, then $f$ admits an analogous factorization
$$ \calA \stackrel{f'}{\rightarrow} \calA' \stackrel{f''}{\rightarrow} \calB,$$
which reduces to the previous factorization after passing to stalks at any point $x \in X$.
Our goal in this section is to show that there is an analogous factorization system on
the $\infty$-category $\Struct_{\calG}(\calX)$, where $\calG$ is any geometry and
$\calX$ any $\infty$-topos (to recover the original situation, we can take
$\calG$ to be the geometry $\calG_{\Zar}$ of Example \ref{summ}, and $\calX$ to be
the $\infty$-topos $\Shv(X)$ of sheaves of spaces on $X$). More precisely, we will prove the following result:

\begin{theorem}\label{gen}
\begin{itemize}
\item[$(1)$]
Let $\calG$ be a geometry, and $\calX$ an $\infty$-topos. Then there exists a factorization system $(S^{\calX}_L, S^{\calX}_R)$ on $\Struct_{\calG}(\calX)$, where
$S^{\calX}_R$ is the collection morphisms belonging to $\Struct^{\loc}_{\calG}(\calX)$. 
\item[$(2)$] The factorization system of $(1)$ depends functorially on $\calX$. In other words, 
given any geometric morphism of $\infty$-topoi $\pi^{\ast}: \calX \rightarrow \calY$, composition
with $\pi^{\ast}$ induces a functor $\Struct_{\calG}(\calX) \rightarrow \Struct_{\calG}(\calY)$ which
carries $S^{\calX}_{L}$ to $S^{\calY}_{L}$ and $S^{\calX}_{R}$ to $S^{\calY}_{R}$.
\end{itemize}
\end{theorem}

\begin{remark}
Let $\calG$ and $\calX$ be as in Theorem \ref{gen}, let $\alpha: \calO' \rightarrow \calO''$ be a natural transformation of $\calG$-structures on $\calX$ with associated factorization
$\calO' \stackrel{\alpha^{R}}{\rightarrow} \calO \stackrel{\alpha^{L}}{\rightarrow} \calO''.$
Applying assertion $(2)$ in the case where $\pi^{\ast}: \calX \rightarrow \SSet$ gives a point $x$ of $\calX$, we deduce that the induced transformation on stalks
$$ \calO'_{x} \rightarrow \calO_x \rightarrow \calO''_{x}$$
is the associated factorization of $\alpha_x: \calO'_{x} \rightarrow \calO''_{x}$ in
$\Struct_{\calG}(\SSet) \simeq \Ind( \calG^{op} )$. If $\calX$ has enough points, then the converse to this assertion holds as well.
\end{remark}

\begin{remark}
The functorial dependence of the factorization system $(S^{\calX}_L, S^{\calX}_{R})$ on the geometry $\calG$ is more subtle; we will return to this point in \S \ref{relspec}.
\end{remark}

In order to prove Theorem \ref{gen}, we will need a good understanding of the
$\infty$-category $\Pro(\calG)$ of pro-objects of $\calG$. We begin by introducing some terminology.

\begin{definition}\label{jun}
Let $\calG$ be a geometry, and let $j: \calG \rightarrow \Pro(\calG)$ denote the Yoneda embedding. We will say that a morphism $f: U \rightarrow X$ in the $\infty$-category $\Pro(\calG)$ is {\it proadmissible} if there exists a small filtered diagram $p: \calI \rightarrow \Fun( \Delta^1, \calG)$ such that each $p(I)$ is an admissible morphism of $\calG$, and $f$ is a limit of the composite diagram
$$ \calI \stackrel{p}{\rightarrow} \Fun( \Delta^1, \calG) \stackrel{j}{\rightarrow} \Fun( \Delta^1, \Pro(\calG) ).$$
For each object $X \in \Pro(\calG)$, we let $\Pro(\calG)^{\proadm}_{/X}$ denote the full subcategory spanned by the proadmissible morphisms.
\end{definition}

\begin{remark}\label{swinnerr}
Let $\calG$ be a geometry. The composition
$$\Fun( \Delta^1, \calG) \times \Delta^1 \rightarrow \calG \stackrel{j}{\rightarrow} \Pro(\calG)$$
classifies a map $\Fun( \Delta^1, \calG) \rightarrow \Fun( \Delta^1, \Pro(\calG) )$. 
Since $\Fun( \Delta^1, \calG)$ is idempotent complete, Proposition \toposref{urgh1} implies that
this map induces an equivalence of $\infty$-categories
$$\phi: \Pro( \Fun( \Delta^1, \calG) ) \simeq \Fun( \Delta^1, \Pro(\calG) ).$$
Let $\Fun^{\adm}( \Delta^1, \calG)$ denote the full subcategory of $\Fun( \Delta^1, \calG)$ spanned by the admissible morphisms of $\calG$. The inclusion of $\Fun^{\adm}( \Delta^1, \calG)$ into
$\Fun( \Delta^1, \calG)$ induces a fully faithful embedding
$\Pro( \Fun^{\adm}( \Delta^1, \calG) ) \rightarrow \Pro( \Fun(\Delta^1, \calG) )$. Composing with the equivalence $\phi$, we obtain a fully faithful embedding
$$ \Pro( \Fun^{\adm}(\Delta^1, \calG) ) \rightarrow \Fun( \Delta^1, \Pro(\calG) ).$$
Unwinding the definitions, we see that the essential image of this functor can be identified with the full subcategory spanned by the proadmissible morphisms of $\Pro(\calG)$.
\end{remark}

\begin{remark}\label{proadmadm}
Let $\calG$ be a geometry, and $f: U \rightarrow X$ a morphism in $\calG$.
If $j(f)$ is proadmissible, then
$j(f)$ can be obtained as the colimit of some filtered diagram $\{ j(f_{\alpha}) \}$, where
each $f_{\alpha}$ is an admissible morphism in $\calG$. Since $j(f)$ is a compact object of
$\Fun( \Delta^1, \Pro(\calG) )$, we conclude that $j(f)$ is a retract of some $j(f_{\alpha})$, so that
$f$ is a retract of $f_{\alpha}$ and therefore admissible.
\end{remark}

\begin{remark}\label{subo}
Let $\calG$ be a geometry, and let $X$ be an object of $\Pro(\calG)$. A morphism
$U \rightarrow X$ is proadmissible if and only if, as an object of $\Pro(\calG)_{/X}$, 
$U$ can be identified with a small filtered limit of morphisms $U_{\alpha} \rightarrow X$
which fit into pullback diagrams
$$ \xymatrix{ U_{\alpha} \ar[r] \ar[d] & X \ar[d] \\
j( U'_{\alpha} ) \ar[r]^{j( f'_{\alpha})} & j(X'), }$$
where $j: \calG \rightarrow \Pro(\calG)$ denotes the Yoneda embedding and each
$f'_{\alpha}$ is an admissible morphism of $\calG$.
\end{remark}

\begin{lemma}\label{swug}
Let $\calG$ be a geometry. Then the class of proadmissible morphisms of $\Pro(\calG)$ constitutes a strongly saturated collection of morphisms of $\Pro(\calG)^{op} \simeq \Ind(\calG^{op})$, which is of small generation.
\end{lemma}

\begin{proof}
Let $\Fun^{\proadm}( \Delta^1, \Pro(\calG) )$ denote the full subcategory of $\Fun( \Delta^1, \Pro(\calG) )$ spanned by the proadmissible morphisms. The description of Remark \ref{swinnerr} implies that
$\Fun^{\proadm}( \Delta^1, \Pro(\calG) )$ is stable under small limits, and generated under small limits by a small subcategory. Remark \ref{subo} shows that the class of proadmissible morphisms is stable under pullbacks. It remains to show that the class of proadmissible morphisms is stable under composition. 

Choose proadmissible morphisms $f: X \rightarrow Y$ and $g: Y \rightarrow Z$ in $\Pro(\calG)$; we wish to show that $g \circ f$ is proadmissible. In view of Remark \ref{subo}, we may assume that
$f$ is a filtered limit of morphisms $\{ f_{\alpha}: U_{\alpha} \rightarrow Y \}$, where
each $f_{\alpha}$ is the pullback of an admissible morphism $f'_{\alpha}$ in $\calG$.
It suffices to show that each composition $g \circ f_{\alpha}$ is admissible. Replacing $f$ by
$f_{\alpha}$, we may assume that there exists a pullback diagram
$$ \xymatrix{ X \ar[r]^{f} \ar[d] & Y \ar[d]^{h} \\
j(X') \ar[r]^{j(f')} & j(Y') }$$
where $f'$ is an admissible morphism of $\calG$.

Applying Remark \ref{subo} again, we may assume that $g$ is the limit of a diagram of morphisms
$\{ g_{\beta}: V_{\beta} \rightarrow Z \}_{\beta \in B}$ indexed by a filtered partially ordered set $B$, such that each $g_{\beta}$ is a pullback of some admissible morphism 
$g'_{\beta}: V'_{\beta} \rightarrow Z'_{\beta}$ in $\calG$. Since $j(Y')$ is a cocompact object of $\Pro(\calG)$, we may assume that $h$ is homotopic to a composition
$ Y \stackrel{ \beta_0}{\rightarrow} V_{\beta_0} \stackrel{ g_{\beta_0}}{\rightarrow} Z$
for some $\beta_0 \in B$. It follows that we can identify $X$ with the limit of
the diagram $\{ V_{\beta} \times_{ j(Y') } j(X') \}_{ \beta \geq \beta_0 }$. It will therefore suffice to show that each of the composite maps
$$V_{\beta} \times_{ j(Y')} j(X') \rightarrow V_{\beta} \rightarrow Z$$
is proadmissible. Each of these composite maps is a pullback of $j(h_{\beta})$, where
$h_{\beta}: V'_{\beta} \times Y' \rightarrow Z'_{\beta} \times X'$ is the product of
$g_{\beta}$ with $f'$. It follows that $h_{\beta}$ is admissible, as desired.
\end{proof}

\begin{corollary}
Suppose given a commutative diagram
$$ \xymatrix{ & Y \ar[dr]^{g} & \\
X \ar[ur]^{f} \ar[rr]^{h} &  & Z }$$
in $\Pro(\calG)$, where $g$ is proadmissible. Then $f$ is proadmissible if and only if $h$ is proadmissible.
\end{corollary}

\begin{proof}
Combine Lemma \ref{swug} with Corollary \toposref{wugg}.
\end{proof}

We are now ready to prove the Theorem \ref{gen} in the special case $\calX = \SSet$.

\begin{proposition}\label{lumm}
Let $\calG$ be a geometry. Then there exists a factorization system $(S_L, S_R)$ on the
$\infty$-category $\Fun^{\lex}( \calG, \SSet) \simeq \Pro(\calG)^{op}$, where:
\begin{itemize}
\item[$(1)$] A morphism of $\Fun^{\lex}(\calG, \SSet)$ belongs to $S_L$ if and only if it is proadmissible, when viewed as a morphism in $\Pro(\calG)$.
\item[$(2)$] A morphism $\alpha: \calO \rightarrow \calO'$ in $\Fun^{\lex}(\calG, \SSet)$ belongs to
$S_R$ if and only if, for every admissible morphism $f: U \rightarrow X$ in $\calG$, the diagram
$$ \xymatrix{ \calO(U) \ar[r] \ar[d] & \calO'(U) \ar[d] \\
\calO(X) \ar[r] & \calO'(X) }$$
is a pullback square in $\SSet$.
\end{itemize}
\end{proposition}

\begin{proof}
Let $S$ denote the class of all morphisms in $\Fun^{\lex}(\calG, \SSet)$ of the form
$j(f)$, where $f: U \rightarrow X$ is an admissible morphism of $\calG$ and
$j: \calG^{op} \rightarrow \Fun^{\lex}( \calG, \SSet)$ denotes the Yoneda embedding.
Using Lemma \ref{swug}, we deduce that $S_{L}$ is the saturated class of morphisms of
$\Fun^{\lex}(\calG, \SSet)$ generated by $S$. It follows from Proposition \toposref{nir} that
$(S_L, S_L^{\perp})$ form a factorization system on $\Fun^{\lex}(\calG, \SSet)$. Since
$S_L$ is generated by $S$ under colimits, Proposition \toposref{swimmm} implies that 
$S_{L}^{\perp} = S^{\perp}$. Finally, the equality $S^{\perp} = S_{R}$ follows by
unwinding the definitions.
\end{proof}

The proof of Theorem \ref{gen} in general will require a few more lemmas.

\begin{lemma}\label{scattra}
Let $\calC$ be an $\infty$-category, $\calC^0 \subseteq \calC$ a localization of $\calC$, and
$Y$ an object of $\calC_0$. Then $\calC^0_{/Y}$ is a localization of $\calC_{/Y}$.
Moreover, a morphism $f: X \rightarrow X'$ in $\calC_{/Y}$ exhibits $X'$ as a $\calC^{0}_{/Y}$-localization of $X$ if and only if $f$ exhibits $X'$ as a $\calC^{0}$-localization of $X$
in the $\infty$-category $\calC$.
\end{lemma}

\begin{proof}
We first prove the ``if'' direction of the last assertion. Choose a morphism $Y' \rightarrow Y$
in $\calC$, where $Y' \in \calC^{0}$. We have a map of homotopy fiber sequences
$$ \xymatrix{ \bHom_{\calC_{/Y}}( X', Y') \ar[r]^{\phi} \ar[d] & \bHom_{ \calC_{/Y} }( X, Y') \ar[d] \\
\bHom_{\calC}( X', Y') \ar[d] \ar[r]^{\phi'} & \bHom_{\calC}(X, Y') \ar[d] \\
\bHom_{\calC}(X', Y) \ar[r]^{\phi''} & \bHom_{\calC}(X, Y). }$$
Since the maps $\phi'$ and $\phi''$ are homotopy equivalences, we conclude that $\phi$ is a homotopy equivalence as desired.

We now show that $\calC^{0}_{/Y}$ is a localization of $\calC_{/Y}$. 
In view of Proposition \toposref{testreflect} and the above argument, it will suffice to show that
for every map $h: X \rightarrow Y$, there exists a factorization
$$ \xymatrix{ & X' \ar[dr]^{g} & \\
X \ar[ur]^{f} \ar[rr]^{h} & & Y }$$
where $f$ exhibits $X'$ as a $\calC^{0}$-localization of $X$. The existence of $f$ follows from
Proposition \toposref{testreflect} (applied to the $\infty$-category $\calC$), and the ability to
complete the diagram follows from the assumption that $Y \in \calC^{0}$.

To complete the proof, we observe that {\em any} $\calC^{0}_{/Y}$-localization
$X \rightarrow X''$ must be equivalent to the morphism $X \rightarrow X'$ constructed above, so that $X \rightarrow X''$ also exhibits $X''$ as a $\calC^{0}$-localization of $X$.
\end{proof}

\begin{lemma}\label{silb}
Let $\calC$ be an $\infty$-category equipped with a factorization system $(S_L, S_R)$, and let
$L: \calC \rightarrow \calC$ be a localization functor such that $L S_R \subseteq S_R$. Then
the full subcategory $L\calC \subseteq \calC$ admits a factorization system
$( S'_L, S'_R)$, where:
\begin{itemize}
\item[$(1)$] A morphism $f'$ in $L \calC$ belongs to $S'_{L}$ if and only if $f'$ is a retract of
$Lf$, for some $f \in S_L$.
\item[$(2)$] A morphism $g$ in $L \calC$ belongs to $S'_{R}$ if and only if $g \in S_R$.
\end{itemize}
\end{lemma}

\begin{proof}
Clearly $S'_{L}$ and $S'_{R}$ are stable under the formation of retracts.
Let $h: X \rightarrow Z$ be a morphism in $L\calC$; we wish to show that
$h$ factors as a composition
$$ X \stackrel{f'}{\rightarrow} Y' \stackrel{g'}{\rightarrow} Z$$
where $f' \in S'_L$ and $g' \in S'_{R}$. First, choose a factorization of $h$ as a composition
$$ X \stackrel{f}{\rightarrow} Y \stackrel{g}{\rightarrow} Z.$$
Then $Lf \in S'_{L}$, $Lg \in S'_{R}$, and $h \simeq Lh \simeq Lg \circ Lf$.

It remains to show that $f' \perp g'$, for $f' \in S'_{L}$ and $g' \in S'_{R}$. Without loss of generality, we may suppose $f' = Lf$ for some $f \in S_L$. Choose a commutative diagram
$$ \xymatrix{ A \ar[d]^{f} \ar[r] & LA \ar[d]^{Lf} \ar[r] & X \ar[d]^{g'} \\
B \ar[r] & LB \ar[r] & Y. }$$
We wish to show that the mapping space $\bHom_{ \calC_{LA/ \, /Y} }(LB, X)$ is weakly contractible.
We have a commutative diagram of fiber sequences
$$ \xymatrix{ \bHom_{ \calC_{LA/ \, /Y}}(LB, X) \ar[r]^{\phi} \ar[d] & \bHom_{\calC_{A/ \, /Y}}(B,X) \ar[d] \\
\bHom_{\calC_{/Y} }(LB, X) \ar[r]^{\phi'} \ar[d] & \bHom_{\calC_{/Y}}(B, X) \ar[d] \\
\bHom_{\calC_{/Y} }(LA, X) \ar[r]^{\phi''} & \bHom_{\calC_{/Y}}(A,X). }$$
 Using Lemma \ref{scattra}, we deduce that $\phi'$ and $\phi''$ are homotopy equivalences.
 It follows that $\phi$ is also a homotopy equivalence. We are therefore reduced to proving that
 $\bHom_{\calC_{A/ \,/Y}}(B,X)$ is contractible, which follows from the orthogonality relation
 $f \perp g'$.
\end{proof}

\begin{lemma}\label{stooge}
Suppose given a pair of adjoint functors
$$ \Adjoint{F}{\calC}{\calD.}{G}$$
Let $(S_L, S_R)$ be a factorization system on $\calC$, and $(S'_{L}, S'_{R})$ a factorization
system on $\calD$. The following conditions are equivalent:
\begin{itemize}
\item[$(1)$] The functor $F$ carries $S_L$ to $S'_{L}$.
\item[$(2)$] The functor $G$ carries $S'_{R}$ to $S_R$.
\end{itemize}
\end{lemma}

\begin{proof}
This follows immediately from Remark \toposref{smule} and Proposition \toposref{swin}.
\end{proof}

\begin{proof}[Proof of Theorem \ref{gen}]
We first show that there exists a factorization system $(\overline{S}^{\calX}_L, \overline{S}^{\calX}_R)$ on $\Fun^{\lex}(\calG, \calX)$,
characterized by the following condition:
\begin{itemize}
\item[$(\ast)$] The collection $\overline{S}^{\calX}_R$ consists of those morphisms $\alpha: \calO \rightarrow \calO'$ such that for every admissible morphism $U \rightarrow X$ in $\calG$, the diagram
$$ \xymatrix{ \calO(U) \ar[r] \ar[d] & \calO'(U) \ar[d] \\
\calO(X) \ar[r] & \calO'(X). }$$
\end{itemize}

If $\calX = \SSet$, this follows from Proposition \ref{lumm}. 
In the general case, we may assume without loss of generality that $\calX = L \calP(\calC)$, where
$\calC$ is a small $\infty$-category and $L$ is a left exact localization functor from
$\calP(\calC)$ to itself. Using Corollary \toposref{funcsys}, we deduce that the $\infty$-category $\Fun^{\lex}( \calG, \calP(\calC)) \simeq \Fun( \calC^{op}, \Fun^{\lex}( \calG, \SSet) )$ admits a factorization
system $(S_L, S_R)$ satisfying $(\ast)$. Composition with $L$ induces a localization functor
from $\Fun^{\lex}(\calG, \calP(\calC))$ to itself; since $L$ is left exact, this functor carries $S_R$ to itself.
It follows from Lemma \ref{silb} that $(S_L, S_R)$ induces a factorization system $(\overline{S}^{\calX}_L, \overline{S}^{\calX}_R)$ on the $\infty$-category $\Fun^{\lex}(\calG, \calX)$, and this factorization system clearly satisfies $(\ast)$ as well.

We now claim that the factorization system on $\Fun^{\lex}(\calG, \calX)$ restricts to a factorization system on the full subcategory $\Struct_{\calG}(\calX)$.
Namely, let $S_L^{\calX}$ denote the collection of all morphisms in $\Struct_{\calG}(\calX)$
belonging to $\overline{S}^{\calX}_L$, and define $S_R^{\calX}$ likewise. We claim that
$(S_L^{\calX}, S_{R}^{\calX})$ is a factorization system on $\Struct_{\calG}(\calX)$. 
It is clear that $f \perp g$ for $f \in S_L^{\calX}$, $g \in S_R^{\calX}$, and that
the collections $S_L^{\calX}$ and $S_R^{\calX}$ are stable under retracts.
To complete the proof, it will suffice to show that every morphism
$h: \calO \rightarrow \calO''$ in $\Struct_{\calG}(\calX)$ factors as a composition
$$ \calO \stackrel{f}{\rightarrow} \calO' \stackrel{g}{\rightarrow} \calO'',$$
where $f \in S_L^{\calX}$ and $g \in S_R^{\calX}$. Since $(\overline{S}^{\calX}_L, \overline{S}^{\calX}_R)$ is a factorization system on
$\Fun^{\lex}(\calG, \calX)$, we can choose such a factorization with $f \in \overline{S}^{\calX}_L$ and
$g \in \overline{S}^{\calX}_R$. To complete the proof it will suffice to show that $\calO' \in \Struct_{\calG}(\calX)$.
In other words, we must show that for every admissible covering $\{ U_{\alpha} \rightarrow X \}$ of
an object $X \in \calG$, the induced map $\psi: \coprod_{\alpha} \calO'(U_{\alpha}) \rightarrow \calO'(X)$
is an effective epimorphism in $\calX$. Using the assumption that $g \in \overline{S}^{\calX}_R$ (and the fact that colimits in $\calX$ are universal), we conclude that the diagram
$$ \xymatrix{ \coprod_{\alpha} \calO'(U_{\alpha}) \ar[r] \ar[d]^{\psi} & \coprod_{\alpha} \calO''(U_{\alpha}) \ar[d]^{\phi} \\
\calO'(X) \ar[r] & \calO''(X) }$$
is a pullback square in $\calX$. Since $\calO'' \in \Struct_{\calG}(\calX)$, the map $\phi$ is an effective epimorphism. Thus $\psi$ is also an effective epimorphism (Proposition \toposref{hintdescent1}), as desired. 

We now complete the proof by showing that the factorization system $(S_{L}^{\calX}, S_{R}^{\calX})$
depends functorially on $\calX$. Let $\pi^{\ast}: \calX \rightarrow \calY$ be a left exact, colimit preserving functor; we wish to show that composition with $\pi^{\ast}$ carries $S_L^{\calX}$ to
$S_{L}^{\calY}$ and $S_{R}^{\calX}$ to $S_{R}^{\calY}$. The second assertion follows immediately from the left exactness of $f^{\ast}$. To prove the first, we will show something slightly stronger:
the induced functor $F: \Fun^{\lex}(\calG, \calX) \rightarrow \Fun^{\lex}(\calG, \calY)$ carries
$\overline{S}_{L}^{\calX}$ to $\overline{S}_{L}^{\calY}$. Note that $F$ admits a right adjoint
$G$, given by composition with a right adjoint $\pi_{\ast}$ to $\pi^{\ast}$. Since
$\pi_{\ast}$ is left exact, the functor $G$ carries $\overline{S}_{R}^{\calY}$ to $\overline{S}_{R}^{\calX}$;
the desired result now follows from Lemma \ref{stooge}.
\end{proof}

\subsection{Classifying $\infty$-Topoi}\label{geo3}

For every topological space $X$, the $\infty$-category $\Shv(X)$ of sheaves of spaces
on $X$ is an $\infty$-topos. Moreover, if the topological space $X$ is {\it sober} (that is, if every
irreducible closed subset of $X$ has a unique generic point), then we can recover $X$ from
$\Shv(X)$: the points $x \in X$ can be identified with isomorphism classes of geometric morphisms
$x^{\ast}: \Shv(X) \rightarrow \SSet$, and open subsets of $X$ can be identified with
subobjects of the unit object $1 \in \Shv(X)$. In other words, the space $X$ and the $\infty$-topos
$\Shv(X)$ are interchangable: either one canonically determines the other. 

The situation described above can be summarized by saying that we can regard the theory of
$\infty$-topoi as a {\em generalization} of the classical theory of topological spaces (more precisely, of the theory of sober topological spaces). In this paper, we have opted to dispense with topological spaces altogether and work almost entirely in the setting of $\infty$-topoi. This extra generality affords us some flexibility which is useful even if we are ultimately interested primarily in ordinary topological spaces. For example, every $\infty$-topos $\calX$ represents a functor from topological spaces
to $\infty$-categories, described by the formula
$$X \mapsto \Fun^{\ast}( \calX, \Shv(X) ).$$
These functors are generally not representable in the category of topological spaces itself.
For example, the functor which carries $X$ to the (nerve of the) category of sheaves of commutative rings on $X$ is representable by the $\infty$-topos $\calP( \calG_{\Zar} )$ (see Proposition \ref{usebil} below). Working in the setting of $\infty$-topoi has the advantage of allowing us to treat the topological space $X$ and the classifying $\infty$-topos $\calP(\calG_{\Zar})$ on the same footing.

Our goal in this section is to exploit the above observation in a systematic fashion.
Let $\calG$ be a geometry. The $\infty$-category $\Struct_{\calG}(\calX)$ of 
Definition \ref{psyab} depends {\em functorially} on the $\infty$-topos $\calX$.
More precisely, composition determines a functor
$$ \Fun^{\ast}( \calX, \calY) \times \Struct_{\calG}( \calX) \rightarrow \Struct_{\calG}(\calY).$$
In particular, if we fix a $\calG$-structure $\calO \in \Struct_{\calG}(\calX)$, then composition with $\calO$ induces a functor $\Fun^{\ast}(\calX, \calY) \rightarrow \Struct_{\calG}(\calY)$.

\begin{definition}
Let $\calG$ be a geometry. We will say that a $\calG$-structure $\calO$ on an $\infty$-topos $\calK$ is
{\it universal} if, for every $\infty$-topos $\calX$, composition with $\calO$ induces an equivalence of
$\infty$-categories $\Fun^{\ast}(\calK, \calX) \rightarrow \Struct_{\calG}(\calX)$. 
In this case, we will also say that {\it $\calO$ exhibits $\calK$ as a classifying $\infty$-topos for $\calG$}, or that {\it $\calK$ is a classifying $\infty$-topos for $\calG$}.
\end{definition}

It is clear from the definition that a classifying $\infty$-topos for a geometry $\calG$ is uniquely determined up to equivalence, provided that it exists. For existence, we have the following construction:

\begin{proposition}\label{usebil}
Let $\calG$ be a geometry. Then the composition
$$ \calG \stackrel{j}{\rightarrow} \calP(\calG) \stackrel{L}{\rightarrow} \Shv(\calG)$$
exhibits $\Shv(\calG)$ as a classifying $\infty$-topos for $\calG$. Here we regard $\calG$ as endowed with a Grothendieck topology as in Remark \ref{switchtop}.
\end{proposition}

\begin{proof}
This follows immediately from Proposition \toposref{igrute}.
\end{proof}

Let $\calG$ be a geometry and $\calX$ an $\infty$-topos. The $\infty$-category
$\Struct_{\calG}(\calX)$ depends only on the underlying $\infty$-category of $\calG$ and its Grothendieck topology, and not on the class of admissible morphisms in $\calG$.
Consequently, a classifying $\infty$-topos for $\calG$ also does not depend on the class of admissible morphisms in $\calG$ (this can also be deduced from Proposition \ref{usebil}). However, the class of admissible morphisms of $\calG$ is needed to define the subcategory
$\Struct^{\loc}_{\calG}(\calX) \subseteq \Struct_{\calG}(\calX)$ of local morphisms, and therefore determines some additional data on a classifying $\infty$-topos for $\calG$. Our next goal is to describe the nature of this additional data.

\begin{definition}\label{cammer}
Let $\calK$ be an $\infty$-topos. A {\it geometric structure} on $\calK$ consists of the specification, for every $\infty$-topos $\calX$, of a factorization system $( S_L^{\calX}, S_{R}^{\calX})$ on
$\Fun^{\ast}(\calK, \calX)$, which depends functorially on $\calX$ in the following sense:
for every geometric morphism $\pi^{\ast}: \calX \rightarrow \calY$, the induced functor
$\Fun^{\ast}( \calK, \calX) \rightarrow \Fun^{\ast}( \calK, \calY)$ (given by composition with
$\pi^{\ast}$) carries $S_{L}^{\calX}$ to $S_{L}^{\calY}$ and $S_{R}^{\calX}$ to $S_{R}^{\calY}$.

If $\calK$ is an $\infty$-topos with geometric structure and $\calX$ is another $\infty$-topos, 
then we let $\Struct^{\loc}_{\calK}( \calX)$ denote the subcategory of $\Fun^{\ast}(\calK, \calX)$
spanned by all the objects of $\Fun^{\ast}(\calK, \calX)$, and all the morphisms which belong to
$S_{R}^{\calX}$. 
\end{definition}

\begin{example}\label{abbun}
Let $\calG$ be a geometry and let $\calO: \calG \rightarrow \calK$ be a universal
$\calG$-structure. Then the $\infty$-topos $\calK$ inherits a geometric structure, which
is characterized by the following property:
\begin{itemize}
\item[$(\ast)$] For every $\infty$-topos $\calX$, the equivalence
$\Fun^{\ast}( \calK, \calX) \rightarrow \Struct_{\calG}(\calX)$ restricts to an equivalence
of $\Struct_{\calK}^{\loc}(\calX)$ with $\Struct_{\calG}^{\loc}(\calX)$.
\end{itemize}
This follows immediately from Theorem \ref{gen}.
\end{example}

\begin{example}\label{fumple}
Let $\calG$ be a {\it discrete} geometry, so that $\Struct_{\calG}^{\loc}(\calX) = \Struct_{\calG}(\calX) = \Fun^{\lex}( \calG, \calX)$ for every $\infty$-topos $\calX$. Then the induced geometric structure
on the classifying $\infty$-topos $\calK \simeq \calP(\calG)$ is trivial, in the sense that for every
$\infty$-topos $\calX$, the corresponding factorization system $(S_L^{\calX}, S_R^{\calX})$ is
given as in Example \toposref{scumm}.
\end{example}

\begin{warning}
There is a potential conflict between the notations introduced in Definitions \ref{psyab} and \ref{cammer}. However, there should be little danger of confusion: $\Struct^{\loc}_{\calK}(\calX)$ is described by Definition \ref{cammer} when $\calK$ is an $\infty$-topos with geometric structure, and by Definition \ref{psyab} if $\calK$ is a geometry.
\end{warning}

We next discuss the functoriality of the $\infty$-categories
$\Struct^{\loc}_{\calK}(\calX)$ (and $\Struct^{\loc}_{\calG}(\calX)$) in the $\infty$-topos $\calX$. We begin by reviewing some definitions from \cite{topoi}.

\begin{notation}
Let $\widehat{\Cat}_{\infty}$ denote the $\infty$-category of (not necessarily small) $\infty$-categories.
We let $\LGeo \subseteq \widehat{\Cat}_{\infty}$ denote the subcategory defined as follows:
\begin{itemize}
\item[$(i)$] An $\infty$-category $\calX$ belongs to $\LGeo$ if and only if $\calX$ is an $\infty$-topos.
\item[$(ii)$] Let $f^{\ast}: \calX \rightarrow \calY$ be a functor between $\infty$-topoi. Then
$f^{\ast}$ is a morphism of $\LGeo$ if and only if $f^{\ast}$ preserves small colimits and finite limits.
\end{itemize}
The inclusion $\LGeo \subseteq \widehat{\Cat}_{\infty}$ classifies a coCartesian fibration
$p: \overline{\LGeo} \rightarrow \LGeo$. We will refer to $p$ as the {\it universal topos fibration}
(see Definition \toposref{skuzz} and Proposition \toposref{surtog2}); note that the the fiber of $p$ over an $\infty$-topos $\calX \in \LGeo$ is canonically equivalent to $\calX$.
\end{notation}

\begin{definition}\label{gcuro}
Let $\calG$ be an geometry. We define a subcategory
$$\LGeo(\calG) \subseteq \Fun( \calG, \overline{\LGeo} )
\times_{ \Fun( \calG, \LGeo) } \LGeo$$
as follows:
\begin{itemize}
\item[$(a)$] Let $\calO \in \Fun( \calG, \overline{\LGeo} )
\times_{ \Fun( \calG, \LGeo) } \LGeo$ be an object. We can identify
$\calO$ with a functor from $\calG$ to $\calX$, for some $\infty$-topos $\calX$.
Then $\calO \in \LGeo(\calG)$ if and only if $\calO$ is a $\calG$-structure on $\calX$.

\item[$(b)$] Let $\alpha: \calO \rightarrow \calO'$ be a morphism in $$\Fun( \calG, \overline{\LGeo} )
\times_{ \Fun( \calG, \LGeo) } \LGeo,$$
where $\calO$ and $\calO'$ belong to
$\LGeo(\calG)$, and let $f^{\ast}: \calX \rightarrow \calY$ denote the image
of the morphism $\alpha$ in $\LGeo$. Then $\alpha$ belongs to $\LGeo(\calG)$ if and only if, for
every admissible morphism $U \rightarrow X$ in $\calG$, the induced diagram
$$ \xymatrix{ f^{\ast} \calO(U) \ar[r] \ar[d] & f^{\ast} \calO(X) \ar[d] \\
{f'}^{\ast} \calO'(U) \ar[r] & \calO'(X) }$$
is a pullback square in the $\infty$-topos $\calY$. 
\end{itemize}

We will refer to the {\em opposite} $\infty$-category $\LGeo(\calG)^{op}$ as the 
{\it $\infty$-category of $\calG$-structured $\infty$-topoi}.
\end{definition}

\begin{definition}\label{tcuro}
Let $\calK$ be an $\infty$-topos equipped with a geometric structure. We define a subcategory
$$\LGeo(\calK) \subseteq \Fun( \calK, \overline{\LGeo} )
\times_{ \Fun( \calK, \LGeo) } \LGeo$$
as follows:
\begin{itemize}
\item[$(a)$] Let $f^{\ast} \in \Fun( \calK, \overline{\LGeo} )
\times_{ \Fun( \calK, \LGeo) } \LGeo$ be an object, which we can identify
with a functor $f^{\ast}: \calK \rightarrow \calX$, where $\calX$ is an $\infty$-topos. Then
$f^{\ast}$ belongs to $\LGeo(\calK)$ if and only if $f^{\ast}$ preserves small colimits and finite limits.
\item[$(b)$] Let $\alpha: f^{\ast} \rightarrow {f'}^{\ast}$ be a morphism in $\Fun( \calK, \overline{\LGeo} )
\times_{ \Fun( \calK, \LGeo) } \LGeo$, where $f^{\ast}$ and ${f'}^{\ast}$ belong to
$\LGeo(\calK)$, and let $g^{\ast}: \calX \rightarrow \calY$ denote the image
of $\alpha$ in $\LGeo$. Then $\alpha$ belongs to $\LGeo(\calK)$ if and only if
the corresponding morphism $g^{\ast} f^{\ast} \rightarrow {f'}^{\ast}$ belongs to
$\Struct^{\loc}_{\calK}(\calY)$.
\end{itemize}

We will refer to the {\em opposite} $\infty$-category
$\LGeo(\calK)^{op}$ as {\it the $\infty$-category of $\calK$-structured $\infty$-topoi}. 
\end{definition}

\begin{remark}
Let $\calG$ be a geometry (or an $\infty$-topos with geometric structure).
The fiber of the map $\LGeo(\calG) \rightarrow \LGeo$ over an $\infty$-topos
$\calX$ is canonically equivalent to (but generally not isomorphic to) $\Struct^{\loc}_{\calG}(\calX)$.
Nevertheless, we will generally abuse notation and identify objects of
$\LGeo(\calG)$ with pairs $(\calX, \calO_{\calX})$, where $\calX$ is an $\infty$-topos
and $\calO_{\calX}: \calG \rightarrow \calX$ is a $\calG$-structure on $\calX$.
\end{remark}

The following result is an easy consequence of Proposition \toposref{doog}:

\begin{proposition}\label{swupa}
Let $\calK$ be an $\infty$-topos with a geometric structure (or a geometry). Then
the projection map $q: \LGeo(\calK) \rightarrow \LGeo$ is a coCartesian
fibration of simplicial sets. Moreover, a morphism
$\alpha: \calO_{\calX} \rightarrow \calO_{\calY}$ in $\LGeo(\calK)$ is
$q$-coCartesian if and only if, for each object $K \in \calK$, the induced
map $\pi^{\ast} \calO_{\calX}(K) \rightarrow \calO_{\calY}(K)$ is
an equivalence in $\calY$; here $\pi^{\ast}: \calX \rightarrow \calY$ denotes the
image of $\alpha$ in $\LGeo$. 
\end{proposition}

In other words, every geometric morphism
of $\infty$-topoi $\pi: \calX \rightarrow \calY$ gives rise to a pullback functor
$\Struct^{\loc}_{\calK}(\calY) \rightarrow \Struct^{\loc}_{\calK}(\calX)$, which is simply
given by composition with the pullback functor $\pi^{\ast}$.

\begin{proposition}\label{swame}
Let $\calG$ be a geometry, and let $\calO: \calG \rightarrow \calK$ be a universal $\calG$-structure.
Then composition with $\calO$ induces an equivalence of $\infty$-categories
$\LGeo(\calK) \rightarrow \LGeo(\calG)$. Here we regard $\calK$ as endowed with the geometric structure of Example \ref{abbun}.
\end{proposition}

\begin{proof}
Combine Proposition \ref{swupa} and Corollary \toposref{usefir}.
\end{proof}

Proposition \ref{swame} implies that the theory of $\calG$-structures on $\infty$-topoi (here $\calG$ is a geometry) can be subsumed into the theory of geometric structures on $\infty$-topoi. It is natural to ask how much benefit we receive from this shift of perspective. For example, we might ask whether a given $\infty$-topos $\calK$ arises as the classifying $\infty$-topos of some geometry. This is not true in general, but if we ignore the geometric structure on $\calK$ it is {\em almost} true in the following sense:
we can always choose a geometric morphism $f^{\ast}: \Shv(\calG) \rightarrow \calK$ which induces
an equivalence after passing to hypercompletions (see \S \toposref{hyperstacks}). Consequently, the gain in generality is slight. However, the gain in {\em functoriality} is considerable.

\begin{definition}
Let $\calK$ and $\calK'$ be $\infty$-topoi equipped with geometric structures, and let
$f^{\ast}: \calK \rightarrow \calK'$ be a geometric morphism. We will say that
$f^{\ast}$ is {\it compatible with the geometric structures} on $\calK$ and $\calK'$ if,
for every $\infty$-topos $\calX$, composition with $f^{\ast}$ carries
$\Struct_{\calK}^{\loc}(\calX)$ into $\Struct_{\calK'}^{\loc}(\calX)$.
\end{definition}

\begin{example}\label{sput}
Let $f: \calG \rightarrow \calG'$ be a transformation of geometries, and let
$\calK$ and $\calK'$ be classifying $\infty$-topoi for $\calG$ and $\calG'$, respectively.
Then $f$ induces a geometric morphism $f^{\ast}: \calK \rightarrow \calK'$, well-defined up to equivalence (in fact, up to a contractible space of choices). Moreover, $f^{\ast}$ is compatible with the geometric structures on $\calK$ and $\calK'$ defined in Example \ref{abbun}.
\end{example}

Let $f^{\ast}: \calK \rightarrow \calK'$ be a geometric morphism which is compatible with
geometric structures on $\calK$ and $\calK'$, respectively. Then composition with
$f^{\ast}$ induces a functor
$$ \LGeo( \calK') \rightarrow \LGeo(\calK).$$ 
Suppose, for example, that $\calK$ and $\calK'$ are classifying $\infty$-topoi for
{\em discrete} geometries $\calG$ and $\calG'$, respectively. Then the $\infty$-category
of geometric morphisms from $\calK$ to $\calK'$ can be identified with the $\infty$-category of left
exact functors from $\calG$ to $\calK' \simeq \calP(\calG')$. This is generally much larger than the $\infty$-category $\Fun^{\lex}(\calG, \calG')$ of left exact functors from $\calG$ to $\calG'$.

\begin{example}
Let $\calK$ be the $\infty$-topos $\SSet$ of spaces. Then
$\calK$ admits a {\em unique} geometric structure, which coincides with
the trivial geometric structure described in Example \ref{fumple}. It follows that $\Struct^{\loc}_{\calK}(\calX) \simeq \Fun^{\ast}( \SSet, \calX)$ for every $\infty$-topos $\calX$.
Proposition \toposref{spacefinall} implies that each
$\Struct^{\loc}_{\calK}(\calX)$ is a contractible Kan complex. Combining Lemma \ref{swupa} with
Corollary \toposref{usefir}, we deduce that forgetful functor
$\LGeo(\calK) \rightarrow \LGeo$ is an equivalence of $\infty$-categories
(even a trivial Kan fibration). In other words, every $\infty$-topos $\calX$
admits an essentially unique $\calK$-structure.
\end{example}

\begin{remark}
Let $\calG$ and $\calG'$ be geometries with classifying $\infty$-topoi $\calK$ and
$\calK'$. Given a geometric morphism $f^{\ast}: \calK \rightarrow \calK'$ compatible with the geometric structures described in Example \ref{abbun}, we might ask: can $f^{\ast}$ be realized as resulting
from a transformation of geometries from $\calG$ to $\calG'$, as in Example \ref{sput}? To phrase the question differently: are there any special properties enjoyed by the geometric morphisms which
arise from the construction of Example \ref{sput}? We will describe one answer to this question in
\S \ref{relspec}: given a transformation $\calG \rightarrow \calG'$ of geometries, the induced functor $\LGeo(\calG) \rightarrow \LGeo(\calG')$ admits a left adjoint.
\end{remark}



\begin{remark}
Fix an $\infty$-topos $\calK$, and regard $\calK$ as endowed with the
{\it trivial} geometric structure described in Example \ref{fumple}.
We can informally summarize Definition \ref{tcuro} as follows: an object of
$\LGeo(\calK)$ is a geometric morphism $\calO_{\calX}: \calK \rightarrow \calX$, and
a morphism in $\LGeo(\calK)$ is a diagram of geometric morphisms
$$ \xymatrix{ \calK \ar[dr]_{ \calO_{\calX} } \ar[rr]^{ \calO_{\calY} } & & \calY \\
& \calX \ar[ur]_{f^{\ast} } & }$$
which commutes up to a natural transformation
$\alpha: f^{\ast} \calO_{\calX} \rightarrow \calO_{\calY}$. However,
we do not assume that $\alpha$ is invertible. Consequently, we can best view
the $\LGeo(\calK)$ as a variant of undercategory $\LGeo_{\calK/ }$, but
defined using the natural {\em $\infty$-bicategory} structure on $\LGeo$.
We will not adopt this point of view, since we do not wish to take the time
to develop the theory of $\infty$-bicategories in this paper.
\end{remark}

We close this section by giving an alternative description of geometric structures on $\infty$-topoi.
First, we review a few facts about limits in the $\infty$-category $\RGeom \simeq \LGeo^{op}$ of $\infty$-topoi.

Recall that the $\infty$-category $\RGeom$ of $\infty$-topoi is naturally {\em cotensored} over
$\Cat_{\infty}$. More precisely, for every $\infty$-topos $\calX$ and every simplicial set $K$,
there exists another $\infty$-topos $\calX^{K}$ and a map
$\theta: K \rightarrow \Fun_{\ast}(\calX^{K}, \calX)$ with the following universal property:
for every $\infty$-topos $\calY$, composition with $\theta$ induces an equivalence of $\infty$-categories
$$ \Fun_{\ast}( \calY, \calX^{K} ) \simeq \Fun(K, \Fun_{\ast}(\calY, \calX) ).$$ 
This is an immediate consequence of Proposition \toposref{cotens}. Note that $\calX^{K}$ is determined up to (canonical) equivalence by $\calX$ and $K$.

\begin{warning}
Our notation $\calX^{K}$ does {\em not} denote the simplicial mapping space $\Fun(K, \calX)$ (though $\Fun(K, \calX)$ is again an $\infty$-topos which can be characterized by similar universal mapping property in the $\infty$-category $\LGeo$).
\end{warning}

Note that $\calX^{K}$ depends functorially on $K$: every map of simplicial sets $K \rightarrow K'$ determines a geometric morphism $\pi_{\ast}: \calX^{K'} \rightarrow \calX^{K}$, well-defined up to homotopy. In the special case $K = \Delta^1$, we will refer to $\calX^{K}$ as a {\it path $\infty$-topos} for $\calX$. The projection $\Delta^1 \rightarrow \Delta^0$ determines a geometric morphism
$\delta_{\ast}: \calX \simeq \calX^{ \Delta^0} \rightarrow \calX^{ \Delta^1}$, which is well-defined up to homotopy; we will refer to $\delta_{\ast}$ as the {\it diagonal}. Note that, for every $\infty$-topos
$\calY$, composition with $\delta_{\ast}$ induces a fully faithful embedding
$$ \Fun_{\ast}( \calY, \calX) \rightarrow \Fun_{\ast}( \calY, \calX^{\Delta^1})
\simeq \Fun( \Delta^1, \Fun_{\ast}(\calY, \calX) )$$
whose essential image is the class of equivalences in $\Fun_{\ast}(\calY, \calX)$. 

Now suppose that $\calK$ is an $\infty$-topos endowed with a geometric structure, which determines a factorization system $(S_L^{\calX}, S_R^{\calX})$ on the $\infty$-category $\Fun^{\ast}(\calK, \calX)$
for each $\infty$-topos $\calX$, depending functorially on $\calX$. In view of Remark \toposref{spill}, 
we obtain an induced factorization system $( \overline{S}_{R}^{\calX}, \overline{S}_{L}^{\calX})$
on the $\infty$-category $\Fun_{\ast}( \calX, \calK) \simeq \Fun^{\ast}(\calK, \calX)^{op}$. Let $\theta: \Delta^1 \rightarrow \Fun_{\ast}( \calK^{\Delta^1}, \calK)$ exhibit $\calK^{\Delta^1}$ as a path $\infty$-topos for $\calK$.
Then $\theta$ admits an (essentially unique) factorization
$\theta \simeq \theta_{L} \circ \theta_{R}$, where $\theta_{L} \in \overline{S}_{L}^{ \calK^{\Delta^1}}$ and
$\theta_{R} \in \overline{S}_{R}^{\calK^{\Delta^1}}$. If $\calX$ is any $\infty$-topos and $\alpha$ is
any morphism in $\Fun_{\ast}(\calX, \calK)$, then there is an equivalence 
$\alpha \simeq \pi_{\ast}(\theta)$ for some geometric morphism $\pi_{\ast}: \calX \rightarrow \calK^{\Delta^1}$. We conclude that $\alpha \simeq \pi_{\ast}( \theta_L) \circ \pi_{\ast}(\theta_R)$ is the factorization of $\alpha$ determined by the factorization system $( \overline{S}_{R}^{\calX},
\overline{S}_{L}^{\calX})$. In particular, $\alpha$ belongs to $\overline{S}_{R}^{\calX}$ if and only if
$\pi_{\ast}(\theta_L)$ is an equivalence, and $\alpha$ belongs to $\overline{S}_{L}^{\calX}$ if and only if $\pi_{\ast}( \theta_{R} )$ is an equivalence. 

The morphisms $\theta_{L}$ and $\theta_{R}$ in $\Fun_{\ast}( \calK^{\Delta^1}, \calK)$ are classified
up to homotopy by objects $$\pi_{\ast}^{L}, \pi_{\ast}^{R} \in \Fun_{\ast}( \calK^{\Delta^1}, \calK^{\Delta^1}).$$ Using Proposition \toposref{swunder}, we can form pullback diagrams
$$ \xymatrix{ \calK^{\Delta^1}_{L} \ar[r] \ar[d] & \calK \ar[d]^{\delta_{\ast}} & \calK^{\Delta^1}_{R} \ar[r] \ar[d] & \calK \ar[d]^{\delta^{\ast}} \\
\calK^{\Delta^1} \ar[r]^{ \pi_{\ast}^{R} } & \calK^{\Delta^1} & \calK^{\Delta^1} \ar[r]^{ \pi_{\ast}^{L} } & \calK^{\Delta^1} }$$
in the $\infty$-category $\RGeom$ of $\infty$-topoi. Invoking Remark \toposref{sablewise}, this construction yields the following:

\begin{proposition}\label{spurin}
Let $\calK$ be an $\infty$-topos with geometric structure, and let $\calK^{\Delta^1}$ denote a
path $\infty$-topos for $\calK$. Then there exists a pair of geometric morphisms
$$ \calK^{\Delta^1}_{L} \rightarrow \calK^{\Delta^1} \leftarrow \calK^{\Delta^1}_{R}$$
with the following universal property: for every $\infty$-topos $\calX$, the induced maps
$$ \Fun_{\ast}( \calX, \calK^{\Delta^1}_{L}) \rightarrow \Fun( \Delta^1, \Fun_{\ast}( \calX, \calK) )$$
$$ \Fun_{\ast}( \calX, \calK^{\Delta^1}_{R} \rightarrow \Fun( \Delta^1, \Fun_{\ast}( \calX, \calK ))$$
are fully faithful embeddings whose essential images are the full subcategories spanned by $\overline{S}^{\calX}_{L}$ and $\overline{S}^{\calX}_{R}$, respectively. 
\end{proposition}

\subsection{$\infty$-Categories of Structure Sheaves}\label{geo4}

In this section, we will study $\infty$-categories of the form $\Struct_{\calG}(\calX)$, where
$\calG$ is a geometry and $\calX$ an $\infty$-topos. Our first goal is to show that $\Struct_{\calG}(\calX)$ admits (small) filtered colimits. This is a consequence of the following more general result:

\begin{proposition}\label{unwi}
Let $\calK$ be an $\infty$-topos with geometric structure, and let $\calX$ be an arbitrary $\infty$-topos.
Then:
\begin{itemize}
\item[$(1)$] The $\infty$-categories $\Fun^{\ast}(\calK, \calX)$ and $\Struct_{\calK}^{\loc}(\calX)$
admit small filtered colimits.

\item[$(2)$] The inclusions $\Struct^{\loc}_{\calK}(\calX) \subseteq \Fun^{\ast}(\calK, \calX) \subseteq \Fun( \calK, \calX)$ preserve small filtered colimits.
\end{itemize}
\end{proposition}

From this, we can immediately deduce some consequences.

\begin{corollary}\label{unwii}
Let $\calK$ be an $\infty$-topos with geometric structure, and let
$f^{\ast}: \calX \rightarrow \calY$ be a geometric morphism of $\infty$-topoi. Then
the functor
$$ \Struct^{\loc}_{\calK}(\calX) \rightarrow \Struct^{\loc}_{\calK}(\calY)$$
given by composition with $f^{\ast}$ preserves small filtered colimits.
\end{corollary}

\begin{corollary}\label{suil}
Let $\calK$ be an $\infty$-topos with geometric structure, and let $\calI$ be a small filtered $\infty$-category. Then every commutative diagram
$$ \xymatrix{ \calI \ar[r] \ar@{^{(}->}[d] & \LGeo(\calK) \ar[d]^{p} \\
\calI^{\triangleright} \ar[r] \ar@{-->}[ur] & \LGeo }$$
can be completed as indicated, such that the dotted arrow is a $p$-colimit diagram
in $\LGeo(\calK)$.
\end{corollary}

\begin{proof}
Combine Proposition \ref{unwi}, Corollary \ref{unwii}, and Corollary \toposref{constrel}.
\end{proof}

\begin{corollary}\label{swip}
Let $\calK$ be an $\infty$-topos with geometric structure. Then:
\begin{itemize}
\item[$(1)$] The $\infty$-category $\LGeo(\calK)$ admits small filtered colimits.
\item[$(2)$] The projection functor $\LGeo(\calK) \rightarrow \LGeo$ preserves small filtered colimits.
\end{itemize}
\end{corollary}

\begin{proof}
Combine Corollary \ref{suil}, Theorem \toposref{sutcar}, and Proposition \toposref{basrel}.
\end{proof}

\begin{remark}\label{swup}
Proposition \ref{unwi} and Corollaries \ref{unwii}, \ref{suil}, and \ref{swip} have evident analogues when the $\infty$-topos $\calK$ is replaced by a geometry $\calG$. To prove these analogues, it suffices to choose a universal $\calG$-structure $\calO: \calG \rightarrow \calK$ and apply the preceding results (together with Proposition \ref{swame}).
\end{remark}

\begin{proof}[Proof of Proposition \ref{unwi}]
Let $\calI$ be a (small) filtered $\infty$-category, let $F: \calI \rightarrow \Fun^{\ast}(\calK, \calX)$
be a diagram, and let $\calO$ be a colimit of $F$ in the $\infty$-category $\Fun( \calK, \calX)$.
We must show:
\begin{itemize}
\item[$(a)$] The functor $\calO$ belongs to $\Fun^{\ast}(\calK, \calX)$.
\item[$(b)$] Suppose that, for every morphism $I \rightarrow J$ in $\calI$, the
associated natural transformations $F(I) \rightarrow F(J)$ belongs to $\Struct^{\loc}_{\calK}(\calX)$. Then
each of the natural transformations $F(I) \rightarrow \calO$ belongs to $\Struct^{\loc}_{\calK}(\calX)$.
\item[$(c)$] Suppose given a natural transformation $\alpha: \calO \rightarrow \calO'$, where
$\calO' \in \Fun^{\ast}(\calK, \calX)$. Suppose further that each of the induced transformations
$\alpha_{I}: F(I) \rightarrow \calO'$ belongs to $\Struct^{\loc}_{\calK}(\calX)$. Then $\alpha$ belongs to $\Struct^{\loc}_{\calK}(\calX)$.
\end{itemize}

Assertion $(a)$ follows immediately from Lemma \toposref{limitscommute} and Example \toposref{tucka}. To prove $(b)$ and $(c)$, we first choose a path $\infty$-topos $\calK^{\Delta^1}$ for $\calK$,
and a geometric morphism $\pi^{\ast}_{R}: \calK^{\Delta^1} \rightarrow \calK^{\Delta^1}_{R}$ as in
Proposition \ref{spurin}, so that composition with $\pi^{\ast}_{R}$ induces a fully faithful embedding
$$ \Fun^{\ast}( \calK^{\Delta^1}_{R}, \calX)
\rightarrow \Fun^{\ast}( \calK^{\Delta^1}, \calX) \simeq \Fun( \Delta^1, \Fun^{\ast}(\calK, \calX) )$$
for every $\infty$-topos $\calX$, whose essential image consists of the class of morphisms in
$\Fun^{\ast}(\calK, \calX)$ which belong to $\Struct^{\loc}_{\calK}(\calX)$. Applying
$(a)$ to the $\infty$-topoi $\calK^{\Delta^1}$ and $\calK^{\Delta^1}_{R}$, we conclude that
the collection of morhisms in $\Fun^{\ast}(\calK, \calX)$ which belong to
$\Struct^{\loc}_{\calK}(\calX)$ is stable under small filtered colimits, which immediately implies $(b)$ and $(c)$.
\end{proof}

We now discuss some ideas which are specific to the theory of geometries.

\begin{proposition}\label{deeder}
Let $\calG$ be a geometry, and let $p: \LGeo(\calG) \rightarrow \LGeo$ denote the projection map.
Suppose that $q: K^{\triangleleft} \rightarrow \LGeo(\calG)$ is a small diagram with the following properties:
\begin{itemize}
\item[$(1)$] The composition $p \circ q: K^{\triangleleft} \rightarrow \LGeo$ is a limit diagram.
\item[$(2)$] The functor $q$ carries each morphism in $K^{\triangleleft}$ to a $p$-coCartesian morphism in $\LGeo(\calG)$.
\end{itemize}
Then $q$ is both a limit diagram and a $p$-limit diagram.
\end{proposition}

The proof requires the following lemma.

\begin{lemma}\label{kooper}
Let $p: X \rightarrow S$ be a coCartesian fibration of simplicial sets classified
by a diagram $\chi: S \rightarrow \Cat_{\infty}$. Let 
$q: K^{\triangleleft} \rightarrow X$ be a diagram with the following properties:
\begin{itemize}
\item[$(1)$] The composition $\chi \circ p \circ q: K^{\triangleleft} \rightarrow \Cat_{\infty}$ is a limit diagram.
\item[$(2)$] The diagram $q$ carries each edge of $K^{\triangleleft}$ to
a $p$-coCartesian morphism in $X$.
\end{itemize}
Then $q$ is a $p$-limit diagram.
\end{lemma}

\begin{proof}
Using Corollary \toposref{tttroke}, we may reduce to the case where $S$ is an $\infty$-category
(so that $X$ is also an $\infty$-category). Choose a categorical equivalence
$K \rightarrow K'$ which is a monomorphism of simplicial sets, where $K'$ is an $\infty$-category.
Since $X$ is an $\infty$-category, the map $q$ factors through ${K'}^{\triangleleft}$. We may therefore replace $K$ by $K'$ and thereby reduce to the case where $K$ is an $\infty$-category.
In view of Corollary \toposref{pannaheave}, we may replace $S$ by
$K^{\triangleleft}$ (and $X$ by the pullback $X \times_{S} K^{\triangleleft}$) and thereby
reduce to the case where $p \circ q$ is an isomorphism.

Consider the map $\pi: K^{\triangleleft} \rightarrow (\Delta^0)^{\triangleleft} \simeq \Delta^1$.
Since $K$ is an $\infty$-category, the map $\pi$ is a Cartesian fibration of simplicial sets.
Let $p': \calC \rightarrow \Delta^1$ be the ``pushforward'' of the coCartesian fibration
$p$, so that $\calC$ is characterized by the universal mapping property
$$ \Hom_{ \Delta^1}(Y, \calC) \simeq \Hom_{ K^{\triangleleft}}( Y \times_{ \Delta^1} K^{\triangleleft}, X).$$
Corollary \toposref{presalad} implies that $p'$ is a coCartesian fibration, associated to some
functor $f$ from $\calC_0 = \calC \times_{ \Delta^1} \{0\}$ to $\calC_1 = \calC \times_{ \Delta^1} \{1\}$. 
We can identify $\calC_0$ with the fiber of $p$ over the cone point of $K^{\triangleleft}$, and
$\calC_1$ with the $\infty$-category of sections of $p$ over $K$. Let $\calC'_1$ denote the full subcategory of $\calC_1$ spanned by the coCartesian sections. Combining Corollary \toposref{presalad}, Proposition \toposref{charcatlimit}, and assumption $(1)$, we deduce that
$f$ determines an equivalence $\calC_0 \rightarrow \calC'_1$.

Let $q_0 = q | K$. We can identify $q_0$ with an object $C \in \calC_1$, and $q$ with a morphism
$\alpha: C' \rightarrow C$ in $\calC$. We then have a commutative diagram
$$ \xymatrix{ \calC_0 \times_{ X } X_{/q} \ar[r] \ar[d] & \calC_0 \times_{ X } X_{/q_0} \ar[d] \\
\calC_0 \times_{ \calC } \calC_{/\alpha} \ar[r] & \calC_0 \times_{ \calC } \calC_{/C}. }$$
We wish to show that the upper horizontal map is a categorical equivalence. Since the vertical maps are isomorphisms, it will suffice to show that the lower horizontal map is a categorical equivalence.
In other words, we wish to show that for every object $C_0 \in \calC_0$, composition with
$\alpha$ induces a homotopy equivalence $\bHom_{\calC_0}(C_0, C') \rightarrow \bHom_{\calC}(C_0, C)$. We have a commutative diagram (in the homotopy category of spaces)
$$ \xymatrix{ \bHom_{\calC_0}(C_0, C') \ar[r] \ar[d] & \bHom_{\calC}( C_0, C) \ar[d] \\
\bHom_{ \calC_1}( f C_0, f C' ) \ar[r] & \bHom_{\calC}( fC_0, C). }$$
Here the right vertical map is a homotopy equivalence. Since $f$ is fully faithful, the left vertical map is also a homotopy equivalence. It therefore suffices to show that the bottom horizontal map is a homotopy equivalence: in other words, that $\alpha$ induces an equivalence $f C' \rightarrow C$.
This is simply a translation of condition $(2)$.
\end{proof}

\begin{proof}[Proof of Proposition \ref{deeder}]
We will prove that $q$ is a $p$-limit diagram; it will then follow from Proposition \toposref{basrel} and
$(1)$ that $q$ is a limit diagram. Let $\calG_0$ denote the discrete geometry having the same underlying $\infty$-category as $\calG$. We can identify $\LGeo(\calG)$ with a subcategory
of $\LGeo(\calG_0)$. Moreover, the induced map $q_0: K^{\triangleleft} \rightarrow \LGeo(\calG_0)$
still satisfies $(1)$ and $(2)$. We will prove the following:
\begin{itemize}
\item[$(\ast)$] Let $v$ denote the cone point of $K^{\triangleleft}$. Suppose that
$(\calX, \calO_{\calX}) \in \LGeo(\calG)$ is an object equipped with a morphism
$\alpha: (\calX, \calO_{\calX}) \rightarrow q(v)$ in $\LGeo(\calG_0)$ 
such that, for every vertex $v_0 \in K$, the composite map
$\alpha_0: (\calX, \calO_{\calX}) \rightarrow q(v_0)$ belongs to $\LGeo(\calG)$. Then
the morphism $\alpha$ belongs to $\LGeo(\calG)$.
\end{itemize}
To prove assertion $(\ast)$, it suffices to observe that a diagram in
the underlying $\infty$-topos of $q(v)$ is a pullback square if and only if its image in
the underlying $\infty$-topos of each $q(v_0)$ is a pullback square, by virtue of assumption
$(1)$. It follows from $(\ast)$ that if $q_0$ is a $p$-limit diagram in $\LGeo(\calG_0)$, then
$q$ is a $p$-limit diagram in $\LGeo(\calG)$. We may therefore replace $\calG$ by $\calG_0$ and thereby reduce to the case where $\calG$ is discrete.

Let $\calC = K^{\triangleleft} \times_{ \LGeo } \overline{\LGeo}$, and let
$\calC' = K^{\triangleleft} \times_{ \LGeo} \LGeo(\calG)$. Let 
$\calD \subseteq \Fun_{K^{\triangleleft}}(K^{\triangleleft}, \calC)$ be the full subcategory
spanned by the coCartesian sections, and let $\calD' \subseteq \Fun_{ K^{\triangleleft} }( K^{\triangleleft}, \calC')$ be defined similarly. Using Propositions \toposref{colimtopoi}, \toposref{charcatlimit}, and assumption $(1)$, we deduce that the evaluation map $\calD \rightarrow \calC_{v}$ is an equivalence of $\infty$-categories.
It follows that the evaluation map
$$ \calD' \simeq \Fun^{\lex}( \calG, \calD) \rightarrow \Fun^{\lex}( \calG, \calC_{v}) \simeq
\calC'_{v}$$
is also an equivalence of $\infty$-categories. Invoking Proposition \toposref{charcatlimit} again, we
deduce that the coCartesian fibration $\calC' \rightarrow K^{\triangleleft}$ is classified by
a limit diagram $K^{\triangleleft} \rightarrow \widehat{\Cat}_{\infty}$. The desired result now follows from Lemma \ref{kooper}.
\end{proof}

\begin{definition}
Let $\calG$ be a geometry, $\calX$ an $\infty$-topos, and $n \geq -2$ an integer. We will say that a $\calG$-structure $\calO: \calG \rightarrow \calX$ is {\it $n$-truncated} if
$\calO(X) \in \calX$ is $n$-truncated, for every $X \in \calG$. We let $\Struct_{\calG}^{\leq n}(\calX)$ denote the full subcategory of $\Struct_{\calG}(\calX)$ spanned by the $n$-truncated $\calG$-structures on $\calX$.
\end{definition}

\begin{remark}
If $\calG$ is $n$-truncated, then every $\calG$-structure on every $\infty$-topos $\calX$ is
$n$-truncated. This follows immediately from Proposition \toposref{eaa}.
\end{remark}

\begin{definition}\label{rab}
Let $f: \calG \rightarrow \calG'$ be a functor between $\infty$-categories which admit finite limits, and let
$n \geq -2$ be an integer. We will say that {\it $f$ exhibits $\calG'$ as an $n$-stub of $\calG$} if the following conditions are satisfied:
\begin{itemize}
\item[$(1)$] The $\infty$-category $\calG'$ is equivalent to an $(n+1)$-category: that is, for every
pair of objects $X, Y \in \calG'$, the space $\bHom_{\calG'}(X,Y)$ is $n$-truncated.
\item[$(2)$] The functor $f$ preserves finite limits.
\item[$(3)$] Let $\calC$ be an $\infty$-category which admits finite limits, and suppose that
$\calC$ is equivalent to an $(n+1)$-category. Then composition with $f$ induces an equivalence of $\infty$-categories
$$ \Fun^{\lex}( \calG', \calC) \rightarrow \Fun^{\lex}( \calG, \calC).$$ 
\end{itemize}
\end{definition}

Let $\calG$ be an $\infty$-category which admits finite limits. It is clear that an $n$-stub of $\calG$ is determined uniquely up to equivalence, provided that it exists. For the existence, we have the following result:

\begin{proposition}\label{snubber}
Let $\calG$ be an $\infty$-category which admits finite limits. Then there exists a functor
$f: \calG \rightarrow \calG'$ which exhibits $\calG'$ as an $n$-stub of $\calG$. Moreover, if $\calG$ is small, then we may assume that $\calG'$ is also small.
\end{proposition}

The proof will require the following preliminary:

\begin{lemma}\label{sull}
Let $\calC$ be an $\infty$-category which admits finite limits. Suppose that $\calC$ is
equivalent to an $n$-category, for some integer $n$. Then $\calC$ is idempotent complete.
\end{lemma}

\begin{proof}
Without loss of generality, we may suppose that $n \geq 1$ and that $\calC$ is an $n$-category.
Let $\Idem$ denote the $\infty$-category defined in \S \toposref{retrus}. We wish to show that every
diagram $p: \Idem \rightarrow \calC$ admits a limit in $\calC$. Equivalently, we must show that
the $\infty$-category $\calC_{/p}$ admits a final object. Let $p_0$ denote the restriction of $p$ to the
$n$-skeleton of $\Idem$. Since $\calC$ is an $n$-category, the restriction map
$\calC_{/p} \rightarrow \calC_{/p_0}$ is an isomorphism of simplicial sets. It will therefore suffice to show that $\calC_{/p_0}$ admits a final object. In other words, we must show that the diagram $p_0$
admits a limit in $\calC$. This follows from our assumption that $\calC$ admits finite limits, since the
$n$-skeleton $\sk^{n} \Idem$ is a finite simplicial set.
\end{proof}

\begin{proof}[Proof of Proposition \ref{snubber}]
Enlarging the universe if necessary, we may suppose that $\calG$ is small. Let $\calC = \Ind(\calG^{op})$, so that $\calC$ is a compactly generated presentable $\infty$-category. Corollary \toposref{hunterygreen} implies that the full subcategory $\tau_{\leq n} \calC$ is again compactly generated, and that the truncation functor $\tau_{\leq n}: \calC \rightarrow \tau_{\leq n} \calC$ preserves compact objects. Let $\calG'$ denote the opposite of the (essentially small) $\infty$-category of compact objects of $\tau_{\leq n} \calC$. It follows that the composition
$$ \calG \stackrel{j}{\rightarrow} \calC^{op} \stackrel{\tau_{\leq n}}{\rightarrow} (\tau_{\leq n} \calC)^{op}$$
factors through a functor $f: \calG \rightarrow \calG'$ (here $j$ denotes the opposite of the Yoneda embedding $\calG^{op} \rightarrow \Ind(\calG^{op})$). We claim that $f$ exhibits $\calG'$ as an $n$-stub of $\calG$. 

It is clear that $f$ preserves finite limits and that $\calG'$ is equivalent to an $(n+1)$-category. To complete the proof, we must show that if $\calG''$ is an $\infty$-category which admits finite limits, and $\calG''$ is equivalent to an $(n+1)$-category, then composition with $f$ induces an equivalence of $\infty$-categories
$$ \Fun^{\lex}( \calG', \calG'') \rightarrow \Fun^{\lex}(\calG, \calG'').$$
By a direct limit argument (using the fact that $\calG$ and $\calG'$ are essentially small), we may reduce to the case where $\calG''$ is itself small. Let $\calD = \Ind( {\calG''}^{op} )$.
We have a homotopy commutative diagram
$$ \xymatrix{ \Fun^{\lex}( \calG', \calG'')^{op} \ar[r]^{h} \ar[d]^{\phi'} & \Fun^{\lex}( \calG, \calG'')^{op} \ar[d]^{\phi} \\
\Fun^{\rex}( {\calG'}^{op}, \calD) \ar[r]^{h'} & \Fun^{\rex}( \calG^{op}, \calD) \\
\LFun( \tau_{\leq n} \calC, \calD) \ar[u]^{\psi'} \ar[r]^{h''} & \LFun(\calC, \calD) \ar[u]^{\psi} }$$
We wish to show that $h$ is an equivalence of $\infty$-categories. The functor $h''$ is an equivalence of $\infty$-categories by Corollary \toposref{truncprop}. Since $\psi$ and $\psi'$ are equivalences
(Propositions \toposref{intprop} and \toposref{sumatch}), it follows that $h'$ is an equivalence.
The functors $\phi$ and $\phi'$ are fully faithful. To complete the proof, it will suffice to show that
their essential images are identified via the equivalence provided by $h'$. In other words, we must show that if $g: {\calG'}^{op} \rightarrow \calD$ is a right exact functor such that
$g \circ f$ factors through the essential image of the Yoneda embedding $j': {\calG''}^{op}
\rightarrow \calD$, then $g$ factors through the essential image of $j'$.

Let $\calE$ be a minimal model for the full subcategory of ${\calG'}^{op}$ spanned by those objects
$X$ such that $gX$ belongs to the essential image of $j'$. Since $j'$ is right exact,
the essential image of $\calE$ in ${\calG'}^{op}$ is stable under finite colimits. It follows that
$\Ind(\calE)$ admits a fully faithful, colimit preserving embedding into $\tau_{\leq n} \calC$.
The essential image of this embedding contains the essential image of the composition
$$ \calG^{op} \stackrel{j}{\rightarrow} \calC \stackrel{\tau_{\leq n}}{\rightarrow} \tau_{\leq n} \calC,$$
and therefore (since it is stable under small colimits) contains the essential image of $\tau_{\leq n}$. 
It follows that $\Ind(\calE)$ is equivalent to $\tau_{\leq n} \calC$. Using Lemma \toposref{stylus}, we conclude that the inclusion $i: \calE \subseteq {\calG'}^{op}$ exhibits $\calG'$ as an idempotent completion of $\calE$. We now invoke Lemma \ref{sull} to complete the proof.
\end{proof}

\begin{definition}
Let $\calG$ be a geometry. We will say that a transformation of geometries
$f: \calG \rightarrow \calG_{\leq n}$ {\it exhibits $\calG_{\leq n}$ as an $n$-stub of $\calG$} if
the following conditions are satisfied:
\begin{itemize}
\item[$(1)$] The underlying functor between $\infty$-categories exhibits $\calG_{\leq n}$ as
an $n$-stub of $\calG$, in the sense of Definition \ref{rab}.
\item[$(2)$] The class of $\calG_{\leq n}$-admissible morphisms and the Grothendieck topology on $\calG_{\leq n}$ are generated by the functor $f$, in the sense of Remark \ref{gener}.
\end{itemize}
\end{definition}

The following result is a more or less immediate consequence of the definitions:

\begin{proposition}\label{lubba1}
Let $f: \calG \rightarrow \calG_{\leq n}$ be a transformation of geometries which exhibits $\calG_{\leq n}$
as an $n$-stub of $\calG$. Then, for every $\infty$-topos $\calX$, composition with
$f$ induces a equivalences of $\infty$-categories
$$\Struct_{ \calG_{\leq n}}(\calX) \rightarrow \Struct^{\leq n}_{\calG}( \calX )$$
$$ \Struct_{\calG_{\leq n}}^{\loc}(\calX) \rightarrow \Struct^{\leq n}_{\calG}(\calX) \cap \Struct^{\loc}_{\calG}(\calX).$$
\end{proposition}

From this we deduce the following Corollary, which also appeared implicitly in the proof of Proposition \ref{snubber}: 

\begin{corollary}
Let $\calG$ and $\calG_{\leq n}$ be small $\infty$-categories which admit finite limits, and let
$f: \calG \rightarrow \calG_{\leq n}$ exhibit $\calG_{\leq n}$ as an $n$-stub of $\calG$. Then
composition with $f$ induces a fully faithful embedding
$\Ind(\calG^{op}_{\leq n}) \rightarrow \Ind(\calG^{op})$ whose essential image is the full subcategory of
$\Ind(\calG^{op})$ spanned by the $n$-truncated objects.
\end{corollary}

\begin{proof}
Regard $\calG$ as a discrete geometry, so that the stub $\calG_{\leq n}$ inherits also the structure of a discrete geometry. We observe that there are canonical isomorphisms
$$ \Ind(\calG^{op}) \simeq \Struct_{\calG}(\SSet) \quad \quad \Ind( \calG_{\leq n} ) \simeq \Struct_{\calG^{op}_{\leq n}}(\SSet).$$
It will suffice to show that the first of these isomorphisms
restricts to an isomorphism $\tau_{\leq n} \Ind(\calG^{op}) \simeq \Struct_{\calG}^{\leq n}(\SSet)$.
In other words, we must show that a left-exact functor $F: \calG \rightarrow \SSet$ is
$n$-truncated as an object of $\Ind(\calG^{op})$ if and only it takes values in the full subcategory
$\tau_{ \leq n} \SSet$. 

Let $e: \Ind(\calG^{op})^{op} \rightarrow \SSet$ denote the functor represented by $F$. Then $F$ is
equivalent to the composition $\calG \stackrel{j}{\rightarrow} \Ind(\calG^{op})^{op} \stackrel{e}{\rightarrow} \SSet$. If $F \in \tau_{\leq n} \Ind(\calG^{op})$, then $e$ factors through $\tau_{\leq n} \SSet$, so that $F$ also factors through $\tau_{\leq n} \SSet$. 

To prove the converse, let
$\calC$ denote the full subcategory of $\Ind( \calG^{op} )$ spanned by those objects
$X$ such that $e(X) \simeq \bHom_{ \Ind(\calG^{op})}( X, F)$ is $n$-truncated. We wish to show that
$\calC = \Ind(\calG^{op})$. Since
the functor $e$ carries colimits in $\Ind(\calG^{op})$ to limits in $\SSet$, we conclude that
$\calC$ is stable under colimits in $\Ind(\calG^{op})$. It will therefore suffice to show that
$\calC$ contains the essential image of the Yoneda embedding $\calG^{op} \rightarrow \Ind(\calG^{op})$, which is equivalent to the assertion that $F \simeq e \circ j$ factors through $\tau_{\leq n} \SSet$.
\end{proof}

\section{Scheme Theory}\label{secscheme}

The theory of geometries presented in \S \ref{geo} can be regarded as a generalization
of the theory of locally ringed spaces. In this section, we will define
a full subcategory $\Sch(\calG) \subseteq \LGeo(\calG)^{op}$, which we will call the
{\it $\infty$-category of $\calG$-schemes}; in the special case where $\calG$ is the geometry
$\calG_{\Zar}$ of Example \ref{summ}, this will recover (a mild generalization of) the classical
theory of schemes.

Our first step is to describe the class of {\em affine} $\calG$-schemes. We take our cue from classical scheme theory: if $A$ is a commutative ring, then the affine scheme $( \SSpec A, \calO_{\SSpec A})$
is characterized by the following universal property:
\begin{itemize}
\item[$(\ast)$] For every locally ringed space $(X, \calO_X)$, the canonical map
$$ \Hom_{ \LRingSpace}( (X, \calO_X), ( \SSpec A, \calO_{ \SSpec A}))
\rightarrow \Hom_{\Comm}( A, \Gamma(X, \calO_X) )$$ is a bijection.
\end{itemize}
To make $(\ast)$ appear more symmetric, we observe that $\Hom_{\Comm}( A, \Gamma(X, \calO_X))$
can be identified with the set of maps from $(X, \calO_X)$ to $(\ast, A)$ in the category of
{\em ringed spaces}. This raises the following general question: given a transformation
of geometries $f: \calG' \rightarrow \calG$, does the induced functor $\LGeo(\calG)^{op}
\rightarrow \LGeo(\calG')^{op}$ admit a right adjoint? We will give an affirmative answer to this question in \S \ref{relspec} (Theorem \ref{exspec}), using a rather abstract construction. Our primary interest is in the situation where $\calG' = \calG$ (the discrete geometry with the same underlying $\infty$-category as $\calG$), and in the restriction of this right adjoint to $\calG'$-structures on the final $\infty$-topos $\SSet$. In this case, we obtain a functor $\Spec^{\calG}: \Ind(\calG^{op}) \rightarrow \LGeo(\calG)$. 

If $A$ is a commutative ring, the affine scheme $( \SSpec A, \calO_{\SSpec A})$ is characterized
by $(\ast)$ but can also be constructed by a very concrete procedure. For example, the underlying topological space $\SSpec A$ can be identified with the set of all prime ideals of $A$, endowed with the Zariski topology. In \S \ref{abspec}, we will generalize this construction to the setting of $\calG$-schemes, where $\calG$ is any geometry, thereby obtaining an explicit model for 
the functor $\Spec^{\calG}$ which is easy to compare with the classical theory of Zariski spectra.

Given the existence of the spectrum functor $\Spec^{\calG}$, we can proceed to define the $\infty$-category $\Sch(\calG)$: it is the full subcategory of $\LGeo(\calG)^{op}$ spanned by those pairs
$(\calX, \calO_{\calX})$ which, locally on the $\infty$-topos $\calX$, belong to the essential image
of $\Spec^{\calG}$. We will then proceed in 
\S \ref{geo6} to define the $\infty$-category $\Sch(\calG)$ and establish some of its basic properties (for example, the existence of finite limits in $\Sch(\calG)$).

The definition of the class of $\calG$-schemes presented in \S \ref{geo6} is analogous to
the usual definition of a scheme as a topological space $X$ equipped with a sheaf
$\calO_X$ of commutative rings. There is another equally important way to think about the category of schemes: every scheme $(X, \calO_X)$ defines a covariant functor
$F_X: \Comm \rightarrow \Set$, described by the formula
$$ F_X(A) = \Hom( (\SSpec A, \calO_{\SSpec A}), (X, \calO_X) ).$$
This construction determines a fully faithful embedding from the category of schemes
to the functor category $\Fun( \Comm, \Set)$. In other words, instead of viewing a scheme
as a special kind of ringed topological space, we may view a scheme as a special kind of
functor from commutative rings to sets. In \S \ref{geo7}, we will prove an analogue of this statement in the setting of $\calG$-schemes: namely, there exists a fully faithful (Yoneda-style) embedding
$$ \phi: \Sch(\calG) \rightarrow \Fun( \Ind(\calG^{op}), \SSet).$$
We can therefore identify $\Sch(\calG)$ with the essential image of this functor, and thereby view $\calG$-schemes as special kinds of space-valued functors on $\Ind(\calG^{op})$. 

Our final objective in this section is to give some examples which illustrate the relationship of our theory with classical geometry (we will consider $\infty$-categorical variations on these examples in \S \ref{app6}, and still more exotic situations in future papers).
Fix a commutative ring $k$. The category $\Comm_{k}$ of commutative $k$-algebras admits many Grothendieck topologies. Of particular interest to us will be the Zariski and \etale topologies on
$\Comm_{k}$. To each of these, we can associate a geometry $\calG$ whose underlying $\infty$-category is equivalent to the (nerve of the) category of affine $k$-schemes of finite presentation. We will study these geometries in \S \ref{exzar} and \S \ref{exet}. In the first
case, the theory of $\calG$-schemes will recover the usual theory of $k$-schemes; in the second, we will recover the theory of Deligne-Mumford stacks over $k$. 

\subsection{Construction of Spectra: Relative Version}\label{relspec}

Let $A$ be a commutative ring. The {\it Zariski spectrum} $\SSpec A$ is defined to be the set of all prime ideals ${\mathfrak p} \subseteq A$. We regard $\SSpec A$ as a topological space, endowed with the {\it Zariski topology} having a basis of open sets $U_{f} = \{ \mathfrak{p} \subseteq A | f \notin \mathfrak{p} \}$, where $f$ ranges over the elements of $A$. There is a sheaf of commutative rings $\calO_{\SSpec A}$ on $\SSpec A$, whose value on an open subset $U_{f} \subseteq \SSpec A$ has the value $\calO_{\SSpec A}(U_{f}) \simeq A[ \frac{1}{f} ]$. The locally ringed space $(\SSpec A, \calO_{\SSpec A})$ can also be described by a universal property: it is univeral among locally ringed spaces
$(X, \calO_X)$ equipped with a ring homomorphism $A \rightarrow \Gamma(X, \calO_X)$.

We wish to obtain a generalization of this picture. We proceed in several steps.

\begin{itemize}
\item[$(1)$] It is in some sense coincidental that $\SSpec A$ is described by a topological space. What arises more canonically is the lattice of open subsets of $\SSpec A$, which is generated by basic open sets of the form $U_{f}$. This lattice naturally forms a {\it locale}, or a
$0$-topos (see \S \toposref{0topoi} for a discussion of this notion). It happens that this locale has enough points, and can therefore be described as the lattice of open subsets of a topological space. However, there are various reasons we might want to disregard this fact:

\begin{itemize}
\item[$(a)$] The existence of enough points for $\SSpec A$ is equivalent to the assertion that every nonzero commutative ring contains a prime ideal, and the proof of this assertion requires the axiom of choice. 

\item[$(b)$] In {\em relative} situations (see $(2)$ below), the relevant construction may well fail to admit enough points, even if the axiom of choice is assumed. However, the underlying locale (and its associated sheaf theory) are still well-behaved.

\item[$(c)$] If we wish to replace the Zariski topology by some other topology (such as the
\etale topology), then we are forced to work with $\SSpec A$ as a {\em topos} rather than simply as a topological space: the category of \etale sheaves on $\SSpec A$ is not generated by subobjects of the final object.
\end{itemize}

When we study derived algebraic geometry, we will want to study sheaves
on $\SSpec A$ of a higher-categorical nature, such as sheaves of spaces or sheaves of spectra.
For these purposes, it will be most convenient to regard $\SSpec A$ as an $\infty$-topos, rather than as a topological space.

\item[$(2)$] Let $A$ be a commutative ring. We can regard $A$ as defining a {\em sheaf} of commutative rings over the space $\ast$ consisting of a single point. Then, for any ringed
space $(X, \calO_X)$, we can identify ring homomorphisms $A \rightarrow \Gamma(X; \calO_X)$
with maps from $(X, \calO_X)$ to $(\ast, A)$ in the category $\RingSpace$ of ringed spaces.
Let $\LRingSpace$ denote the subcategory of {\em locally} ringed spaces. Then we can
reformulate condition $(\ast)$ as follows: for every commutative ring $A$ and every locally ringed
space $(X, \calO_X)$, we have a canonical bijection
$$ \Hom_{\LRingSpace}( (X, \calO_X), ( \SSpec A, \calO_{ \SSpec A}) )
\rightarrow \Hom_{ \RingSpace}( (X, \calO_X), (\ast, A) ).$$
More generally, we might ask if there is an analogue of the locally ringed space $( \SSpec A, \calO_{ \SSpec A})$ for {\em any} ringed space $(Y, \calA)$. In other words, we might ask if
the inclusion $\LRingSpace \subseteq \RingSpace$ admits a right adjoint.

\item[$(3)$] For any topological space $X$, a sheaf $\calO$ of local commutative rings
on $X$ can be identified with a $\calG_{\Zar}$-structure on the $\infty$-topos $\Shv(X)$, where
$\calG_{\Zar}$ is the geometry described in Example \ref{summ}. This construction allows us to identify the usual category $\LRingSpace$ with a full subcategory of the $\infty$-category
$\LGeo(\calG_{\Zar})^{op}$ of {\em locally ringed $\infty$-topoi}. Similarly, the category of ringed spaces
can be identified with a full subcategory of $\LGeo(\calG')^{op}$, where $\calG'$ denotes
denotes the discrete geometry having the same underlying $\infty$-category as $\calG_{\Zar}$.
The evident transformation of geometries $\calG' \rightarrow \calG_{\Zar}$ induces a functor
$$ \LGeo( \calG_{\Zar})^{op} \rightarrow \LGeo(\calG')^{op},$$
generalizing the inclusion $\LRingSpace \subseteq \RingSpace$ of $(2)$. More generally,
we can consider an analogous restriction functor associated to {\em any} 
transformation of geometries $\calG' \rightarrow \calG_{\Zar}$.
\end{itemize}

We can now state the main result of this section.

\begin{theorem}\label{exspec}
Let $f: \calG \rightarrow \calG'$ be a transformation of geometries. Then
the induced functor $\LGeo(\calG') \rightarrow \LGeo(\calG)$ admits a left adjoint.
\end{theorem}

\begin{definition}\label{defspect}
Given a transformation of geometries $f: \calG \rightarrow \calG'$, we let
$\Spec_{\calG}^{\calG'}$ denote a left adjoint to the restriction functor
$\LGeo(\calG') \rightarrow \LGeo(\calG)$. We will refer to $\Spec_{\calG}^{\calG'}$ as the
{\it relative spectrum functor} associated to $f$.

Let $\calG$ be a geometry, and let $\calG_0$ be the discrete geometry with the same underlying $\infty$-category as $\calG$. We let $\Spec^{\calG}$ denote the composition
$$ \Ind({\calG}^{op}) \rightarrow \LGeo(\calG_0) \stackrel{ \Spec^{\calG_0}_{\calG} }{\rightarrow} \LGeo(\calG).$$
Here the first functor is given by the identification of $\Ind(\calG^{op}) \simeq \Struct_{\calG_0}(\SSet)$
with the identification of $\Struct_{\calG_0}(\SSet)$ with the fiber of $\LGeo(\calG_0) \times_{ \LGeo } \{ \SSet \}$. We will refer to $\Spec^{\calG}$ as the {\it absolute spectrum functor} associated to the geometry $\calG$.
\end{definition}

Our goal in this section is to prove Theorem \ref{exspec}. The basic idea is reasonably simple but perhaps unenlightening: it uses somewhat abstract constructions such as fiber products of
$\infty$-topoi, and therefore yields a poor understanding of the resulting object. We will address this inadequacy in \S \ref{abspec} by giving a much more explicit construction of 
the absolute spectrum functor $\Spec^{\calG}$ (see Theorem \ref{scoo}).

For the following discussion, let us fix a transformation $f: \calG \rightarrow \calG'$ of geometries.
Given objects $(\calX, \calO) \in \LGeo(\calG)$, and $(\calX', \calO') \in \LGeo(\calG')$, we will say that a
morphism $\theta: (\calX, \calO) \rightarrow (\calX', \calO' \circ f)$ {\it exhibits $(\calX', \calO')$ as a
relative spectrum of $(\calX, \calO)$} if, for every object $(\calY, \calO_{\calY}) \in \LGeo(\calG')$,
composition with $\theta$ induces a homotopy equivalence
$$ \bHom_{ \LGeo(\calG)}( (\calX', \calO'), (\calY, \calO_{\calY}) ) \rightarrow
\bHom_{ \LGeo(\calG')}( (\calX, \calO), (\calY, \calO_{\calY} \circ f)).$$
Theorem \ref{exspec} can be formulated as follows: for every object
$(\calX, \calO) \in \LGeo(\calG)$, there exists an object $(\calX', \calO') \in \LGeo(\calG')$ and a
morphism $(\calX, \calO) \rightarrow (\calX', \calO' \circ f)$ which exhibits $(\calX', \calO')$ as a
relative spectrum of $(\calX, \calO)$. Our first step is to reduce to the proof to a universal case.

\begin{lemma}\label{specred}
Let $f: \calG \rightarrow \calG'$ be a transformation of geometries. Suppose that
$(\calX, \calO) \in \LGeo(\calG)$, $(\calX', \calO') \in \LGeo(\calG')$, and that
$\alpha: (\calX, \calO) \rightarrow (\calX', \calO' \circ f)$ is a morphism in $\LGeo(\calG)$
which exhibits $(\calX', \calO')$ as a relative spectrum of $(\calX, \calO)$.
Suppose given a pushout diagram
$$ \xymatrix{ \calX \ar[d] \ar[r]^{{g}^{\ast}} & \calY \ar[d] \\
\calX' \ar[r]^{{g'}^{\ast}} & \calY' }$$
in the $\infty$-category $\LGeo$. Then the induced map
$(\calY, {g}^{\ast} \circ \calO) \rightarrow (\calY', {g'}^{\ast} \circ \calO' \circ f)$ exhibits
$(\calY', {g'}^{\ast} \circ \calO')$ as a relative spectrum of $( \calY, {g}^{\ast} \circ \calO)$.
\end{lemma}

The proof is a simple matter of untangling definitions. To apply Lemma \ref{specred}, we use the following observation: let $\calO_0: \calG \rightarrow \calK$ be a universal $\calG$-structure on
an $\infty$-topos $\calK$. Then {\em any} object $(\calY, \calO) \in \LGeo(\calG)$ is equivalent to
$(\calY, \pi^{\ast} \calO_0)$ for some geometric morphism $\pi^{\ast}: \calK \rightarrow \calY$,
which is uniquely determined up to homotopy. Hence, to prove that $(\calY, \calO)$ admits a relative spectrum, it will suffice to show that $(\calK, \calO_0)$ admits a relative spectrum. 

We observe that for every object $(\calX, \calO) \in \LGeo(\calG)$, the mapping space
$\bHom_{ \LGeo(\calG)}( (\calK, \calO_0), (\calX, \calO) )$ can be identified with the largest
Kan complex contained in $\Struct^{\loc}_{\calG}(\calX)^{/ \calO}$. Theorem \ref{exspec} is now an immediate consequence of the following result:

\begin{proposition}\label{ingred}
Let $f: \calG \rightarrow \calG'$ be a transformation of geometries. Then there
exists an $\infty$-topos $\calK_{\calG}^{\calG'}$, objects $\calO \in \Struct_{\calG}(\calK^{\calG'}_{\calG})$ and $\calO' \in \Struct_{\calG'}(\calK^{\calG'}_{\calG})$, and a local morphism of $\calG$-structures
$\alpha: \calO \rightarrow \calO' \circ f$ with the following universal property:
for every object $(\calX, \calO_{\calX}) \in \LGeo(\calG')$, composition with $\alpha$ induces a homotopy equivalence from
$\bHom_{ \LGeo(\calG') }( ( \calK^{\calG'}_{\calG}, \calO'), (\calX, \calO_{\calX}) )$ to the largest Kan complex contained in $\Struct^{\loc}_{\calG}(\calX)^{/\calO_{\calX} \circ f}$. 
\end{proposition}

For a general object $(\calX, \calO) \in \LGeo(\calG)$, we then have
$\Spec^{\calG'}_{\calG}(\calX, \calO) \simeq ( \calX \times_{ \calK } \calK^{\calG'}_{\calG}, \calO')$, where
the fiber product is taken in $\RGeom$, and $\calO'$ is the pullback of the $\calG'$-structure on
$\calK^{\calG'}_{\calG}$ appearing in the statement of Proposition \ref{ingred}.

We will prove Proposition \ref{ingred} via a somewhat lengthy construction.
First, choose a correspondence $\calM \rightarrow \Delta^1$ associated to the functor
$f: \calG \rightarrow \calG'$, so that we have isomorphisms $\calG \simeq \calM \times_{ \Delta^1 } \{0\}$ and $\calG' \simeq \calM \times_{ \Delta^1} \{1\}$. 

\begin{notation}
For every $\infty$-topos $\calY$, we define a {\it $\calM$-structure} on $\calY$ to be a functor
$\calO_Y: \calM \rightarrow \calY$ such that $\calO_{Y}| \calG \in \Struct_{\calG}(\calY)$ and
$\calO_{Y} | \calG' \in \Struct_{\calG'}(\calY)$. We let $\Struct_{\calM}(\calY)$ denote the
full subcategory of $\Fun( \calM, \calY )$ spanned by the $\calM$-structures.
\end{notation}

\begin{lemma}\label{kofm}
There exists a {\em universal} $\calM$-structure $\calO_{\calM}: \calM \rightarrow \calK(\calM)$,
such that for every $\infty$-topos $\calX$, composition with $\calO_{\calM}$ induces an
equivalence of $\infty$-categories
$$ \Fun^{\ast}( \calK(\calM), \calY) \rightarrow \Struct_{\calM}(\calX).$$
\end{lemma}

\begin{proof}
Using Proposition \toposref{cupper1}, we can construct a functor $\beta: \calM \rightarrow \calM'$ with the following properties:
\begin{itemize}
\item[$(a)$] The $\infty$-category $\calM'$ is small and admits finite limits.
\item[$(b)$] The restrictions $\beta | \calG$ and $\beta | \calG'$ are left exact.
\item[$(c)$] For every $\infty$-category $\calC$ which admits finite limits, composition with
$\beta$ induces an equivalence of $\infty$-categories
$$ \Fun^{\lex}( \calM', \calC) \rightarrow \Fun^{\lex}(\calG,\calC) \times_{ \Fun(\calG, \calC) }
\Fun(\calM, \calC) \times_{ \Fun(\calG', \calC) } \Fun^{\lex}(\calG',\calC).$$
\end{itemize}
We now define $\calK(\calM)$ to be the fiber product
$$ \Shv(\calG) \times_{ \calP(\calG) } \calP( \calM') \times_{ \calP(\calG') } \Shv(\calG'),$$
taken in the $\infty$-category $\RGeom$, so that $\calK(\calM)$ is an accessible left exact localization of $\calP(\calM')$. Let $L: \calP(\calM') \rightarrow \calK(\calM)$ be a localization functor, and define
$\calO_{\calM}$ to be the composition
$$ \calM \stackrel{\beta}{\rightarrow} \calM' \stackrel{j}{\rightarrow} \calP(\calM')
\stackrel{L}{\rightarrow} \calK(\calM),$$
where $j$ denotes the Yoneda embedding. It is not difficult to see that $\calO_{\calM}$ has the desired universal property.
\end{proof}

\begin{notation}
For every $\infty$-topos $\calX$, let $\Struct_{\calM}^{(1)}(\calX)$ denote the full subcategory of
$\Struct_{\calM}(\calX)$ consisting of those $\calM$-structures $\calO: \calM \rightarrow \calX$ such that
$\calO$ is a right Kan extension of $\calO | \calG'$. Similarly, we define
$\Struct_{\calM}^{(0)}(\calX)$ to be the full subcategory of $\Struct_{\calM}(\calX)$ spanned by
those $\calM$-structures $\calO: \calM \rightarrow \calX$ which are left Kan extensions of
$\calO| \calG$.
\end{notation}

\begin{lemma}\label{skipun}
\begin{itemize}
\item[$(1)$] For every $\infty$-topos $\calX$, the restriction functors
$$ \Struct^{(0)}_{\calM}(\calX) \rightarrow \Struct_{\calG}(\calX)$$
$$ \Struct^{(1)}_{\calM}(\calX) \rightarrow \Struct_{\calG'}(\calX)$$
are trivial Kan fibrations. 
\item[$(2)$] For every geometric morphism $\pi^{\ast}: \calX \rightarrow \calY$,
the induced map $\Struct_{\calM}(\calX) \rightarrow \Struct_{\calM}(\calY)$ carries
$\Struct^{(0)}_{\calM}(\calX)$ to $\Struct^{(0)}_{\calM}(\calY)$ and
$\Struct^{(1)}_{\calM}(\calX)$ to $\Struct^{(1)}_{\calM}(\calY)$.
\end{itemize}
\end{lemma}

\begin{proof}
In view of Proposition \toposref{lklk}, assertion $(1)$ can be rephrased as follows:
\begin{itemize}
\item[$(1a)$] Let $\calO_0 \in \Struct_{\calG}( \calX)$ and let $\calO: \calM \rightarrow \calX$
be a left Kan extension of $\calO_0$. Then $\calO| \calG' \in \Struct_{\calG'}(\calM)$.
\item[$(1b)$] Let $\calO_1 \in \Struct_{\calG}( \calX)$ and let $\calO: \calM \rightarrow \calX$ be a right Kan extension of $\calO_1$. Then $\calO| \calG \in \Struct_{\calG}(\calM)$.
\end{itemize}
Assertion $(1b)$ follows from the fact that $\calO | \calG \simeq f \circ \calO_0$, and assertion
$(1a)$ is a consequence of the following result, whose proof will be given below:
\begin{lemma}\label{sillyfunt}
Let $\calG$ and $\calG'$ be small $\infty$-categories which admit finite limits, let
$f: \calG \rightarrow \calG'$ be a left exact functor, and let $\calX$ be an $\infty$-topos.
Then left Kan extension along $f$ carries $\Fun^{\lex}(\calG, \calX)$ into
$\Fun^{\lex}(\calG', \calX)$. 
\end{lemma}

To prove $(2)$, we observe that $\pi^{\ast}$ preserves small colimits, and therefore left Kan extensions along inclusions of small $\infty$-categories. To prove that composition with $\pi^{\ast}$ carries
$\Struct^{(1)}_{\calM}(\calX)$ to $\Struct^{(1)}_{\calM}(\calY)$, it suffices to observe that
$\calO \in \Struct_{\calM}(\calZ)$ belongs to $\Struct^{(1)}_{\calM}(\calZ)$ if and only if
$\calO(\alpha)$ is an equivalence in $\calZ$, for every $p$-Cartesian morphism in
$\calM$; here $p: \calM \rightarrow \Delta^1$ denotes the projection.
\end{proof}

We now turn to the proof of Lemma \ref{sillyfunt}. 

\begin{definition}\label{jip}
Let $\calM \rightarrow \Delta^1$ be a (small) correspondence from an $\infty$-category
$\calC = \calM \times_{ \Delta^1} \{0\}$ to another $\infty$-category $\calD = \calM \times_{ \Delta^1} \{1\}$. The Yoneda embedding for $\calM$ determines a functor $\calM^{op} \times \calM \rightarrow \SSet$, which determines by restriction a functor $F: \calC^{op} \times \calD \rightarrow \SSet$, such that $F(C,D)$ is homotopy equivalent to $\bHom_{\calM}(C,D)$. 

We will say that a bifunctor $F': {\calC'}^{op} \times \calD' \rightarrow \SSet$ is {\it associated to $\calM$} if there exist equivalences $\alpha: \calC' \rightarrow \calC$ and $\beta: \calD' \rightarrow \calD$ such
that $F' \circ (\alpha \times \beta)$ is homotopic to $F$. In this case, we will also say that
$\calM$ is {\it associated to $F'$}.
\end{definition}

\begin{example}\label{skai}
Let $f: \calD \rightarrow \calC$ be a functor between small $\infty$-categories. Composing with the Yoneda embedding $\calC \rightarrow \calP(\calC)$, we obtain a map
$\calD \rightarrow \calP(\calC)$ which is adjoint to a bifunctor $F: \calC^{op} \times \calD \rightarrow \SSet$. A correspondence $\calM$ from $\calC$ to $\calD$ is associated to $F$ (in the sense of Definition \ref{jip}) if and only if it is associated to $f$ (in the sense of Definition \toposref{fibas}). 
\end{example}

\begin{remark}\label{repman}
Let $F: {\calC}^{op} \times \calD \rightarrow \SSet$ be a bifunctor associated to a correspondence
$\calM \rightarrow \Delta^1$, and let $\calC_0 \subseteq \calC$, $\calD_0 \subseteq \calD$ be full subcategories. Then the restriction $F_0 = F | (\calC_0^{op} \times \calD_0)$ is associated
to the full subcategory of $\calM$ spanned by the essential images of $\calC_0$ and $\calD_0$.
\end{remark}




\begin{lemma}\label{haberspaz}
Let $\calC$ and $\calD$ be small $\infty$-categories, and let $F: \calC^{op} \times \calD \rightarrow \SSet$ be a functor, so that $F$ determines maps $f^{0}: \calC \rightarrow \calP(\calD^{op})^{op}$
and $f^1: \calD \rightarrow \calP(\calC)$. Let $j: \calC \rightarrow \calP(\calC)$ denote the Yoneda embedding. Then:
\begin{itemize}
\item[$(1)$] The functor $j$ admits a left Kan extension along $f^0$, which we will denote by
$f^0_{!}(j)$. 
\item[$(2)$] The composition of $f^0_{!}(j)$ with the Yoneda embedding 
$j': \calD \rightarrow \calP(\calD^{op})^{op}$ is equivalent to $f^1$.
\end{itemize}
\end{lemma}

\begin{proof}
To prove $(1)$, it suffices to observe that for every object $X \in \calP(\calD^{op})^{op}$, the
fibers of the right fibration $\calP(\calD^{op})^{op}_{/X}$ are essentially small, so that the fiber product
$\calC \times_{ \calP( \calD^{op} )^{op} } \calP( \calD^{op})^{op}_{/X}$ is essentially small; now invoke 
the fact that $\calP(\calC)$ admits small colimits (and Lemma \toposref{kan2}). 

To prove $(2)$, let $\overline{\calM}$ be a correspondence from $\calC$ to $\calP( \calD^{op} )^{op}$
associated to $f^{0}$, and let $\overline{J}: \overline{\calM} \rightarrow \calP(\calC)$ be a left
Kan extension of $j$. Let $\calM'$ denote the full subcategory of $\overline{\calM}$ spanned by the objects of $\calC$ and the essential image of $j'$, let $\calM$ be a minimal model for $\calM'$
(so that $\calM$ is small), and let $J = \overline{J} | \calM$. We observe that $\calM$ is equipped
with equivalences $\alpha: \calC \rightarrow \calM \times_{ \Delta^1 } \{0\}$, $\beta: \calD \rightarrow \calM \times_{ \Delta^1} \{1\}$ which exhibit $\calM$ as associated to the bifunctor $F$. 
Moreover, $j' \circ f^0_{!}(j)$ can be identified with $J \circ \beta$. Invoking Lemma \toposref{kanspaz}, we
can also identify $J \circ \beta$ with the composition
$$ \calD \stackrel{\beta}{\rightarrow} \calM \rightarrow \calP(\calM) \rightarrow \calP(\calC).$$
Since $\calM$ is associated to $F$, this composition is homotopic to $f^1$ as desired.
\end{proof}

\begin{proof}[Proof of Lemma \ref{sillyfunt}]
Let $h: \calG \rightarrow \calX$ be a left exact functor; we wish to show that the left Kan extension
$f_{!}(h): \calG' \rightarrow \calX$ is also left exact. Using Theorem \toposref{charpresheaf}, we may assume without loss of generality that $h$ is a composition
$\calG \stackrel{j}{\rightarrow} \calP(\calG) \stackrel{H}{\rightarrow} \calX$, where
$j$ is the Yoneda embedding and $H$ preserves small colimits. Using Proposition
\toposref{natash}, we conclude that $H$ is left exact. Using Lemma \toposref{kanspaz}, we conclude
that $f_{!}(h)$ is equivalent to the composition
$$ \calG' \stackrel{j'}{\rightarrow} \calP(\calG') \stackrel{\circ f}{\rightarrow} \calP(\calG)
\stackrel{H}{\rightarrow} \calX.$$
As a composition of left-exact functors, we conclude that $f_{!}(h)$ is left exact.
\end{proof}

We now put Lemma \ref{skipun} into practice. Let $\calK(\calM)$ be as in Lemma \ref{kofm}, and 
let $\calK$ and $\calK'$ denote classifying $\infty$-topoi for $\calG$ and $\calG'$, respectively.
By general nonsense, we deduce the existence of geometric morphisms
$$ \calK \stackrel{\phi^{\ast}}{\rightarrow} \calK(\calM) \stackrel{\psi^{\ast}}{\leftarrow} \calK'$$
which determine, for every $\infty$-topos $\calX$, a homotopy commutative diagram
$$ \xymatrix{ \Fun^{\ast}(\calK, \calX) \ar[r] \ar[d] & \Fun( \calK(\calM), \calX) \ar[d] & \Fun^{\ast}(\calK', \calX) \ar[d] \ar[l] \\
\Struct_{\calG}(\calX) \ar[r] & \Struct_{\calM}(\calX) & \Struct_{\calG'}(\calX) \ar[l] }$$
where the vertical maps are categorical equivalences and the lower horizontal maps are obtained
by choosing sections of the trivial fibrations of Lemma \ref{skipun}. 

\begin{notation}\label{sluck}
We let $\calK^{\calG'}_{\calG}$ denote the fiber product
$$ (\calK^{\Delta^1}_{R} \times {\calK'}^{\Delta^1}_{L}) \times_{ \calK^{\Delta^1} \times {\calK'}^{\Delta^1} }
\calK(\calM)^{\Delta^1} \times_{ \calK(\calM) \times \calK(\calM) } (\calK \times \calK'),$$ taken in
the $\infty$-category $\RGeom \simeq \LGeo^{op}$ of $\infty$-topoi. By construction, we have arranged
that $\calK^{\calG'}_{\calG}$ is equipped with a map $T: \calM \times \Delta^1 \rightarrow \calK^{\calG'}_{\calG}$ with the following properties, where $\calX = \calK^{\calG'}_{\calG}$:
\begin{itemize}
\item[$(1)$] The restriction of $T$ to $\calG \times \Delta^1$ belongs to $\Struct^{\loc}_{\calG}(\calX)$.
\item[$(2)$] The restriction of $T$ to $\calG' \times \Delta^1$ belongs to $\Struct^{L}_{\calG'}(\calG)$.
\item[$(3)$] The restriction of $T$ to $\calM \times \{0\}$ belongs to $\Struct^{(0)}_{\calM}(\calX)$.
\item[$(4)$] The restriction of $T$ to $\calM \times \{1\}$ belongs to $\Struct^{(1)}_{\calM}(\calX)$.
\end{itemize}
Moreover, $\calK^{\calG'}_{\calG}$ can be described by the following universal property: for any
$\infty$-topos $\calX$, composition with $T$ induces a fully faithful embedding
$\Fun^{\ast}( \calK^{\calG'}_{\calG}, \calX) \rightarrow \Fun( \calM \times \Delta^1, \calX)$ whose
essential image is spanned by those functors $\calM \times \Delta^1 \rightarrow \calX$ which
satisfy conditions $(1)$ through $(4)$.

Let $\calO = T | (\calG \times \{0\})$ and $\calO' = T | ( \calG' \times \{1\})$. Condition $(4)$
allows us to identify $T | \calG \times \{1\}$ with $\calO' \circ f$, so that
(according to condition $(1)$) the restriction $T | (\calG \times \Delta^1)$ determines a morphism $\alpha: \calO \rightarrow \calO' \circ f$ in $\Struct^{\loc}_{\calG}(\calK^{\calG'}_{\calG})$, well-defined up to homotopy.
\end{notation}

\begin{proof}[Proof of Proposition \ref{ingred}]
We proceed as in Notation \ref{sluck} to define $\calK^{\calG'}_{\calG}$, a pair of functors $\calO \in \Struct_{\calG}( \calK^{\calG'}_{\calG})$, $\calO' \in \Struct_{\calG'}( \calK^{\calG'}_{\calG} )$, and a natural transformation $\alpha: \calO \rightarrow \calO' \circ f$. We will show that the desired universal property is satisfied.
In other words, we will show that for every object $(\calX, \calO_{\calX})$, composition with $\alpha$ induces a homotopy equivalence from
$\bHom_{ \LGeo(\calG') }( ( \calK^{\calG'}_{\calG}, \calO'), (\calX, \calO_{\calX}) )$ to the largest Kan complex contained in $\Struct^{\loc}_{\calG}(\calX)^{/\calO_{\calX} \circ f}$. Let
$\overline{\calO}_{\calX}: \calM \rightarrow \calX$ be a right Kan extension of $\calO_{\calX}$.
Using Lemma \ref{skipun} and the universal property of $\calK^{\calK'}_{\calK}$, we can identify
$\bHom_{ \LGeo(\calG')}( (\calK^{\calG'}_{\calG}, \calO'), (\calX, \calO_{\calX}) )$
with the the largest Kan complex contained in
$$ \calC \subseteq \Fun( \calM \times \Delta^2, \calX) \times_{ \Fun( \calM \times \{2\}, \calX) } \{ \overline{\calO}_{\calX} \},$$
where $\calC$ denotes the full subcategory spanned by those functors
$\overline{T}: \calM \times \Delta^2 \rightarrow \calX$ such that
$\overline{T} | \calM \times \Delta^{ \{0,1\} }$ satisfies conditions $(1)$ through $(4)$ of Notation \ref{sluck}, together with the following additional condition:
\begin{itemize}
\item[$(5)$] The restriction $\beta = \overline{T} | (\calG' \times \Delta^{ \{1,2\} })$ is a morphism of
$\Struct_{\calG'}^{\loc}(\calX)$.
\end{itemize}
Note that condition $(4)$ (and the fact that $\overline{\calO}_{\calX}$ is a right Kan extension
of $\calO_{\calX}$) imply that $\beta' = \overline{T} | ( \calG \times \Delta^{ \{1,2\} })$ can be identified with
the image of $\beta$ under $f$, so that $\beta'$ is a morphism of $\Struct^{\loc}_{\calG}(\calX)$. 
Consequently, if condition $(5)$ is assumed, then Propositions \toposref{swimmm} and \toposref{swin} permit the following reformulation of condition $(1)$:
\begin{itemize}
\item[$(1')$] The restriction $\overline{T} | (\calG \times \Delta^{ \{0,2\} })$ is a morphism of
$\Struct^{\loc}_{\calG}(\calX)$. 
\end{itemize}

Unwinding the definitions, we are reduced to proving the following assertion:
\begin{itemize}
\item[$(\ast)$] The inclusion $\calG \times \Delta^{ \{0,2\} } \subseteq \calM \times \Delta^2$
induces a trivial Kan fibration $\theta: \calC \rightarrow \Struct^{\loc}_{\calG}(\calX)^{/ \overline{\calO}_{\calX}|\calG}$. 
\end{itemize}
To prove $(\ast)$, we factor the map $\theta$ as a composition
$$ \calC \stackrel{\theta_0}{\rightarrow} \calC_1
\stackrel{\theta_1}{\rightarrow} \calC_2 \stackrel{\theta_2}{\rightarrow} \calC_3
\stackrel{\theta_3}{\rightarrow} \calC_4 \stackrel{\theta_4}{\rightarrow} \Struct^{\loc}_{\calG}(\calX)^{/ \overline{\calO}_{\calX} | \calG},$$
where:
\begin{itemize}
\item[$(a)$] Let $\calJ \subseteq \calM \times \Delta^2$ denote the full subcategory spanned
by those objects $(X,i)$ such that either $i \neq 1$ or $X \in \calG'$. Then
$\calC_1$ denotes the full subcategory of 
$$\Fun( \calJ, \calX) \times_{ \Fun( \calM \times \{2\}, \calX) } \{ \overline{\calO}_{\calX} \}$$
spanned by those functors which satisfy conditions $(1')$, $(2)$, $(3)$, and $(5)$.
Using Proposition \toposref{lklk}, we deduce that the restriction map
$\theta_0: \calC \rightarrow \calC_1$ is a trivial Kan fibration.
\item[$(b)$] The $\infty$-category $\calC_3$ is defined to be the full subcategory
of 
$$\Fun( \calM \times \Delta^{ \{0,2\} }, \calX) \times_{ \Fun( \calM \times \{2\}, \calX)} \{ \overline{\calO}_{\calX} \}$$
spanned by those functors which satisfy conditions $(1')$ and $(3)$.
Set
$$\calC_2 = \calC_3 \times_{ \Fun( \Delta^{ \{0,2\} }, \Struct_{\calG'}(\calX)) }
\Fun'( \Delta^2, \Struct_{\calG'}(\calX) ),$$
where $\Fun'( \Delta^2, \Struct_{\calG'}(\calX) )$ denotes the full subcategory of
$\Fun(\Delta^2, \Struct_{\calG'}(\calX))$ spanned by those functors satisfying
$(2)$ and $(5)$. The projection map $\theta_2: \calC_2 \rightarrow \calC_3$ is a
pullback of the projection $\Fun'(\Delta^2, \Struct_{\calG'}(\calX) ) \rightarrow
\Fun( \Delta^{ \{0,2\} }, \Struct_{\calG'}(\calX) )$, and therefore a trivial Kan fibration
by Proposition \toposref{canfact}.

\item[$(c)$] Let $\theta_1: \calC_1 \rightarrow \calC_2$ be the evident restriction map.
To show that $\theta_1$ is a trivial fibration of simplicial sets, it suffices to show that
the inclusion 
$$(\calM \times \Delta^{ \{0,2\} }) \coprod_{ \calG' \times \Delta^{ \{0,2\} }}
(\calG' \times \Delta^2) \hookrightarrow \calJ_1$$
is a categorical equivalence. This follows from Proposition \toposref{basechangefunky},
applied to the Cartesian fibration $\calM \times \Delta^2 \rightarrow \Delta^1 \times \Delta^2$.

\item[$(d)$] Let $\calJ'$ denote the full subcategory of $\calM \times \Delta^{ \{0,2\} }$ spanned by
those objects having the form $(X, i)$, where either $i = 2$ or $X \in \calG$, and let
$\calC_4$ denote the full subcategory of 
$$ \Fun( \calJ', \calX) \times_{ \Fun( \calM \times \{2\}, \calX) } \{ \overline{\calO}_{\calX} \}$$
spanned by those functors satisfying $(1')$. Proposition \toposref{lklk} implies that the
restriction map $\theta_3: \calC_3 \rightarrow \calC_4$ is a trivial Kan fibration.

\item[$(e)$] To prove that the restriction map $\theta_4: \calC_4 \rightarrow
\Struct^{\loc}_{\calG}(\calX)^{/ \overline{\calO}_{\calX} | \calG}$ is a trivial Kan fibration,
it suffices to show that the inclusion
$$ ( \calG \times \Delta^{ \{0,2\} }) \coprod_{ \calG \times \{2\} }
( \calM \times \{2\} ) \hookrightarrow \calJ'$$
is a categorical equivalence. This again follows from Proposition \toposref{basechangefunky}, applied to the Cartesian fibration $\calM \times \Delta^2 \rightarrow \Delta^1 \times \Delta^2$.
\end{itemize}
\end{proof}

\subsection{Construction of Spectra: Absolute Version}\label{abspec}

Let $\calG$ be a geometry, and let $\calG_{\disc}$ be the underlying discrete geometry. In \S \ref{relspec}, we gave a construction of the relative spectrum functor $\Spec^{\calG}_{\calG_{\disc}}$, which restricts to the absolute spectrum functor
$$ \Spec^{\calG}: \Ind(\calG^{op} ) \rightarrow \LGeo(\calG).$$
However, the construction was perhaps rather opaque. We will correct that deficiency in this section, by giving a second (much more explicit) construction of $\Spec^{\calG}$. We begin by introducing a bit of notation.

\begin{notation}\label{sipper}
Let $p: \overline{\LGeo} \rightarrow \LGeo$ denote the universal topos fibration; we may therefore view objects of $\overline{\LGeo}$ as pairs $(\calX, X)$ where $\calX$ is an $\infty$-topos and $X$ is an object of $\calX$. Let $E \in \overline{\LGeo}$ be an object corresponding to the pair
$(\SSet, 1_{\SSet})$ (so that $\SSet$ is an initial object of $\LGeo$ by Proposition \toposref{spacefinall},
and $1_{\SSet}$ is a final object of $\SSet$).

We let $\Gamma: \overline{\LGeo} \rightarrow \SSet$ denote a functor corepresented by
$E$. We will refer to $\Gamma$ (which is well-defined up to a contractible space of choices) as {\it the global sections functor}. For every $\infty$-topos $\calX$, the restriction of $\Gamma$ to the fiber $\overline{\LGeo} \times_{ \LGeo} \{ \calX \} \simeq \calX$ is a functor corepresented by the final object $1_{\calX} \in \calX$. 

For any geometry $\calG$, we have a canonical evaluation map
$\LGeo(\calG) \times \calG \rightarrow \overline{\LGeo}$. Composing with the global sections functor $\Gamma$, we obtain a pairing
$$ \LGeo(\calG) \times \calG \rightarrow \SSet,$$
which we may view as a functor $\LGeo(\calG) \rightarrow \Fun( \calG, \SSet)$. 
We observe that this functor factors through $\Fun^{\lex}(\calG, \SSet) \simeq \Ind(\calG^{op})$.
We let $\Gamma_{\calG}: \LGeo(\calG) \rightarrow \Ind( \calG^{op} )$ denote the resulting map; we will refer to $\Gamma_{\calG}$ as the {\it $\calG$-structured global sections functor}.
\end{notation}

Let $\calG$ be a geometry, and let $\calG_0$ be a discrete geometry having the same
underlying $\infty$-category as $\calG$. 
By construction, the $\calG$-structured global sections functor $\Gamma_{\calG}$ factors as
a composition
$$ \LGeo(\calG) \rightarrow \LGeo(\calG_0) \stackrel{ \Gamma_{\calG_0} }{\rightarrow} \Ind(\calG^{op} ).$$
We observe that $\Gamma_{\calG_0}$ can be identified with a right adjoint to the fully faithful embedding
$$ \Ind(\calG^{op}) \simeq \Struct_{\calG}(\SSet) \simeq \LGeo(\calG_0) \times_{ \LGeo } \{ \SSet \} \subseteq \LGeo(\calG_0).$$
Consequently, we may identify $\Gamma_{\calG}$ with a right adjoint to the absolute spectrum
functor $\Spec^{\calG}: \Ind(\calG^{op}) \rightarrow \LGeo(\calG)$ of Definition \ref{defspect}.
We will use description to explicitly construct $\Spec^{\calG} X$, where $X \in \Ind(\calG^{op}) \simeq \Pro(\calG)^{op}$. We begin with a few preliminaries.

\begin{notation}\label{ilk}
Let $\calG$ be a geometry. We will say that a morphism $f: U \rightarrow X$ in
$\Pro(\calG)$ is {\it admissible} if there exists a pullback diagram
$$ \xymatrix{ U \ar[r] \ar[d]^{f} & j(U') \ar[d]^{j(f')} \\
X \ar[r] & j(X') }$$
in $\Pro(\calG)$, where $j: \calG \rightarrow \Pro(\calG)$ denotes the Yoneda embedding
and $f': U' \rightarrow X'$ is an admissible morphism in $\calG$. 
\end{notation}

\begin{remark}
Every admissible morphism in $\Pro(\calG)$ is proadmissible. In particular, if $f: U \rightarrow X$
is a morphism in $\calG$ such that $j(f)$ is admissible in $\Pro(\calG)$, then $f$ is admissible in
$\calG$ (Remark \ref{proadmadm}).
\end{remark}

\begin{lemma}\label{capsule}
Let $\calG$ be a geometry. Then:
\begin{itemize}
\item[$(1)$] Every equivalence in $\Pro(\calG)$ is admissible.
\item[$(2)$] The collection of admissible morphisms in $\Pro(\calG)$ is stable under the formation of pullbacks.
\item[$(3)$] Let $$ \xymatrix{ & Y \ar[dr]^{g} & \\
X \ar[ur]^{f} \ar[rr]^{h} & & Z }$$
be a commutative diagram in $\Pro(\calG)$, where $g$ is admissible.
Then $f$ is admissible if and only if $h$ is admissible.
\end{itemize}
\end{lemma}

\begin{proof}
Assertions $(1)$ and $(2)$ are obvious. Let us now prove $(3)$. We first establish a bit of notation.
Since $g$ is admissible, we can choose a pullback diagram
$$ \xymatrix{ Y \ar[r] \ar[d]^{g} & j(Y_0) \ar[d]^{j(g_0)} \\
Z \ar[r]^{t_0} & j(Z_0) }$$
where $g_0: Y_0 \rightarrow Z_0$ is an admissible morphism in $\calG$.
Write $Z$ as the filtered limit of a diagram $\{ Z_{\alpha} \}$ in $\calG_{/Z_0}$, and set
$Y_{\alpha} = Z_{\alpha} \times_{ Z_0} Y_0$, so that $Y$ is the filtered limit of the diagram
$\{ Y_{\alpha} \}$ in $\calG$. 

We now prove the ``only'' if direction of $(3)$. Suppose that $f$ is admissible, so there exists a pullback diagram
$$ \xymatrix{ X \ar[r] \ar[d]^{f} & j(X_1) \ar[d]^{j(f_1)} \ar[d] \\
Y \ar[r]^{\phi} & j(Y_1) }$$
for some admissible morphism $f_1: X_1 \rightarrow Y_1$ in $\calG$. The
map $\phi$ factors as a composition
$Y \rightarrow j(Y_{\alpha}) \rightarrow j(Y_1)$ for $\alpha$ sufficiently large.
We then have a diagram
$$ \xymatrix{ X \ar[d] \ar[r] & j(X_1 \times_{Y_1} Y_{\alpha}) \ar[d] \\
Y \ar[r] \ar[d] & j(Y_{\alpha}) \ar[d] \\
Z \ar[r] & j(Z_{\alpha}). }$$
Since the upper and lower squares are pullback diagrams, we conclude that the outer square is also a pullback diagram. We now observe that $h$ is a pullback of $j(h_0)$, where $h_0$ denotes the composition
$$ X_1 \times_{Y_1} Y_{\alpha} \rightarrow Y_{\alpha} \rightarrow Z_{\alpha}.$$
Since the collection of admissible morphisms in $\calG$ is stable under pullbacks and composition, we conclude that $h_0$ is admissible. Then $h$ is also admissible, as desired.

We now prove the ``if'' direction of $(3)$. Assume that $h$ is admissible, so that we have a pullback diagram
$$ \xymatrix{ X \ar[r] \ar[d]^{h} & j(X_2) \ar[d]^{ j(h_2)} \\
Z \ar[r]^{t_2} & j(Z_2) }$$
for some admissible morphism $h_2: X_2 \rightarrow Z_2$ in $\calG$.
Replacing $Z_0$ by $Z_0 \times Z_2$ (and repeating the construction of the first part of the proof),
we may suppose that $Z_0 = Z_2$ and $t_0 = t_2$. Let $X_{\alpha} = Z_{\alpha} \times_{ Z_2} X_2$, so that $X$ is the filtered limit of the diagram $\{ X_{\alpha} \}$ in $\calG$.
By a compactness argument, we can choose a commutative diagram
$$ \xymatrix{ X \ar[r] \ar[d]^{f} & j(X_{\alpha}) \ar[d]^{j(f_{\alpha})} \\
Y \ar[r] & j(Y_0) }$$
for some index $\alpha$. The map $f_{\alpha}$ determines a commutative diagram
$$ \xymatrix{ X_{\alpha} \ar[rr]^{ \overline{f}_{\alpha} } \ar[dr] & & Y_{\alpha} \ar[dl] \\
& Z_{\alpha}. & }$$
Since the diagonal arrows are both admissible, we deduce that $\overline{f}_{\alpha}$ is admissible. We have a commutative diagram
$$ \xymatrix{ X \ar[d] \ar[r]^{f} \ar[r] & j( X_{\alpha}) \ar[d]^{ j(\overline{f}_{\alpha}) } \\
Y \ar[d] \ar[r] & j(Y_{\alpha}) \ar[d] \\
Z \ar[r] & j(Z_{\alpha}). }$$
The outer and lower squares are pullbacks, so the upper square is a pullback as well.
It follows that $f$ is admissible, as desired.
\end{proof}

\begin{warning}
If $\calG$ is a geometry, then the collection of admissible morphisms in $\Pro(\calG)$
is not necessarily stable under the formation of retracts.
\end{warning}

\begin{notation}\label{scun}
Let $\calG$ be a geometry, and let $X$ be an object of $\Pro(\calG)$. We let
$\Pro(\calG)^{\adm}_{/X}$ denote the full subcategory of $\Pro(\calG)_{/X}$ spanned by the admissible morphisms $U \rightarrow X$. In view of Lemma \ref{capsule}, we can also identify $\Pro(\calG)^{\adm}_{/X}$ with the $\infty$-category $( \Pro(\calG)^{\adm} )_{/X}$, where $\Pro(\calG)^{\adm}$ denotes the subcategory of $\Pro(\calG)$ spanned by the admissible morphisms.

We regard $\Pro(\calG)_{/X}$ as endowed with the {\em coarsest} Grothendieck topology having the following property: for every admissible morphism $U \rightarrow X$, every admissible covering $\{ V'_{\alpha} \rightarrow U' \}$ of an object $U' \in \calG$ and every morphism $U \rightarrow j(U')$ in $\Pro(\calG)$, the collection of admissible morphisms
$\{ j(V'_{\alpha}) \times_{ j(U') } U \rightarrow U \}$ generates a covering sieve on
$U \in \Pro(\calG)^{\adm}_{/X}$.  
\end{notation}

\begin{remark}\label{sunn}
Let $\calG$ be a geometry, and let $X \in \Pro(\calG)$. Every admissible morphism
$U \rightarrow X$ in $\Pro(\calG)$ arises from some pullback diagram
$$ \xymatrix{ U \ar[r] \ar[d]^{f} & j(U') \ar[d]^{j(f')} \\
X \ar[r] & j(X') }$$
in $\Pro(\calG)$. It follows that the collection of equivalence classes of admissible morphisms
$U \rightarrow X$ is small, provided that $X$ has been fixed in advance. In particular, the
$\infty$-category $\Pro(\calG)^{\adm}_{/X}$ is essentially small. We may therefore proceed as usual
to define a presheaf $\infty$-category $\calP( \Pro(\calG)^{\adm}_{/X} ) = \Fun( (\Pro(\calG)^{\adm}_{/X})^{op}, \SSet)$ and the subcategory $\Shv( \Pro(\calG)^{\adm}_{/X}) \subseteq \calP( \Pro(\calG)^{\adm}_{/X})$ of sheaves with respect to the Grothendieck topology described in Notation
\ref{scun}. The inclusion $\Shv( \Pro(\calG)^{\adm}_{/X} ) \subseteq \calP( \Pro(\calG)^{\adm}_{/X} )$
admits a left exact left adjoint $L$, and $\Shv( \Pro(\calG)^{\adm}_{/X} )$ is an $\infty$-topos.
\end{remark}

\begin{remark}
Let us say that a geometry $\calG$ is {\it finitary} if $\calG$ is small, and the Grothendieck topology on $\calG$ is finitely generated in the following sense: for every covering sieve
$\calG^{0}_{/Y} \subseteq \calG_{/Y}$ on an object $Y \in \calG$, there exists a finite collection of admissible morphisms $\{ W_{i} \rightarrow Y \}_{1 \leq i \leq n}$ belonging to $\calG^{0}_{/Y}$ which themselves generate a covering sieve on $\calG^{0}_{/Y}$. 

Suppose that $\calG$ is a finitary geometry, and let $X$ be an object of $\Pro(\calG)$. Then
the Grothendieck topology on $\Pro(\calG)_{/X}^{\adm}$ of Notation \ref{scun} can be described as follows. A sieve $S$ on an object $U \rightarrow X$ of $\Pro(\calG)$ is covering if and only if there
exists an object $U' \in \calG$, a finite collection of admissible morphisms
$\{ V'_{i} \rightarrow U' \}_{1 \leq i \leq n}$ which generate a covering sieve on $U'$, and
a map $U \rightarrow j(U')$ such that each of the pullback maps
$U \times_{ j(U') } j( V'_i) \rightarrow U$ belongs to the sieve $S$.

If $\calG$ is not finitary, then the condition given above is sufficient to guarantee that $S$ is a covering sieve, but is generally not necessary.
\end{remark}

We are now ready to proceed with our construction.

\begin{definition}
Let $\calG$ be a geometry and let $X$ be an object of $\Pro(\calG)$. We
define $\USpec X$ to be the $\infty$-topos $\Shv( \Pro(\calG)_{/X}^{\adm} )$
(see Notation \ref{scun} and Remark \ref{sunn}). Let
$\calO_{ \USpec X}$ denote the composite functor
$$ \calG \stackrel{ \widetilde{\calO}_{\USpec X} }{\rightarrow} \calP( \Pro(\calG)^{\adm}_{/X} )
\stackrel{L}{\rightarrow} \USpec X$$
where $L$ denotes a left adjoint to the inclusion $\Shv( \Pro(\calG)^{\adm}_{/X} )
\subseteq \calP( \Pro( \calG^{\adm}_{/X} ) )$, and the functor $\widetilde{\calO}_{\USpec X}$ is
adjoint to the composite map
$$ \calG \times (\Pro(\calG)^{\adm}_{/X})^{op} \rightarrow
\calG \times \Pro(\calG)^{op} = \calG \times \Fun^{\lex}(\calG, \SSet) \rightarrow \SSet.$$
\end{definition}

\begin{remark}\label{juna}
Suppose that the Grothendieck topology of Notation \ref{scun} is {\it precanonical} in the sense that
for every object $Y \in \calG$ and every object $X \in \Pro(\calG)$, the functor 
$$(\Pro(\calG)_{/X}^{\adm})^{op} \rightarrow \Pro(\calG)^{op} = \Fun^{\lex}(\calG, \SSet)
\stackrel{e_{Y}}{\rightarrow} \SSet$$
is a sheaf, where $e_Y$ denotes evaluation at $Y$. Then the functor
$\widetilde{\calO}_{\USpec X}$ already takes values in $\USpec X$, so further sheafification
is unnecessary and we can identify $\calO_{\USpec X}$ with $\widetilde{\calO}_{\USpec X}$.
In particular, $\Gamma_{\calG}( \USpec X, \calO_{ \USpec X} )$ can be identified with
$X \in \Pro(\calG)^{op} \simeq \Fun^{\lex}( \calG, \SSet)$. This will be the case in many of the examples that we consider.
\end{remark}

\begin{proposition}
Let $\calG$ be a geometry and $X \in \Pro(\calG)$ an object. Then
the functor $\calO_{ \USpec X}: \calG \rightarrow \USpec X$ is a $\calG$-structure on
$\USpec X$. 
\end{proposition}

\begin{proof}
It is clear from the construction that the functor $\calO_{ \USpec X}$ is left exact.
To complete the proof, it will suffice to show that if $\{ V_{\alpha} \rightarrow Y \}$ is
a collection of admissible morphisms which generate a covering sieve on an object
$Y \in \calG$, then the induced map $\coprod \calO_{\USpec X}(V_{\alpha}) \rightarrow \calO_{\USpec X}(Y)$ is an effective epimorphism in $\Shv( \Pro(\calG)_{/X}^{\adm})$. 
Let $U \in \Pro(\calG)_{/X}^{\adm}$, and let
$\eta \in \pi_0 \calO_{\USpec X}(Y)(U)$; we wish to show that, locally on $U$, the section
$\eta$ belongs to the image of $\pi_0 \calO_{X}(V_{\alpha})(U)$ for some index $\alpha$. Without loss of generality, we may suppose that $\eta$ arises from a map $U \rightarrow j(Y)$ in $\Pro(\calG)$, where $j: \calG \rightarrow \Pro(\calG)$ denotes the Yoneda embedding.
Then the fiber products $U_{\alpha} = j(V_{\alpha}) \times_{j(Y)} U$ form an admissible cover of $U$, and each $\eta_{\alpha} = \eta | U_{\alpha} \in \pi_0 \calO_{\USpec X}(Y)(U_{\alpha})$ lifts to
$\pi_0 \calO_{\USpec X}(V_{\alpha})(U_{\alpha})$.
\end{proof}

Consequently, we can view $( \USpec X, \calO_{\USpec X} )$ as an object 
of $\LGeo(\calG)$. Our next goal is to show that this object can be identified
with $\Spec^{\calG} X$. To formulate this result more precisely, we begin by observing that the global sections functors 
$$\Gamma: \USpec X \rightarrow \SSet$$
$$\Gamma': \calP( \Pro(\calG)^{\adm}_{/X} ) \rightarrow \SSet$$
can be identified with the functors given by evaluation on the final object
$\id_X: X \rightarrow X$ of $\Pro(\calG)^{\adm}_{/X}$. In particular, the composition
$$ \calG \stackrel{ \widetilde{\calO}_{\calX} }{\rightarrow} \calP( \Pro(\calG)^{\adm}_{/X} ) 
\stackrel{ \Gamma' }{\rightarrow} \SSet$$
is canonically equivalent to the proobject $X \in \Pro(\calG) = \Fun^{\lex}( \calG, \SSet)^{op}$ itself.
The sheafification map $\widetilde{\calO}_{\USpec X} \rightarrow \calO_{ \USpec X}$
induces a natural transformation 
$\alpha: X \rightarrow \Gamma_{\calG}( \USpec X, \calO_{\USpec X})$ in the
$\infty$-category $\Ind(\calG^{op}) \simeq \Pro(\calG)^{op}$. 

\begin{theorem}\label{scoo}
Let $\calG$ be a geometry and let $X$ be an object of $\Pro(\calG)$. Then the 
natural transformation $\alpha: X \rightarrow \Gamma_{\calG}( \USpec X, \calO_{\USpec X} )$
in $\Pro(\calG)^{op}$ is adjoint to an equivalence $\Spec^{\calG} X \rightarrow ( \USpec X, \calO_{ \USpec X} )$ in $\LGeo(\calG)$.
\end{theorem}

In other words, for every object $(\calY, \calO_{\calY}) \in \LGeo(\calG)$, composition with
$\alpha$ induces a homotopy equivalence
$$ \bHom_{ \LGeo(\calG) }( (\USpec X, \calO_{ \USpec X}), (\calY, \calO_{\calY}))
\rightarrow \bHom_{ \Pro(\calG)^{op} }( X, \Gamma_{\calG}( \calY, \calO_{\calY}) ).$$  
In view of Propositions \toposref{intprop} and \toposref{sumatch}, we may assume that $\calO_{\calY}$
factors as a composition
$$ \calG \stackrel{j}{\rightarrow} \Pro(\calG) \stackrel{ \overline{\calO}_{\calY} }{\rightarrow} \calY,$$
where $j$ denotes the Yoneda embedding and the functor $\overline{\calO}_{\calY}$ preserves small limits. Unwinding the definitions, we can identify the mapping space
$\bHom_{ \Pro(\calG)^{op} }( X, \Gamma_{\calG}( \calY, \calO_{\calY}) )$ with
$\bHom_{\calY}( 1_{\calY}, \overline{\calO}_{\calY}(X) )$, where $1_{\calY}$ denotes the final object of $\calY$. Consequently, we may reformulate Theorem \ref{scoo} as follows:

\begin{proposition}\label{proscoo}
Let $\calG$ be a geometry, $X$ an object of $\Pro(\calG)$. Let
$\calY$ be an $\infty$-topos and $\overline{\calO}_{\calY}: \Pro(\calG) \rightarrow \calY$ a functor
which preserves small limits, such that the composition
$\calO_{\calY} = \overline{\calO}_{\calY} \circ j: \calG \rightarrow \calY$ belongs to
$\Struct_{\calG}(\calY)$. Then the canonical map
$$ \theta: \bHom_{ \LGeo(\calG) }( ( \USpec X, \calO_{ \USpec X}), (\calY, \calO_{\calY}) )
\rightarrow \bHom_{\calY}( 1_{\calY}, \calO_{\calY}(X) )$$
is a homotopy equivalence.
\end{proposition}

The proof will require the following preliminary result, which is an easy consequence of
Lemma \ref{haberspaz}:

\begin{lemma}\label{blahann}
Let $\calG$ be a geometry and let $X$ be an object of $\Pro(\calG)$. 
Let $j: \Pro(\calG)^{\adm}_{/X} \rightarrow \USpec X$ be the composition of the Yoneda
embedding $\Pro(\calG)^{\adm}_{/X} \rightarrow \calP( \Pro(\calG^{\adm})_{/X})$
with the sheafification functor $L: \calP( \Pro(\calG^{\adm})_{/X}) \rightarrow \USpec X$
(that is, a left adjoint to the inclusion), and let $\alpha: \Pro(\calG)^{\adm}_{/X} \rightarrow \Pro(\calG)$
be the canonical projection. Then:
\begin{itemize}
\item[$(1)$] The functor $j$ admits a left Kan extension along $\alpha$, which we will denote by
$\alpha_{!}(j): \Pro(\calG) \rightarrow \USpec X$.
\item[$(2)$] The composition of $\alpha_{!}(j)$ with the Yoneda embedding
$\calG \rightarrow \Pro(\calG)$ is canonically equivalent to the structure sheaf $\calO_{ \USpec X}$. 
\end{itemize}
\end{lemma}




\begin{proof}[Proof of Proposition \ref{proscoo} (and Theorem \ref{scoo})]
Let $\overline{\calO}_0$ denote the composition
$$ \Pro(\calG)^{\adm}_{/X} \rightarrow \Pro(\calG) \stackrel{\overline{\calO}_{\calY}}{\rightarrow} \calY.$$
Let $\calI_0$ denote the simplicial set
$$ ( \{X\} \times \Delta^1 ) \coprod_{ \{X \} \times \{1\} } ( \Pro(\calG)^{\adm}_{/X} \times \{1\} ),$$
and let $\calI$ denote the essential image of $\calI_0$ in $\Pro(\calG)^{\adm}_{/X} \times \Delta^1.$
Since the inclusion $\calI_0 \subseteq \calI$ is a categorical equivalence, the
induced map
$$ \Fun( \calI, \calY) \rightarrow \Fun( \calI_0, \calY)$$
is a trivial Kan fibration. 

Let $\calC$ denote the
the full subcategory of $\Fun( \Pro(\calG)^{\adm}_{/X} \times \Delta^1, \calY)$ spanned by those functors
$F$ which satisfy the following conditions:
\begin{itemize}
\item[$(i)$] The functor $F$ is a right Kan extension of $F | \calI$. More concretely, 
for every admissible morphism $U \rightarrow X$, the diagram
$$ \xymatrix{ F(U, 0) \ar[r] \ar[d] & F( U, 1) \ar[d] \\
F(X,0) \ar[r] & F(X,1) }$$
is a pullback diagram in $\calY$.
\item[$(ii)$] The object $F(X, 0)$ is final in $\calY$.
\end{itemize}

Using Proposition \toposref{lklk}, we deduce that the forgetful
functors
$$ \calC \rightarrow \Fun^{0}( \calI, \calY) \rightarrow \Fun^{0}( \calI_0, \calY)$$
are trivial Kan fibrations, where $\Fun^{0}( \calI, \calY)$ and $\Fun^{0}( \calI_0, \calY)$
denote the full subcategories of $\Fun( \calI, \calY)$ and $\Fun( \calI_0, \calY)$ spanned
by those functors $F$ which satisfy condition $(ii)$. Form a pullback diagram
$$ \xymatrix{ \calC_0 \ar[r] \ar[d] & Z \ar[d] \ar[r] & \{ \overline{\calO}_0 \} \ar[d] \\
\calC \ar[r] & \Fun^{0}( \calI_0, \calY) \ar[r] & \Fun( \Pro(\calG)^{\adm}_{/X} \times \{1\}, \calY). }$$
Then $Z$ is a Kan complex, which we can identify with the space
$\bHom_{\calY}( 1_{\calY}, \overline{\calO}_{\calY}(X))$. The projection map $\calC_0 \rightarrow Z$ is a trivial Kan fibration, so that $\calC_0$ is another Kan complex which models the homotopy type
$\bHom_{\calY}( 1_{\calY}, \overline{\calO}_{\calY}(X) )$. 

The inclusion $\Pro(\calG)^{\adm}_{/X} \times \{0\} \subseteq \Pro(\calG)^{\adm}_{/X} \times \Delta^1$
induces a functor $f: \calC_0 \rightarrow \Fun( \Pro(\calG)^{\adm}_{/X}, \calY )$. In terms of the identification
above, we can view this functor as associating to each global section $1_{\calY} \rightarrow \overline{\calO}_{\calY}(X)$ the functor
$$U \mapsto \overline{\calO}_{\calY}(U) \times_{ \overline{\calO}_{\calY}(X) } 1_{\calY}.$$
It follows that the essential image of $f$ belongs to the full subcategory
$$\Fun^{0}( \Pro(\calG)^{\adm}_{/X}, \calY) \subseteq \Fun( \Pro(\calG)^{\adm}_{/X}, \calY)$$ spanned by those functors $F$ which satisfy the following conditions:
\begin{itemize}
\item[$(a)$] The functor $F$ preserves finite limits.
\item[$(b)$] For every admissible covering $\{ V_{\alpha} \rightarrow V \}$ of an object
$V \in \calG$, every object $U \rightarrow X$ in $\Pro(\calG)^{\adm}_{/X}$, and every
map $U \rightarrow j(V)$ in $\Pro(\calG)$ (here $j: \calG \rightarrow \Pro(\calG)$ denotes the Yoneda embedding), the induced map $\coprod_{\alpha} F( j(V_{\alpha}) \times_{ j(V) } U) \rightarrow F(U)$
is an effective epimorphism in $\calY$.
\end{itemize}

The map $\theta$ fits into a homotopy pullback diagram
$$ \xymatrix{ \bHom_{ \LGeo(\calG)}( ( \USpec X, \calO_{ \USpec X}) , (\calY, \calO_{\calY}) ) \ar[r]^-{\theta} \ar[d] & \calC_0 \ar[d]^{f} \\
\Fun^{\ast}( \USpec X, \calY) \ar[r]^-{\theta'} & \Fun^{0}( \Pro(\calG)^{\adm}_{/X}, \calY).}$$
Here $\theta'$ is induced by composition with the map
$$ \Pro(\calG)^{\adm}_{/X} \rightarrow \calP( \Pro(\calG)^{\adm}_{/X} ) 
\stackrel{L}{\rightarrow} \Shv( \calG^{\adm}_{/X} ),$$
where the first map is a Yoneda embedding and $L$ is a sheafification functor.
Using Propositions \toposref{natash} and \toposref{igrute} (and the definition of the Grothendieck topology on $\Pro(\calG)^{\adm}_{/X}$), we deduce that $\theta'$ is an equivalence
of $\infty$-categories. Consequently, to show that $\theta$ is a homotopy equivalence, it will suffice to show
that it induces a homotopy equivalence after passing to the fiber over every geometric morphism
$g^{\ast}: \USpec X \rightarrow \calY$.

Let $\alpha: \Pro(\calG)^{\adm}_{/X} \rightarrow \Pro(\calG)$ denote the projection, and let
$\overline{\calO}'_{\calY}$ denote a left Kan extension of $g^{\ast} \circ j: \Pro(\calG)^{\adm}_{/X} \rightarrow \calY$ along $\alpha$. Lemma \ref{blahann} allows us to identify the composition
$$ \calG \rightarrow \Pro(\calG) \stackrel{ \overline{\calO}'_{\calY} }{\rightarrow} \calY$$
with $g^{\ast} \circ \calO_{ \USpec X}$, so that we canonical homotopy equivalences
$$  \bHom_{ \Fun( \calG, \calY )}( g^{\ast} \circ \calO_{ \USpec X}, \calO_{ \calY} ) \simeq
\bHom_{ \Fun( \Pro(\calG), \calY ) }( \overline{\calO}'_{\calY}, \overline{\calO}_{\calY} )
\simeq \bHom_{ \Fun( \Pro(\calG)^{\adm}_{/X}, \calY)}( g^{\ast} \circ j, \overline{\calO}_0).$$
To complete the proof, it will suffice to show that if
$\beta: g^{\ast} \circ \calO_{ \USpec X} \rightarrow \calO_{\calY}$ and
$\beta': g^{\ast} \circ j \rightarrow \overline{\calO}_0$ are morphisms which correspond
under this homotopy equivalence, then $\beta$ belongs to $\Struct^{\loc}_{\calG}( \calY )$
if and only if $\beta'$ satisfies condition $(i)$. This is a special case of Proposition \ref{prespaz}.
\end{proof}

\begin{corollary}\label{tabletime}
Let $\calG$ be a geometry, $X$ an object of $\Pro(\calG)$, and set
$(\calX, \calO_{\calX}) = \Spec^{\calG} X$. Suppose that:
\begin{itemize}
\item[$(\ast)$] For every admissible morphism $U \rightarrow X$ in $\Pro(\calG)$, the
object $U$ is $n$-truncated when viewed as an object of $\Ind(\calG^{op})$. 
\end{itemize}
Then $\calO_{\calX}$ is an $n$-truncated $\calG$-structure on $\calX$.
\end{corollary}

\begin{proof}
In view of Theorem \ref{scoo}, we may suppose that $\calX = \USpec A$ and
that $\calO_{\calX} = \calO_{ \USpec X}$. Let $V \in \calG$; we wish to show that
$\calO_{\USpec X}(V)$ is an $n$-truncated object of 
$\USpec X \simeq \Shv( \Pro(\calG)^{\adm}_{/X} )$. By definition,
$\calO_{\USpec X}(V)$ is the sheafification of the presheaf defined by the composition
$$ \Ind(\calG^{op})^{\adm}_{X/} \rightarrow \Ind(\calG^{op}) =
\Fun^{\lex}(\calG, \SSet) \rightarrow \SSet,$$
where the second map is given by evaluation at $V \in \calG$. 
It will therefore suffice to show that this presheaf takes $n$-truncated values.
The value of this presheaf on an admissible morphism $U \rightarrow X$ is the space
$\bHom_{ \Pro(\calG)}(U, j(V) )$, where $j: \calG \rightarrow \Pro(\calG)$ denotes the Yoneda embedding, and therefore $n$-truncated by virtue of assumption $(\ast)$.
\end{proof}

\subsection{$\calG$-Schemes}\label{geo6}

Recall that a {\it scheme} is a ringed topological space $(X, \calO_X)$ such that
$X$ admits an open covering $\{ U_{\alpha} \subseteq X \}$ such that each
$(U_{\alpha}, \calO_X|U_{\alpha})$ is isomorphic (in the category of ringed spaces)
to $( \SSpec A, \calO_{ \SSpec A} )$, for some commutative ring $A$. In this section, 
we will introduce an analogous definition, for an arbitrary geometry $\calG$. We will discuss the relationship of this definition with the classical theory of schemes in \S \ref{exzar}. We begin with a few general remarks concerning \etale morphisms between $\calG$-structured $\infty$-topoi.

\begin{definition}\label{sabl}
Let $\calG$ be a geometry. We will say that a morphism $(\calX, \calO_{\calX}) \rightarrow
(\calY, \calO_{\calY})$ in $\LGeo(\calG)$ is {\it \etale} if the following conditions are satisfied:
\begin{itemize}
\item[$(1)$] The underlying geometric morphism $f^{\ast}: \calX \rightarrow \calY$ of 
$\infty$-topoi is \etale.
\item[$(2)$] The induced map $f^{\ast} \calO_{\calX} \rightarrow \calO_{\calY}$ is
an equivalence in $\Struct_{\calG}( \calU)$.
\end{itemize} 
We let $\LGeo(\calG)_{\mathet}$ denote the subcategory of $\LGeo(\calG)$ spanned by the
\etale morphisms, and $\LGeo_{\mathet}$ the subcategory of $\LGeo$ spanned by the \etale morphisms.
\end{definition}

\begin{remark}\label{sumptu}
Let $f$ be a morphism in $\LGeo(\calG)$. Condition $(2)$ of Definition \ref{sabl} is equivalent to the requirement that $f$ be $p$-coCartesian, where $p: \LGeo(\calG) \rightarrow \LGeo$
denotes the projection. Corollary \toposref{relativeKan} implies that $p$ restricts to a
left fibration $\LGeo(\calG)_{\mathet} \rightarrow \LGeo(\calG)$. 
\end{remark}

\begin{notation}
Let $(\calX, \calO_{\calX})$ be a $\calG$-structured $\infty$-topos, for some geometry $\calG$.
If $U$ is an object of $\calX$, we let $\calO_{\calX} | U$ denote the $\calG$-structure
on $\calX_{/U}$ determined by the composition
$$ \calG \stackrel{\calO_{\calX}}{\rightarrow} \calX \stackrel{\pi^{\ast}}{\rightarrow} \calX_{/U},$$
where $\pi^{\ast}$ is a right adjoint to the projection $\calX_{/U} \rightarrow \calX$.
Then we have a canonical \etale morphism
$$ ( \calX, \calO_{\calX} ) \rightarrow (\calX_{/U}, \calO_{\calX} |U )$$
in $\LGeo(\calG)$. 
\end{notation}

\begin{remark}\label{unipop}
Let $\calG$ be a geometry, and let $f: (\calX, \calO_{\calX}) \rightarrow (\calY, \calO_{\calG})$ be
a morphism in $\LGeo(\calG)$. For every object $U \in \calX$, we have a diagram of spaces
$$ \xymatrix{ \bHom_{ \LGeo(\calG) }( (\calX_{/U}, \calO_{\calX}|U), (\calY, \calO_{\calY}))
\ar[r] \ar[d] & \bHom_{ \LGeo(\calG) }( (\calX, \calO_{\calX}), (\calY,\calO_{\calY} )) \ar[d] \\
\bHom_{\LGeo}( \calX_{/U}, \calY ) \ar[r] & \bHom_{ \LGeo }( \calX, \calY ) }$$
which commutes up to canonical homotopy. 
Taking the vertical homotopy fibers over a point given by a geometric morphism
$\phi^{\ast}: \calX_{/U} \rightarrow \calY$ (and its image
$\phi_0^{\ast} \in \Fun^{\ast}(\calX, \calY)$), 
we obtain the homotopy equivalence
$$ \bHom_{ \Struct_{\calG}(\calY)}( \phi^{\ast}(\calO_{\calX}|U), \calO_{\calY} )
\simeq \bHom_{ \Struct_{\calG}(\calY)}( \phi_0^{\ast} \calO_{\calX}, \calO_{\calY}).$$
It follows that the above diagram is a homotopy pullback square. Combining this observation with Remark \toposref{goodilk}, we deduce the existence of a fiber sequence
$$ \bHom_{\calY}(1_{\calY}, \phi^{\ast}(U)) \rightarrow
 \bHom_{ \LGeo(\calG) }( (\calX_{/U}, \calO_{\calX}|U), (\calY, \calO_{\calY}))
\rightarrow \bHom_{ \LGeo(\calG) }( (\calX, \calO_{\calX}), (\calY,\calO_{\calY} ))$$
(here the fiber is taken over the point determined by $\phi^{\ast}$).
\end{remark}

Some of basic properties of the class of \etale morphisms are summarized in the following result:¶

\begin{proposition}\label{scanh}
Let $\calG$ be a geometry. Then:
\begin{itemize}
\item[$(1)$] Every equivalence in $\LGeo(\calG)^{op}$ is \etale. 
\item[$(2)$] Suppose given a commutative diagram
$$ \xymatrix{ & (\calY, \calO_{\calY}) \ar[dr]^{g} & \\
(\calX, \calO_{\calX}) \ar[ur]^{f} \ar[rr]^{h} & & (\calZ, \calO_{\calZ} ) }$$
in $\LGeo(\calG)^{op}$, where $g$ is \etale. Then $f$ is \etale if and only if $h$ is \etale.
\end{itemize}
In particular, the collection of \etale morphisms in $\LGeo(\calG)^{op}$ is stable under composition, and
therefore spans a subcategory $\LGeo(\calG)^{op}_{\mathet} \subseteq \LGeo(\calG)^{op}$.
\begin{itemize}
\item[$(3)$] The $\infty$-category $\LGeo(\calG)^{op}_{\mathet}$ admits small colimits.
\item[$(4)$] The inclusion $\LGeo(\calG)^{op}_{\mathet} \subseteq \LGeo(\calG)^{op}$ preserves small colimits (note that $\LGeo(\calG)^{op}$ need not admit small colimits in general).
\item[$(5)$] Fix an object $(\calX, \calO_{\calX}) \in \LGeo(\calG)^{op}$. Then the $\infty$-category
$( \LGeo(\calG)^{op}_{\mathet})_{/ (\calX, \calO_{\calX} )}$ is canonically equivalent with $\calX$.
\end{itemize}
\end{proposition}

\begin{corollary}
Let $\calG$ be a geometry. The collection of \etale morphisms in $\LGeo(\calG)^{op}$ is stable under
the formation of retracts.
\end{corollary}

\begin{proof}[Proof of Proposition \ref{scanh}]
Assertion $(1)$ is obvious. Assertion $(2)$ follows from Proposition \toposref{protohermes}, Remark \ref{sumptu}, and Corollary \toposref{toadscan}.

Using Propositions \toposref{needed17} and Theorem \toposref{prescan}, we deduce the following more precise version of $(3)$:
\begin{itemize}
\item[$(3')$] The $\infty$-category $\LGeo(\calG)^{op}_{\mathet}$ admits small colimits. Moreover,
a small diagram $p: K^{\triangleright} \rightarrow \LGeo(\calG)^{op}_{\mathet}$ is a colimit if
and only the induced map $K^{\triangleright} \rightarrow \LGeo^{op}_{\mathet}$ is a colimit diagram.
\end{itemize}

To prove $(4)$, let us consider a small colimit diagram $K^{\triangleright} \rightarrow \LGeo(\calG)^{op}_{\mathet}$. Using $(3')$ and Theorem \toposref{prescan}, we conclude that the composition
$$ K^{\triangleright} \stackrel{p}{\rightarrow} \LGeo(\calG)^{op} \stackrel{q}{\rightarrow} \LGeo^{op}$$
is a colimit diagram. Proposition \ref{deeder} implies that $p$ is a limit diagram.
Assertion $(5)$ is an immediate consequence of Remark \ref{sumptu}.
\end{proof}

\begin{remark}\label{lopus}
The condition that a morphism $f: (\calX, \calO_{\calX}) \rightarrow (\calY, \calO_{\calY} )$ in $\LGeo(\calG)^{op}$ be \etale is {\it local} in the following sense: if
there exists an effective epimorphism $\coprod U_{\alpha} \rightarrow 1_{\calX}$ in $\calX$ such that each of the induced maps $f_{\alpha}: (\calX_{/U_{\alpha}}, \calO_{\calX} | U_{\alpha} ) \rightarrow (\calY, \calO_{\calY} )$ is \etale, then $f$ is \etale. To prove this, we let $\calX^{0}$ denote the full subcategory of $\calX$ spanned by those objects $U$ for which the map
$(\calX_{/U}, \calO_{\calX} | U ) \rightarrow (\calY, \calO_{\calY} )$ is \etale. Proposition
\ref{scanh} implies that $\calX^{0}$ is stable under the formation of small colimits. In particular,
$U_0 = \coprod U_{\alpha}$ belongs to $\calX^{0}$. Let $U_{\bigdot}$ be the simplicial object
of $\calX$ given by the \Cech nerve of the effective epimorphism $U_0 \rightarrow 1_{\calX}$.
Since $\calX^{0}$ is a sieve, we deduce that each $U_{n} \in \calX^{0}$. Then
$| U_{\bigdot} | \simeq 1_{\calX} \in \calX^{0}$, so that $f$ is \etale as desired.
\end{remark}

\begin{example}\label{simpus}
Let $\calG$ be a geometry, and let $f: U \rightarrow X$ be an admissible morphism in
$\Pro(\calG)$. Then the induced map $\Spec^{\calG} U \rightarrow \Spec^{\calG} X$ is
\etale. This follows from Theorem \ref{scoo}, but we can also give a direct proof as follows.
Since $\Spec^{\calG}$ preserves finite limits, we can reduce to the case where $f$
is arises from an admissible morphism $f_0: U_0 \rightarrow X_0$ in $\calG$.
Let $\Spec^{\calG} X = (\calX, \calO_{\calX})$, so that $\calO_{\calX}(X_0)$ has a canonical
global section $\eta: 1_{\calX} \rightarrow \calO_{\calX}(X_0)$. Set
$Y = 1_{\calX} \times_{ \calO_{\calX}(X_0) } \calO_{\calX}(U_0)$ and
$(\calY, \calO_{\calY} ) = ( \calX_{/Y}, \calO_{\calX} | Y )$. Then there is a canonical global section
$\eta'$ of $\calO_{\calY}( U_0)$. Unwinding the definitions, we deduce that $\eta'$ exhibits
$(\calY, \calO_{\calY})$ as an absolute spectrum $U$, so that
$\Spec^{\calG}(f)$ can be identified with the \etale map $( \calY, \calO_{\calY}) 
\rightarrow (\calX, \calO_{\calX})$.
\end{example}

\begin{definition}
Let $\calG$ be a geometry. We will say that an object $(\calX, \calO_{\calX}) \in \LGeo(\calG)^{op}$ is
an {\it affine $\calG$-scheme} if there exists an object $A \in \Pro(\calG)$ and an equivalence
$( \calX, \calO_{\calX} ) \simeq \Spec^{\calG} A$. 

We will say that $(\calX, \calO_{\calX})$ is a {\it $\calG$-scheme} if the there exists a collection of objects
$\{ U_{\alpha} \}$ of $\calX$ with the following properties:
\begin{itemize}
\item[$(1)$] The objects $\{ U_{\alpha} \}$ cover $\calX$: that is, the canonical map
$\coprod_{\alpha} U_{\alpha} \rightarrow 1_{\calX}$ is an effective epimorphism, where
$1_{\calX}$ denotes the final object of $\calX$.
\item[$(2)$] For every index $\alpha$, there exists an
equivalence $( \calX_{/U_{\alpha}}, \calO_{\calX} | U_{\alpha} ) \simeq \Spec^{\calG} A_{\alpha}$, for
some $A_{\alpha} \in \Pro(\calG)$.
\end{itemize} 

We will say that a $\calG$-scheme $(\calX, \calO_{\calX})$ is {\it locally of finite presentation} if it is possible to choose the covering $\{ U_{\alpha} \}$ such that each $A_{\alpha}$ belongs to the essential image of the Yoneda embedding $\calG \rightarrow \Pro(\calG)$. We let $\Sch(\calG)$ denote the full subcategory of $\LGeo(\calG)^{op}$ spanned by the collection of all $\calG$-schemes, and $\Sch^{\fin}(\calG)$ the full subcategory spanned by the $\calG$-schemes which are locally of finite presentation.
\end{definition}

We now summarize some of the basic properties of the class of $\calG$-schemes.

\begin{proposition}\label{sizem}
Let $\calG$ be a geometry. Then:
\begin{itemize}
\item[$(1)$] Let $f: (\calX, \calO_{\calX}) \rightarrow (\calY, \calO_{\calY})$ be an
\etale morphism in $\LGeo(\calG)^{op}$. If $(\calY, \calO_{\calY})$ is a $\calG$-scheme, then
so is $(\calX, \calO_{\calX})$. 

\item[$(2)$] Let $(\calX, \calO_{\calX})$ be an object of $\LGeo(\calG)^{op}$. Suppose that there exists an effective epimorphism
$\coprod U_{\alpha} \rightarrow 1_{\calX}$ in $\calX$ such that each
$(\calX_{/ U_{\alpha}}, \calO_{\calX} | U_{\alpha} )$ is a $\calG$-scheme. Then
$(\calX, \calO_{\calX})$ is a $\calG$-scheme.

\item[$(3)$] Let $\Sch(\calG)_{\mathet}$ denote the subcategory of $\Sch(\calG)$ spanned by the
\etale morphisms. Then $\Sch(\calG)_{\mathet}$ is stable under small colimits in $\LGeo(\calG)^{op}_{\mathet}$; in particular, $\Sch(\calG)_{\mathet}$ admits small colimits.
\end{itemize}
\end{proposition}

\begin{proof}
Assertion $(2)$ follows immediately from the definitions, and assertion
$(3)$ follows immediately from $(2)$ and Proposition \ref{scanh}. Let us prove $(1)$.
In view of $(2)$, the assertion is local on the $\infty$-topos $\calY$; we may therefore assume without loss of generality that $(\calY, \calO_{\calY} ) = \Spec^{\calG} A$ for some $A \in \Pro(\calG)$.
Let $(\calX, \calO_{\calX}) \simeq (\calY_{/U}, \calO_{\calY} | U )$. 
Theorem \ref{scoo} allows us to identify $\calY$ with the $\infty$-category of
sheaves $\Shv( \Pro(\calG)_{/A}^{\adm} )$. Consequently, there exists an effective epimorphism
$\coprod_{\alpha} V_{\alpha} \rightarrow U$, where each $V_{\alpha} \in \Shv( \Pro(\calG)_{/A}^{\adm} )$ is the sheafification of the functor represented by some $B_{\alpha} \in \Pro(\calG)_{/A}^{\adm}$.
It now suffices to observe that there is an equivalence of
$( \calY_{/V_{\alpha}}, \calO_{\calY} | V_{\alpha} )$ with $\Spec^{\calG} B_{\alpha}$
(see Example \ref{simpus}).
\end{proof}

\begin{lemma}\label{spak}
Let $\calG$ be a geometry. Then the $\infty$-category $\Sch(\calG)_{\mathet}$ is generated under small colimits by the full subcategory spanned by the affine $\calG$-schemes.
\end{lemma}

\begin{proof}
Let $(\calX, \calO_{\calX})$ be a $\calG$-scheme, and let
$\calX^{0} \subseteq \calX$ be the full subcategory spanned by those objects $V$ for which
$(\calX_{/V}, \calO_{\calX} |V )$ is \etale. Let $\calX^{1}$ be the smallest full subcategory of
$\calX$ which contains $\calX^{0}$ and is stable under small colimits. In view of Remark \toposref{postit}, it will suffice to show that $\calX^{1}$ contains every object $X \in \calX$.

Since $(\calX, \calO_{\calX})$ is a $\calG$-scheme, there exists a collection of objects
$U_{\alpha} \in \calX^{0}$ such that the induced map $\coprod_{\alpha} U_{\alpha} \rightarrow 1_{\calX}$ is an effective epimorphism. Let
$X_0 = \coprod_{ \alpha} ( U_{\alpha} \times X)$. Let $X_{\bigdot}$ 
$X_{\bigdot}$ be the simplicial object of $\calX$ given by the \Cech nerve of the
map $X_0 \rightarrow X$. Then
$X$ is equivalent to the geometric realization of $X_{\bigdot}$. It will therefore
suffice to show that each $X_{n}$ belongs to $\calX^{1}$. Since colimits in
$\calX$ are universal, $X_{n}$ can be identified with a coproduct of
objects having the form $U_{\alpha_0} \times \ldots \times U_{\alpha_{n}}  \times X$. It
will therefore suffice to show that every such product belongs to $\calX^{1}$.
We may therefore replace $X$ by $U_{\alpha_0} \times \ldots \times U_{\alpha_{n}} \times X$,
and thereby reduce to the case where $X$ admits a map $X \rightarrow U$, $U \in \calX^{0}$.
Replacing $\calX$ by $\calX_{/U}$, we may further reduce to the case
where $(\calX, \calO_{\calX}) \simeq \Spec^{\calG} A$ is affine. In this case, Theorem \ref{scoo} implies that we can identify $\calX$ with the $\infty$-category
$\Shv( \Pro(\calG)_{/A}^{\adm} )$. Consider the composition
$$ j: \Pro(\calG_{/A}^{\adm}) \rightarrow \Fun( (\Pro(\calG)_{/A}^{\adm})^{op}, \SSet)
\rightarrow \Shv( \Pro(\calG)_{/A}^{\adm}).$$
We now observe that for every admissible morphism $B \rightarrow A$ in $\Pro(\calG)$, Theorem \ref{scoo} allows us to identify the $\calG$-scheme $(\calX_{/j(B)}, \calO_{\calX} | j(B) )$
with $\Spec^{\calG} B$. Consequently, the map $j$ factors through $\calX^{0} \subseteq \calX$.
Since $\calX$ is generated under small colimits by the essential image of $j$, we conclude that
every object of $\calX$ belongs to $\calX^1$, as desired.
\end{proof}

\begin{proposition}\label{sizem2}
Let $\calG$ be a geometry. Then the $\infty$-category $\Sch(\calG)$ is locally small.
\end{proposition}

\begin{proof}
We will prove the following stronger claim:
\begin{itemize}
\item[$(\ast)$] Let $(\calX, \calO_{\calX})$ and $(\calY, \calO_{\calY} )$ be
objects of $\LGeo(\calG)^{op}$. Assume that $(\calY, \calO_{\calY})$ is
a $\calG$-scheme. Then the mapping space
$$ \bHom_{ \LGeo(\calG)^{op} }( (\calX, \calO_{\calX} ), (\calY, \calO_{\calY} ))$$
is essentially small.
\end{itemize}

Consider the composition
$$ \calX^{op} \times \calY
\simeq \LGeo(\calG)^{(\calX, \calO_{\calX})/}_{\mathet}
\times (\LGeo(\calG)^{(\calY, \calO_{\calY})/}_{\mathet})^{op}
\simeq \LGeo(\calG) \times \LGeo(\calG)^{op} \rightarrow \widehat{\SSet},$$
described more informally by the formula
$$ (U,V) \mapsto \bHom_{ \LGeo(\calG)^{op} }( (\calX_{/U}, \calO_{\calX} | U),
(\calY_{/V}, \calO_{\calY} | V) ).$$
Proposition \ref{scanh} implies that this functor preserves limits in the first variable, and therefore
determines a map $\chi: \calY \rightarrow \Shv_{ \widehat{\SSet} }(\calX)$.
Let $\calX'$ denote the full subcategory of $\Shv_{\widehat{\SSet}}(\calX)$ spanned by those
functors $F: \calX^{op} \rightarrow \widehat{\SSet}$ such that $F(U)$ is essentially small,
for each $U \in \calX$; according to Proposition \toposref{representable}, the Yoneda embedding determines an equivalence $\calX \rightarrow \calX'$. We can then reformulate $(\ast)$ as follows:

\begin{itemize}
\item[$(\ast')$] The object $\chi( 1_{\calY} ) \in \Shv_{\widehat{\SSet}}( \calX)$ belongs
to $\calX'$.
\end{itemize}

Let $\calY^{0}$ denote the full subcategory of $\calY$ spanned by those objects
$V \in \calY$ for which $\chi(V) \in \calX'$. 
Note that if $(\calY_{/V}, \calO_{\calY}|V) \simeq \Spec^{\calG} A$ is affine, then 
$$\chi(V)(U) \simeq \bHom_{\Pro(\calG)}(\Gamma( \calX_{/U}, \calO_{\calX}|U),A)$$
is essentially small for each $U \in \calX$, so that $V \in \calY^{0}$. 
To complete the proof, it will suffice to show that $\calY^{0} = \calY$. In view of Lemma \ref{spak}, it will suffice to show that $\calY^{0}$ is stable under small colimits in $\calY$.
Since $\calX'$ is stable under small colimits in $\Shv_{ \widehat{\SSet}}(\calX)$
(see Remark \toposref{quest}), it will suffice to show that the functor
$\chi$ preserves small colimits. We will prove the following more general assertion:

\begin{itemize}
\item[$(\ast'')$] Let $\{ V_{\alpha} \}$ be a small diagram in $\calY$ with colimit
$V \in \calY$, and let $\eta: F \rightarrow \chi(V)$ be a morphism in
$\Shv_{ \widehat{\SSet} }( \calX)$. Then the induced diagram
$\{ F \times_{ \chi(V) } \chi(V_{\alpha} ) \}$ has colimit $F \in \Shv_{ \widehat{\SSet}}(\calX)$.
\end{itemize}

The collection of $F \in \Shv_{\widehat{\SSet}}(\calX)$ which satisfy $(\ast'')$ is stable under
(not necessarily small) colimits, since colimits are universal in
$\Shv_{ \widehat{\SSet}}(\calX)$ (Remark \toposref{quest}). It will therefore suffice to prove
$(\ast'')$ in the special case where $F \in \Shv_{\widehat{\SSet}}(\calX)$ is a representable functor corresponding to some $U \in \calX$. We can then identify $\eta$ with a morphism
$( \calX_{/U}, \calO_{\calX}|U) \rightarrow (\calY_{/V}, \calO_{\calY}|V)$
in $\LGeo(\calG)^{op}$. Let $f^{\ast}: \calY_{/V} \rightarrow \calX_{/U}$ denote the underlying geometric morphism of $\infty$-topoi. Using Remark \ref{unipop}, we can identify
each $F \times_{ \chi(V) } \chi(V_{\alpha})$ with the functor represented by
$f^{\ast} V_{\alpha} \in \calX$. The desired result now follows from
Remark \toposref{quest} (and the fact that $f^{\ast}$ preserves small colimits).
\end{proof}

One respect in which our theory of $\calG$-schemes generalizes the classical theory of schemes is that we work with arbitrary $\infty$-topoi, rather than ordinary topological spaces. Our next result (which can be regarded as a converse to Proposition \ref{sizem}) shows that, for practical purposes, it often suffices to work with much more concrete objects (such as $1$-topoi):

\begin{theorem}\label{top4}
Let $\calG$ be an $n$-truncated geometry, for some $n \geq -1$. Suppose that
$(\calX, \calO_{\calX})$ is a $\calG$-scheme. Then there exists an $(n+1)$-localic $\calG$-scheme
$(\calY, \calO_{\calY})$, an $(n+2)$-connective object $U \in \calY$, and an equivalence
$$(\calX, \calO_{\calX} ) \simeq ( \calY_{/U}, \calO_{\calY} | U ).$$ 
\end{theorem}

Before giving the proof, we need a few preliminary results about \etale maps between $\infty$-topoi.

\begin{lemma}\label{camb}
Let $\calX$ be an $n$-localic $\infty$-topos for some $n \geq 0$ and let $U$ be an object of $\calX$. The following conditions are equivalent:
\begin{itemize}
\item[$(1)$] The pullback functor $f^{\ast}: \calX \rightarrow \calX_{/U}$ induces an equivalence
on $(n-1)$-truncated objects.
\item[$(2)$] The object $U \in \calX$ is $(n+1)$-connective.
\end{itemize}
\end{lemma}

\begin{proof}
The implication $(2) \Rightarrow (1)$ follows from Lemma \toposref{nicelemma}.
To prove the reverse implication, choose an $(n+1)$-connective $\alpha: U \rightarrow U'$, where
$U'$ is $n$-truncated. The implication $(2) \Rightarrow (1)$ guarantees that
$\tau_{\leq n-1} \calX_{/U} \simeq \tau_{\leq n-1} \calX_{/U'}$. Replacing $U$ by $U'$, we may reduce to the case where $U$ is $n$-truncated; we wish to show that $U$ is a final object of $\calX$.

Since $\calX$ is $n$-localic, there exists an effective epimorphism $\phi: V \rightarrow U$, where
$V$ is $(n-1)$-truncated. The assumption that $U$ is $n$-truncated implies that $\phi$ is 
$(n-1)$-truncated. Invoking $(1)$, we deduce that $V \simeq U \times X$ for some $(n-1)$-truncated object $X \in \calX$. It follows that we have an isomorphism $\phi^{\ast}(\pi_{n} U) \simeq
\pi_{n} V$ in the topos $\h{(\tau_{\leq 0} \calX_{/V})}$. Since $V$ is $(n-1)$-truncated, we deduce that
$\phi^{\ast}( \pi_{n} U) \simeq \ast$. Because $\phi$ is an effective epimorphism, we conclude that
$\pi_{n} U \simeq \ast$, so that $U$ is $(n-1)$-truncated. We then have an equivalence
$ \tau_{\leq n-1} \calX_{/U} \simeq ( \tau_{\leq n-1} \calX)_{/U}$, so that $(1)$ implies that the projection
$( \tau_{\leq n-1} \calX)_{/U} \rightarrow \tau_{\leq n-1} \calX$ is an equivalence. This implies that
$U$ is a final object of $\tau_{\leq n-1} \calX$ (hence a final object of $\calX$), as desired.
\end{proof}

\begin{lemma}\label{pretub}
Let $\calC$ be an $\infty$-category and $S$ a collection of morphisms in $\calC$.
Let $U$ an $S$-local object of $\calC$, let $\pi: \calC_{/U} \rightarrow \calC$ denote the projection, and let $T = \pi^{-1}(S)$. Then $T^{-1} \calC_{/U} = \calC_{/U} \times_{\calC} S^{-1} \calC$
(as full subcategories of $\calC_{/U}$).
\end{lemma}

\begin{proof}
It clearly suffices to prove Lemma \ref{pretub} in the special case where $S$ consists
of a {\em single} morphism $A \rightarrow B$ in $\calC$. Let $\eta: X \rightarrow U$
be an object of $\calC_{/U}$. We have a homotopy commutative diagram
$$ \xymatrix{ \bHom_{\calC}( B, X) \ar[r]^{\phi_{B}} \ar[d]^{\psi} & \bHom_{\calC}( B, U) \ar[d]^{\psi_0} \\
\bHom_{\calC}(A, X) \ar[r]^{\phi_{A}} & \bHom_{\calC}(A, U) }$$
By definition, $\eta$ belongs to $\calC_{/U} \times_{\calC} S^{-1} \calC$ if and only if $
\psi$ is a homotopy equivalence. Since $U$ is $S$-local, the map $\psi_0$ is a homotopy equivalence; thus $\eta \in \calC_{/U} \times_{\calC} S^{-1} \calC$ if and only if
$\psi$ induces an equivalence after passing to the homotopy fiber over any point $\delta: \bHom_{\calC}(B, U)$. If we identify $\delta$ with the corresponding commutative diagram
$$ \xymatrix{ A \ar[rr] \ar[dr]^{\alpha} & & \ar[dl]^{\beta} B \\
& U, & }$$
regarded as an element of $T$, then we can identify the induced map of homotopy fibers with the map
$$ \bHom_{ \calC_{/U}}( \beta, \eta) \rightarrow \bHom_{\calC_{/U}}( \alpha, \eta)$$
induced by composition with $\delta$. Consequently, $\eta \in \calC_{/U} \times_{\calC} S^{-1} \calC$ if and only if $\eta$ is $T$-local, as desired.
\end{proof}

\begin{lemma}\label{tub}
Let $\calX$ be an $n$-localic $\infty$-topos for some $n \geq 0$, and let $U$ be an object of $\calX$. The following conditions are equivalent:
\begin{itemize}
\item[$(1)$] The object $U$ is $n$-truncated.
\item[$(2)$] The $\infty$-topos $\calX_{/U}$ is $n$-localic.
\end{itemize}
\end{lemma}

\begin{proof}
We begin by showing that $(1) \Rightarrow (2)$. Consider first the case where
$\calX = \calP(\calC)$, where $\calC$ is a small $n$-category which admits finite limits. 
Corollary \toposref{swapKK} implies that $\calP(\calC)_{/U}$ is equivalent to the
presheaf $\infty$-category $\Fun( {\calC'}^{op}, \SSet)$, where 
$\calC' = \calC \times_{ \calP(\calC) } \calP(\calC)_{/U}$. We note that the canonical projection
$p: \calC' \rightarrow \calC$ is a right fibration associated to the functor
$U: \calC^{op} \rightarrow \SSet$; in particular, $\calC'$ is essentially small.
For every pair of objects $x,y \in \calC'$, we have a fiber sequence
$$ \bHom_{\calC'}(x,y) \rightarrow \bHom_{\calC}(x,y) \rightarrow
\bHom_{\SSet}( U(y), U(x) ).$$
Since $U$ is $n$-truncated and $\calC$ is an $n$-category, we conclude that
$\bHom_{\calC'}(x,y)$ is $(n-1)$-truncated, so that $\calC'$ is equivalent to an $n$-category.
Using Propositions \toposref{needed17} and \toposref{yonedaprop}, we conclude that
$\calC'$ admits finite limits, so that the presheaf $\infty$-category
$\Fun( {\calC'}^{op}, \SSet)$ is an $n$-localic $\infty$-topos as desired.

We now prove that $(1) \Rightarrow (2)$ in general. Since $\calX$ is $n$-localic, we may assume without loss of generality that $\calX = S^{-1} \calP(\calC)$, where $\calC$ is a small $n$-category which admits finite limits, and $S$ is topological (see Definition \toposref{deftoploc}). 
Lemma \ref{pretub} allows us to identify $\calX_{/U}$ with $T^{-1} \calP(\calC)_{/U}$, where
$T$ is the inverse image of $S$ in $\calP(\calC)_{/U}$. The first part of the proof shows that
$\calP(\calC)_{/U}$ is $n$-localic. According to Proposition \toposref{useiron}, it will suffice to show that the strongly saturated class of morphisms $T$ is topological. 

Let $T_0 \subseteq T$ be the smallest strongly saturated class of morphisms in $\calP(\calC)_{/U}$ which is stable under pullbacks and contains every monomorphism belonging to $T$. 
We wish to show that $T_0 = T$. Consider an arbitrary diagram
$$ \xymatrix{ X \ar[dr] \ar[rr]^{f} & & Y \ar[dl] \\
& U & }$$
in $\calP(\calC)$, corresponding to an element of $T$. We then have a pullback diagram
in $\calP(\calC)_{/U}$:
$$ \xymatrix{ X \ar[r]^{f} \ar[d] & Y \ar[d] \\
X \times U \ar[r]^{f \times \id_{U}} & Y \times U. }$$
Since $T_0$ is stable under pullbacks, it will suffice to show that $f \times \id_{U}$ belongs to
$T_0$. Let $S_0$ be the collection of all morphisms $g \in S$ such that
$g \times \id_{U} \in T_0$. Then $S_0$ is strongly saturated, stable under pullbacks, and
contains every monomorphism in $S$. Since $S$ is topological, $S_0 = S$ so that
$f \in S$, as desired. This completes the proof of the implication $(1) \Rightarrow (2)$.

Now suppose that $(2)$ is satisfied. Choose an $(n+1)$-connective morphism
$U \rightarrow U'$, where $U'$ is $n$-truncated. The first part of the proof shows that
$\calX_{/U'}$ is $n$-localic. We may therefore replace $U$ by $U'$ and thereby reduce to the case where $U$ is $(n+1)$-connective; we wish to show that $U$ is a final object of $\calX$. Lemma \ref{camb} implies that the geometric morphism
$\pi^{\ast}: \calX \rightarrow \calX_{/U}$ induces an equivalence on $(n-1)$-truncated objects.
Since $\calX$ and $\calX_{/U}$ are both $n$-localic, we conclude that $\pi^{\ast}$ is an equivalence of $\infty$-categories, so that $U$ is a final object of $\calX$ as desired.
\end{proof}

\begin{lemma}\label{upsight}
Let $\calG$ be a geometry, and let $(\calX, \calO_{\calX})$ be an affine $\calG$-scheme.
Then $\calX$ is generated under small colimits by fiber products 
$1_{\calX} \times_{ \calO_{\calX}(X)} \calO_{\calX}(U)$, determined by
admissible morphisms $U \rightarrow X$ in $\calG$ and global sections
$1_{\calX} \rightarrow \calO_{\calX}(X)$. 
\end{lemma}

\begin{proof}
In view of Theorem \ref{scoo}, we may suppose that $(\calX, \calO_{\calX}) =
( \USpec A, \calO_{ \USpec A})$, for some object $A \in \Pro(\calG)$.
The $\infty$-topos $\USpec A$ is a localization of the $\infty$-category
$\calX' = \Fun( (\Pro(\calG)^{\adm}_{/A})^{op}, \SSet)$, and $\calX'$ is generated under small colimits by
corepresentable functors $e_{\phi}$, where $\phi$ ranges over admissible morphisms
$B \rightarrow A$ in $\Pro(\calG)$. Let $\overline{e}_{\phi} \in \calX$ denote the image of
$e_{\phi}$ under the localization functor $\calX' \rightarrow \calX$. 

By definition, every admissible morphism $B \rightarrow A$ in $\Pro(\calG)$
fits into a pullback diagram
$$ \xymatrix{ B \ar[r] \ar[d] & j(U) \ar[d] \\
A \ar[r] & j(X) }$$
where $j: \calG \rightarrow \Pro(\calG)$ denotes the Yoneda embedding, and $U \rightarrow X$ is an admissible morphism in $X$. Unwinding the definitions, we deduce that
$\overline{e}_{\phi} \simeq 1_{\calX} \times_{ \calO_{\calX}(X) } \calO_{\calX}(U),$
where the section $1_{\calX} \rightarrow \calO_{\calX}(X)$ is determined by the map
$A \rightarrow j(X)$. We now conclude by observing that the functors $\overline{e}_{\phi}$
generate $\calX$ under small colimits.
\end{proof}

\begin{proof}[Proof of Theorem \ref{top4}]
Let $\tau_{\leq n} \calX$ denote the underlying $n$-topos of $\calX$ (that is, the full subcategory of
$\calX$ spanned by the $n$-truncated objects). We let $\calY$ denote the {\it $(n+1)$-localic reflection} of $\calX$. It is characterized up to equivalence by the following properties:
\begin{itemize}
\item[$(a)$] The $\infty$-topos $\calY$ is $(n+1)$-localic.
\item[$(b)$] There is a geometric morphism $\pi^{\ast}: \calY \rightarrow \calX$ which induces an equivalence of $\infty$-categories $\tau_{\leq n} \calY \rightarrow \tau_{\leq n} \calX$.
\end{itemize}
Since $\calG$ is $n$-truncated, the functor $\calO_{\calX}: \calG \rightarrow \calX$ automatically
factors through $\tau_{\leq n} \calX$. We may therefore assume without loss of generality that
$\calO_{\calX} = \pi^{\ast} \circ \calO_{\calY}$ for some $\calG$-structure $\calO_{\calY}: \calG \rightarrow \calY$ (which is determined uniquely up to equivalence). We have an evident morphism
$( \calY, \calO_{\calY}) \rightarrow (\calX, \calO_{\calX})$ in $\LGeo(\calG)$. 

We will show that $\pi^{\ast}$ is \etale: that is, that $\pi^{\ast}$ induces an equivalence
$\calY_{/U} \simeq \calX$ for some object $U \in \calX$. Lemma \ref{camb} will then imply that
$U$ is $(n+2)$-connective. In particular, $U \rightarrow 1_{\calY}$ is an effective epimorphism. Then Proposition \ref{sizem} will imply that $(\calY, \calO_{\calY})$ is a $\calG$-scheme, and the proof will be complete.

Let $\calX^0$ denote the full subcategory of $\calX$ spanned by those objects $V \in \calX$
such that the induced geometric morphism $\pi_{V}^{\ast}: \calY \rightarrow \calX_{/V}$ is \etale.
We wish to show that $\calX^{0} = \calX$. Proposition \ref{scanh} implies that $\calX^{0}$ is stable under small colimits in $\calX$. In view of Lemma \ref{spak}, it will suffice to show that
$\calX^{0}$ contains every object $V$ for which the
$\calG$-scheme $(\calX_{/V}, \calO_{\calX} | V)$ is affine. Assumption $(b)$ implies that
$\tau_{\leq n} V \simeq \pi^{\ast} W$, for some $n$-truncated object $W \in \calY$.
The truncation map $\alpha: V \rightarrow \pi^{\ast} W$ determines a homotopy commutative diagram of $\infty$-topoi
$$ \xymatrix{ \calY \ar[r]^{\pi^{\ast}} \ar[d] & \calX \ar[d] \\
\calY_{/W} \ar[r]^{\phi^{\ast} } & \calX_{/V}. }$$
To complete the proof, it will suffice to show that $\phi^{\ast}$ is an equivalence of $\infty$-categories.

By assumption, $(\calX_{/V}, \calO_{\calX}|V)$ is an affine $\calG$-scheme. Since the geometry $\calG$ is $n$-truncated, Theorem \ref{scoo} implies that $\calX_{/V}$ is $(n+1)$-localic. Since $\calY$ is
$(n+1)$-localic and $W$ is $n$-truncated, Lemma \ref{tub} implies that $\calY_{/W}$ is $(n+1)$-localic.
It will therefore suffice to show that $\phi^{\ast}$ induces an equivalence of
$(n+1)$-topoi $\phi^{\ast}_{\leq n}: \tau_{\leq n} \calY_{/W} \rightarrow \tau_{\leq n} \calX_{/V}$.
Since $W$ is $n$-truncated, we have canonical equivalences
$$\tau_{\leq n} \calY_{/W} = ( \tau_{\leq n} \calY)_{/W}
\simeq ( \tau_{\leq n} \calX)_{/ \pi^{\ast} W}
\simeq \tau_{\leq n} \calX_{/\pi^{\ast} W}.$$
Since the morphism $\alpha$ is $(n+1)$-connective, Lemma \toposref{nicelemma} implies that
the pullback functor $\tau_{\leq n} \calX_{ / \pi^{\ast} W } \rightarrow \tau_{\leq n} \calX_{/ V}$
is fully faithful. This proves that $\phi_{\leq n}^{\ast}$ is fully faithful.

To complete the proof, it will suffice to show that $\phi_{\leq n}^{\ast}$ is essentially surjective.
Let $\calC$ denote the essential image of $\phi_{\leq n}^{\ast}$. Since $\phi_{\leq n}^{\ast}$ is preserves finite limits, small colimits, and is fully faithful, the subcategory $\calC \subseteq \tau_{\leq n} \calX_{/V}$ is stable under finite limits and small colimits. Because $\phi^{\ast} \circ ( \calO_{\calY} | W) \simeq (\calO_{\calX} | V)$, the $\infty$-category $\calC$ contains the essential image of $\calO_{\calX} | V$. 
Since $( \calX_{/V}, \calO_{\calX} | V)$ is an affine $\calG$-scheme, Lemma \ref{upsight} implies that
$\calX_{/V}$ is generated under small colimits by $\calC \subseteq \calX_{/V}$. It follows that
$\tau_{\leq n} \calX_{/V}$ is generated under small colimits by $\calC$. Since
$\calC$ is stable under small colimits in $\tau_{\leq n} \calX_{/V}$, we deduce that
$\calC = \tau_{\leq n} \calX_{/V}$ as desired.
\end{proof}

We now discuss the behavior of the $\infty$-categories of $\calG$-schemes as $\calG$ varies.

\begin{proposition}
Let $f: \calG \rightarrow \calG'$ be a transformation of geometries, and
$\Spec^{\calG'}_{\calG}: \LGeo(\calG)^{op} \rightarrow \LGeo(\calG')^{op}$ the relative spectrum functor
constructed in \S \ref{relspec}. Then:
\begin{itemize}
\item[$(1)$] The functor $\Spec^{\calG'}_{\calG}$ preserves \etale morphisms. More precisely,
suppose that $(\calX, \calO_{\calX}) \in \LGeo(\calG)$. Let $U \in \calX$, let
$(\calX', \calO'_{\calX'}) \in \LGeo(\calG')$, and let $\alpha: (\calX', \calO'_{\calX'} \circ f) \rightarrow
(\calX, \calO_{\calX})$ be a morphism in $\LGeo(\calG)^{op}$ which exhibits
$(\calX', \calO'_{\calX'})$ as a relative spectrum of $(\calX, \calO_{\calX})$. Then the induced map
$$ ( \calX'_{/ \alpha^{\ast} U}, (\calO_{\calX'} | \alpha^{\ast} U) \circ f)
\rightarrow (\calX_{/U}, \calO_{\calX} |U)$$ exhibits
$(\calX'_{/ \alpha^{\ast} U}, \calO_{\calX'} | \alpha^{\ast} U)$ as a relative spectrum
of $(\calX_{/U}, \calO_{\calX} | U )$. 
\item[$(2)$] The relative spectrum functor $\Spec_{\calG}^{\calG'}$ carries
affine $\calG$-schemes to affine $\calG'$-schemes.
\item[$(3)$] The relative spectrum functor $\Spec_{\calG}^{\calG'}$ carries
$\calG$-schemes to $\calG'$-schemes.
\item[$(4)$] The relative spectrum functor $\Spec_{\calG}^{\calG'}$ carries
$\calG$-schemes which are locally of finite presentation to $\calG'$-schemes which are locally of finite presentation.
\end{itemize}
\end{proposition}

\begin{proof}
Assertion $(1)$ follows from Lemma \ref{specred} and $(2)$ from the homotopy commutative diagram
$$ \xymatrix{ \Pro(\calG) \ar[r]^{ \Spec^{\calG} } \ar[d] & \LGeo(\calG)^{op} \ar[d]^{ \Spec_{\calG}^{\calG'} } \\ 
\Pro(\calG') \ar[r]^{ \Spec^{\calG'} } & \LGeo(\calG')^{op}. }$$
Assertions $(3)$ and $(4)$ follow from $(2)$, together with Lemma \ref{specred}.
\end{proof}

\begin{remark}\label{armur}
Let $\calG$ be a geometry. Suppose that the topology on
$\Pro(\calG)$ is {\it precanonical} (see Remark \ref{juna}). Then for any pair of objects,
$A, B \in \Pro(\calG)$, we have canonical homotopy equivalences
\begin{eqnarray*}
\bHom_{ \Sch(\calG)}( \Spec^{\calG} A, \Spec^{\calG} B) & \simeq & 
\bHom_{ \Pro(\calG)}( A, \Gamma_{\calG}( \Spec^{\calG} B) ) \\
& \simeq & \bHom_{\Pro(\calG)}( A, \Gamma_{\calG}( \USpec B, \calO_{\USpec B} ) \\
& \simeq & \bHom_{ \Pro(\calG)}( A, B).
\end{eqnarray*}
In other words, the functor $\Spec^{\calG}: \Pro(\calG) \rightarrow \Sch(\calG)$ is fully faithful.
\end{remark}

We close with a brief discussion of the existence of limits of $\calG$-schemes.

\begin{remark}\label{pre4}
Let $\calG$ be a geometry, and let $f: (\calX, \calO_{\calX}) \rightarrow (\calY, \calO_{\calY})$ be a morphism in $\LGeo(\calG)^{op}$. For every object $U \in \calY$, we have a pullback diagram
$$ \xymatrix{ ( \calX_{/ f^{\ast} U}, \calO_{\calX} | f^{\ast} U) \ar[r] \ar[d] & (\calY_{/U}, \calO_{\calY} | U ) \ar[d] \\
(\calX, \calO_{\calX}) \ar[r] & ( \calY, \calO_{\calY});}$$
this follows easily from Remark \ref{unipop}.
\end{remark}

\begin{proposition}\label{cuble}
Let $\calG$ be a geometry. Then:
\begin{itemize}
\item[$(1)$] The full subcategory $\Sch^{\fin}(\calG) \subseteq \Sch(\calG)$ is admits finite limits, and the inclusion $\Sch^{\fin}(\calG) \subseteq \LGeo(\calG)^{op}$ preserves finite limits.
\item[$(2)$] Suppose that the Grothendieck topology on $\Pro(\calG)$ is precanonical. Then
$\Sch(\calG)$ admits finite limits, and the inclusion $\Sch(\calG) \subseteq \LGeo(\calG)^{op}$ preserves finite limits.
\end{itemize}
\end{proposition}

Before giving the proof, we need to establish a bit of notation. Let $\widehat{\SSet}$ denote the
$\infty$-category of spaces which are not necessarily small. We will say that a functor
$F: \LGeo(\calG) \rightarrow \widehat{\SSet}$ {\it satisfies descent} if, for every
object $(\calX, \calO_{\calX}) \in \LGeo(\calG)$, the composition
$$ \calX^{op} \simeq \LGeo(\calG)_{(\calX, \calO_{\calX})/}^{\mathet}
\rightarrow \LGeo(\calG) \stackrel{F}{\rightarrow} \widehat{\SSet}$$
preserves small limits. Let $\widehat{\Shv}(\LGeo(\calG)^{op})$
denote the full subcategory of $\Fun( \LGeo(\calG), \widehat{\SSet} )$ spanned by those functors which satisfy descent
(the $\infty$-category $\widehat{\Shv}(\LGeo(\calG)^{op})$ can be regarded as
an $\infty$-topos in a larger universe, but we will not need this fact). Proposition \ref{sizem2} implies that the Yoneda embedding $j: \LGeo(\calG)^{op} \rightarrow \Fun( \LGeo(\calG), \widehat{\SSet} )$
factors through $\widehat{\Shv}(\LGeo(\calG)^{op})$.

\begin{lemma}\label{junner}
Suppose given a commutative diagram
$$ \xymatrix{ F \ar[rr]^{\phi} \ar[dr] & & F' \ar[dl] \\
& j( \calX, \calO_{\calX} ) & }$$
in $\widehat{\Shv}(\LGeo(\calG)^{op})$. Suppose further that there exists a collection of objects $U_{\alpha} \in \calX$ with the following properties:
\begin{itemize}
\item[$(a)$] The objects $U_{\alpha}$ cover $\calX$; in other words, the map
$\coprod U_{\alpha} \rightarrow 1_{\calX}$ is an effective epimorphism.
\item[$(b)$] For each index $\alpha$, the induced map
$$F \times_{ j(\calX, \calO_{\calX}) } j(\calX_{/U_{\alpha}}, \calO_{\calX}|U_{\alpha})
\rightarrow F' \times_{ j(\calX, \calO_{\calX}) } j(\calX_{/U_{\alpha}}, \calO_{\calX}|U_{\alpha})$$
is an equivalence.
\end{itemize}
Then $\phi$ is an equivalence.
\end{lemma}

\begin{proof}
Let $(\calY, \calO_{\calY})$ be an object of $\LGeo(\calG)$, and let
$\eta$ be a point of $F'( \calY, \calO_{\calY})$. We wish to show that the homotopy fiber
of the map $F( \calY, \calO_{\calY}) \rightarrow F'(\calY, \calO_{\calY})$ over $\eta$ is contractible.
Let $f$ be the functor given by the composition
$$ \calY^{op} \simeq \LGeo(\calG)_{(\calY, \calO_{\calY})/}^{\mathet}
\rightarrow \LGeo(\calG) \stackrel{F}{\rightarrow} \widehat{\SSet},$$
let $f': \calY^{op} \rightarrow \widehat{\SSet}$ be defined similarly, and let
$g: \calY^{op} \rightarrow \widehat{\SSet}$ be the functor described by the formula
$g(V) = f(V) \times_{ f'(V) } \{ \eta \}$ (where we identify $\eta$ with its image in
$F'( \calY_{/V}, \calO_{\calY} | V )$. Let $\calY^{0}$ denote the full subcategory of $\calY$ spanned by those objects $V \in \calY$ such that $g(V)$ is contractible. Since $g$ preserves small limits, $\calY^{0}$ is stable under small colimits in $\calY$. To complete the proof, it will suffice to show that
$\calY^{0}$ contains the final object $1_{\calY} \in \calY$.

The point $\eta$ determines a geometric morphism $\phi_{\ast}: \calY \rightarrow \calX$.
For each index $\alpha$, let $V_{\alpha} = \phi^{\ast} U_{\alpha} \in \calY$. 
Let $Y_0 = \coprod_{\alpha} V_{\alpha}$. Assumption $(a)$ guarantees that
the projection map $p: Y_0 \rightarrow 1_{\calY}$ is an effective epimorphism.
Let $Y_{\bigdot}$ be the simplicial object of $\calY$ given by the \Cech nerve of
$p$, so that 
$$Y_{n} \simeq \coprod_{ \alpha_0, \ldots, \alpha_n } V_{\alpha_0} \times \ldots \times V_{\alpha_n}.$$
Hypothesis
$(b)$ implies that $\calY^{0}$ contains the sieve generated by the objects $V_{\alpha}$;
in particular, it contains every product $V_{\alpha_0} \times \ldots \times V_{\alpha_n}$.
Since $\calY^{0}$ is stable under small colimits, it contains each $Y_{n}$, and therefore
contains the geometric realization $| Y_{\bigdot} | \simeq 1_{\calY}$ as desired.
\end{proof}

\begin{lemma}\label{sado}
Let $\alpha: F \rightarrow j(\calX, \calO_{\calX})$ be a morphism in
$\widehat{\Shv}(\LGeo(\calG)^{op})$. Suppose that there exists a family of objects
$\{ U_{\alpha} \in \calX \}$ with the following properties:
\begin{itemize}
\item[$(a)$] The objects $U_{\alpha}$ cover $\calX$; in other words, the map
$\coprod U_{\alpha} \rightarrow 1_{\calX}$ is an effective epimorphism.
\item[$(b)$] For each index $\alpha$, the fiber product
$$F \times_{ j(\calX, \calO_{\calX}) } j(\calX_{/U_{\alpha}}, \calO_{\calX}|U_{\alpha})$$
is representable by a $\calG$-scheme $(\calY_{\alpha}, \calO_{\calY_{\alpha}})$. 
\end{itemize}
Then $F$ is representable by a $\calG$-scheme $(\calY, \calO_{\calY})$.
\end{lemma}

\begin{remark}\label{sadoprime}
In the situation of Lemma \ref{sado}, the $\calG$-schemes $(\calY_{\alpha}, \calO_{\calY_{\alpha}})$
form an \etale covering of $(\calY, \calO_{\calY})$. In particular, if each $(\calY_{\alpha}, \calO_{\calY_{\alpha}})$ is locally of finite presentation, then $(\calY, \calO_{\calY})$ is locally of finite presentation. 
\end{remark}

\begin{proof}
Let $\calX^{0}$ denote the full subcategory of $\calX$ spanned by those objects
$U$ for which the fiber product
$$F \times_{ j(\calX, \calO_{\calX}) } j(\calX_{/U}, \calO_{\calX}|U)$$
is representable by a $\calG$-scheme $X_U$. Remark \ref{pre4} implies that
$\calX^{0}$ is a sieve in $\calX$: that is, if $f: V \rightarrow U$ is a morphism in
$\calX$ and $U \in \calX^{0}$, then $V \in \calX^{0}$. We will prove:
\begin{itemize}
\item[$(\ast)$] The full subcategory $\calX^{0} \subseteq \calX$ is stable under small colimits.
\end{itemize}
Assuming $(\ast)$ for the moment, we can complete the argument as follows. By hypothesis, each $U_{\alpha}$ belongs to $\calX^{0}$. By $(\ast)$, the object $X_0 = \coprod_{\alpha} U_{\alpha}$ belongs to $\calX^{0}$. Let $X_{\bigdot}$ be the simplicial object of $\calX$ given by the
\Cech nerve of $X_0 \rightarrow 1_{\calX}$ (so that $X_{n} \simeq X_0^{n+1}$). Since
$\calX^{0}$ is a sieve, each $X_{n}$ belongs to $\calX^{0}$. Applying $(\ast)$ again, we deduce that
$1_{\calX} \simeq | X_{\bigdot} |$ belongs to $\calX^{0}$, which is equivalent to the desired assertion.

To prove $(\ast)$, let us consider a small diagram $p: K \rightarrow \calX^{0}$ and let
$U = \colim(p) \in \calX$; we wish to show that $U \in \calX^{0}$. Let $p': K \rightarrow \Sch(\calG)$ be the functor described by the formula $p'(v) = X_{p(v)}$. Using Remark \ref{pre4}, we deduce that $p'$ factors through $\Sch(\calG)_{\mathet}$. Proposition \ref{sizem} implies that $p'$ admits a colimit $(\calZ, \calO_{\calZ})$ in $\Sch(\calG)_{\mathet}$. Since $F$ satisfies descent, we obtain
a canonical map $j(\calZ, \calO_{\calZ}) \rightarrow F$. It follows from Lemma \ref{junner} that this map is an equivalence.
\end{proof}

\begin{proof}[Proof of Proposition \ref{cuble}]
Let $\calC$ denote the essential image of the composite map
$$ \Sch(\calG) \subseteq \LGeo(\calG)^{op} \stackrel{j}{\rightarrow} \Fun( \LGeo(\calG), \widehat{\SSet}),$$
and let $\calC^{\fin} \subseteq \calC$ be defined similarly. Assertions $(1)$ and $(2)$ can be reformulated as follows:
\begin{itemize}
\item[$(1')$] The full subcategory $\calC^{\fin} \subseteq \Fun( \LGeo(\calG), \widehat{\SSet})$ is stable under finite limits.
\item[$(2')$] If the Grothendieck topology on $\Pro(\calG)$ is precanonical, then the full subcategory $\calC \subseteq \Fun( \LGeo(\calG), \widehat{\SSet})$ is stable under finite limits.
\end{itemize}
We will prove $(1')$ and $(2')$ under the assumption that the Grothendieck topology on
$\Pro(\calG)$ is precanonical, indicating where necessary how to eliminate this hypothesis in the proof of $(1')$. We begin by observing that $\calC^{\fin} \subseteq \calC$ contains the final object of
$\Fun( \LGeo(\calG), \widehat{\SSet})$, since this final object is representable by the spectrum of the final object of $\Pro(\calG)$. It is therefore sufficient to show that $\calC^{\fin}$ and $\calC$ are stable under the formation of pullbacks.

Consider a pullback diagram
$$ \xymatrix{ F' \ar[r] \ar[d] & G' \ar[d]^{\alpha} \\
F \ar[r]^{\beta} & G }$$
in $\Fun( \LGeo(\calG), \widehat{\SSet})$, where $G'$, $G$ and $F$ belong to $\calC$ (or $\calC^{\fin})$. Then $G$ is representable by a $\calG$-scheme $(\calX, \calO_{\calX})$ (which is locally of finite presentation). We wish to show that $F' \in \calC$ (or $\calC^{\fin}$). 
In view of Lemma \ref{sado} (and Remark \ref{sadoprime}), the assertion is local on $\calX$; we may therefore assume without loss of generality that $(\calX, \calO_{\calX}) = \Spec^{\calG} A$ is affine
(where $A$ belongs to the essential image of the Yoneda embedding $\calG \rightarrow \Pro(\calG)$).

By assumption, the functor $G'$ is representable by a $\calG$-scheme $(\calY, \calO_{\calY})$. 
The desired conclusion is also local on $\calY$, so we may assume without loss of generality that $(\calY, \calO_{\calY}) = \Spec^{\calG} A'$ is affine (and $A'$ belongs to the essential image of the Yoneda embedding $\calG \rightarrow \Pro(\calG)$). Since the Grothendieck topology on $\Pro(\calG)$ is precanonical, Remark \ref{armur} implies that the map $\alpha: G' \rightarrow G$ is induced by a morphism $A' \rightarrow A$ in $\Pro(\calG)$ (in the proof of $(1')$, we must work a bit harder.
The universal property of $\Spec^{\calG} A$ implies that $\alpha: G' \rightarrow G$ is determined by
a point $\eta \in \calO_{\calY}(A)$. The construction of Theorem \ref{scoo} allows us to identify
$\calO_{\calY}(A)$ with the sheafification of the presheaf $(\calG_{/A'}^{\adm})^{op} \rightarrow \SSet$
described by the formula $A'' \mapsto \bHom_{\calG}(A'',A)$. It follows that after further localization on $\calY$, we may assume that $\alpha$ is induced by a map $A' \rightarrow A$ in $\calG$).

Using the same argument, we can suppose that $F$ is representable by the affine $\calG$-scheme $\Spec^{\calG} B$, where $B \in \Pro(\calG)$ (and belongs to the essential image of the Yoneda embedding $\calG \rightarrow \Pro(\calG)$), and $\beta$ is induced by a morphism
$B \rightarrow A$ in $\Pro(\calG)$. It follows that $F'$ is representable by the 
affine $\calG$-scheme $\Spec^{\calG} (B \times_{A} A')$, so that $F' \in \calC$
($F' \in \calC^{\fin}$), as desired.
\end{proof}

\subsection{The Functor of Points}\label{geo7}

In classical algebraic geometry, we can often understand algebraic varieties (or schemes) as arising as the solutions to {\em moduli problems}. For example, the $n$-dimensional projective space
$\mathbf{P}^{n}$ can be characterized as follows: it is universal among schemes over which
there is a line bundle $\calL$ generated by $(n+1)$ global sections. In particular, for any commutative ring $A$, the set $\Hom( \SSpec A, \mathbf{P}^{n})$ can be identified with the set of isomorphism classes of pairs $(L, \eta: A^{n+1} \rightarrow L)$ where $L$ is an invertible $A$-module and
$\eta$ is a surjective map of $A$ modules (such a pair is determined up to unique isomorphism
by the submodule $\ker(\eta) \subseteq A^{n+1})$. 

More generally, {\em any} scheme $X$ determines a covariant functor from commutative rings to sets, given by the formula
$$ A \mapsto \Hom( \SSpec A, X).$$
This functor determines $X$ up to canonical isomorphism. More precisely, the above construction yields a fully faithful embedding from the category of schemes to the presheaf category $\Fun( \Comm, \Set)$.
Consequently, it is possible to think of schemes as objects of $\Fun( \Comm, \Set)$, rather than the category of locally ringed spaces. This point of view is often valuable: frequently it is easier to describe the functor represented by a scheme $X$ than it is to give an explicit construction of $X$ as a locally ringed space. Moreover, this perspective becomes essential when we wish to study more general algebro-geometric objects, such as algebraic stacks.

Our goal in this section is to obtain an analogous understanding of the $\infty$-category
$\Sch(\calG)$ of $\calG$-schemes, where $\calG$ is an arbitrary geometry.
We begin by reviewing a bit of notation. Let $\SSet$ denote the $\infty$-category of (small) spaces, and $\widehat{\SSet}$ the larger $\infty$-category of spaces which are not necessarily small.
Fix a geometry $\calG$, and let $j: \LGeo(\calG)^{op} \rightarrow \Fun( \LGeo(\calG), \widehat{\SSet})$ denote the Yoneda embedding. The main result of this section is the following:

\begin{theorem}\label{skil}
Let $\calG$ be a geometry, and let $\phi$ denote the composite functor
$$ \Sch(\calG) \subseteq \LGeo(\calG)^{op}
\stackrel{j}{\rightarrow} \Fun( \LGeo(\calG), \widehat{\SSet} )
\stackrel{\circ \Spec^{\calG} }{\rightarrow} \Fun( \Ind(\calG^{op}), \widehat{\SSet} ),$$
where $\Spec^{\calG}: \Ind(\calG^{op}) \rightarrow \LGeo(\calG)$ denotes the absolute spectrum
functor constructed in \S \ref{abspec}. Then:
\begin{itemize}
\item[$(1)$] The essential image of $\phi$ is contained in the essential image of the inclusion
$$ \Fun( \Ind(\calG^{op}), \SSet) \subseteq \Fun( \Ind(\calG^{op}), \widehat{\SSet}).$$ 
\item[$(2)$] The functor $\phi$ is fully faithful.
\end{itemize}
\end{theorem}

Assertion $(1)$ is equivalent to the statement that
the mapping space $\bHom_{ \LGeo(\calG) }( (\calX, \calO_{\calX}), \Spec^{\calG} A)$
is essentially small whenever $(\calX, \calO_{\calX})$ is a $\calG$ scheme and
$A \in \Pro(\calG)$; this follows from Proposition \ref{sizem2}. The proof of $(2)$ will occupy the remainder of this section. 

\begin{definition}
Let $\calG$ be a geometry.
We will say that a functor $F: \Ind(\calG^{op}) \rightarrow \widehat{\SSet}$ is {\it representable by a $\calG$-scheme} if $F$ belongs to the essential image of the functor $\phi$ appearing in
Theorem \ref{skil}. 
\end{definition}

Our first step is to isolate one key feature of representable functors $F: \Ind(\calG^{op}) \rightarrow \widehat{\SSet}$. Namely, they are {\em sheaves} with respect to the natural Grothendieck topology on $\Pro(\calG)$. However, we must exercise a bit of care because the $\infty$-category $\Pro(\calG)$ is not small.

\begin{definition}\label{sputt}
Let $\calG$ be a geometry, and let $F: \Ind(\calG^{op}) \rightarrow \widehat{\SSet}$ be a
functor. We will say that $F$ is a {\it sheaf} if, for every object $A \in \Ind(\calG^{op})$, the
composition
$$ \Ind(\calG^{op})_{A/} \rightarrow \Ind(\calG^{op}) \stackrel{F}{\rightarrow} \widehat{\SSet}$$
is a sheaf in the sense of Definition \ref{ury} (here we regard $(\Ind(\calG^{op})_{/A})^{op} \simeq
\Pro(\calG)_{/A}$ as endowed with the Grothendieck topology described in Notation \ref{scun}). 
We let $\widehat{\Shv}( \Pro(\calG) )$ denote the full subcategory of
$\Fun( \Ind(\calG^{op}), \widehat{\SSet})$ spanned by the sheaves, and set
$$\Shv(\Pro(\calG) ) = \widehat{\Shv}(\Pro(\calG)) \cap \Fun( \Ind(\calG^{op}), \SSet).$$
\end{definition}

The basic properties of sheaves are summarized in the following result:

\begin{proposition}\label{kumber}
Let $\calG$ be a geometry.
\begin{itemize}
\item[$(1)$] The inclusion $\widehat{\Shv}( \Pro(\calG) )
\subseteq \Fun( \Ind(\calG^{op}), \widehat{\SSet} )$ admits a left adjoint $\widehat{L}$.
\item[$(2)$] If $F \in \Fun( \Ind(\calG^{op}), \SSet )$, then $\widehat{L}F \in \widehat{\Shv}(\Pro(\calG))$
belongs to the essential image of the inclusion $\Shv( \Pro(\calG) ) \subseteq \widehat{\Shv}( \Pro(\calG) )$. 
\item[$(3)$] The inclusion $\Shv( \Pro(\calG) ) \subseteq \Fun( \Ind(\calG^{op}), \SSet)$ admits a left adjoint $L$.
\item[$(4)$] The $\infty$-category $\Shv( \Pro(\calG) )$ admits small limits and colimits. Moreover, a small diagram $p: K^{\triangleright} \rightarrow \Shv(\Pro(\calG))$ is a colimit diagram if and only if,
for every object $A \in \Pro(\calG)$, the composition
$$ K^{\triangleright} \rightarrow \Shv(\Pro(\calG)) \rightarrow \Shv( \Pro(\calG)_{/A}^{\adm} )$$
is a colimit diagram.
\item[$(5)$] The $\infty$-category $\widehat{\Shv}( \Pro(\calG) )$ admits limits and colimits. Moreover, a diagram $p: K^{\triangleright} \rightarrow \widehat{\Shv}(\Pro(\calG))$ is a colimit diagram if and only if,
for every object $A \in \Pro(\calG)$, the composition
$$ K^{\triangleright} \rightarrow \widehat{\Shv}(\Pro(\calG)) \rightarrow \widehat{\Shv}( \Pro(\calG)_{/A}^{\adm} )$$
is a colimit diagram.
\item[$(6)$] The inclusion $\Shv( \Pro(\calG) ) \subseteq \widehat{\Shv}( \Pro(\calG) )$ preserves small limits and colimits.
\end{itemize}
\end{proposition}

\begin{warning}
The $\infty$-category $\Shv(\Pro(\calG) )$ is usually {\em not} an $\infty$-topos, because it is not
presentable. However, this is merely a technical annoyance. Note that the (very large)
$\infty$-category $\widehat{\Shv}( \Pro(\calG) )$ {\em can} be regarded as an $\infty$-topos
after a change of universe.
\end{warning}

The proof of Proposition \ref{kumber} will require a few preliminary results concerning the theory of sheaves. We first discuss the pushforward operation associated to a morphism between Grothendieck sites.

\begin{lemma}\label{stulbus}
Let $\calC$ be a nonempty $\infty$-category. Assume that for every pair of objects
$X,Y \in \calC$, there exists a product $X \times Y \in \calC$. Then $\calC$ is weakly contractible.
\end{lemma}

\begin{proof}
Fix an object $X \in \calC$. By assumption, the projection $F: \calC_{/X} \rightarrow \calC$ admits
a right adjoint $G$. Then we have natural transformations
$$ \id_{\calC} \leftarrow FG \rightarrow T,$$
where $T: \calC \rightarrow \calC$ is the constant functor taking the value $X$. This
proves that the identity map $\id_{\calC}$ is simplicially homotopic to a constant map, so that
$\calC$ is weakly contractible as desired.
\end{proof}

\begin{lemma}\label{facus}
Let $\calT$ and $\calT'$ be small $\infty$-categories equipped with Grothendieck topologies.
Assume that $\calT$ admits finite limits. Let $f: \calT \rightarrow \calT'$ be a functor with the following properties:
\begin{itemize}
\item[$(a)$] The functor $f$ preserves finite limits.
\item[$(b)$] For every collection of morphisms $\{ U_{\alpha} \rightarrow X\}$ which generate a covering sieve on $X \in \calT$, the resulting collection of morphisms $\{ fU_{\alpha} \rightarrow fX \}$ in $\calT'$ generates a covering sieve on $fX \in \calT'$.
\end{itemize}
Then:
\begin{itemize}
\item[$(1)$] For every $\infty$-category $\calC$, composition with
$f$ induces a functor $f_{\ast}: \Shv_{\calC}( \calT' ) \rightarrow \Shv_{\calC}(\calT)$.
\item[$(2)$] If $\calC = \SSet$, then $f_{\ast}$ is a geometric morphism of $\infty$-topoi.
\end{itemize}
\end{lemma}

\begin{proof}
We first prove $(1)$. Let $\calO: {\calT'}^{op} \rightarrow \calC$ belong to $\Shv_{\calC}(\calT')$. We wish to show that $\calO \circ f$ belongs to $\Shv_{\calC}(\calT)$. In other words, we wish to show that
for every object $X \in \calT$ and every covering sieve $\calT^{(0)}_{/X} \subseteq \calT_{/X}$, the
functor $\calO \circ f$ exhibits $\calO(fX)$ as a limit of the diagram
$$ (\calT^{(0)}_{/X})^{op} \rightarrow \calT^{op} \stackrel{f}{\rightarrow} {\calT'}^{op}
\stackrel{\calO}{\rightarrow} \calC.$$
Let ${\calT'}^{(1)}_{/fX} \subseteq \calT'_{/fX}$ be the sieve generated by
$f \calT^{(0)}_{/X}$. Assumption $(b)$ implies that ${ \calT'}^{(1)}_{/fX}$ is a covering sieve
on $fX$. Since $\calO$ is a sheaf, the diagram
$$ ({\calT'}^{(1)}_{/fX})^{op} \rightarrow {\calT'}^{op} \stackrel{\calO}{\rightarrow} \calC$$
is a limit diagram. It will therefore suffice to show that the functor $f$ induces a cofinal map
$\calT^{(0)}_{/X} \rightarrow {\calT'}^{(1)}_{/fX}$. In view of Corollary \toposref{hollowtt}, 
we must show that for every morphism $Y \rightarrow fX$ belonging to ${\calT'}^{(1)}_{/fX}$, the
$\infty$-category $$\calC = \calT^{(0)}_{/X} \times_{ {\calT'}^{(1)}_{/fX} } {\calT'}^{(1)}_{Y/\, /fX}$$
is weakly contractible. Since ${\calT'}^{(1)}_{/fX}$ is generated by
$\calT^{(0)}_{/X}$, the $\infty$-category $\calC$ is nonempty. Assumption $(a)$ guarantees
that $\calC$ admits pairwise products. The contractibility of $\calC$ now follows from Lemma \ref{stulbus}.

We now prove $(2)$. Let $\overline{f}_{\ast}: \calP(\calT') \rightarrow \calP(\calT)$ be given by composition with $f$, and let $\overline{f}^{\ast}: \calP(\calT) \rightarrow \calP(\calT')$ be a left
adjoint to $\overline{f}_{\ast}$. Proposition \toposref{adjobs} implies that $\overline{f}^{\ast}$ fits into a homotopy commutative diagram
$$ \xymatrix{ \calT \ar[r]^{f} \ar[d] & \calT' \ar[d] \\
\calP(\calT) \ar[r]^{ \overline{f}^{\ast} } & \calP(\calT'), }$$
where the vertical arrows are Yoneda embeddings. Applying $(a)$ and Proposition \toposref{natash}, we conclude that $\overline{f}^{\ast}$ is left exact. We now observe
that $f_{\ast}$ has a left adjoint $f^{\ast}$ given by the composition
$$ \Shv(\calT) \subseteq \calP(\calT) \stackrel{ \overline{f}^{\ast}}{\rightarrow}
\calP(\calT') \stackrel{L}{\rightarrow} \Shv(\calT'),$$
where $L$ is a left adjoint to the inclusion $\Shv(\calT') \subseteq \calP(\calT')$ (and therefore left exact). As a composition of left exact functors, $f^{\ast}$ is left exact as desired.
\end{proof}

\begin{lemma}\label{stulbin}
Let $\calT$ be a small $\infty$-category equipped with a Grothendieck topology.
Let $X$ be an object of $\calT$, and let $\phi: \calP( \calT) \rightarrow \calP( \calT_{/X} )$
be given by composition with the projection $\calT_{/X} \rightarrow \calT$. 
Let $\alpha: F \rightarrow F'$ be a natural transformation in $\calP(\calT)$ which exhibits
$F'$ as a sheafification of $F$. Then $\phi(\alpha): \phi F \rightarrow \phi F'$ exhibits
$\phi F'$ as a sheafification of $\phi F$ (with respect to the induced Grothendieck topology on $\calT_{/X}$).
\end{lemma}

\begin{proof}
Let $L: \calP( \calT) \rightarrow \Shv(\calT)$ denote a left adjoint to the inclusion, and let
$S$ be the collection of all morphisms $f$ in $\calP(\calT)$ such that $Lf$ is an equivalence.
Let $L_{X}: \calP( \calT_{/X} ) \rightarrow \Shv( \calT_{/X} )$ and $S_X$ be defined similarly.
It follows immediately from the definition that $F'_{X} \in \Shv( \calT_{/X} )$. To complete the proof, we must show that $\phi(\alpha) \in S_X$. Since $\alpha \in S$, it will suffice to show that
$\phi(S) \subseteq S_{X}$. 

Let $j: \calT \rightarrow \calP(\calT)$ denote the Yoneda embedding.
The functor $\phi$ preserves finite limits (in fact all limits) and small colimits. Consequently,
$\phi^{-1} S_X$ is a strongly saturated class of morphisms in $\calP(\calT)$ which is stable under pullbacks. Thus, in order to prove that $S \subseteq \phi^{-1} S_X$, it will suffice to show that
$\phi^{-1} S_X$ contains every monomorphism of the form $U \rightarrow j(Y)$, where
$U$ corresponds to a covering sieve $\calT^{(0)}_{/Y} \subseteq \calT_{/Y}$. 
In other words, we must show that the induced map $\phi(U) \rightarrow \phi(j(Y))$
belongs to $S_{X}$. Let $j_X: \calT_{/X} \rightarrow \calP( \calT_{/X} )$ denote the Yoneda embedding for $\calT_{/X}$. Since $S_{X}$ is stable under small colimits, it will suffice to show that
for every object $X' \rightarrow X$ of $\calT_{/X}$ and every map
$\beta: j_X(X') \rightarrow \phi j(Y)$, the induced map
$$ i: \phi(U) \times_{ \phi(j(Y) ) } j_X(X') \rightarrow j_{X}(X')$$
belongs to $S_X$. We can identify that map $\beta$ with a morphism
$X' \rightarrow Y$ in $\calT$. We now observe that $i$ is a monomorphism
classified by the sieve
$$\calT^{(0)}_{/Y} \times_{ \calT_{/Y} } \calT_{/X'} \subseteq \calT_{/X'}.$$
Since this sieve is a covering (with respect to the Grothendieck topology on $\calT_{/X'}$),
the morphism $i$ belongs to $S_X$ as desired.
\end{proof}

\begin{lemma}\label{scur}
Let $p: \calX \rightarrow \calY$ be a functor between $\infty$-categories,
$\calX^{0}$ a full subcategory of $\calX$, and $p_0 = p | \calX^{0}$. Assume that the following conditions are satisfied:
\begin{itemize}
\item[$(a)$] The functors $p$ and $p_0$ are Cartesian fibrations.
\item[$(b)$] The inclusion $\calX^{0} \subseteq \calX$ carries $p_0$-Cartesian morphisms
in $\calX^{0}$ to $p$-Cartesian morphisms in $\calX$.
\item[$(c)$] For every object $Y \in \calY$, the inclusion of fibers
$\calX^{0}_{Y} \subseteq \calX_{Y}$ admits a left adjoint.
\end{itemize}
Let $\calC$ denote the $\infty$-category $\Fun_{ \calY}( \calY, \calX)$ denote the
$\infty$-category of sections of $p$, and let $\calC^{0} = \Fun_{\calY}( \calY, \calX^{0}) \subseteq \calC$ be defined similarly. Then:
\begin{itemize}
\item[$(1)$] The inclusion $\calC^{0} \subseteq \calC$ admits a left adjoint.
\item[$(2)$] A morphism $X \rightarrow X'$ in $\calC$ exhibits
$X'$ as a $\calC^{0}$-localization of $X$ if and only if, for each object $Y \in \calY$, the induced map
$X(Y) \rightarrow X'(Y)$ exhibits $X'(Y)$ as a $\calX^{0}_{Y}$-localization of $X(Y) \in \calX_{Y}$.
\end{itemize}
\end{lemma}

\begin{proof}
This is a special case of Proposition \deformationref{reladjprop}.
\end{proof}

\begin{lemma}\label{biggier}
Let $\calC$ be an essentially small $\infty$-category equipped with a Grothendieck topology.
Let $\alpha: F \rightarrow F'$ be a morphism in $\calP(\calC)$ be a morphism which
exhibits $F'$ as a $\Shv(\calC)$-localization of $F$ in $\calP(\calC)$. Then
$\alpha$ also exhibits $F'$ as a $\Shv_{ \widehat{\SSet}}( \calC)$-localization of
$F$ in $\Fun( \calC^{op}, \widehat{\SSet} )$.
\end{lemma}

\begin{proof}
Without loss of generality, we may suppose that $\calC$ is small.
Let $L$ be a left adjoint to the inclusion $\Shv(\calC) \subseteq \calP(\calC)$, and let
$S$ be the collection of all morphisms $\beta$ in $\calP(\calC)$ such that
$L(\beta)$ is an equivalence. Let 
$$\widehat{S} \subseteq \Hom_{\sSet}( \Delta^1, \Fun( \calC^{op}, \widehat{\SSet} ))$$
be defined similarly. To prove the desired assertion, we must show that
$F' \in \Shv_{ \widehat{\SSet}}(\calC)$ and that $\alpha \in \widehat{S}$.
The first assertion is obvious; to prove the second, it will suffice to show that
$S \subseteq \widehat{S}$.

Let $S'$ be the collection of all morphisms in $\calP(\calC)$ which belong to
$\widehat{S}$. Then $S'$ is strongly saturated and stable under pullbacks; we wish to show that $S \subseteq S'$. For this, it suffices to observe that $S'$ contains every monomorphism
$\beta: \chi \rightarrow j(C)$, where $j: \calC \rightarrow \calP(\calC)$ is the Yoneda embedding,
$C \in \calC$ is an object, and $\chi$ is the subobject of $j(C)$ associated to a covering
sieve $\calC^{0}_{/C} \subseteq \calC_{/C}$.
\end{proof}

\begin{proof}[Proof of Proposition \ref{kumber}]
The proof is similar to that of Lemma \toposref{kumba}.
Let $\Fun^{\adm}( \Delta^1, \Ind(\calG^{op}))$ denote the full subcategory of $\Fun( \Delta^1, \Ind(\calG^{op}))$ spanned by the admissible morphisms in $\Ind(\calG^{op})$, and let
$e: \Fun^{\adm}( \Delta^1, \Ind(\calG^{op})) \rightarrow \Ind(\calG^{op})$ be given by evaluation at the vertex
$\{0\} \in \Delta^1$. Since the collection of admissible morphisms of $\Ind(\calG^{op})$ is stable under pushouts, the map $e$ is a coCartesian fibration.

We define a simplicial set $\calX$ equipped with a projection
$p: \calX \rightarrow \Ind(\calG^{op})$ so that the following universal property is satisfied: for every
simplicial set $K$, we have a natural bijection
$$ \Hom_{\Ind(\calG^{op})}(K, \calX) = \Hom_{\sSet}(
K \times_{ \Ind(\calG^{op}) } \Fun^{\adm}(\Delta^1, \Ind(\calG^{op})), \SSet ).$$
Then $\calX$ is an $\infty$-category, whose objects can be identified with pairs
$(A, F)$, where $X \in \Ind(\calG^{op})$ and $F: \Ind(\calG^{op})^{X/, \adm} \rightarrow \SSet$
is a functor. It follows from Corollary \toposref{skinnysalad} that the projection $p$ is a Cartesian fibration, and that
a morphism $(A,F) \rightarrow (A', F')$ is $p$-Cartesian if and only if, for every
admissible morphism $A \rightarrow B$, the canonical map
$F(B) \rightarrow F'(A' \coprod_{A} B)$ is an equivalence in $\SSet$.

For every admissible morphism $f$ in $\Ind(\calG^{op})$, classified by a map
$\Delta^1 \rightarrow \Ind(\calG^{op})$, the pullback
$$\Fun^{\adm}(\Delta^1, \Ind(\calG^{op})) \times_{ \Ind(\calG^{op}) } \Delta^1 \rightarrow \Delta^1$$
is a Cartesian fibration. Applying Corollary \toposref{skinnysalad} again, we conclude that the projection $\calX \times_{\Ind(\calG^{op})} \Delta^1$ is a coCartesian fibration. In other words, given an object $(A,F) \in \calX$ and an admissible morphism $f: A \rightarrow A'$ in $\Ind(\calG^{op})$, there exists a locally $p$-coCartesian morphism $\overline{f}: (A,F) \rightarrow (A',F')$ lifting $f$, which is characterized up to equivalence by the requirement that the map $F(B) \rightarrow F'(B)$
is an equivalence for every admissible morphism $A' \rightarrow B$. Invoking Corollary \toposref{grutt1}, we conclude that $\overline{f}$ is actually $p$-coCartesian.

Let $\calY = \Fun_{ \Ind(\calG^{op})}( \Ind(\calG^{op}), \calX)$ denote the $\infty$-category of sections of $p$. Unwinding the definition, we can identify $\calY$ with the $\infty$-category
$\Fun( \Fun^{\adm}( \Delta^1, \Ind(\calG^{op})),\SSet)$.
Let $\Ind(\calG^{op})'$ denote the essential image of the (fully faithful) diagonal embedding
$\Ind(\calG^{op}) \rightarrow \Fun^{\adm}( \Delta^1, \Ind(\calG^{op}))$. Consider the following conditions on a section $s: \Ind(\calG^{op}) \rightarrow \calX$ of $p$:
\begin{itemize}
\item[$(a)$] The functor $s$ carries admissible morphisms in $\Ind(\calG^{op})$ to $p$-coCartesian morphisms in $\calX$.
\item[$(b)$] Let $S: \Fun^{\adm}( \Delta^1, \Ind(\calG^{op})) \rightarrow \SSet$ be the functor
corresponding to $s$. Then, for every commutative diagram
$$ \xymatrix{ & B \ar[dr] & \\
A \ar[ur] \ar[rr] & & C }$$
of admissible morphisms in $\Pro(\calG)$, the induced map
$S(A \rightarrow C) \rightarrow S(B \rightarrow C)$ is an equivalence in $\widehat{\SSet}$.
\item[$(c)$] For every admissible morphism $A \rightarrow B$ in $\Pro(\calG)$, the canonical
map $S( \id_{A} ) \rightarrow S(A \rightarrow B)$ is an equivalence in $\SSet$.
\item[$(d)$] The functor $S$ is a left Kan extension of $S | \Ind(\calG^{op})'$.
\end{itemize}
Unwinding the definitions, we see that $(a) \Leftrightarrow (b) \Rightarrow (c) \Leftrightarrow (d)$.
Moreover, the implication $(c) \Rightarrow (b)$ follows by a two-out-of-three argument.
Let $\calY'$ denote the full subcategory of $\calY$ spanned by those sections which
satisfy the equivalent conditions $(a)$ through $(d)$; it follows from Proposition \toposref{lklk}
that composition with the diagonal embedding $\Ind(\calG^{op}) \rightarrow \Fun^{\adm}( \Delta^1, \Ind(\calG^{op}))$ induces an equivalence
$$ \theta: \calY' \rightarrow \Fun( \Ind(\calG^{op})', \SSet) \rightarrow \Fun( \Ind(\calG^{op}), \SSet).$$ 

Let $\calX_0$ denote the full subcategory of $\calX$ spanned by those objects
$(A, F)$, where $F: (\Ind(\calG^{op})^{\adm})^{A/} \rightarrow \SSet$ is a sheaf.
Let $\calY_0$ denote the full subcategory of $\calY$ spanned by those sections which
factor through $\calX_0$, and let $\calY'_0 = \calY' \cap \calY_0$. Unwinding the definitions, we see that $\theta$ restricts to an equivalence $\theta_0: \calY'_0 \rightarrow \Shv( \Pro(\calG) )$. 
Consequently, to prove $(3)$, it will suffice to show that the inclusion $\calY'_0 \subseteq \calY'$
admits a left adjoint. For this, we will show that the inclusion $\calY_0 \subseteq \calY$ admits a left adjoint $L$, and that $L \calY' \subseteq \calY'_0$. Since $p$ restricts to a
Cartesian fibration $p_0: \calX_0 \rightarrow \Ind(\calG^{op})$ (Lemma \ref{facus}), the first assertion follows from Lemma \ref{scur}. The second is then a translation of Lemma \ref{stulbin}.

Assertion $(1)$ can be proven using an analogous argument. Namely, we let
$\widehat{\calX}$, $\widehat{\calX}_0$, $\widehat{\calY}$, $\widehat{\calY}_0$, 
$\widehat{\calY}'$, and $\widehat{\calY}'_0$ be defined in the same way as
$\calX$, $\calX_0$, $\calY$, $\calY_0$, $\calY'$, and $\calY'_0$, with the exception that
the $\infty$-category $\SSet$ of small spaces is replaced by the larger $\infty$-category
$\widehat{\SSet}$ of all spaces. Then we have a commutative diagram
$$ \xymatrix{ \widehat{ \calY}'_0 \ar[r]^{i} \ar[d] & \widehat{\calY}' \ar[d] \\
\widehat{\Shv}( \Pro(\calG) ) \ar[r] & \Fun( \Ind(\calG^{op}), \widehat{\SSet} )}$$ 
where the vertical maps are categorical equivalences; it therefore suffices to show that
$i$ admits a left adjoint $\widehat{L}$, which follows from the arguments above. (We could also argue more directly by viewing $\widehat{\Shv}( \Pro(\calG) )$ as an $\infty$-topos of sheaves on a Grothendieck site in a larger universe.)

To prove $(2)$, it will suffice to show that if $\alpha: F \rightarrow F'$ is a morphism
in $\calY'$ which exhibits $F'$ as a $\calY'_0$-localization of $F$, then
$\alpha$ also exhibits $F'$ as a $\widehat{\calY}'_0$-localization of $F$
in $\widehat{\calY}'$. In view of Lemma \ref{scur}, it will suffice to test this assertion fiberwise
on $\Ind(\calG^{op})$, where it reduces to Lemma \ref{biggier}.

We now give the proof of $(4)$; the proof of $(5)$ is nearly identical and left to the reader. The case of limits is clear (since $\Shv( \Pro(\calG) )$ is stable under small limits in $\Fun( \Ind(\calG^{op}), \SSet)$). Since $\theta_0$ is an equivalence, the assertion regarding colimits can be reformulated as follows:
\begin{itemize}
\item[$(4')$] The $\infty$-category $\calY'_0$ admits small colimits. Moreover, a small diagram
$q: K^{\triangleright} \rightarrow \calY'_0$ is a colimit diagram if and only if, for every
object $A \in \Ind(\calG)^{op}$, the induced map $q_A: K^{\triangleright} \rightarrow
\Shv( \Pro(\calG)^{\adm}_{/A} )$ is a colimit diagram.
\end{itemize}
Since the map $p_0$ is a Cartesian fibration, Proposition \toposref{limiteval} immediately yields the following analogue of $(4')$:
\begin{itemize}
\item[$(4'')$] The $\infty$-category $\calY_0$ admits small colimits. Moreover, a small diagram
$q: K^{\triangleright} \rightarrow \calY_0$ is a colimit diagram if and only if, for every
object $A \in \Ind(\calG)^{op}$, the induced map $q_A: K^{\triangleright} \rightarrow
\Shv( \Pro(\calG)^{\adm}_{/A} )$ is a colimit diagram.
\end{itemize}
To deduce $(4')$ from $(4'')$, it suffices to observe that $\calY'_0$ is stable under
small colimits in $\calY_0$, since for every admissible morphism $f: A \rightarrow B$
in $\Ind(\calG^{op})$ the induced functor $f^{\ast}: \Shv( \Pro(\calG)^{\adm}_{/A})
\rightarrow \Shv( \Pro(\calG)^{\adm}_{/B} )$ is a geometric morphism of $\infty$-topoi
(Lemma \ref{facus}) and therefore preserves small colimits.

To prove $(6)$, we must show that
the inclusion $\Shv(\Pro(\calG)) \subseteq \widehat{\Shv}( \Pro(\calG) )$ preserves small limits and colimits. In the case of limits, this follows from the observation that the inclusion
$\Fun( \Ind(\calG)^{op}, \SSet) \rightarrow \Fun( \Ind(\calG)^{op}, \widehat{\SSet} )$ preserves small limits (which in turn follows from Proposition \toposref{limiteval}). To handle the case of colimits, it will suffice to show that the inclusion $\calX_0 \subseteq \widehat{\calX}_0$ preserves small colimits after passing to the fiber over every object $A \in \Ind(\calG^{op})$. In other words, we must show that
the inclusion $\Shv( \Pro(\calG)^{\adm}_{/A}) \subseteq \widehat{\Shv}( \Pro(\calG)^{\adm}_{/A} )$
preserves small colimits. This follows from Proposition \ref{selwus} and Remark \toposref{quest}.
\end{proof}

\begin{remark}\label{spittle}
Although the $\infty$-category $\Shv( \Pro(\calG) )$ is not an $\infty$-topos, it can be identified with the limit of a (large) diagram of $\infty$-topoi. It therefore has many features in common with $\infty$-topoi. For example, there is a good theory of effective epimorphisms in $\Shv( \Pro(\calG) )$: we say that a map $\alpha: F_0 \rightarrow F$ in $\Shv(\Pro(\calG) )$ is an effective epimorphism if it restricts to an effective epimorphism in each of the $\infty$-topoi $\Shv( \Pro(\calG)^{\adm}_{/A} )$ (equivalently, $\alpha$ is an effective epimorphism if it is an effective epimorphism in the larger $\infty$-category
$\widehat{\Shv}(\Pro(\calG))$, which is an $\infty$-topos in a larger universe). As in an ordinary $\infty$-topos, we have effective descent: namely, if $F_{\bigdot}$ denotes the \Cech nerve of $\alpha$, then
the canonical map $| F_{\bigdot} | \rightarrow F$ is an equivalence in $\Shv( \Pro(\calG) )$. 
\end{remark}

\begin{lemma}\label{bodkin}
Let $\calG$ be a geometry, let $(\calX, \calO_{\calX})$ be an object of
$\LGeo(\calG)$, and let $e: \LGeo(\calG) \rightarrow \widehat{\SSet}$ be the functor
corepresented by $e$. Then the composite map
$$ \Ind(\calG^{op}) \stackrel{ \Spec^{\calG} }{\rightarrow} \LGeo(\calG)
\stackrel{e}{\rightarrow} \widehat{\SSet} $$
belongs to $\widehat{ \Shv}( \Pro(\calG) )$.
\end{lemma}

\begin{proof}
Fix an object $A \in \Ind(\calG)^{op}$; we wish to show that the composition
$$ \Ind( \calG^{op} )_{A/}^{\adm}
\stackrel{f}{\rightarrow} \Ind(\calG^{op})
\stackrel{ \Spec^{\calG} }{\rightarrow} \LGeo(\calG) \stackrel{e}{\rightarrow} \widehat{\SSet}$$
is a sheaf on $\Ind(\calG^{op})_{A/}^{\adm}$. We note that the composition
$\Spec^{\calG} \circ f$ can also be written as a composition
$$ \Ind( \calG^{op})_{A/}^{\adm} \stackrel{g}{\rightarrow} \Shv( \Pro(\calG)_{/A}^{\adm})^{op}
\simeq \LGeo(\calG)_{\mathet}^{\Spec^{\calG} A/} \rightarrow \LGeo(\calG),$$
where $g$ is the (opposite of) the composition of the Yoneda embedding
$j: \Pro(\calG)_{/A}^{\adm} \rightarrow \Fun( \Pro(\calG)_{/A}^{\adm}, \SSet)$
with a left adjoint to the inclusion $\Shv( \Pro(\calG)_{/A}^{\adm})
\subseteq \Fun( \Pro(\calG)_{/A}^{\adm}, \SSet)$. In view of Proposition \ref{selwus}, it will suffice to show that the composition
$$ \LGeo(\calG)_{\mathet}^{\Spec^{\calG} A/} \rightarrow \LGeo(\calG)
\stackrel{e}{\rightarrow} \widehat{\SSet}$$
preserves small limits, which follows immediately from Proposition \ref{scanh}.
\end{proof}

\begin{lemma}\label{kokin}
Let $\calG$ be a geometry. Then the composition
$$ \LGeo(\calG)^{op}
\rightarrow \Fun( \LGeo(\calG), \widehat{\SSet} )
\rightarrow \Fun( \Ind(\calG^{op}), \widehat{\SSet} )$$
factors through a functor $\overline{\phi}: \LGeo(\calG)^{op} 
\rightarrow \widehat{\Shv}( \Pro(\calG) )$. Moreover, the composition
$$ \LGeo(\calG)^{op}_{\mathet} \subseteq \LGeo(\calG)^{op}
\stackrel{ \overline{\phi}}{\rightarrow} \widehat{\Shv}( \Pro(\calG) )$$
preserves small colimits.
\end{lemma}

\begin{proof}
The first assertion is merely a translation of Lemma \ref{bodkin}.
To prove the second, it will suffice to show that for every object $A \in \Pro(\calG)$, the
induced map
$$ \overline{\phi}_{A}: \LGeo( \calG)^{op}_{\mathet} \rightarrow \widehat{\Shv}( \Pro(\calG)_{/A}^{\mathet}) $$
preserves small colimits (Proposition \ref{kumber}). Using Proposition \ref{selwus} and Theorem \ref{scoo}, we can identify the target with $\Shv_{ \widehat{\SSet}}( \calY)$, where
$( \calY, \calO_{\calY} ) = \Spec^{\calG} A$.  

Let us consider an arbitrary small diagram
$\{ ( \calX_{\alpha}, \calO_{\calX_{\alpha}}) \}$ in $\LGeo(\calG)^{op}$, having
colimit $(\calX, \calO_{\calX})$. We then have canonical identifications
$$ (\calX_{\alpha}, \calO_{\calX_{\alpha}}) \simeq (\calX_{/U_{\alpha}}, \calO_{\calX} | U_{\alpha} )$$
for some diagram $\{ U_{\alpha} \}$ in $\calX$ having colimit $1_{\calX}$. To prove that
$\overline{\phi}_A$ preserves this colimit diagram, it will suffice to show that the composition
$$ \calX \simeq (\LGeo(\calG)_{\mathet}^{(\calX, \calO_{\calX})/})^{op}
\rightarrow \LGeo(\calG)^{op}_{\mathet} \stackrel{\overline{\phi}_A}{\rightarrow} \Shv_{\widehat{\SSet}}(\calY)$$
preserves small colimits. This follows from the proof of Proposition \ref{sizem2} (more specifically, assertion $(\ast'')$)
\end{proof}

We now return to our main result.

\begin{proof}[Proof of Theorem \ref{skil}]
In view of the first assertion of the theorem and Lemma \ref{kokin}, we may regard
$\phi$ as a functor $\Sch(\calG) \rightarrow \Shv( \Pro(\calG) )$. We wish to show that
this functor is fully faithful. In other words, we must show that for every pair of objects
$X,Y \in \Sch(\calG)$, the induced map
$$ \psi_{X,Y}: \bHom_{\Sch(\calG)}(X,Y) \rightarrow \bHom_{ \Shv(\Pro(\calG))}( \phi X, \phi Y).$$
is a homotopy equivalence. 

Let us regard $Y$ as fixed, and let $\Sch(\calG)^{0}_{\mathet}$ be the full subcategory of
$\Sch(\calG)_{\mathet}$ spanned by those objects $X$ for which $\psi_{X,Y}$ is a homotopy equivalence. Using Lemma \ref{kokin}, we conclude that $\Sch(\calG)^{0}_{\mathet}$ is
stable under small colimits in $\Sch(\calG)_{\mathet}$. We wish to show that
$\Sch(\calG)^{0}_{\mathet} = \Sch(\calG)_{\mathet}$. In view of Lemma \ref{spak}, it will suffice to show that $\Sch(\calG)^{0}_{\mathet}$ contains every {\em affine} $\calG$-scheme. In other words, we are reduced to proving that $\psi_{X,Y}$ is a homotopy equivalence when $X$ is affine, which is obvious.
\end{proof}

\begin{warning}
In the classical theory of schemes, a scheme $(X, \calO_X)$ is locally of finite presentation
(over the integer $\Z$, say) if and only if the associated functor $F_X: \Comm \rightarrow \Set$
preserves filtered colimits. The analogous assertion is false in the present setting:
if $(\calX, \calO_{\calX})$ is a $\calG$-scheme which is locally of finite presentation, then
the associated functor $\Ind(\calG^{op}) \rightarrow \SSet$ need not preserve filtered colimits
in general. However, it remains true if we assume that $\calX$ is $n$-localic for some
$n \geq 0$. We will give a more detailed discussion in a future paper.
\end{warning}

Theorem \ref{skil} allows us to identify $\Sch(\calG)$ with a full subcategory of
$\Shv( \Pro(\calG) ) \subseteq \Fun( \Ind(\calG^{op}), \SSet)$. We might now ask for a characterization of this subcategory. In other words, given a ``moduli functor'' $F: \Ind(\calG^{op}) \rightarrow \SSet$, under what conditions is $F$ representable by a $\calG$-scheme? We will return to this question in \cite{elliptic1}. For the time being, we will content ourselves with a few easy observations.

\begin{definition}\label{psab}
Let $\calG$ be a geometry, and let $\alpha: F' \rightarrow F$ be a natural transformation of functors $F,F' \in \Shv( \Pro(\calG) )$. We will say that $\alpha$ is {\it \etale} if
the following condition is satisfied:
\begin{itemize}
\item[$(\ast)$] Let $\beta: G \rightarrow F$ be a natural transformation in $\Shv(\Pro(\calG))$, where
$G$ is representable by a $\calG$-scheme $(\calX, \calO_{\calX})$. Then the fiber product
$G \times_{F} F'$ is representable by a $\calG$-scheme which is \etale over $(\calX, \calO_{\calX})$.
\end{itemize}
\end{definition}

\begin{remark}\label{saab}
In the situation of Definition \ref{psab}, it suffices to test condition $(\ast)$ in the case where
$G$ is representable by an {\em affine} $\calG$-scheme. For suppose that this weaker version of condition $(\ast)$ is satisfied, and let $G$ be representable by an arbitrary $\calG$-scheme
$(\calX, \calO_{\calX})$. Let $\calX^{0}$ denote the full subcategory of $\calX$ spanned by those objects $V$ for which $(\calX_{/V}, \calO_{\calX}|V)$ represents a functor $G_V$ for which
$G_{V} \times_{F} F'$ is representable by a $\calG$-scheme \etale over $(\calX_{/V}, \calO_{\calX}|V)$. Using Proposition \ref{scanh}, we deduce that $\calX^{0}$ is stable under small colimits in
$\calX$, so that $\calX = \calX^{0}$ by Lemma \ref{spak}. 
\end{remark}

We now record the basic properties of the class of \etale morphisms in $\Shv( \Pro(\calG) )$. 

\begin{proposition}
Let $\calG$ be a geometry.
\begin{itemize}
\item[$(1)$] Every equivalence in $\Shv( \Pro(\calG) )$ is \etale.
\item[$(2)$] Every pullback of an \etale morphism in $\Shv( \Pro(\calG) )$ is \etale.
\item[$(3)$] Let $\alpha: F \rightarrow G$ be a morphism in $\Shv( \Pro(\calG) )$, where
$G$ is representable by a $\calG$-scheme $(\calX, \calO_{\calX})$. Then $\alpha$ is \etale if and only if $F$ is representable by a $\calG$-scheme which is \etale over $(\calX, \calO_{\calX})$.
\item[$(4)$] Every retract of an \etale morphism in $\Shv( \Pro(\calG) )$ is \etale.
\item[$(5)$] Suppose given a commutative diagram
$$ \xymatrix{ & G \ar[dr]^{g} & \\
F \ar[ur]^{f} \ar[rr]^{h} & & H }$$
in $\Shv( \Pro(\calG) )$ such that $g$ is \etale. Then $f$ is \etale if and only if $h$ is \etale.
In particular, the collection of \etale morphisms is stable under composition.

\item[$(6)$] Let $f_{\alpha}: F_{\alpha} \rightarrow G$ be a collection of \etale morphisms in $\Shv( \Pro(\calG) )$. Suppose
that each $F_{\alpha}$ is representable by a $\calG$-scheme $(\calX_{\alpha}, \calO_{\calX_{\alpha}})$ and that the induced map $\coprod F_{\alpha} \rightarrow G$ is an effective epimorphism
(see Remark \ref{spittle}). Then $G$ is representable
by a $\calG$-scheme $(\calY, \calO_{\calY} )$. (Each map of $\calG$-schemes
$(\calX_{\alpha}, \calO_{\calX_{\alpha}}) \rightarrow (\calY, \calO_{\calY})$ is then automatically \etale, by virtue of
$(3)$).

\item[$(7)$] Let $\Shv(\Pro(\calG))_{\mathet}$ denote the subcategory of
$\Shv( \Pro(\calG) )$ spanned by the \etale morphisms. Then the $\infty$-category 
$\Shv( \Pro(\calG) )_{\mathet}$ admits small colimits, and the inclusion
$\Shv( \Pro(\calG) )_{\mathet} \subseteq \Shv( \Pro(\calG) )$ preserves small colimits.

\end{itemize}
\end{proposition}

\begin{proof}
Assertions $(1)$ and $(2)$ follow immediately from the definition. 
The ``only if'' direction of $(3)$ follows immediately from the definition, and the
``if'' direction follows from Remark \ref{pre4}. To prove $(4)$, let us suppose that
$\alpha: F \rightarrow F'$ is a retract of an \etale morphism $\beta: G \rightarrow G'$. 
Let $F'_0$ be a functor representable by a $\calG$-scheme $(\calX, \calO_{\calX} )$. Let
$\calC$ denote the full subcategory of $\Shv( \Pro(\calG) )_{/F'_0}$ spanned by those functors which are representable by $\calG$-schemes which are \etale over $(\calX, \calO_{\calX})$. We wish to prove that $F \times_{F'} F'_0$ belongs to $\calC$. By assumption, $F \times_{F'} F'_0$ is a retract of
$G \times_{ G'} F'_0 \in \calC$. It will therefore suffice to show that $\calC$ is idempotent complete, which follows from the observation that $\calC \simeq \calX$.

We now prove $(5)$. Suppose first that $f$ is \etale; we wish to prove that $h$ is \etale.
Without loss of generality, we may suppose that $H$ is representable by a $\calG$-scheme
$(\calX, \calO_{\calX})$. Since $g$ is \etale, $G$ is representable by a $\calG$-scheme
$(\calX_{/U}, \calO_{\calX} | U )$. Since $f$ is \etale, we conclude that $F$ is representable
by $(\calX_{/V}, \calO_{\calX} | V)$ for some morphism $V \rightarrow U$ in $\calX$.

For the converse, let us suppose that $h$ is \etale; we wish to prove that $f$ is \etale. 
Consider a morphism $G_0 \rightarrow G$, where $G_0$ is representable by a $\calG$-scheme
$(\calX, \calO_{\calX})$; we wish to show that $F \times_G G_0$ is representable by a
$\calG$-scheme \etale over $(\calX, \calO_{\calX})$. Pulling back by the composite map
$G_0 \rightarrow G \rightarrow H$, we may assume that $H$ is representable by 
$(\calX, \calO_{\calX})$. Since $g$ and $h$ are \etale, the functors $F$ and $G$
are representable by $\calG$-schemes $(\calX_{/U}, \calO_{\calX} | U)$ and
$(\calX_{/V}, \calO_{\calX} | V)$, respectively. Then $f$ is induced by an \etale map 
of $\calG$-schemes, classified by a morphism $U \rightarrow V$ in $\calX$ (Remark \ref{unipop}), and therefore \etale (by virtue of $(3)$).

To prove $(6)$, let $F_0 = \coprod F_{\alpha} \in \Shv( \Pro(\calG) )$, so that $F_0$ is also representable by a $\calG$-scheme (Proposition \ref{kokin}). Let $F_{\bigdot}$ be the simplicial object of $\Shv( \Pro(\calG) )$ determined by the \Cech nerve of the effective epimorphism $F_0 \rightarrow G$. We observe that the map $F_0 \rightarrow G$ is \etale (this is a special case of $(7)$, but is also easy to check directly). Using assertions $(2)$ and $(5)$, we conclude that $F_{\bigdot}$
can be regarded as a simplicial object in $\Shv( \Pro(\calG) )_{\mathet}$. Since $F_0$ is representable by a $\calG$-scheme, assertion $(3)$ implies that each $F_{n}$ is representable by a $\calG$-scheme, so that we obtain a simplicial object $( \calX_{\bigdot}, \calO_{\calX_{\bigdot}})$ in 
$\Sch(\calG)_{\mathet}$. Let $(\calX, \calO_{\calX}) \in \Sch(\calG)$ denote the colimit of this diagram.
Lemma \ref{kokin} implies that $(\calX, \calO_{\calX})$ represents the functor
$| F_{\bigdot} | \simeq G \in \Shv( \Pro(\calG))$. 

To prove assertion $(7)$, let us consider a small diagram $\{ F_{\alpha} \}$ in
$\Shv( \Pro(\calG) )_{\mathet}$ having a colimit $F$ in $\Shv( \Pro(\calG) )$. We first claim that
each of the maps $F_{\alpha} \rightarrow F$ is \etale. Since colimits in $\Shv( \Pro(\calG) )$ are universal, we may assume without loss of generality that $F$ is representable by an
affine $\calG$-scheme $\Spec^{\calG} A = (\calX, \calO_{\calX})$ (see Remark \ref{saab}), so there is a canonical point $\eta \in F(A)$. For each index $\alpha$, let 
$G_{\alpha}: \Ind(\calG^{op})_{A/} \rightarrow \SSet$ denote the fiber of the map
$$F_{\alpha} | \Ind(\calG^{op})^{\adm}_{A/} \rightarrow F | \Ind( \calG^{op})^{\adm}_{A/}$$
over the point determined by $\eta$. Using Theorem \ref{scoo}, we can identify $G_{\alpha}$ with
an object $U_{\alpha} \in \calX$. Proposition \ref{kumber} implies that $\colim \{ U_{\alpha} \} \simeq 1_{\calX}$. Applying Theorem \toposref{charleschar} (in the very large $\infty$-topos $\widehat{\Shv}( \Pro(\calG) )$), we conclude that each of the diagrams
$$ \xymatrix{ F'_{\alpha} \ar[r] \ar[d] & F_{\alpha} \ar[d] \\
F' \ar[r] & F }$$
is a pullback square. However, since $\colim \{ U_{\alpha} \} \simeq 1_{\calX}$, the bottom horizontal map is an equivalence (Lemma \ref{kokin}). It follows that the upper horizontal map is also an equivalence, so that $F_{\alpha}$ is representable by a $\calG$-scheme which is \etale over
$(\calX, \calO_{\calX})$.

To complete the proof of $(7)$, it will suffice to show that a morphism $f: F \rightarrow G$
in $\Shv(\Pro(\calG))$ is \etale if and only if each of the composite maps $f_{\alpha}: F_{\alpha} \rightarrow
F \rightarrow G$ is \etale. The ``only if'' direction follows from the argument given above. Let us therefore suppose that each of the maps $f_{\alpha}$ is \etale; we wish to show that $f$ is \etale. Without loss of generality, we may assume that $G$ is representable by a
$\calG$-scheme $(\calY, \calO_{\calY})$. Then each $F_{\alpha}$ is representable by
a $\calG$-scheme $(\calY_{/V_{\alpha}}, \calO_{\calY} | V_{\alpha})$. Using assertion
$(3)$ and Lemma \ref{kokin}, we conclude that $F$ is representable by
a $\calG$-scheme $(\calX, \calO_{\calX})$, which is covered by the
$\calG$-schemes $(\calY_{/V_{\alpha}}, \calO_{\calY} | V_{\alpha} )$. The desired
result now follows from assertion $(3)$ and Remark \ref{lopus}.
\end{proof}


\subsection{Algebraic Geometry (Zariski topology)}\label{exzar}

Throughout this section, we fix a commutative ring $k$. Our goal is to show how to recover
the classical theory of $k$-schemes from our general formalism of geometries. More precisely, we will introduce a geometry $\calG_{\Zar}(k)$, such that $\calG_{\Zar}(k)$-structures on an $\infty$-topos
$\calX$ can be identified with ``sheaves of commutative local $k$-algebras'' on $\calX$. We begin
with a concrete discussion of sheaves of $k$-algebras.

\begin{definition}
Let $\calX$ be an $\infty$-topos. A {\it commutative $k$-algebra in $\calX$} is
a $k$-algebra object in the underlying topos $\h{(\tau_{\leq 0} \calX)}$ of discrete objects
of $\calX$. In other words, a commutative $k$-algebra in $\calX$ is a discrete object
$A \in \calX$, equipped with addition and multiplication maps
$$ A \times A \stackrel{+}{\rightarrow} A \quad \quad A \times A \stackrel{\times}{\rightarrow} A$$
and scalar multiplication maps $A \stackrel{\lambda}{\rightarrow} A$
for each $\lambda \in k$, which are required to satisfy the usual axioms for a commutative $k$-algebra. The commutative $k$-algebras in $\calX$ form a category, which we will denote by 
$\Comm_{k}(\calX)$. 
\end{definition}

\begin{example}
Let $\calX = \SSet$. Then $\Comm_{k}(\calX)$ can be identified with the usual category of commutative $k$-algebras. We will denote this (ordinary) category by $\Comm_{k}$.
\end{example}

Let $A$ be an arbitrary discrete object of an $\infty$-topos $\calX$. Then $A$ represents
a functor $e_{A}: \h{\calX}^{op} \rightarrow \Set$, described by the formula
$$ e_{A}(X) = \Hom_{ \h{\calX}}( X, A) \simeq \pi_0 \bHom_{\calX}( X,A).$$ 
According to Yoneda's lemma, giving a commutative $k$-algebra structure on $A$ is
equivalent to giving a commutative $k$-algebra structure on the functor $e_A$: in other words, producing a factorization
$$ \xymatrix{ & \Comm_{k} \ar[dd] \\
\h{\calX}^{op} \ar[ur]^{\overline{e}_A} \ar[dr]^{e_A} & \\
& \Set, }$$
where the vertical arrow denotes the evident forgetful functor.
In other words, we can identify commutative $k$-algebras in $\calX$ with functors
$\overline{e}_{A}: \h{\calX}^{op} \rightarrow \Comm_{k}$, whose underlying set-valued functor
is representable. Such a functor can be identified with a map of $\infty$-categories
$\widetilde{e}_{A}: \calX^{op} \rightarrow \Nerve(\Comm_{k})$. In view of
Proposition \toposref{representable}, the representability condition is equivalent to the
requirement that $\widetilde{e}_{A}$ preserve small limits. This proves the following result:

\begin{proposition}\label{skup}
Let $\calX$ be an $\infty$-topos. Then there is a canonical equivalence of $\infty$-categories
$$ \Shv_{ \Comm_{k} }(\calX) \simeq \Nerve \Comm_{k}(\calX),$$
where the left side is described in Definition \ref{swinging}.
\end{proposition}

\begin{notation}
Let $\Comm_{k}^{\fin}$ denote the full subcategory of $\Comm_{k}$ spanned by those
commutative $k$-algebras which are of finite presentation; that is, $k$-algebras of the form
$$ k[x_1, \ldots, x_n] / ( f_1, \ldots, f_m ).$$
We let $\calG(k)$ denote the $\infty$-category $\Nerve( \Comm_{k}^{\fin} )^{op}$,
regarded as a {\em discrete geometry}.
\end{notation}

\begin{remark}
We can identify $\calG(k)$ with the (nerve of the) category of affine $k$-schemes
which are of finite presentation over $k$. 
\end{remark}

\begin{remark}\label{hunner}
The category $\Comm_{k}$ is compactly generated, and the compact objects of
$\Comm_{k}$ are precisely the finitely generated commutative $k$-algebras. Consequently, we have a canonical equivalence of $\infty$-categories $\Nerve(\Comm_{k}) \simeq \Ind( \calG(k)^{op})$. 
\end{remark}

Combining Proposition \ref{skup}, Remark \ref{tunner}, and Remark \ref{hunner},  we obtain the following:

\begin{proposition}\label{skuupp}
Let $\calX$ be an $\infty$-topos. Then there is a canonical equivalence of $\infty$-categories
$$ \Struct_{ \calG(k) }( \calX ) \simeq \Nerve \Comm_{k}(\calX).$$
\end{proposition}

To recover the classical theory of schemes, we will view $\Nerve( \Comm_{k}^{fin})^{op}$
as a geometry in a slightly different way.

\begin{definition}\label{jin}
Let $k$ be a commutative ring. 
We define a geometry $\calG_{\Zar}(k)$ as follows:
\begin{itemize}
\item[$(1)$] The underlying $\infty$-category of $\calG_{\Zar}(k)$ is $\Nerve(\Comm_{k}^{\fin})^{op}$.
\item[$(2)$] Let $f$ be a morphism in $\calG_{\Zar}(k)$, which we can identify with a map
$A \rightarrow B$ of commutative $k$-algebras. Then $f$ is admissible if and only if there exists
an element $a \in A$ such that $f$ induces an isomorphism $A[ \frac{1}{a} ] \simeq B$.
\item[$(3)$] Suppose given a collection of admissible morphisms $\{ \phi_{\alpha}: U_{\alpha} \rightarrow X \}$ in $\calG_{\Zar}(k)$, corresponding to maps $A \rightarrow A[ \frac{1}{a_{\alpha} }]$
of commutative $k$-algebras. Then the admissible morphisms $\phi_{\alpha}$ generate a covering sieve on $X$ if and only if the elements $\{ a_{\alpha} \}$ generate the unit ideal of $A$.
\end{itemize}
\end{definition}

\begin{remark}
Condition $(3)$ of Definition \ref{jin} is equivalent to the requirement that the collection of maps
$\{ \phi_{\alpha}: U_{\alpha} \rightarrow X \}$ are jointly surjective, when viewed as maps of affine $k$-schemes. To see this, let $\mathfrak{a} \subseteq A$ denote the ideal generated by the elements
$\{ a_{\alpha} \}$. If $\mathfrak{a} = A$, then no prime ideal $\mathfrak{p}$ of $A$ can contain every
element $a_{\alpha}$, so that $\mathfrak{p}$ belongs to the image of some $\phi_{\alpha}$. 
Conversely, if $\mathfrak{a} \neq A$, then $\mathfrak{a}$ is contained in some prime ideal
$\mathfrak{p}$ of $A$, which does not belong to the image of any $\phi_{\alpha}$. See the proof of Lemma \ref{stu2} for more details. 
\end{remark}

\begin{remark}
Every admissible morphism in $\calG_{\Zar}(k)$ corresponds to an open immersion
in the category of affine $k$-schemes. However, the converse is false: not every open immersion
$U \rightarrow X$ of affine $k$-schemes arises from a localization morphism $A \rightarrow A[ \frac{1}{a} ]$. For example, let $k$ be an algebraically closed field, $E$ an elliptic curve over $k$,
$X = E - \{ x \}$ and $U = E - \{x,y \}$; here $x$ and $y$ denote closed points of $X$. Then $U$ is the nonvanishing locus of a regular function on $X$ if and only if the difference $x-y$ is torsion with respect to the group structure on $E$.

In Definition \ref{jin}, we can enlarge the class of admissible morphisms to include {\em all} open immersions between affine $k$-schemes of finite presentation; the resulting theory is the same.
We will prove a version of this assertion in \S \ref{app6}.
\end{remark}

\begin{remark}\label{testlo}
Let $X$ be a topological space, and let $\calO_{X}$ be a sheaf of commutative $k$-algebras on $X$
(in the usual sense). Using Proposition \ref{skuupp}, we can identify $\calO_X$ with a
$\calG(k)$-structure $\overline{\calO}_X: \calG(k) \rightarrow \Shv(X)$. The relationship between
$\calO_X$ and $\overline{\calO}_X$ can be described more precisely as follows: if $A$ is a commutative $k$-algebra of finite presentation, then $\overline{\calO}_X(A)$ is a sheaf (of spaces) on $X$ whose value on an open set
$U \subseteq X$ is the (discrete) set $\Hom_{\Comm_k}( A, \calO_X(U) )$ of $k$-algebra homomorphisms from $A$ to $\calO_X(U)$. 

We note that $\overline{\calO}_X$ belongs to $\Struct_{\calG_{\Zar}(k)}( \Shv(X) )$ if and only if, for every
collection of homomorphisms $\{ A \rightarrow A[ \frac{1}{a_{\alpha} }] \}$ such that 
the elements $a_{\alpha}$ generate the unit ideal in $A$, the induced map
$$ \coprod_{\alpha} \overline{\calO}_X( A[ \frac{1}{a_{\alpha} }] ) \rightarrow \overline{\calO}_X(A)$$
is an epimorphism of sheaves on $X$. Unwinding the definitions, this condition asserts that
for every open set $U$ and every $k$-algebra homomorphism
$$ \phi: A \rightarrow \calO_{X}(U),$$
if we define $U_{\alpha} \subseteq U$ to be the largest open subset over which the section
$\phi(a_{\alpha}) \in \calO_X(U)$ is invertible, then the open subsets $U_{\alpha}$ cover $U$.
In other words, at every point $x \in X$, at least one of the sections $\phi(a_{\alpha})$ is invertible.
This is equivalent to the requirement that every stalk $\calO_{X,x}$ be a {\em local} ring.
\end{remark}

\begin{remark}\label{testhigh}
Let $f: (X, \calO_X) \rightarrow (Y, \calO_Y)$ be a map of topological spaces ringed by commutative $k$-algebras. In other words, suppose we have a continuous map of topological spaces $f: X \rightarrow Y$, and a homomorphism $f^{\ast} \calO_Y \rightarrow \calO_X$ of sheaves of commutative $k$-algebras on $X$. As in Remark \ref{testlo}, we can identify $\calO_X$ and $\calO_Y$ with 
$\calG(k)$-structures $\overline{\calO}_X: \calG(k) \rightarrow \Shv(X)$ and
$\overline{\calO}_Y: \calG(k) \rightarrow \Shv(Y)$, respectively. The map
$f$ itself determines a morphism
$$ \overline{f}: ( \Shv(X), \overline{\calO}_X) \rightarrow ( \Shv(Y), \overline{\calO}_{Y} )$$
in the $\infty$-category $\LGeo( \calG(k) )^{op}$. Suppose furthermore that
$\calO_X$ and $\calO_Y$ are sheaves of local rings, so that
$(\Shv(X), \overline{\calO}_X)$ and $( \Shv(Y), \overline{\calO}_{Y} )$ belong to the subcategory
$\LGeo( \calG_{\Zar}(k) )^{op} \subseteq \LGeo( \calG(k) )^{op}$. The morphism
$\overline{f}$ belongs to $\LGeo( \calG_{\Zar}(k) )^{op}$ if and only if, for every
admissible morphism $A \rightarrow A[ \frac{1}{a} ]$ between commutative $k$-algebras of finite presentation, the diagram
$$ \xymatrix{ f^{\ast} \overline{\calO}_Y( A[ \frac{1}{a} ]) \ar[r] \ar[d] & \overline{\calO}_{X}(A[ \frac{1}{a} ]) \ar[d] \\
f^{\ast} \overline{\calO}_{Y}(A) \ar[r] & \overline{\calO}_{X}(A). }$$
Unwinding the definitions, this amounts to the following condition: let $U$ be an open subset of $Y$, and let $a \in \calO_{Y}(U)$. Then, for every point $x \in f^{-1} U \subseteq X$, the restriction
$f^{\ast}(a) \in \calO_X( f^{-1} U)$ is invertible at $x$ if and only if the section $a$ is invertible
at $f(x)$. In other words, $\overline{f}$ belongs to $\LGeo( \calG_{\Zar}(k) )^{op}$ if and only if
$f: (X, \calO_X) \rightarrow (Y, \calO_Y)$ is a morphism in the category of {\em locally} ringed spaces.
\end{remark}

\begin{definition}
Let $\calG$ be a geometry, $n \geq 0$ a nonnegative integer, and let $(\calX, \calO_{\calX}) \in \LGeo(\calG)$. We will say that $(\calX, \calO_{\calX} )$ is {\it $n$-localic} if the $\infty$-topos
$\calX$ is $n$-localic, in the sense of Definition \toposref{stuffera}. 
\end{definition}

We recall that an $\infty$-topos $\calX$ is $0$-localic if and only if $\calX$ is equivalent to the
$\infty$-category of sheaves (of spaces) of some locale (see \S \toposref{0topoi}). If $\calX$ has enough points, then we can identify this locale with the lattice of open subsets of a sober topological space $X$
(recall that a topological space $X$ is said to be {\it sober} if every irreducible subset of $X$ has a unique generic point). Combining this observation with Remarks \ref{testlo} and \ref{testhigh}, we obtain the following result:

\begin{proposition}\label{stum}
\begin{itemize}
\item[$(1)$] Let $\LGeo'$ denote the full subcategory of $\LGeo$ spanned by those
$0$-localic $\infty$-topoi with enough points, and let $\Top_{\sob}$ denote the category of
{\em sober} topological spaces and continuous maps. Then there is a canonical equivalence of $\infty$-categories 
$$ \Nerve( \Top_{\sob} ) \simeq (\LGeo')^{op}.$$

\item[$(2)$] Let $\LGeo'( \calG(k) )$ denote the full subcategory of
$\LGeo(\calG(k))$ spanned by those pairs $(\calX, \calO_{\calX} )$, where
$\calX$ is a $0$-localic $\infty$-topos with enough points. Let $\RingSpace_{k}$ denote the
ordinary category of pairs $(X, \calO_X)$, where $X$ is a sober topological space and
$\calO_X$ is a sheaf of commutative $k$-algebras on $X$. Then there is a canonical
equivalence of $\infty$-categories
$$ \Nerve( \RingSpace_{k} ) \simeq \LGeo'(\calG(k) )^{op}.$$

\item[$(3)$] Let $\LGeo'( \calG_{\Zar}(k) )$ denote the full subcategory of $\LGeo(\calG_{\Zar}(k))$
spanned by those pairs $(\calX, \calO_{\calX} )$ where $\calX$ is a $0$-localic $\infty$-topos with enough points, and let $\RingSpace_{k}^{\loc}$ denote the ordinary category of pairs
$(X, \calO_X)$ where $X$ is a sober topological space and $\calO_X$ a sheaf of commutative $k$-algebras on $X$ with local stalks (morphisms are required to induce local homomorphisms on each stalk). Then there is a canonical equivalence of $\infty$-categories
$$ \alpha: \Nerve( \RingSpace_{k}^{\loc} ) \simeq \LGeo'( \calG_{\Zar}(k) )^{op}.$$
\end{itemize}
\end{proposition}

We can now describe the relationship of our theory with classical algebraic geometry.

\begin{theorem}\label{stumm}
Let $\Sch_{k}$ denote the category of $k$-schemes, regarded as a full subcategory
of $\RingSpace_{k}^{\loc}$. Then the equivalence $\alpha$ of Proposition \ref{stum} induces
a fully faithful embedding
$$ \Nerve( \Sch_{k} ) \rightarrow \Sch( \calG_{\Zar}(k) ).$$
The essential image of this embedding consists of those $\calG_{\Zar}(k)$-schemes
$(\calX, \calO_{\calX} )$ for which the $\infty$-topos $\calX$ is $0$-localic.
\end{theorem}

The proof of Theorem \ref{stumm} will require a few preliminaries. 

\begin{lemma}\label{stu11}
Let $X$ be a topological space. Suppose that $X$ is locally sober: that is, $X$ is the union of its sober open subsets. Then $X$ is sober.
\end{lemma}

\begin{proof}
Let $K$ be an irreducible closed subset of $X$. Since $K$ is nonempty, there exists a sober
open subset $U \subseteq X$ such that $U \cap K \neq \emptyset$. Since $K$ is irreducible
conclude that $K$ is the closure of $U \cap K$ in $X$, and that $U \cap K$ is an irreducible closed subset of $U$. Since $U$ is sober, $U \cap K$ is the closure (in $U$) of some point $x \in U \cap K$.
Then $K$ is the closure of $\{x\}$ in $X$; in other words, $x$ is a generic point of $K$. Let
$y$ be another generic point of $K$, so that $K = \overline{ \{y\} }$. Since $K \cap U \neq \emptyset$, we must have $y \in K$, so that $y$ is a generic point of $K \cap U$ in $U$. Since $U$ is sober, we
conclude that $y = x$.
\end{proof}

\begin{lemma}\label{stu2}
Let $A$ be a commutative $k$-algebra. Then the functor $\alpha$ of Proposition \ref{stum} carries
the affine $k$-scheme $\SSpec A$ to the $\calG_{\Zar}(k)$-scheme $\Spec^{\calG_{\Zar}(k)} A$.
\end{lemma}

\begin{remark}
In the statement of Lemma \ref{stu2}, we have implicitly identified 
$\Pro(\calG_{\Zar}(k) )$ with the (opposite of the nerve of the) category of commutative $k$-algebras.
\end{remark}

\begin{proof}
It is possible to prove Lemma \ref{stu2} by showing that $\SSpec A$ and $\Spec^{\calG_{\Zar}(k)} A$ can be described by the same universal property. We opt instead for a more concrete approach, using the equivalence $\Spec^{\calG_{\Zar}(k)} A \simeq ( \USpec A, \calO_{ \USpec A} )$ supplied by Theorem \ref{scoo}. We begin by observing that 
$\Pro( \calG_{\Zar}(k) )^{\adm}_{/A} \simeq \Nerve(\calC)$, where $\calC$ is the opposite of the category of commutative
$A$-algebras having the form $A[ \frac{1}{a}]$, for some $a \in A$. We observe that the category $\calC$
is equivalent to a partially ordered set (in other words, there is at most one morphism between any pair of objects of $\calC$). It follows that $\USpec A \simeq \Shv( \calC )$ is a $0$-localic $\infty$-topos, and
is therefore determined by its underlying locale $\calU$ of subobjects of the unit object.

Our first goal is to describe the locale $\calU$ more explicitly. By definition, $\calU$ is given by the partially ordered set of sieves $\calC^{0} \subseteq \calC$ which are {\em saturated} in the following sense:
\begin{itemize}
\item[$(\ast)$] If $a$ is an element of $A$, $\{ b_{\alpha} \}$ a collection of elements of
$A[ \frac{1}{a} ]$ which generate the unit ideal in $A[ \frac{1}{a} ]$, and each
localization $A[ \frac{1}{a} ][ \frac{1}{b_{\alpha}} ]$ belongs to $\calC^{0}$, then
$A[\frac{1}{a} ] \in \calC^{0}$.
\end{itemize}
For every saturated sieve $\calC^{0} \subseteq \calC$, we let $I( \calC^{0} ) = \{ a \in A: A[ \frac{1}{a} ] \in \calC^{0} \}$. It is clear that $I(\calC^{0})$ determines
the sieve $\calC^{0}$. We observe that the set $I( \calC^{0} )$ has the following properties:
\begin{itemize}
\item[$(1)$] If $a \in I(\calC^{0})$ and $\lambda \in A$, then $\lambda a \in I(\calC^{0})$ (since $\calC^{0}$ is a sieve).
\item[$(2)$] If $a, b \in I(\calC^{0})$, then $a + b \in I(\calC^{0})$ (since $A[ \frac{1}{a+b}, \frac{1}{a}],
A[ \frac{1}{a+b}, \frac{1}{b}] \in \calC^{0}$, and the elements $a$ and $b$ generate the unit ideal
$A[ \frac{1}{a+b} ]$).
\item[$(3)$] If $a \in A$ and $a^{n} \in I(\calC^{0})$ for some $n > 0$, then $a \in I(\calC^{0})$
(since there is a canonical isomorphism $A[ \frac{1}{a} ] \simeq A[ \frac{1}{a^n} ]$). 
\end{itemize}
In other words, $I(\calC^{0})$ is a radical ideal of $A$. 

Conversely, if $I \subseteq A$ is any radical ideal, then we can define $\calC^{0}$ to be the collection of all commutative $A$-algebras which are isomorphic to $A[ \frac{1}{a} ]$, for some $a \in A$. Since
$I$ is closed under multiplication, we conclude that $\calC^{0}$ is a sieve on $\calC$. We claim
that $\calC^{0}$ is saturated and that $I = I( \calC^{0})$. We first prove the second claim in the following form: if $a \in A$ is such that $A[ \frac{1}{a} ] \in \calC^{0}$, then $a \in I$. For in this case, we have
an isomorphism of $A$-algebras $A[ \frac{1}{a} ] \simeq A[ \frac{1}{b} ]$ for some $b \in I$. It
follows that $b$ is invertible in $A[ \frac{1}{a} ]$, so that $\lambda b = a^n$ for some $\lambda \in A$,
$n > 0$. Since $I$ is a radical ideal, we conclude that $a \in I$ as desired.

We now claim that $\calC^{0}$ is saturated. For this, we must show that if
$\{ b_\alpha \}$ is a collection of elements of $A[ \frac{1}{a} ]$ which generate the unit ideal and
each $A[ \frac{1}{a} ][ \frac{1}{b_{\alpha}} ] \in \calC^{0}$, then $a \in I$. Without loss of generality,
we may assume that each $b_{\alpha}$ arises from an element of $A$ (which, by abuse of notation, we will continue to denote by $b_{\alpha}$). The condition that the $b_{\alpha}$ generate the unit ideal in $A[ \frac{1}{a} ]$ guarantee an equation of the form $\sum \lambda_{\alpha} b_{\alpha} = a^n$. 
Since each product $a b_{\alpha}$ belongs to $I$, we conclude that 
$a^{n+1} = \sum_{ \alpha} \lambda_{\alpha} a b_{\alpha} \in I$. Since $I$ is a radical ideal, we deduce that $a \in I$, as desired.

The above argument shows that we can identify the locale $\calU$ with the collection of all radical ideals in $A$, partially ordered by inclusion. Unwinding the definitions, we can identify
{\em points} of $\calU$ with proper radical ideals $\mathfrak{p} \subset A$ satisfying the following additional condition:
\begin{itemize}
\item If $\mathfrak{p}$ is contains the intersection $I \cap I'$ of two radical ideals of $A$, then
either $I \subseteq \mathfrak{p}$ or $I' \subseteq \mathfrak{p}$.
\end{itemize}
If $\mathfrak{p}$ fails to satisfy this condition, then we can choose $a \in I$, $b
\in I'$ such that $a, b \notin \mathfrak{p}$. The product $ab$ belongs to 
$I \cap I' \subseteq \mathfrak{p}$, so that $\mathfrak{p}$ is not prime. Conversely, suppose that
$\mathfrak{p}$ is prime. If $\mathfrak{p}$ fails to contain some element $a \in I$, then
for every $b \in I'$, the inclusion $ab \in I \cap I' \subseteq \mathfrak{p}$ guarantees that
$b \in \mathfrak{p}$, so that $I' \subseteq \mathfrak{p}$. This proves that the set of points of
$\calU$ can be identified with the Zariski spectrum $\SSpec A$ consisting of all prime ideals of $A$.
Note that the induced topology on $\SSpec A$ agrees with the usual Zariski topology: the closed subsets are exactly those of the form $\{ \mathfrak{p} \subseteq A: I \subseteq \mathfrak{p} \}$, where
$I$ is a radical ideal of $A$. 

We next claim that $\calU$ has enough points. Unwinding the definitions, this is equivalent to the assertion that if $I \subset I'$ is a proper inclusion of radical ideals of $A$, then there exists a prime
ideal of $A$ which contains $I$ but not $I'$. Replacing $A$ by $A/I$, we may reduce to the case
$I = 0$. Since $I \neq I'$, there exists a nonzero element $a \in I'$. Since $I = 0$ is a radical ideal,
$a$ is not nilpotent. Replacing $A$ by $A[ \frac{1}{a} ]$, we may reduce to the case $I' = A$.
We are therefore reduced to the following classical fact: every nonzero commutative ring contains a prime ideal.

Since $\calU$ has enough points, it can be identified with the collection of open subsets of
the topological space $\SSpec A$. In other words, we have a canonical equivalence of
$\infty$-topoi $\USpec A \simeq \Shv( \SSpec A)$. To complete the proof, it will suffice to show that
the structure sheaves of $\USpec A$ and $\SSpec A$ agree. In other words, we must show that the functor
$\calO_{ \USpec A}: \calG_{\Zar}(k) \rightarrow \Shv( \SSpec A)$ can be described by the formula
$$ \calO_{\USpec A}(R)(A[ \frac{1}{a}]) \simeq \Hom_{ \Comm_{k} }( R, A[ \frac{1}{a} ]).$$
By definition, $\calO_{ \USpec A}(R)$ is the {\em sheaf} associated to the presheaf described by this formula. The desired result now follows from the fact (already implicit in the statement of the lemma)
that this presheaf is already a sheaf.
\end{proof}

\begin{lemma}\label{stu1}
Let $(X, \calO_X)$ be a scheme. Then the topological space $X$ is sober.
\end{lemma}

\begin{proof}
In view of Lemma \ref{stu11}, we may assume that $( X, \calO_X )$ is affine, so that $X$ is the
Zariski spectrum of some commutative ring $A$. The desired result now follows from Lemma \ref{stu2}.
\end{proof}

We next prove an analogue of Lemma \ref{stu11} in the context of $\infty$-topoi.

\begin{lemma}\label{stu3}
Let $\calX$ be an $\infty$-topos, and let $\{ U_{\alpha} \}$ be a collection of objects of
$\calX$ which cover the final object of $\calX$. Suppose that each
of the $\infty$-topoi $\calX_{/U}$ has enough points. Then $\calX$ has enough points.
\end{lemma}

\begin{proof}
Let $f: X \rightarrow Y$ be a morphism of $\calX$ which is not an equivalence. We wish to show that there exists a point $\phi^{\ast}: \calX \rightarrow \SSet$ such that $\phi^{\ast}(f)$ is not an equivalence.
For each index $\alpha$, let $\psi_{\alpha}^{\ast}: \calX \rightarrow \calX_{/U_{\alpha}}$ be a left adjoint to the projection. Since the objects $\{ U_{\alpha} \}$ cover $\calX$, there exists an index $\alpha$ such that $\psi_{\alpha}^{\ast}(f)$ is not an equivalence in $\calX_{/U_{\alpha}}$. Since
$\calX_{/ U_{\alpha} }$ has enough points, there exists a geometric morphism
$\phi_{\alpha}^{\ast}: \calX_{/ U_{\alpha} } \rightarrow \SSet$ such that
$\phi_{\alpha}^{\ast} \psi^{\ast}_{\alpha}(f)$ is not an equivalence. We can therefore take
$\phi^{\ast} = \phi_{\alpha}^{\ast} \circ \psi^{\ast}_{\alpha}$.
\end{proof}

\begin{lemma}\label{stu4}
Suppose given a map of ringed topological spaces $f: (Y, \calO_Y) \rightarrow (X, \calO_X)$
satisfying the following conditions:
\begin{itemize}
\item[$(1)$] The underlying map of topological spaces $Y \rightarrow X$ is a surjective local homeomorphism.
\item[$(2)$] The map $f^{\ast} \calO_{X} \rightarrow \calO_Y$ is an isomorphism of sheaves on $Y$.
\end{itemize}
If $(Y, \calO_Y)$ is a scheme, then $(X, \calO_X)$ is also a scheme.
\end{lemma}

\begin{proof}
Let $x$ be a point of $X$. Since $f$ is surjective, we can choose a point $y \in Y$ such that
$f(y) = x$. Let $U$ be an open subset of $Y$ such that the restriction $f | U$ is a homeomorphism from $U$ onto $V$, where $V$ is an open subset of $X$ containing $x$. Since $Y$ is a scheme, we may
(after shrinking $U$ if necessary) suppose that $(U, \calO_Y | U)$ is an affine scheme. Then
$(V, \calO_X | V) \simeq (U, \calO_Y| U)$ is also an affine scheme.
\end{proof}

\begin{proof}[Proof of Theorem \ref{stumm}] 
The first assertion follows immediately from Lemma \ref{stu2}. To complete the proof, let us
consider an arbitrary $0$-localic $\calG_{\Zar}(k)$-scheme $(\calX, \calO_{\calX})$. 
Using Lemmas \ref{stu2} and \ref{stu3}, we conclude that $\calX$ has enough points; we may therefore assume that $\calX = \Shv(X)$, where $X$ is a (sober) topological space, and we can identify
$\calO_{\calX}$ with a sheaf of local rings $\calO_X$ on $X$. We wish to show that
$(X, \calO_X)$ is a scheme.

Since $(\calX, \calO_{\calX})$ is a $\calG_{\Zar}(k)$-scheme, there exists an
\etale surjection $\coprod \Spec^{\calG_{\Zar}(k)} A_{\alpha} \rightarrow ( \calX, \calO_{\calX} )$,
for some collection of commutative $k$-algebras $\{ A_{\alpha} \}$. This induces a map
of locally ringed spaces $(Y, \calO_Y) \rightarrow (X, \calO_X)$ satisfying the hypotheses of
Lemma \ref{stu4}, so that $(X, \calO_X)$ is a scheme as desired.
\end{proof}

\begin{warning}\label{jui}
Let $(\calX, \calO_{\calX})$ be a $\calG_{\Zar}(k)$-scheme. As explained in \S \ref{geo7}, we
can identify $(\calX, \calO_{\calX})$ with its underlying ``functor of points''
$F: \Nerve(\Comm_{k}) \rightarrow \SSet$, which is a Zariski sheaf on $\Comm_{k}$. If $(\calX, \calO_{\calX})$ is $0$-localic, then $F$ can be identified with the usual (set-valued) functor associated to the underlying $k$-scheme. In this case, $F$ is a sheaf with respect to many other Grothendieck topologies on $\Comm_{k}$ (for example, the flat topology). However, this is {\em not} true for a general
$\calG_{\Zar}(k)$-scheme $(\calX, \calO_{\calX})$, even if $(\calX, \calO_{\calX})$ is $1$-localic.
For this reason, we will generally not consider $\calG_{\Zar}(k)$-schemes which are not $0$-localic; if we want to allow more general underlying $\infty$-topoi, then it is important to switch from the Zariski topology on $\Comm_{k}$ to the \etale topology (see \S \ref{exet}).
\end{warning}

\subsection{Algebraic Geometry (\Etale topology)}\label{exet}

Throughout this section, we fix a commutative ring $k$. In \S \ref{exzar}, we explained how
to use our formalism of geometries to recover the classical theory of $k$-schemes: namely, they are precisely the $0$-localic $\calG_{\Zar}(k)$-schemes, where $\calG_{\Zar}(k)$ is the geometry of Definition \ref{jin}. As the notation suggests, the collection of admissible morphisms and admissible coverings in $\calG_{\Zar}(k)$ is specifically geared towards the study of the Zariski topology on commutative rings (and the associated notion of a {\em local} commutative ring). In this section, we wish to describe an analogous geometry $\calG_{\mathet}(k)$ which is instead associated to the \etale topology on commutative rings. Our main result is Theorem \ref{sup}, which asserts that
$\calG_{\mathet}(k)$-schemes are closely related to the usual theory of Deligne-Mumford stacks over $k$.

\begin{warning}
The definitions presented in this section are slightly nonstandard, in that we do not require our
algebraic spaces or Deligne-Mumford stacks to satisfy any separatedness conditions. These can always be imposed later, if so desired.
\end{warning}

We begin by reviewing some definitions.

\begin{notation}\label{sint}
Let $A$ be a commutative ring. We let $\Comm_{A}^{\mathet}$ denote the category of all
commutative $A$-algebras which are \etale over $A$. We regard $(\Comm_{A}^{\mathet})^{op}$ as equipped with the following Grothendieck topology: a collection of \etale morphisms $\{ B \rightarrow B_{\alpha} \}$ is a {\it covering} if there exists a finite collection of indices $\{ \alpha_1, \ldots, \alpha_n \}$ such that the map $B \rightarrow \prod_{1 \leq i \leq n} B_{\alpha_n}$ is faithfully flat.
An {\it \etale sheaf of sets} on $A$ is a functor $\calF: \Comm_{A}^{\mathet} \rightarrow \Set$
which is a sheaf with respect to this topology: that is, for every \etale covering
$\{ B \rightarrow B_{\alpha} \}$, the associated diagram
$$\xymatrix{ \calF(B) \ar[r] & \prod_{\alpha} \calF(B_{\alpha}) \ar@<.4ex>[r] \ar@<-.4ex>[r] & 
\prod_{\alpha, \beta} \calF(B_{\alpha} \otimes_{B} B_{\beta})}$$
is an equalizer. We let $\Shv_{\Set}^{\mathet}(A)$ denote the full subcategory of
$\Fun( \Comm_{A}^{\mathet}, \Set)$ spanned by the sheaves of sets on $A$.

Let $\phi: A \rightarrow A'$ be a homomorphism of commutative rings. Composition
with the base change functor $B \mapsto B \otimes_{A} A'$ induces a pushforward functor
$\phi_{\ast}: \Shv_{\Set}^{\mathet}(A') \rightarrow \Shv_{\Set}^{\mathet}(A)$. This pushforward functor has a left adjoint, which we will denote by $\phi^{\ast}$.

Let us now fix a commutative ground rink $k$. Suppose that $A$ is a commutative $k$-algebra
and that $\calF \in \Shv_{\Set}^{\mathet}(A)$. We define a functor
$\widehat{\calF}: \Comm_{k} \rightarrow \Set$ by the following formula:
$$ \widehat{\calF}(B) = \{ (\phi, \eta): \phi \in \Hom_{k}(A,B), \eta \in (\phi^{\ast} \calF)(B) \}.$$
\end{notation}

\begin{example}
Let $k$ be a commutative ring, $A$ a commutative $k$-algebra, and let $\calF_0 \in \Shv_{\Set}^{\mathet}(A)$ be a final object (so that $\calF_0(B) \simeq \ast$ for every \etale $A$-algebra $B$). Then
$\widehat{\calF_0}: \Comm_{k} \rightarrow \Set$ can be identified with the functor $\underline{A}$ 
corepresented by $A$: that is, we have canonical bijections
$$ \widehat{\calF_0}(B) \simeq \underline{A}(B) \simeq \Hom_{k}(A, B).$$
More generally, if $\calF$ is any object of $\Shv_{\Set}^{\mathet}(A)$, then we have a unique
map $\calF \rightarrow \calF_0$ in $\Shv_{\Set}^{\mathet}(A)$, which induces a natural transformation
$\widehat{\calF} \rightarrow \underline{A}$.
\end{example}

\begin{definition}\label{grag}
Let $k$ be a commutative ring, $A$ a commutative $k$-algebra, and 
$F: \Comm_{k} \rightarrow \Set$ any functor equipped with a natural transformation
$\alpha: F \rightarrow \underline{A}$. We will say that $\alpha$ {\it exhibits $F$ as
an algebraic space \etale over $\SSpec A$} it there exists an object
$\calF \in \Shv_{\Set}^{\mathet}(A)$ and an isomorphism $F \simeq \widehat{\calF}$ in
$\Fun( \Comm_{k}, \Set)_{/ \underline{A} }$. 
\end{definition}

\begin{remark}\label{kiu}
Definition \ref{grag} is slightly more general than the usual definition of an algebraic space \etale over $\SSpec A$, as found in \cite{knutson}. Let $A$ be a commutative ring and let $\calF \in \Shv_{\Set}^{\mathet}(A)$. Choosing a set of sections $\eta_{\alpha} \in \calF( A_{\alpha} )$ which generate
$\calF$, we obtain an effective epimorphism
$$ \coprod \underline{A}_{\alpha} \rightarrow \widehat{\calF}$$
in the category of sheaves of sets on $\Comm_{k}$. However, the maps
$\underline{A}_{\alpha} \rightarrow \widehat{\calF}$ need not be relatively representable by schemes.
However, this {\em is} true if there exists a monomorphism $\widehat{\calF} \rightarrow \underline{B}$, for some commutative $k$-algebra $B$. For in this case, each fiber product
$$ \underline{A}_{\alpha} \times_{ \widehat{\calF} } \underline{A}_{\beta}
\simeq \underline{A}_{\alpha} \times_{ \underline{B} } \underline{A}_{\beta}
\simeq \underline{ A_{\alpha} \otimes_{B} A_{\beta} }$$
is representable by an affine $k$-scheme.

In the general case, each fiber product $\calF_{\alpha,\beta} = \underline{A}_{\alpha} \times_{ \widehat{\calF} } \underline{A}_{\beta}$ is again relative algebraic space \etale over $A$ (in the sense of Definition \ref{grag}), which admits a monomorphism $\calF_{\alpha, \beta} \hookrightarrow
\underline{ A_{\alpha} \otimes_{k} A_{\beta} }$. It follows that $\calF_{\alpha, \beta}$ {\em is}
an algebraic space in the more restrictive sense of \cite{knutson} (so that $\calF$ can be described as the
quotient of an \etale equivalence relation in this more restrictive setting).
\end{remark}

\begin{notation}
We will abuse notation by identifying the corepresentable functor $\underline{A}: \Comm_{k} \rightarrow \Set$ with the induced space-valued functor $\Nerve( \Comm_{k} ) \rightarrow \SSet$; in this context we can identify $\underline{A}$ with the functor corepresented by $A$ in the $\infty$-category
$\Nerve( \Comm_{k} )$. 

Given a functor $F: \Nerve( \Comm_{k} ) \rightarrow \SSet$ and a natural transformation
$\alpha: F \rightarrow \underline{A}$, we will say that $\alpha$ exhibits $F$ as an algebraic space
\etale over $\SSpec A$ if $F(B)$ is a discrete object of $\SSet$ for every commutative $k$-algebra $B$
(in other words, $\pi_{i}( F(B), x) \simeq \ast$ for every $i > 0$ and every base point $x \in F(B)$),
and the induced functor
$$ \Comm_{k} \simeq \h{ \Nerve(\Comm_{k} )} \rightarrow \h{\SSet} \stackrel{\pi_0}{\rightarrow} \Set$$
is an algebraic space \etale over $\SSpec A$, in the sense of Definition \ref{grag}.
\end{notation}

\begin{definition}\label{spei}
Let $k$ be a commutative ring. The {\it \etale topology} on $\Nerve(\Comm_{k}^{op})$ is defined as follows:
given a commutative $k$-algebra $A$ and a sieve $\calC^{0} \subseteq \Nerve( \Comm_{k}^{op})_{/A}$, we will say that $\calC^{0}$ is {\it covering} if it contains a finite collection of \etale morphisms
$\{ A \rightarrow A_{i} \}_{1 \leq i \leq n}$ such that the map $A \rightarrow \prod_{1 \leq i \leq n} A_i$
is faithfully flat.
\end{definition}

\begin{remark}
We let $\Shv( \Comm_{k}^{op} )$ denote the full subcategory of 
$\Fun( \Nerve( \Comm_{k}^{op}), \SSet)$ spanned by those functors which are sheaves with respect to the
\etale topology of Definition \ref{spei}. Though $\Shv( \Nerve(\Comm_{k})^{op} )$ is not an $\infty$-topos
(because $\Comm_{k}$ is not a small category), it nevertheless behaves like one for practical purposes; for example, there is a good theory of effective epimorphisms in $\Shv( \Nerve( \Comm_{k} )^{op} )$
(see \S \ref{geo7}).
\end{remark}

\begin{definition}\label{spud}
Let $k$ be a commutative ring. We will say that a natural transformation
$\alpha: F \rightarrow F'$ in $\Fun( \Nerve(\Comm_{k}), \SSet)$ {\it exhibits
$F$ as a relative algebraic space \etale over $F'$} if the following condition is satisfied:
\begin{itemize}
\item[$(\ast)$] Let $A$ be a commutative $k$-algebra and $\eta \in F'(A)$, classifying a map
$\underline{A} \rightarrow F'$. Then the induced map
$F \times_{F'} \underline{A} \rightarrow \underline{A}$ exhibits $F \times_{F'} \underline{A}$
as an algebraic space \etale over $\SSpec A$.
\end{itemize}

A {\it Deligne-Mumford stack over $k$} is a functor
$F: \Nerve( \Comm_{k} ) \rightarrow \SSet$ satisfying the following conditions:
\begin{itemize}
\item[$(1)$] The functor $F$ is a sheaf with respect to the \etale topology of Definition \ref{spei}.
\item[$(2)$] There exists a small collection of commutative $k$-algebras $\{ A_{\alpha} \}$ and
points $\eta_{\alpha} \in F(A_{\alpha})$ with the following properties: 
\begin{itemize}
\item[$(a)$] For each index $\alpha$, the induced map $\underline{A}_{\alpha} \rightarrow F$
is a relative algebraic space \etale over $F$.
\item[$(b)$] The induced map
$$ \coprod_{\alpha} \underline{A_{\alpha}} \rightarrow F$$
is an effective epimorphism in $\Shv( \Nerve(\Comm_{k} )^{op} )$.
\end{itemize}
\end{itemize}
\end{definition}

\begin{remark}\label{twab}
Let us compare Definition \ref{spud} with the usual definitions of Deligne-Mumford stack, as found (for example) in \cite{stacks}. To begin, we consider here functors $F$ which take values in the
$\infty$-category $\SSet$ of spaces. However, this results in no additional generality:
we have required the existence of an effective epimorphism
$$ \phi: \coprod_{\alpha} \underline{A}_{\alpha} \rightarrow F$$
where the domain $\coprod_{\alpha} \underline{A}_{\alpha}$ is discrete and the fibers of
$\phi$ are discrete. Consequently, $F$ is a $1$-truncated object of
$\Shv( \Nerve( \Comm_{k} )^{op} )$, and therefore takes values in the $\infty$-category
$\tau_{\leq 1} \SSet$ of $1$-truncated spaces, which is equivalent to the $2$-category of small
groupoids. 

The other principal difference in our definition is that we allow more freedom in our definition of an algebraic space (see Remark \ref{kiu}).
\end{remark}

\begin{remark}\label{sue}
Consider a pullback diagram
$$ \xymatrix{ F_0 \ar[r] \ar[d]^{\alpha_0} & F \ar[d]^{\alpha} \\
F'_0 \ar[r] & F'}$$
in $\Shv( \Nerve(\Comm_{k})^{op})$. If $\alpha$ exhibits $F$ as a relative algebraic space \etale over $F'$, then $\alpha_0$ exhibits $F_0$ as a relative algebraic space \etale over $F$.
\end{remark}

We now reformulate the theory of Deligne-Mumford stacks using our language of geometries.

\begin{definition}
Let $k$ be a commutative ring. We define a geometry $\calG_{\mathet}(k)$ as follows:
\begin{itemize}
\item[$(1)$] The underlying $\infty$-category is $\calG_{\mathet}(k) = \Nerve( \Comm_{k}^{\fin})^{op}$, the (nerve of the) category of affine $k$-schemes of finite presentation.
\item[$(2)$] A morphism $f$ in $\calG_{\mathet}(k)$ is admissible if and only if the corresponding morphism $A \rightarrow B$ is an \etale map of commutative $k$-algebras.
\item[$(3)$] The Grothendieck topology on $\calG_{\mathet}(k)$ is the (restriction of the) \etale topology described in Definition \ref{spei}.  
\end{itemize}
\end{definition}

\begin{remark}
As in Remark \ref{hunner}, we have a canonical equivalence
$\Ind( \calG_{\mathet}(k)^{op} ) \simeq \Nerve( \Comm_{k} )$. 
\end{remark}

\begin{remark}\label{obba}
Let $\calO$ be a left exact functor from $\calG_{\mathet}(k) \simeq \calG(k)$ to $\SSet$,
corresponding to a commutative $k$-algebra $A$. Then $\calO$ defines a $\calG_{\mathet}(k)$-structure on $\SSet$ if and only if the following condition is satisfied:
\begin{itemize}
\item[$(\ast)$] For every commutative $k$-algebra $B$ of finite presentation, and every
finite collection of \etale morphisms $\{ B \rightarrow B_{\alpha} \}$ which induce a faithfully flat
map $B \rightarrow \prod_{\alpha} B_{\alpha}$, the induced map
$$ \coprod \calO(B_{\alpha}) \rightarrow \calO( B)$$
is an effective epimorphism. In other words, every map of commutative $k$-algebras
$B \rightarrow A$ factors through some $B_{\alpha}$.
\end{itemize}
This definition is equivalent to the requirement that $A$ be a strictly Henselian local ring.

More generally, let $\calX$ be an $\infty$-topos with enough points. Then a
$\calG_{\mathet}(k)$-structure on $\calX$ can be identified with a commutative
$k$-algebra $\calA$ in the underlying topos $\h{(\tau_{\leq 0} \calX)}$ such that,
for every point $x$ of $\calX$, the stalk $\calA_{x}$ is strictly Henselian (see Remark \ref{henhun}).
We may therefore refer informally to $\calG_{\mathet}(k)$-structures on $\calX$
as {\it strictly Henselian sheaves of $k$-algebras on $\calX$}. 
\end{remark}

\begin{remark}\label{clump}
We can identify the $\infty$-category $\Pro(\calG_{\mathet}(k))$ with
(the nerve of) the category $\Comm^{op}_{k}$ of affine $k$-schemes. Under this identification, a morphism $A \rightarrow B$ in $\Comm_{k}$ is admissible (in the sense of Notation \ref{ilk})
if and only if it is \etale in the usual sense (this follows from the observation that every \etale homomorphism of $k$-algebras is the pushout of an \etale homomorphism between finitely presented $k$-algebras; for a stronger version of this assertion we refer the reader to Proposition \ref{swimm}),
and the Grothendieck topology on $\Pro( \calG_{\mathet}(k))$ reduces to the \etale topology of Notation \ref{sint}.
\end{remark}

\begin{theorem}\label{sup}
Let $k$ be a commutative ring, and let $F: \Nerve( \Comm_{k} ) \rightarrow \SSet$ be a functor.
The following conditions are equivalent:
\begin{itemize}
\item[$(1)$] The functor $F$ is a Deligne-Mumford stack over $k$ (in the sense of Definition \ref{spud}).
\item[$(2)$] The functor $F$ is representable by a $\calG_{\mathet}(k)$-scheme $(\calX, \calO_{\calX})$ such that $\calX$ is $1$-localic.
\end{itemize}
\end{theorem}

The proof will require a number of preliminaries.

\begin{lemma}\label{psyward}
Let $\calG$ be an $n$-truncated geometry, $\calX$ an $n$-localic $\infty$-topos, and $\calO_{\calX}: \calG \rightarrow \calX$ a $\calG$-structure on $\calX$. Then
$(\calX, \calO_{\calX})$ is an $n$-truncated object of $\LGeo(\calG)^{op}$.
\end{lemma}

\begin{proof}
Let $(\calY, \calO_{\calY})$ be any object of $\LGeo(\calG)^{op}$. We have a canonical map
$$ \phi: \bHom_{ \LGeo(\calG)^{op} }( (\calY, \calO_{\calY} ),
(\calX, \calO_{\calX}) ) \rightarrow \bHom_{ \LGeo^{op} }( \calY, \calX).$$
Since $\calX$ is $n$-localic, we can identify the right side with the space of
geometric morphisms of $n$-topoi from $\tau_{\leq n} \calY$ to $\tau_{\leq n} \calX$, which is $n$-truncated. It will therefore suffice to show that the homotopy fibers of $\phi$ are $n$-truncated. But the homotopy fiber of $\phi$ over a point $f_{\ast}: \calY \rightarrow \calX$ can be identified with
$\bHom_{ \Struct_{\calG}(\calY)}( f^{\ast} \calO_{\calX}, \calO_{\calY})$, which is $n$-truncated
since the geometry $\calG$ is $n$-truncated.
\end{proof}

\begin{lemma}\label{surpuk}
Let $\alpha: F' \rightarrow F$ be a morphism in $\Shv( \Nerve(\Comm_{k})^{op} )$, and suppose
that $F$ is representable by a $\calG_{\mathet}(k)$-scheme $(\calX, \calO_{\calX})$. The following conditions are equivalent:
\begin{itemize}
\item[$(1)$] The map $\alpha$ exhibits $F'$ as a relative algebraic space \etale over $F$.
\item[$(2)$] The functor $F'$ is representable by a $\calG_{\mathet}(k)$-scheme
$(\calY, \calO_{\calY})$, and $\alpha$ induces an equivalence
$(\calY, \calO_{\calY}) \simeq (\calX_{/U}, \calO_{\calX}|U)$ for some discrete object
$U \in \calX$.
\end{itemize}
\end{lemma}

\begin{proof}
We first show that $(1) \Rightarrow (2)$. The problem is local on $\calX$, so we may assume without loss of generality that $( \calX, \calO_{\calX} ) \simeq \Spec^{\calG_{\mathet}} A$ is affine.
In this case, condition $(1)$ implies that $F' \simeq \widehat{\calF}$, where
$\calF$ is an \etale sheaf (of sets) on $\Comm_{A}^{\mathet}$. Unwinding the definitions (see Remark \ref{clump}), we can identify $\calF$ with a discrete object $U$ of the $\infty$-topos
$\Shv( \Pro( \calG_{\mathet}(k))^{\adm}_{/A} )$. Theorem \ref{scoo} allows us to identify this
$\infty$-topos $\calX$, and Remark \ref{unipop} implies that the $\calG_{\mathet}(k)$-scheme
$( \calX_{/U}, \calO_{\calX}|U)$ represents the functor $F'$.

The reverse implication is proven using exactly the same argument.
Now suppose that $(2)$ is satisfied. We wish to show that $F'$ is a relative algebraic space \etale over $F$. Using Remark \ref{pre4}, we can reduce to the case where
$F$ is representable by the affine scheme $\Spec^{\calG_{\mathet}(k)} A$, for some
commutative $k$-algebra $A$. Using Theorem \ref{scoo} and Remark \ref{clump}, we can identify the discrete object $U \in \calX$ with a sheaf of sets $\calF$ on $\Comm_{A}^{\mathet}$. Remark \ref{unipop} now furnishes an equivalence $F' \simeq \widehat{\calF}$.
\end{proof}

\begin{corollary}\label{simj}
Suppose given a commutative diagram
$$ \xymatrix{ & F' \ar[dr]^{\beta} & \\
F \ar[ur]^{\alpha} \ar[rr]^{\gamma} & & F'' }$$
in $\Shv( \Nerve(\Comm_{k})^{op} )$, where $\beta$ exhibits $F'$ as a relative algebraic space \etale over $F''$. Then $\alpha$ exhibits $F$ as a relative algebraic space \etale over $F'$ if and only if $\gamma$ exhibits $F$ as a relative algebraic space \etale over $F''$.
\end{corollary}

\begin{proof}
First we prove the ``only if'' direction. We wish to show that $\gamma$ exhibits
$F$ as a relative algebraic space \etale over $F''$. Without loss of generality, we
may suppose that $F''$ is representable by an affine $\calG_{\mathet}(k)$-scheme
$\Spec^{\calG_{\mathet}}(k) A = (\calX, \calO_{\calX} )$. Applying Lemma \ref{surpuk} to the morphism $\beta$, we conclude that $F'$ is representable by $(\calX_{/U}, \calO_{\calX}|U)$ for some discrete object $U \in \calX$. Applying Lemma \ref{surpuk} again, we conclude that
$F$ is representable by $(\calX_{/V}, \calO_{\calX} | V)$ for some discrete object
$V \in \calX_{/U}$. Since $U$ is discrete, we conclude that $V$ is discrete when viewed as an object of $\calX$, so the desired result follows from Lemma \ref{surpuk}.

Let us now prove the ``if'' direction. We must show that for every pullback diagram
$$ \xymatrix{ F_0 \ar[r]^{\alpha_0} \ar[d] & \underline{A} \ar[d] \\
F \ar[r]^{\alpha} & F', }$$
the map $\alpha_0$ exhibits $F_0$ as a relative algebraic space \etale over $A$.
Pulling back by the composite map $\underline{A} \rightarrow F' \rightarrow F'',$
we may assume without loss of generality that $F''$ is representable by the affine
$\calG_{\mathet}(k)$-scheme $\Spec^{\calG_{\mathet}(k)} A = (\calX, \calO_{\calX})$. 
Applying Lemma \ref{surpuk}, we conclude that $F$ and $F'$ are representable
by $( \calX_{/U}, \calO_{\calX}|U)$ and $(\calX_{/V}, \calO_{\calX}|V)$, respectively, for some pair of discrete objects $U, V \in \calX$. Using Remark \ref{unipop}, we can identify
$\alpha$ with the map induced by a morphism $f: U \rightarrow V$ in $\calX$.
Since $U$ and $V$ are discrete, the map $f$ exhibits $U$ as a discrete object in
$\calX_{/V}$. Applying Lemma \ref{surpuk}, we deduce that $\alpha$
exhibits $F$ as a relative algebraic space \etale over $F'$, as desired.
\end{proof}

\begin{proof}[Proof of Theorem \ref{sup}]
We first prove that $(1) \Rightarrow (2)$. If $(1)$ is satisfied, then there exists a collection
of commutative $k$-algebras $A_{\alpha}$ and \etale maps $\underline{A}_{\alpha} \rightarrow F$ such that, if $F_0$ denotes the coproduct of the family $\underline{A}_{\alpha}$ in
$\Shv( \Nerve(\Comm_{k})^{op} )$, then the induced map $f: F_0 \rightarrow F$ is an effective epimorphism. Note that $F_0$ is representable by the affine $\calG_{\Zar}(k)$-scheme
$\coprod_{\alpha} \Spec^{\calG_{\Zar}(k)} A_{\alpha} = (\calX_0, \calO_{\calX_0})$. 

Let $F_{\bigdot}$ be the simplicial object of $\Shv( \Nerve(\Comm_{k})^{op} )$ given by the
\Cech nerve of $f$, so that $F_{n} \simeq F_0 \times_{F} \ldots \times_{F} F_0$.
Applying Remark \ref{sue} repeatedly, we deduce that each of the projection maps
$F_{n} \rightarrow F_0$ exhibits $F_{n}$ as a relative algebraic space \etale over $F_0$.
In particular, each $F_{n}$ is representable by a $\calG_{\mathet}(k)$-scheme
which is \etale over $(\calX_0, \calO_{\calX_0})$ (Lemma \ref{surpuk}). 
Applying Theorem \ref{skil}, we may assume that $F_{\bigdot}$ is the image
of a simplicial object $( \calX_{\bigdot}, \calO_{\calX_{\bigdot}})$ in the
$\infty$-category $\Sch( \calG_{\mathet}(k))$ of $\calG_{\mathet}(k)$-schemes.

Our next goal is to prove that the simplicial object 
$( \calX_{\bigdot}, \calO_{\calX_{\bigdot}})$ actually takes values in the
subcategory $\Sch(\calG_{\mathet}(k))_{\mathet} \subseteq \Sch( \calG_{\mathet}(k))$; in other words, we claim that for each morphism $i: [m] \rightarrow [n]$ in $\cDelta$, the induced map
$( \calX_{n}, \calO_{\calX_n}) \rightarrow (\calX_{m}, \calO_{\calX_m})$ is \etale.
Choosing a map $[0] \rightarrow [m]$ and applying Proposition \ref{scanh} to the resulting diagram
$$ \xymatrix{ & ( \calX_{m}, \calO_{\calX_{m}}) \ar[dr] & \\
(\calX_n, \calO_{\calX_n}) \ar[ur] \ar[rr] & & ( \calX_0, \calO_{\calX_0} ), }$$
we can reduce to the case where $m =0$, which now follows from Lemma \ref{surpuk}. 

Let $(\calX, \calO_{\calX})$ be the geometric realization of the simplicial object
$(\calX_{\bigdot}, \calO_{\calX_{\bigdot} })$ in $\Sch( \calG_{\mathet}(k) )$ (which exists
by virtue of Proposition \ref{sizem}). Lemma \ref{kokin} implies that $(\calX, \calO_{\calX})$ represents the geometric realization $| F_{\bigdot} | \in \Shv( \Nerve(\Comm_{k})^{op})$. Since $F_0 \rightarrow F$ is an effective epimorphism, we can identify this geometric realization with $F$.
To complete the proof, it will suffice to show that $(\calX, \calO_{\calX})$ is $1$-localic.

Theorem \ref{top4} implies the existence of an equivalence
$(\calX, \calO_{\calX}) \simeq (\calY_{/U}, \calO_{\calY}|U)$, where
$(\calY, \calO_{\calY})$ is a $1$-localic $\calG_{\mathet}(k)$-scheme and
$U \in \calY$ is $2$-connective. To complete the proof, we will show that
$U$ is a final object of $\calY$. Since $U$ is $2$-connective, it will suffice to show that
$U$ is $1$-truncated. This question is local on $\calY$; it will therefore suffice to show
that for every \etale map $\phi: \Spec^{\calG_{\mathet}(k)} A \rightarrow (\calY, \calO_{\calY})$,
the pullback $\phi^{\ast} U$ is a final object in the underlying $\infty$-topos
of $\Spec^{\calG_{\mathet}(k)} A$. Using Theorem \ref{scoo} and Remark \ref{clump}, we can
identify this $\infty$-topos with the $\infty$-category of sheaves (of spaces) on the
category of \etale $A$-algebras. It will therefore suffice to show that for every
\etale $A$-algebra $B$, the space $\phi^{\ast}(U)(B)$ is $1$-truncated.
Replacing $A$ by $B$ and invoking Remark \ref{unipop}, we are reduced to showing that
the homotopy fibers of the map 
$$F(A) \rightarrow \bHom_{ \Sch(\calG_{\mathet}(k))}( \Spec^{\calG_{\mathet}(k)} A,
(\calY, \calO_{\calY}))$$
are $1$-truncated. We now complete the argument by observing that
$F(A)$ is $1$-truncated by assumption $(1)$ (see Remark \ref{twab}), and the target
$\bHom_{ \Sch(\calG_{\mathet}(k)}( \Spec^{\calG_{\mathet}(k)} A,
(\calY, \calO_{\calY})$ is $1$-truncated by Lemma \ref{psyward}.

We now prove that $(2) \Rightarrow (1)$. Assume that $F$ is representable
by a $1$-localic $\calG_{\mathet}(k)$-scheme $(\calX, \calO_{\calX})$.
Then there exists a collection of objects $U_{\alpha} \in \calX$ with the following properties:
\begin{itemize}
\item[$(a)$] The canonical map $\coprod U_{\alpha} \rightarrow 1_{\calX}$ is an effective epimorphism in $\calX$.
\item[$(b)$] Each of the $\calG_{\mathet}(k)$-schemes
$( \calX_{/ U_{\alpha}}, \calO_{\calX} | U_{\alpha} )$ is affine.
\end{itemize}
Let $F_{\alpha} \in \Shv( \Nerve(\Comm_{k})^{op})$ be the functor represented by
$( \calX_{/U_{\alpha}}, \calO_{\calX} | U_{\alpha} )$. It follows from $(a)$ the the map
$\coprod F_{\alpha} \rightarrow F$ is an effective epimorphism. It will therefore suffice to show that
each $F_{\alpha}$ is a relative algebraic space \etale over $F$. In view of Lemma \ref{surpuk}, we are reduced to proving that each of the objects $U_{\alpha}$ is discrete. Using Remark \ref{unipop}, this translates to the condition that for every commutative $k$-algebra $B$, the homotopy fibers of the map $F_{\alpha}(B) \rightarrow F(B)$ are discrete. To prove this, it suffices to observe that
$F(B)$ is $1$-truncated (by Lemma \ref{psyward}), and $F_{\alpha}(B)$ is discrete
(in fact, if $F_{\alpha}$ is represented by the affine $\calG_{\mathet}(k)$-scheme $\Spec^{\calG} A_{\alpha}$, then $F_{\alpha}(B)$ is homotopy equivalent to the discrete set
$\Hom_{\Comm_{k}}(A_{\alpha}, B)$).
\end{proof}

\begin{warning}
Fix a commutative ring $k$.
We have an evident transformation of geometries $\calG_{\Zar}(k) \rightarrow \calG_{\mathet}(k)$
(which, at the level of the underlying $\infty$-categories, is simply given by the identity functor). 
We therefore have a relative spectrum functor $\Spec_{\calG_{\Zar}(k)}^{\calG_{\mathet}(k)}: \Sch( \calG_{\Zar}(k)) \rightarrow \Sch( \calG_{\mathet}(k) )$. When restricted to $0$-localic $\calG_{\Zar}(k)$-schemes, we recover the usual embedding of the category of $k$-schemes into the $2$-category
of Deligne-Mumford stacks over $k$, which is fully faithful. However, 
$\Spec_{\calG_{\Zar}(k)}^{\calG_{\mathet}(k)}$ is {\em not} fully faithful in general, because the cohomology of a scheme with respect to the Zariski topology (with constant coefficients, say) generally does not agree with its cohomology with respect to the \etale topology.
\end{warning}

Theorem \ref{sup} gives a description of the $\infty$-category of $1$-localic
$\calG_{\mathet}(k)$-schemes in reasonably classical terms. It is natural to ask what
happens if we consider more general $\calG_{\mathet}(k)$-schemes. The following consequence of Theorem \ref{top4} implies that there is no essential gain in generality:

\begin{proposition}
Let $k$ be a commutative ring, and let $(\calX, \calO_{\calX}) \in \LGeo( \calG_{\mathet}(k) )$. 
The following conditions are equivalent:
\begin{itemize}
\item[$(1)$] The pair $(\calX, \calO_{\calX})$ is an $\calG_{\mathet}(k)$-scheme.
\item[$(2)$] There exists a $1$-localic $\calG_{\mathet}(k)$-scheme $(\calY, \calO_{\calY})$, a
$2$-connective object $U \in \calY$, and an equivalence
$$( \calX, \calO_{\calX}) \simeq ( \calY_{/U}, \calO_{\calY} | U ).$$
\end{itemize}
\end{proposition}

\section{Smoothness}\label{appus}

Let $k$ be a commutative ring, and $\calC$ an $\infty$-category which admits finite limits. Suppose that we wish
to define the notion of a {\it commutative $k$-algebra} in $\calC$. Following the ideas introduced in \S \ref{geo}, we might proceed as follows. Let $\calG(k)$ be the (nerve of the) {\em opposite} of the category of finitely presented commutative $k$-algebras (in other words, the category of affine $k$-schemes which are locally of finite presentation). Then we can define a commutative $k$-algebra in $\calC$
to be a left exact functor from $\calG(k)$ to $\calC$. Note that because $\calG(k)$ is the nerve of an ordinary category, any left exact functor from $\calG(k)$ to $\calC$ automatically factors through
the full subcategory $\tau_{\leq 0} \calC \subseteq \calC$ spanned by the discrete objects
(Proposition \toposref{eaa}).

Let $\calT(k) \subseteq \calG(k)$ denote the full subcategory spanned by the {\em free} $k$-algebras: that is, $k$-algebras of the form $k[x_1, \ldots, x_n]$. The relationship between $\calG(k)$ and $\calT(k)$ can be summarized as follows:
\begin{itemize}
\item[$(a)$] The $\infty$-category $\calT(k)$ admits finite products.
\item[$(b)$] The $\infty$-category $\calG(k)$ is discrete (that is, equivalent to the nerve of a category) and admits finite limits.
\item[$(c)$] The inclusion $\calT(k) \subseteq \calG(k)$ preserves finite products. Moreover, 
$\calG(k)$ is {\em universal} among discrete $\infty$-categories which admit finite limits and receive
a product-preserving functor from $\calT(k)$ (see \S \ref{geoenv} for a more detailed discussion of this universal property). 
\end{itemize}

It follows from assertion $(c)$ that we can also define a commutative $k$-algebra in an $\infty$-category $\calC$ to be a product-preserving functor from $\calT$ into $\tau_{\leq 0} \calC$. Consequently, we can view the category of commutative $k$-algebras in $\calC$ as a full subcategory of the
larger $\infty$-category $\Fun^{\times}(\calT(k), \calC)$ of {\em all} product-preserving functors from $\calT(k)$ to $\calC$. In the case where $\calC$ is the $\infty$-category of spaces, we can identify
$\Fun^{\times}(\calT(k), \calC)$ with the $\infty$-category underlying the model category of {\it simplicial commutative $k$-algebras}. Loosely speaking, we can summarize the situation as follows: the
category $\calG(k)$ knows about algebraic geometry (over $k$), but the category $\calT(k)$ knows about {\em derived} algebraic geometry.

To pursue this idea further, we first observe that $\calG(k)$ can be viewed as a {\it geometry}, in the sense of Definition \ref{psyob}. In fact, it can be viewed as a geometry in several different ways (see \S \ref{exzar} and \S \ref{exet}); for definiteness, let us regard $\calG(k)$ as endowed with the \etale topology described in \S \ref{exet}. The subcategory $\calT(k) \subseteq \calG(k)$ does not interact well with the geometry structure on $\calG(k)$. The problem is easy to describe: if we identify $\calG(k)$ with the category of affine $k$-schemes of finite presentation, then $\calT(k)$ corresponds to the subcategory spanned by the {\em affine spaces} over $k$: that is, finite products of the affine line with itself. The \etale topology on $\calT(k)$ is not very interesting, because there are not many \etale maps between affine spaces. However, there is a natural replacement: the larger subcategory $\calT_{\mathet}(k) \subseteq \calG(k)$ spanned by those affine $k$-schemes which are \etale over affine spaces (in particular, every object of $\calT_{\mathet}(k)$ is a smooth $k$-scheme). This subcategory contains $\calT$ and is not {\em too} much larger, by virtue of the fact that every smooth $k$-scheme can be locally realized as an \etale cover of an affine space.

The category $\calT_{\mathet}(k)$ is much like a geometry: it has a robust theory of admissible morphisms (in this case, \etale maps between $k$-schemes) and an associated Grothendieck topology. It fails to qualify as a geometry only because $\calT_{\mathet}(k)$ does not admit finite limits. Instead, $\calT_{\mathet}(k)$ is an example of a {\it pregeometry} (see \S \ref{app1} for a precise definition). Moreover, this pregeometry {\em determines} the geometry $\calG(k)$: namely, we will prove an analogue of the universal property asserted in $(c)$, which makes reference to the Grothendieck topologies on $\calT_{\mathet}(k)$ and $\calG(k)$ (see Proposition \ref{sturman}).
The situation can be summarized by saying that $\calG(k)$ is a {\it $0$-truncated envelope} of $\calT_{\mathet}(k)$: we refer the reader to \S \ref{geoenv} for an explanation of this terminology. 

Our objective in this section is to describe a theory of pregeometries which is analogous to the theory of geometries developed in \S \ref{geo}. Although this material is not logically necessary for developing the foundations of derived algebraic geometry, it has many uses:

\begin{itemize}
\item[$(1)$] A single pregeometry $\calT$ can generate a variety of geometries. For example,
the pregeometry $\calT_{\mathet}$ described above can give rise to either classical algebraic geometry or derived algebraic geometry, depending on what discreteness conditions we impose.

\item[$(2)$] Often it is easier to describe a pregeometry $\calT$ than it is to describe the associated geometry $\calG$. For example, in the complex analytic setting, it is easier to describe the class of complex analytic manifolds than the class of complex analytic spaces (see \S \ref{app6}).

\item[$(3)$] If $\calG$ is a geometry associated to a pregeometry $\calT$,
then we can introduce an associated theory of {\em smooth} morphisms between $\calG$-schemes,
which specializes to the usual notion of smoothness when we take $\calT = \calT_{\mathet}$.
\end{itemize}

We now outline the contents of this section. We will begin in \S \ref{app1} by introducing the definition of a {\it pregeometry}. Given a pregeometry $\calT$ and an $\infty$-topos $\calX$, we will define an $\infty$-category $\Struct_{\calT}(\calX)$ of {\it $\calT$-structures} on $\calX$. In \S \ref{app2}, we will discuss the functoriality of $\Struct_{\calT}(\calX)$ in $\calT$, and describe situations in which two pregeometries $\calT$ and $\calT'$ give rise to the same notion of ``structure''. In \S \ref{app5}, we will describe some examples of $\calT$-structures, which are obtained by mirroring the constructions of \S \ref{abspec}.

Every pregeometry $\calT$ admits a {\it geometric envelope} $\calG$: a geometry with the property that for every $\infty$-topos $\calX$, there is a canonical equivalence $\Struct_{\calG}(\calX) \simeq \Struct_{\calT}(\calX)$. We will provide a construction of $\calG$ in \S \ref{app4}. Consequently, we can view the theory of pregeometries presented here as a less general version of our previous theory of geometries.
However, the $\infty$-category $\Struct_{\calG}(\calX)$ has special features in the case where
the geometry $\calG$ arises as the envelope of a pregeometry: for example, $\Struct_{\calG}(\calX)$
admits sifted colimits. We will establish this and other properties in \S \ref{app3}.

\subsection{Pregeometries}\label{app1}

\begin{definition}\label{geot}
A {\it pregeometry} is an $\infty$-category $\calT$ equipped with an admissibility
structure (see Definition \ref{stubb}) such that $\calT$ admits finite products.
\end{definition}





\begin{remark}
In the situation of Definition \ref{geot}, we will generally abuse terminology by identifying a pregeometry with its underlying $\infty$-category $\calT$; we implicitly understand that
a Grothendieck topology on $\calT$ and a class of admissible morphisms has been specified as well.
\end{remark}

\begin{example}
Let $\calT$ be any $\infty$-category which admits finite products. Then we can regard $\calT$ as a pregeometry as follows:
\begin{itemize}
\item The Grothendieck topology on $\calT$ is discrete: that is, a sieve $\calT^{0}_{/X} \subseteq \calT_{/X}$ is covering if and only if it contains the whole of $\calT_{/X}$.
\item A morphism in $\calT$ is admissible if and only if it is an equivalence in $\calT$.
\end{itemize}
We will refer to a pregeometry as {\em discrete} if it arises via this construction.
\end{example}

\begin{definition}\label{spey}
Let $\calT$ be a pregeometry, and let $\calX$ be an $\infty$-topos.
A {\it $\calT$-structure} on $\calX$ is a functor $\calO: \calT \rightarrow \calX$ with the following
properties:
\begin{itemize}
\item[$(1)$] The functor $\calO$ preserves finite products.
\item[$(2)$] Suppose given a pullback diagram
$$ \xymatrix{ U' \ar[r] \ar[d] & U \ar[d]^{f} \\
X' \ar[r] & X }$$
in $\calT$, where $f$ is admissible. Then the induced diagram
$$ \xymatrix{ \calO(U') \ar[r] \ar[d] & \calO(U) \ar[d] \\
\calO(X') \ar[r] & \calO(X) }$$
is a pullback square in $\calX$.
\item[$(3)$] Let $\{ U_{\alpha} \rightarrow X \}$ be a collection of admissible morphisms
in $\calT$ which generate a covering sieve on $X$. Then the induced map
$$ \coprod_{\alpha} \calO( U_{\alpha} ) \rightarrow \calO(X)$$
is an effective epimorphism in $\calX$.
\end{itemize}

Given a pair of $\calT$-structures $\calO$ and $\calO'$, a morphism of $\calT$-structures
$\alpha: \calO \rightarrow \calO'$ is {\it local} if the following condition is satisfied:
for every admissible morphism $U \rightarrow X$ in $\calT$,
the resulting diagram
$$ \xymatrix{ \calO(U) \ar[d] \ar[r] & \calO'(U) \ar[d] \\
\calO(X) \ar[r] & \calO'(X) }$$
is a pullback square in $\calX$. 

We let $\Struct_{\calT}(\calX)$ denote the full subcategory of $\Fun( \calT, \calX)$ spanned by the
$\calT$-structures on $\calX$, and $\Struct_{\calT}^{\loc}(\calX)$ the subcategory of
$\Struct_{\calT}(\calX)$ spanned by the local morphisms of $\calT$-structures.
\end{definition}

\begin{warning}\label{jar}
Let $\calT$ be a pregeometry. If $\calT$ admits finite limits, then we can also regard $\calT$ as a geometry. However, the notion of $\calT$-structure introduced in Definition \ref{spey} does {\em not} agree with Definition \ref{psyab}. To avoid confusion, we will always use the symbol $\calT$ to denote a pregeometry, so that $\Struct_{\calT}(\calX)$ and $\Struct_{\calT}^{\loc}(\calX)$ are defined via Definition \ref{spey} regardless of whether or not $\calT$ is also a geometry.
\end{warning}

\begin{example}
Let $\calT$ be an $\infty$-category which admits finite products, regarded as a discrete geometry.
Then $\Struct_{\calT}(\SSet)$ can be identified with the $\infty$-category $\calP_{\Sigma}( \calT^{op} )$
defined in \S \toposref{stable11}. 
Let $\calC \subseteq \calP_{\Sigma}( \calT^{op} )$ denote the smallest full subcategory
of $\calP_{\Sigma}( \calT^{op} )$ which contains the image of the Yoneda embedding
$j: \calT^{op} \rightarrow \calP_{\Sigma}( \calT^{op} )$ and is stable under finite colimits. 
Then, for every $\infty$-category $\calX$ which admits small colimits, the functor
$$ \Fun( \calC^{op}, \calX) \rightarrow \Fun( \calT, \calX)$$
given by composition with $j$ induces an equivalence from
$\Fun^{\lex}( \calC^{op}, \calX)$ to the full subcategory of $\Fun( \calT, \calX)$ spanned by those functor which preserve finite products. In the case where $\calX$ is an $\infty$-topos, we obtain equivalences
$$ \Struct_{\calT}( \calX) \leftarrow \Fun^{\lex}(\calC^{op}, \calX) \simeq \Shv_{ \calP_{\Sigma}(\calT^{op}) }( \calX).$$
\end{example}

\begin{definition}
Let $\calT$ be a pregeometry. We will say that a morphism $f: X \rightarrow S$ in
$\calT$ is {\it smooth} if there exists a collection of commutative diagrams
$$ \xymatrix{ U_{\alpha} \ar[r]^{u_{\alpha}} \ar[d]^{v_{\alpha}} & X \ar[d]^{f} \\
S \times V_{\alpha} \ar[r] & S }$$
where $u_{\alpha}$ and $v_{\alpha}$ are admissible, the bottom horizontal map
is given by projection onto the first factor, and the morphisms $u_{\alpha}$ generate
a covering sieve on $X$.
\end{definition}

\begin{proposition}
Let $\calT$ be a pregeometry, and let
$$ \xymatrix{ X' \ar[r] \ar[d]^{f'} & X \ar[d]^{f} \\
S' \ar[r] & S }$$
be a pullback diagram in $\calT$, where $f$ is smooth. Then:
\begin{itemize}
\item[$(1)$] The map $f'$ is smooth.
\item[$(2)$] Let $\calX$ be an arbitrary $\infty$-topos, and $\calO: \calT \rightarrow \calX$
a $\calT$-structure in $\calX$. Then the diagram
$$ \xymatrix{ \calO(X') \ar[r] \ar[d] & \calO(X) \ar[d] \\
\calO(S') \ar[r] & \calO(S) }$$
is a pullback square in $\calX$.
\end{itemize}
\end{proposition}

\begin{proof}
Choose a collection of commutative diagrams
$$ \xymatrix{ U_{\alpha} \ar[r]^{u_{\alpha}} \ar[d]^{v_{\alpha}} & X \ar[d]^{f} \\
S \times V_{\alpha} \ar[r] & S }$$
such that the morphisms $\{ U_{\alpha} \rightarrow X \}_{\alpha \in A}$ generate a covering sieve on $X$. Let
$U'_{\alpha} = U_{\alpha} \times_{X} X'$. Then the collection of commutative
diagrams
$$ \xymatrix{ U'_{\alpha} \ar[r] \ar[d] & X' \ar[d]^{f'} \\
S' \times V_{\alpha} \ar[r] & S' }$$
shows that $f'$ is smooth. This proves $(1)$.

We now prove $(2)$. Let $W_0 = \coprod_{\alpha \in A} \calO( U_{\alpha} )$ and
$W'_0 = \coprod_{ \alpha \in A} \calO( U'_{\alpha} )$. Let $W_{\bigdot}$ denote
the \Cech nerve of the canonical map $\phi: W_0 \rightarrow \calO(X)$, and define
$W'_{\bigdot}$ to be the \Cech nerve of the canonical map $\phi': W'_0 \rightarrow \calO(X')$. 
Since $\phi$ and $\phi'$ are effective epimorphisms, we have a commutative diagram
$$ \xymatrix{ | W'_{\bigdot} | \ar[r] \ar[d] & | W_{\bigdot} | \ar[d] \\
\calO(X') \ar[r] \ar[d] & \calO(X) \ar[d] \\
\calO(S') \ar[r] & \calO(S) }$$
in which the upper vertical arrows are equivalences. It will therefore suffice to show that
the outer square is a pullback. 

For each $\overline{\alpha} = ( \alpha_0, \ldots, \alpha_n) \in A^{n+1}$, let
$U_{ \overline{\alpha} } = U_{\alpha_0} \times_{X} \ldots \times_{X} U_{\alpha_n}$, and define
$U'_{\overline{\alpha} } = U_{\overline{\alpha}} \times_{X} X'$. We have canonical equivalences
$$W_{n} \simeq \coprod_{ \overline{\alpha} \in A^{n+1} } \calO(U_{\overline{\alpha}})$$
$$ W'_{n} \simeq \coprod_{ \overline{\alpha} \in A^{n+1} } \calO( U'_{\overline{\alpha}}).$$
Since colimits in $\calX$ are universal, it will suffice to show that each of the diagrams
$$ \xymatrix{ \calO( U'_{\overline{\alpha}} ) \ar[r] \ar[d] & \calO( U_{ \overline{\alpha} } ) \ar[d] \\
\calO(S') \ar[r] & \calO(S) }$$
is a pullback square. To prove this, we consider the larger diagram
$$ \xymatrix{ \calO( U'_{\overline{\alpha}} ) \ar[r] \ar[d] & \calO( U_{\overline{\alpha} }) \ar[d] \\
\calO( S' \times V_{\alpha_0} ) \ar[r] \ar[d] & \calO( S \times V_{\alpha_0}) \ar[d] \\
\calO(S') \ar[r] & \calO(S). }$$
The upper square is a pullback because the projection $U_{\overline{\alpha}} \rightarrow S \times V_{\alpha_0}$ is admissible, and the lower square is a pullback because the functor $\calO$ preserves finite products.
\end{proof}

We conclude this section by introducing a bit of notation:

\begin{definition}\label{curoo}
Let $\calT$ be a pregeometry. We define a subcategory
$$\LGeo(\calT) \subseteq \Fun( \calT, \overline{\LGeo} ) \times_{ \Fun( \calK, \LGeo) } \LGeo$$
as follows:
\begin{itemize}
\item[$(a)$] Let $f^{\ast} \in \Fun( \calT, \overline{\LGeo} )
\times_{ \Fun( \calT, \LGeo) } \LGeo$ be an object, which we can identify
with a functor $\calO: \calT \rightarrow \calX$, where $\calX$ is an $\infty$-topos. Then
$\calO$ belongs to $\LGeo(\calT)$ if and only if $\calO$ is a $\calT$-structure on $\calX$.

\item[$(b)$] Let $\alpha: \calO \rightarrow \calO'$ be a morphism in $\Fun( \calT, \overline{\LGeo} )
\times_{ \Fun( \calT, \LGeo) } \LGeo$, where $\calO$ and $\calO'$ belong to
$\LGeo(\calT)$, and let $f^{\ast}: \calX \rightarrow \calY$ denote the image
of $\alpha$ in $\LGeo$. Then $\alpha$ belongs to $\LGeo(\calT)$ if and only if
the induced map $f^{\ast} \circ \calO \rightarrow \calO'$ is a morphism
of $\Struct^{\loc}_{\calT}(\calY)$.
\end{itemize}
\end{definition}

\begin{remark}
Let $\calT$ be a pregeometry. The $\infty$-category $\LGeo( \calT)$ is equipped with a canonical
coCartesian fibration $p: \LGeo(\calT) \rightarrow \LGeo$. The fiber over an object $\calX \in \LGeo$
is isomorphic to $\Struct^{\loc}_{\calT}(\calX')$, where $\calX' = \overline{\LGeo} \times_{ \LGeo } \{ \calX \}$
is an $\infty$-topos canonically equivalent to $\calX$. If
$f^{\ast}: \calX \rightarrow \calY$ is a geometric morphism of $\infty$-topoi, then the
coCartesian fibration associates to $f^{\ast}$ the functor
$$ \Struct^{\loc}_{\calT}( \calX') \simeq \Struct^{\loc}_{\calT}(\calX)
\stackrel{ \circ f^{\ast} }{\rightarrow} \Struct^{\loc}_{\calT}( \calY) \simeq \Struct^{\loc}_{\calT}(\calY').$$
\end{remark}

\begin{remark}
Our notation is somewhat abusive, since the notations of Definitions \ref{gcuro}, \ref{tcuro}, and \ref{curoo} overlap. However, there should not be any cause for confusion: as explained in Warning \ref{jar}, the symbol $\calT$ will always denote a pregeometry, so that $\LGeo(\calT)$ is always defined via Definition \ref{curoo}.
\end{remark}

\subsection{Transformations and Morita Equivalence}\label{app2}

The theory of pregeometries is a formalism which allows us to begin with a collection
of ``smooth'' objects (namely, the objects of a pregeometry $\calT$) and to extrapolate an
$\infty$-category consisting of ``singular'' versions of the same objects (namely, the
$\infty$-category $\LGeo(\calT)^{op}$ of $\calT$-structured $\infty$-topoi).
Often there are many different choices for the pregeometry $\calT$ which give rise to the same
theory of $\calT$-structures. Our goal in this section is to give a precise account of this phenomenon. 
We begin by introducing a few definitions.

\begin{definition}\label{carbus}
Let $\calT$ and $\calT'$ be pregeometries. A {\it transformation of pregeometries} from
$\calT$ to $\calT'$ is a functor $F: \calT \rightarrow \calT'$ satisfying the following conditions:
\begin{itemize}
\item[$(i)$] The functor $F$ preserves finite products.
\item[$(ii)$] The functor $F$ carries admissible morphisms in $\calT$ to admissible morphisms in $\calT'$.
\item[$(iii)$] Let $\{ u_{\alpha}: U_{\alpha} \rightarrow X \}$ be a collection of admissible morphisms
in $\calT$ which generates a covering sieve on $X$. Then the morphisms 
$\{ F(u_{\alpha}): FU_{\alpha} \rightarrow FX \}$ generate a covering sieve on $FX \in \calT'$.
\item[$(iv)$] Suppose given a pullback diagram
$$ \xymatrix{ U' \ar[r] \ar[d] & U \ar[d]^{f} \\
X' \ar[r] & X }$$
in $\calT$, where $f$ is admissible. Then the induced diagram
$$ \xymatrix{ FU' \ar[r] \ar[d] & FU \ar[d] \\
FX' \ar[r] & FX }$$
is a pullback square in $\calT'$.
\end{itemize}
\end{definition}

\begin{definition}
Let $F: \calT \rightarrow \calT'$ be a transformation of pregeometries. Then,
for any $\infty$-topos $\calX$, composition with $F$ induces a functor
$$ \Struct^{\loc}_{\calT'}(\calX) \rightarrow \Struct^{\loc}_{\calT}(\calX).$$
We will say that $F$ is a {\it Morita equivalence} of pregeometries if
this functor is an equivalence of $\infty$-categories, for every $\infty$-topos $\calX$.
\end{definition}

\begin{remark}
Let $F: \calT \rightarrow \calT'$ be a Morita equivalence of pregeometries. Then
$F$ induces equivalences
$ \Struct_{\calT'}(\calX) \rightarrow \Struct_{\calT}(\calX)$ for every $\infty$-topos $\calX$;
see Remark \ref{sumbus}.
\end{remark}

\begin{remark}
The notion of Morita equivalence is most naturally formulated in the language of {\it classifying $\infty$-topoi}. We will later see that a transformation $F: \calT \rightarrow \calT'$ is a Morita equivalence if and only if it induces an equivalence $\calK \rightarrow \calK'$ of $\infty$-topoi with geometric structure (Remark \ref{sumbus}). Here $\calK$ is a classifying $\infty$-topos for $\calT$-structures (Definition \ref{tinner}) and $\calK'$ is defined similarly.
\end{remark}

The main results of this section are Propositions \ref{spunkk} and \ref{silver}, which give criteria for establishing that a transformation of pregeometries is a Morita equivalence.

\begin{proposition}\label{spunkk}
Let $F: \calT \rightarrow \calT'$ be a transformation of pregeometries. Suppose that:
\begin{itemize}
\item[$(1)$] The underlying $\infty$-categories of $\calT$ and $\calT'$ coincide, and
$F$ is the identity functor. Moreover, the Grothendieck topologies on $\calT$ and $\calT'$ are the same.
\item[$(2)$] For every $\calT'$-admissible morphism $U \rightarrow X$, 
there exists a collection of $\calT$-admissible morphisms $\{ V_{\alpha} \rightarrow U \}$ which
generate a covering sieve on $U$, such that each composite map $V_{\alpha} \rightarrow X$
is $\calT$-admissible.
\end{itemize}
Then $F$ is a Morita equivalence.
\end{proposition}

\begin{proof}
Let $\calX$ be an $\infty$-topos. We will show that the subcategories $\Struct^{\loc}_{\calT}(\calX), \Struct^{\loc}_{\calT'}( \calX) \subseteq \Fun( \calT, \calX)$ coincide. Our first step is to show that
if $\calO: \calT \rightarrow \calX$ is a $\calT$-structure on $\calX$, then $\calO$ is also a
$\calT'$-structure. In other words, we must show that for every pullback diagram
$$ \xymatrix{ U' \ar[r] \ar[d] & U \ar[d]^{f} \\
X' \ar[r] & X }$$
in $\calT$, if $f$ is $\calT'$-admissible, then the associated diagram
$$ \xymatrix{ \calO(U') \ar[r] \ar[d] & \calO(U) \ar[d] \\
\calO(X') \ar[r] & \calO(X) }$$
is a pullback square in $\calX$. Choose a collection of $\calT$-admissible morphisms
$\{ g_{\alpha}: V_{\alpha} \rightarrow U \}_{\alpha \in A}$ which cover $U$, such that each composition
$f \circ g_{\alpha}: V_{\alpha} \rightarrow X$ is again $\calT$-admissible. For each
$\alpha \in A$, let $V'_{\alpha}$ denote the fiber product $V_{\alpha} \times_{U} U'
\simeq V_{\alpha} \times_{X} X'$.

For every finite sequence $\overline{\alpha} = (\alpha_0, \ldots, \alpha_n)$ of elements of $A$, set
$$ V_{ \overline{\alpha} } = V_{\alpha_0} \times_{U} \ldots \times_{U} V_{\alpha_n },$$
$$ V'_{ \overline{\alpha} } = V'_{\alpha_0} \times_{U'} \ldots \times_{U'} V'_{\alpha_n}.$$
We then have a pullback diagram
$$ \xymatrix{ V'_{\overline{\alpha} } \ar[r] \ar[d] & V_{\overline{\alpha}} \ar[d] \\
X' \ar[r] & X. }$$
Let $Z_0 = \coprod_{ \alpha \in A} \calO(V_{\alpha})$, and define $Z'_0$ similarly. 
Let $Z_{\bigdot}$ denote the \Cech nerve of the canonical map
$\phi: Z_0 \rightarrow \calO( U)$, and $Z'_{\bigdot}$ the \Cech nerve of the canonical
map $\phi': Z'_0 \rightarrow \calO( U')$. Since each map $g_{\alpha}$ is $\calT$-admissible, we can identify each $Z_{n}$ with the coproduct $\coprod_{ \overline{\alpha} \in A^{n+1} } \calO( V_{\overline{\alpha}} )$
and each $Z'_{n}$ with the coproduct $\coprod_{ \overline{\alpha} \in A^{n+1}} \calO( V'_{\overline{\alpha} })$. It follows that each diagram
$$ \xymatrix{ Z'_{n} \ar[r] \ar[d] & Z_{n} \ar[d] \\
\calO(X') \ar[r] & \calO(X) }$$
is a pullback square in $\calX$. Since colimits in $\calX$ are universal, we deduce that the
outer rectangle in the diagram
$$ \xymatrix{ | Z'_{\bigdot} | \ar[r] \ar[d] & | Z_{\bigdot} | \ar[d] \\
\calO(U') \ar[r] \ar[d] & \calO(U) \ar[d] \\
\calO(X') \ar[r] & \calO(X) }$$
is a pullback square. Because
the morphisms $\{ V_{\alpha} \rightarrow U \}$ form a covering of $U$, the maps
$\phi$ and $\phi'$ are effective epimorphisms. It follows that the upper vertical maps in the preceding diagram are equivalences, so that the lower square is a pullback as well. This completes the proof
that $\calO$ is a $\calT'$-structure on $\calX$. 

We now show that every morphism $f: \calO' \rightarrow \calO$ in $\Struct^{\loc}_{\calT}(\calX)$ belongs to $\Struct^{\loc}_{\calT'}(\calX)$. In other words, we must show that if $f: U \rightarrow X$ is 
$\calT$-admissible, then the diagram
$$ \xymatrix{ \calO'(U) \ar[r] \ar[d] & \calO(U) \ar[d] \\
\calO'(X) \ar[r] & \calO(X) }$$
is a pullback square in $X$. Choose a covering $\{ g_{\alpha}: V_{\alpha} \rightarrow U \}_{\alpha \in A}$
and define $Z_{\bigdot}$ as in the first part of the proof, but now set
$Z'_0 = \coprod_{ \alpha \in A} \calO'( V_{\alpha} )$ and let $Z'_{\bigdot}$ be \Cech nerve
of the map $Z'_0 \rightarrow \calO'(U)$. As above, we deduce that each
$$ \xymatrix{ Z'_{n} \ar[r] \ar[d] & Z_{n} \ar[d] \\
\calO'(X) \ar[r] & \calO(X) }$$
is a pullback square. Since colimits in $\calX$ are universal, the outer rectangle in the diagram
$$ \xymatrix{ | Z'_{\bigdot} | \ar[r] \ar[d] & | Z_{\bigdot} | \ar[d] \\
\calO'(U) \ar[r] \ar[d] & \calO(U) \ar[d] \\
\calO'(X) \ar[r] & \calO(X) }$$
is a pullback square, and the upper vertical arrows are equivalences. It follows that the lower square
is a pullback as well, as desired.
\end{proof}

\begin{lemma}\label{turbine}
Let $\calX$ be an $\infty$-topos. Suppose given a finite collection of small diagrams
$\{ f_i: \calJ_i \rightarrow \calX \}_{i \in I}$. Let $\calJ = \prod_{i \in I} \calJ_i$, and for
each $i \in I$ let $g_i$ denote the composition $\calJ \rightarrow \calJ_i \stackrel{f_i}{\rightarrow} \calX$.
Let $f \in \Fun( \calJ, \calX)$ be a product of the functors $\{ g_i \}_{i \in I}$. Then
the canonical map
$$ \colim(f) \rightarrow \prod_{i \in I} \colim( g_i) \rightarrow \prod_{i \in I} \colim(f_i)$$
is an equivalence in $\calX$.
\end{lemma}

\begin{proof}
Choose colimit diagrams $\overline{f}_i: \calJ_i^{\triangleright} \rightarrow \calX$
extending $f_i$, for each $i \in I$. Let $\overline{\calJ} = \prod_{i \in I} \calJ_i^{\triangleright}$, let
$\overline{g}_i: \overline{\calJ} \rightarrow \calX$ be the composition of $\overline{f}_i$ with
the projection $\overline{\calJ} \rightarrow \calJ_{i}^{\triangleright}$, and let
$\overline{f} \in \Fun( \overline{\calJ}, \calX)$ be a product of the functors
$\{ \overline{g}_i \}_{i \in I}$. It will suffice to show that $\overline{f}$ is a left
Kan extension of $f = \overline{f} | \calJ$.

Without loss of generality, we may suppose that $I = \{ 1, \ldots, n \}$ for some nonnegative integer
$n$. For $0 \leq i \leq n$, let $\overline{\calJ}(i)$ denote the product
$$ \calJ_{1}^{\triangleright} \times \ldots \times \calJ_{i}^{\triangleright} \times
\calJ_{i+1} \times \ldots \times \calJ_{n}$$
so that we have a filtration
$$ \calJ = \overline{\calJ}(0) \subseteq \overline{\calJ}(1) \subseteq \ldots
\subseteq \overline{\calJ}(n) = \overline{\calJ}.$$
In view of Proposition \toposref{acekan}, it will suffice to show that
$\overline{f} | \overline{\calJ}(i)$ is a left Kan extension of
$\overline{f} | \overline{\calJ}(i-1)$ for each $0 < i \leq n$. Since $f_i$ is a colimit diagram, 
this follows from the fact that colimits in $\calX$ are stable under products.
\end{proof}

\begin{lemma}\label{spaz}
Let $\calT \rightarrow \Delta^1$ be an (essentially small) correspondence from a pregeometry
$\calT_0 = \calT \times_{ \Delta^1} \{0\}$ to another pregeometry
$\calT_1 = \calT \times_{ \Delta^1} \{1\}$. Assume that $\calT$ satisfies
conditions $(i)$ and $(ii)$ of Proposition \ref{prespaz}, together with the following additional condition:
\begin{itemize}
\item[$(iii)$] The inclusion $\calT_1 \subseteq \calT$ preserves finite products.
\end{itemize}
(This condition is automatically satisfied if $\calT$ is the correspondence associated to a functor
$\calT_0 \rightarrow \calT_1$.) Let $\calX$ be an $\infty$-topos, and let
$F: \Fun( \calT_0, \calX) \rightarrow \Fun( \calT_1, \calX)$ be given by left Kan extension along
the correspondence $\calT$. Then $F$ carries $\calT_0$-structures on $\calX$ to
$\calT_1$-structures on $\calX$.
\end{lemma}

\begin{proof}
Let $\calO_0: \calT_0 \rightarrow \calX$ be a $\calT_0$-structure, and let
$\calO: \calT \rightarrow \calX$ be a left Kan extension of $\calO_0$. We
wish to show that $\calO_1 = \calO | \calT_1$ is a $\calT_1$-structure on $\calX$.
In view of Proposition \ref{prespaz}, it will suffice to show that
$\calO_1$ preserves finite products.

Let $I$ be a finite set, and let $\{ X_i \}_{i \in I}$ be a finite collection of objects of $\calT_1$. We regard $I$ as a discrete simplicial set, and the collection $\{ X_i \}_{i \in I}$ as a diagram $p:I \rightarrow \calT_1$. Extend this
to a limit diagram $\overline{p}: I^{\triangleleft} \rightarrow \calT_1$. Set $X = \overline{p}(v)$, where
$v$ denotes the cone point of $I^{\triangleleft}$, so that $X \simeq \prod_{i \in I} X_i$ in
$\calT_1$ (and also in $\calT$, by virtue of assumption $(iii)$). Let $\calJ$ denote the full subcategory of 
$$ \Fun_{\Delta^1}( I^{\triangleleft} \times \Delta^1, \calT) \times_{
 \Fun( I^{\triangleleft} \times \{1\}, \calT_1) } \{ \overline{p} \}$$
spanned by those functors $F$ such that $F | ( I^{\triangleleft} \times \{0\} )$ is a limit diagram in $\calT_0$. Let $\phi: \calJ \rightarrow \calT_0^{/X}$ be given by evaluation at $v$. The functor $\phi$
admits a left adjoint, and is therefore cofinal (by Theorem \toposref{hollowtt}). Let $f$ denote the composite functor
$$ \calJ \stackrel{\phi}{\rightarrow} \calT_0^{/X} \rightarrow \calT_0 \stackrel{\calO_0}{\rightarrow} \calX.$$
Using Proposition \toposref{lklk}, we deduce that evaluation at the points of $i$ induces a
trivial Kan fibration $\phi_i: \calJ \rightarrow \prod_{i \in I} \calT_0^{/X_i}$. 
For each $i \in I$, let $f_i$ denote the composition
$$ \calT_0^{/X_i} \rightarrow \calT_0 \stackrel{\calO_0}{\rightarrow} \calX.$$
We have a homotopy commutative diagram
$$ \xymatrix{ \colim(f) \ar[r] \ar[d]^{\psi'} & \calO_1(X) \ar[d]^{\psi} \\
\prod_{i \in I} \colim(f_i) \ar[r] & \prod_{i \in I} \calO_1(X_i). }$$
Since $\calO$ is a left Kan extension of $\calO_0$, the horizontal
arrows are equivalences in $\calX$. Consequently, to show that $\psi$ is an equivalence,
it will suffice to show that $\psi'$ is an equivalence, which follows from Lemma \ref{turbine}.
\end{proof}

\begin{proposition}\label{silver}
Let $\calT_0 \subseteq \calT$ be pregeometries satisfying the following conditions:
\begin{itemize}
\item[$(1)$] The $\infty$-category $\calT_0$ is a full subcategory of $\calT$, which is stable under finite products.
\item[$(2)$] A morphism $f: U \rightarrow X$ in $\calT_0$ is admissible in
$\calT_0$ if and only if it is admissible in $\calT$.
\item[$(3)$] A collection of admissible morphisms $\{ f_{\alpha}: U_{\alpha} \rightarrow X \}$ in
$\calT_0$ generate a covering sieve on $X$ in $\calT_0$ if and only if they generate a covering sieve
on $X$ in $\calT$.
\item[$(4)$] Suppose given a pullback diagram
$$ \xymatrix{ U' \ar[r] \ar[d] & U \ar[d]^{f} \\
X' \ar[r] & X }$$
in $\calT$, where $f$ is admissible. If $U$ and $X'$ belong to $\calT_0$, then $U'$ belongs to $\calT_0$.
\item[$(5)$] For every $X \in \calT$, there exists a collection of admissible morphisms
$\{ U_{\alpha} \rightarrow X \}$ which generates a covering sieve on $X$, such that each
$U_{\alpha} \in \calT_0$.
\end{itemize}
Then the inclusion $\calT_0 \subseteq \calT$ is a Morita equivalence.
\end{proposition}

\begin{proof}
Let $\calX$ be an $\infty$-topos. We define a subcategory $\Struct'_{\calT}(\calX) \subseteq \Fun(\calT, \calX)$ as follows:
\begin{itemize}
\item[$(a)$] A functor $\calO: \calT \rightarrow \calX$ belongs to $\Struct'_{\calT}(\calX)$ if and only if 
$\calO | \calT_0 \in \Struct_{\calT_0}(\calX)$, and $\calO$ is a left Kan extension of
$\calO | \calT_0$.
\item[$(b)$] A natural transformation $\alpha: \calO \rightarrow \calO'$ between objects of $\Struct'_{\calT}(\calX)$ belongs to $\Struct'_{\calT}( \calX)$ if and only if the induced transformation
$\calO | \calT_0 \rightarrow \calO' | \calT_0$ belongs to $\Struct^{\loc}_{\calT_0}(\calX)$. 
\end{itemize}
Proposition \toposref{lklk} implies that the restriction functor
$\Struct'_{\calT}(\calX) \rightarrow \Struct^{\loc}_{\calT_0}( \calX)$ is a trivial Kan fibration.
It will therefore suffice to show that $\Struct^{\loc}_{\calT}(\calX) = \Struct'_{\calT}(\calX)$.
We first show that the equality holds at the level of objects. Lemma \ref{spaz}
implies that every object of $\Struct'_{\calT}(\calX)$ belongs to $\Struct^{\loc}_{\calT}(\calX)$.

To prove the reverse inclusion, we begin by studying an arbitrary functor $\calO: \calT \rightarrow \calX$. Fix an
object $X \in \calX$. Using $(5)$, we can choose a covering of
$X$ by admissible morphisms $\{ u_{\alpha}: U_{\alpha} \rightarrow X \}_{\alpha \in A}$,
where each $U_{\alpha} \in \calT_0$. Let $\cDelta_{+}^{A}$ denote the category
whose objects are finite linearly ordered sets $I$ equipped with a map $I \rightarrow A$, and
whose morphisms are commutative diagrams
$$ \xymatrix{ I \ar[rr]^{f} \ar[dr] & & J \ar[dl] \\
& A & }$$
where $f$ is an order-preserving map. Let $\cDelta_{+}^{A, \leq n}$ denote the full subcategory
of $\cDelta_{+}^{A}$ spanned by those objects whose underlying linearly ordered set $I$ has
cardinality $\leq n+1$, and $\cDelta^{A}$ the subcategory spanned by those objects whose underlying linearly ordered set $I$ is nonempty. The admissible morphisms $u_{\alpha}$ determine a 
map $U^{+}_{\leq 0}: \Nerve( \cDelta_{+}^{A, \leq 0} )^{op} \rightarrow \calX$. We let
$u^{+}: \Nerve( \cDelta_{+}^{A} )^{op} \rightarrow \calX$ be a right Kan extension of $u^{+}_{\leq 0}$, so that
$u^{+}$ associates to an ordered sequence $\overline{\alpha} = ( \alpha_0, \ldots, \alpha_n )$ of elements of $A$ the iterated fiber product
$$ u^{+}_{\overline{\alpha}} = U_{\alpha_0} \times_{X} \times \ldots \times_{X} U_{\alpha_n}. $$
Finally, let $u = u^{+} | \Nerve( \cDelta^{A} )^{op}$. Condition $(4)$ guarantees that
$u$ takes values in $\calT_0 \subseteq \calT$.
We will prove the following:
\begin{itemize}
\item[$(\ast)$] Suppose that $\calO_0 = \calO | \calT_0 \in \Struct_{\calT_0}(\calX)$. Then
$\calO$ is a left Kan extension of $\calO_0$ at $X$ if and only if 
$\calO \circ u^{+}$ exhibits $\calO(X)$ as a colimit of the diagram $\calO \circ u$.
\end{itemize}

Let $U_{\bigdot}: \Nerve(\cDelta)^{op} \rightarrow \calX_{/\calO(X)}$ be obtained by left
Kan extension of $\calO \circ U$ along the canonical projection
$\Nerve( \cDelta^{A})^{op} \rightarrow \Nerve( \cDelta)^{op}$, so that we can
identify each $U_{n}$ with the coproduct $\coprod_{ \overline{\alpha} \in A^{n+1} }
\calO( U_{\overline{\alpha} } )$. Then we can identify $\colim (\calO \circ u)$ with the geometric realization of $U_{\bigdot}$. 

Let $\calT'$ denote the full subcategory of $\calT_{/X}$ spanned by those morphisms
$g: V \rightarrow X$ which belong to the sieve generated by the maps
$\{ U_{\alpha} \rightarrow X \}$, and such that $V \in \calT_0$. Since $\calO_0$ is a sheaf
with respect to the topology on $\calT_0$, the restriction $\calO_0 | ( \calT_{/X} \times_{\calT} \calT_0)$
is a left Kan extension of $\calO_0 | \calT'$. In view of Lemma \toposref{kan0}, it will suffice to show
that $U_{\bigdot}$ induces a cofinal map from $\Nerve( \cDelta )^{op} \rightarrow \calT'$.
According to Corollary \toposref{hollowtt}, we must show that for every morphism
$V \rightarrow X$ in $\calT'$, the $\infty$-category
$Y = \Nerve( \cDelta )^{op} \times_{ \calT' } \calT'_{V/}$ is a weakly contractible simplicial set.
We observe that the projection $Y \rightarrow \Nerve( \cDelta)^{op}$ is a left fibration, classified
by a simplicial object $Y_{\bigdot}$ in the $\infty$-category $\SSet$ of spaces. In view of Corollary
\toposref{needka}, it will suffice to show that the geometric realization $| Y_{\bigdot} |$ is weakly contractible. We note that $Y_{\bigdot}$ can be identified with a \Cech nerve of the projection
$Y_0 \rightarrow \ast$. Since $\SSet$ is an $\infty$-topos, we are reduced to showing that $p$ is an effective epimorphism. In other words, we must show that the space $Y_0$ is nonempty; this follows from
our assumption that the map $W \rightarrow X$ belongs to the sieve generated by the morphisms
$\{ U_{\alpha} \rightarrow X\}$. This completes the proof of $(\ast)$.

Suppose now that $\calO \in \Struct_{\calT}(\calX)$. Replacing $\calT_0$ by $\calT$ and
applying the proof of $(\ast)$, we conclude that $\calO \circ U^{+}$ is a colimit diagram.
Invoking $(\ast)$, we deduce that $\calO \in \Struct'_{\calT}( \calX)$. This completes the proof
that $\Struct_{\calT}(\calX)$ and $\Struct'_{\calT}(\calX)$ have the same objects.

It is obvious that every morphism of $\Struct^{\loc}_{\calT}(\calX)$ is also a morphism of $\Struct'_{\calT}(\calX)$. It remains to show that every morphism of $\Struct'_{\calT}( \calX)$ belongs to
$\Struct^{\loc}_{\calT}(\calX)$. Let $\alpha: \calO' \rightarrow \calO$ be a natural transformation in
$\Fun( \calT, \calX)$, where $\calO$ and $\calO'$ are $\calT$-structures on $\calX$.
Suppose further that the induced map
$\calO' | \calT_0 \rightarrow \calO | \calT_0$ is a morphism of $\Struct^{\loc}_{\calT_0}(\calX)$. We
wish to show that $\alpha$ belongs to $\Struct^{\loc}_{\calT}(\calX)$. For this, we must show that if
$U \rightarrow X$ is an admissible morphism in $\calT$, then the diagram $\tau:$
$$ \xymatrix{ \calO'(U) \ar[r] \ar[d] & \calO(U) \ar[d] \\
\calO'(X) \ar[r] & \calO(X) }$$
is a pullback square. 

We begin by treating the special case where $U \in \calT_0$. 
Choose an admissible covering $\{ V_{\alpha} \rightarrow X \}_{\alpha \in A}$, where
each $V_{\alpha}$ belongs to $\calT_0$. Let
$\{ W_{\alpha} \rightarrow U \}_{\alpha \in A}$ be the induced covering, where
$W_{\alpha} \simeq U \times_{X} V_{\alpha}$. Let $v, w: \Nerve( \cDelta^{A})^{op} \rightarrow \calX$
be defined as in the proof of $(\ast)$ using the functor $\calO$, and let
$v',w': \Nerve( \cDelta^{A} )^{op} \rightarrow \calX$ be defined using the functor $\calO'$. 
We then have a commutative diagram of functors
$$ \xymatrix{ w' \ar[r] \ar[d] & w \ar[d] \\
v' \ar[r] & v, }$$
which is a pullback square in virtue of our assumption on $\alpha$. Moreover, the vertical arrows
carry morphisms in $\cDelta^{A}$ to pullback squares in $\calX$. Using Lemma \ref{gooduse}, we deduce that the induced diagram
$$ \xymatrix{ \colim w' \ar[r] \ar[d] & \colim w \ar[d] \\
\colim v' \ar[r] & \colim v }$$
is a pullback square in $\calX$. Combining this with $(\ast)$, we deduce that $\tau$ is a pullback square as desired.

We now treat the general case. Choose an admissible covering $\{ V_{\alpha} \rightarrow U \}_{\alpha \in A}$ by objects of $\calT_0$. Let $v,v': \Nerve( \cDelta^{A} )^{op} \rightarrow \calX$ be defined as in the proof of $(\ast)$, using the functors $\calO$ and $\calO'$ respectively. For each
$\overline{\alpha} \in \cDelta^{A}$, we have a commutative diagram
$$ \xymatrix{ v'( \overline{\alpha}) \ar[r] \ar[d] & v( \overline{\alpha} ) \ar[d] \\
\calO'(X) \ar[r] & \calO(X). }$$
The special case treated above guarantees that this is a pullback square. Since colimits in
$\calX$ are universal, we obtain a pullback square
$$ \xymatrix{ \colim v' \ar[r] \ar[d] & \colim v \ar[d] \\
\calO'(X) \ar[r] & \calO(X). }$$
Assertion $(\ast)$ allows us to identify this square with $\tau$ and complete the proof.

\end{proof}

\subsection{$\infty$-Categories of $\calT$-Structures}\label{app3}

Our goal in this section is to study the $\infty$-category $\Struct_{\calT}(\calX)$, where $\calT$ is a pregeometry and $\calX$ an $\infty$-topos. Proposition \ref{sku} implies the
existence of an equivalence $\Struct_{\calT}(\calX) \simeq \Struct_{\calG}(\calX)$, where
$\calG$ is a geometric envelope of $\calT$ (see \S \ref{app4}); consequently, many of the results of \S \ref{geo4} can be applied to $\Struct_{\calT}(\calX)$. For example, Remark \ref{swame} implies that$\Struct^{\loc}_{\calT}(\calX)$ admits filtered colimits. However, we can prove a stronger result
in the setting of pregeometries: 

\begin{proposition}
Let $\calT$ be a pregeometry, and $\calX$ an $\infty$-topos. Then:
\begin{itemize}
\item[$(i)$] The $\infty$-category $\Struct^{\loc}_{\calT}(\calX)$ admits sifted colimits.
\item[$(ii)$] The inclusion $\Struct^{\loc}_{\calT}(\calX) \subseteq \Fun( \calT, \calX)$ preserves
sifted colimits.
\end{itemize}
\end{proposition}

\begin{proof}
Let $\calC$ be a (small) sifted $\infty$-category, let $F: \calC \rightarrow \Struct^{\loc}_{\calT}(\calX)$
be a diagram, and let $\calO$ be a colimit of $F$ in the $\infty$-category $\Fun( \calT, \calX)$.
We must show:
\begin{itemize}
\item[$(a)$] The functor $\calO$ belongs to $\Struct^{\loc}_{\calT}(\calX)$.
\item[$(b)$] For every object $C \in \calC$, the canonical map $\alpha_{C}: F(C) \rightarrow \calO$
is local.
\item[$(c)$] Given any object $\calO' \in \Struct^{\loc}_{\calT}(\calX)$ and a morphism
$\beta: \calO \rightarrow \calO'$ in $\Fun( \calT, \calX)$, if each composition
$\beta \circ \alpha_{C}$ is local, then $\beta$ is local.
\end{itemize}

Assertions $(b)$ and $(c)$ follow immediately from Lemma \ref{gooduse}.
To prove $(a)$, we must show that $\calO$ satisfies conditions $(1)$, $(2)$ and $(3)$
of Definition \ref{spey}. Since $\calC$ is sifted, the formation of finite products in
$\calX$ is compatible with $\calC$-indexed colimits. It follows immediately that $\calO$ preserves
finite products, which proves $(1)$. 

To prove $(2)$, we must show that for every pullback diagram
$$ \xymatrix{ U' \ar[r] \ar[d] & U \ar[d]^{f} \\
X' \ar[r] & X }$$ 
in $\calT$ such that $f$ is admissible, the associated diagram
$$ \xymatrix{ \calO(U') \ar[r] \ar[d] & \calO(U) \ar[d] \\
\calO(X') \ar[r] & \calO(X) }$$
is a pullback in $\calX$. In view of Lemma \ref{gooduse}, it suffices to show that
each of the diagrams
$$ \xymatrix{ F(C)(U') \ar[r] \ar[d] & \calO(U) \ar[d] \\
F(C)(X') \ar[r] & \calO(X) }$$
is a pullback square. These diagrams can be enlarged to commutative rectangles
$$ \xymatrix{ F(C)(U') \ar[r] \ar[d] & F(C)(U) \ar[d] \ar[r] & \calO(U) \ar[d] \\
F(C)(X') \ar[r] & F(C)(X) \ar[r] & \calO(X) }$$
where the left square is a pullback because $F(C) \in \Struct_{\calT}(\calX)$, and
the right square is a pullback by Lemma \ref{gooduse}.

It remains to show that $\calO$ satisfies condition $(3)$. Let
$X$ be an object of $\calT$ and let $\{ U_{\alpha} \rightarrow X \}$ be a collection
of admissible morphisms which generate a covering sieve on $X$. We wish to show that
the induced map $\coprod_{\alpha} \calO( U_{\alpha} ) \rightarrow \calO(X)$ is an effective
epiomorphism in $\calX$. Since $\calO(X) \simeq \colim_{C} F(C)(X)$, we have
an effective epimorphism $\coprod_{C} F(C)(X) \rightarrow \calO(X)$. It will
therefore suffice to show that, for each $C \in \calC$, the induced map
$$ \coprod_{\alpha} ( \calO(U_{\alpha}) \times_{ \calO(X) } F(C)(X) ) \rightarrow F(C)(X)$$
is an effective epimorphism. Using $(b)$, we can identify the left side with
$\coprod_{\alpha} F(C)(U_{\alpha})$, so that the desired result follows from
the assumption that $F(C) \in \Struct^{\loc}_{\calT}(\calX)$.
\end{proof}

We now study the behavior of $\calT$-structures under truncation.

\begin{definition}\label{bluha}
Let $\calT$ be a pregeometry, and $\calX$ an $\infty$-topos, and $n \geq -1$ an integer.
A $\calT$-structure $\calO$ on $\calX$ is {\it $n$-truncated} if, for
every object $X \in \calT$, the image $\calO(X)$ is an $n$-truncated object of $\calX$.
We let $\Struct_{\calT}^{\leq n}( \calX)$ denote the full subcategory of
$\Struct_{\calT}(\calX)$ spanned by the $n$-truncated $\calT$-structures on $\calX$.

We will say that $\calT$ is {\it compatible with $n$-truncations} if, for every
$\infty$-topos $\calX$, every $\calT$-structure $\calO: \calT \rightarrow \calX$, and
every admissible morphism $U \rightarrow X$ in $\calT$, the induced diagram
$$ \xymatrix{ \calO(U) \ar[r] \ar[d] & \tau_{\leq n} \calO(U) \ar[d] \\
\calO(X) \ar[r] & \tau_{\leq n} \calO(X) }$$
is a pullback square in $\calX$.
\end{definition}

\begin{proposition}\label{rubb}
Let $n \geq -1$ be an integer, $\calT$ a pregeometry which is compatible with $n$-truncations, 
$\calX$ an $\infty$-topos, and $\calO$ an object of $\Struct_{\calT}(\calX)$. Then:
\begin{itemize}
\item[$(1)$] The composition $\tau_{\leq n} \circ \calO$ is a
$\calT$-structure on $\calX$, where $\tau_{\leq n}: \calX \rightarrow \calX$ denotes
the $n$-truncation functor.

\item[$(2)$] The canonical map $\alpha: \calO \rightarrow ( \tau_{\leq n} \circ \calO )$ is local.

\item[$(3)$] For every object $\calO' \in \Struct^{\leq n}_{\calT}(\calX)$, composition with 
$\alpha$ induces a homotopy equivalence
$$ \bHom_{ \Struct^{\loc}_{\calT}(\calX) }( \tau_{\leq n} \circ \calO, \calO' ) \rightarrow
\bHom_{ \Struct^{\loc}_{\calT}(\calX) }( \calO, \calO' ).$$
\item[$(4)$] Composition with $\tau_{\leq n}$ induces a functor
$\Struct^{\loc}_{\calT}(\calX) \rightarrow \Struct_{\calT}^{\leq n}(\calX)
\cap \Struct_{\calT}^{\loc}(\calX)$, which is left adjoint to the inclusion
$\Struct^{\leq n}_{\calT}(\calX)  \cap \Struct_{\calT}^{\loc}(\calX) \subseteq \Struct^{\loc}_{\calT}(\calX)$.
\end{itemize}
\end{proposition}

The proof is based on the following lemma:

\begin{lemma}\label{swimple}
Let $\calX$ be an $\infty$-topos and $n \geq -2$. Suppose given a commutative diagram
$$ \xymatrix{ U \ar[r] \ar[d] & \tau_{\leq n} U \ar[r] \ar[d] & U' \ar[d] \\
X \ar[r] & \tau_{\leq n} X \ar[r] & X', }$$
in $\calX$, where $U'$ and $X'$ are $n$-truncated. If the outer square and the left square
are pullback diagrams, then the right square is a pullback diagram as well.
\end{lemma}

\begin{proof}
We wish to show that the canonical map $\alpha: U \mapsto U' \times_{X'} \tau_{\leq n} X$
exhibits $U' \times_{X'} \tau_{\leq n} X$ as an $n$-truncation of $U$ in the $\infty$-topos
$\calX$. Because $U'$ is $n$-truncated, it will suffice to show that
$\alpha$ exhibits $U' \times_{X'} \tau_{\leq n} X$ as an $n$-truncation of $U$ in
the $\infty$-topos $\calX_{/U'}$ (this follows from Lemma \toposref{trunccomp}). 
This follows from Proposition \toposref{compattrunc}, since the functor
$U' \times_{ X' } \bigdot$ preserves small colimits and finite limits.
\end{proof}

\begin{proof}[Proof of Proposition \ref{rubb}]
Assertion $(2)$ is simply a reformulation of the condition that $\calT$ is compatible with
$n$-truncations, and assertion $(4)$ follows immediately from $(1)$, $(2)$ and $(3)$. It
will therefore suffice to prove $(1)$ and $(3)$. 

To prove $(1)$, we must show that $\tau_{\leq n} \circ \calO$ satisfies the conditions
of Definition \ref{spey}:
\begin{itemize}
\item The functor $\tau_{\leq n} \circ \calO$ preserves finite products. This follows
from the fact that $\calO$ and $\tau_{\leq n}$ preserve finite products (the second
assertion follows from Lemma \toposref{slurpy}).
\item The functor $\tau_{\leq n} \circ \calO$ preserves pullbacks by admissible morphisms.
Suppose given a pullback diagram 
$$\xymatrix{ U' \ar[r] \ar[d] & U \ar[d] \\
X' \ar[r] & X }$$
in $\calT$, where the vertical arrows are admissible. We wish to show that the right square
appearing in the diagram
$$ \xymatrix{ \calO(U') \ar[r] \ar[d] & \tau_{\leq n} \calO(U') \ar[r] \ar[d] & \tau_{\leq n} \calO(U) \ar[d] \\
\calO(X') \ar[r] & \tau_{\leq n} \calO(X') \ar[r] & \tau_{\leq n} \calO(X) }$$
is a pullback square. The left square is a pullback because $\calT$ is compatible with $n$-truncations.
In view of Lemma \ref{swimple}, it will suffice to show that the outer square is a pullback. For this,
we consider the diagram
$$ \xymatrix{ \calO(U') \ar[r] \ar[d] & \calO(U) \ar[d] \ar[r] & \tau_{\leq n} \calO(U) \ar[d] \\
\calO(X') \ar[r] & \calO(X) \ar[r] & \tau_{\leq n} \calO(X). }$$
The left square is a pullback diagram because $\calO \in \Struct_{\calT}(\calX)$, and the right
square is a pullback diagram because $\calT$ is compatible with $n$-truncations.
It follows that the outer square is a pullback diagram, as desired.
\item The functor $\tau_{\leq n} \circ \calO$ carries every covering sieve
$\{ U_{\alpha} \rightarrow X \}$ to an effective epimorphism
$$ \coprod \tau_{\leq n} \calO(U_{\alpha} ) \rightarrow \tau_{\leq n} \calO(X).$$
To prove this, we consider the commutative diagram
$$ \xymatrix{ \coprod \calO(U_{\alpha} ) \ar[r] \ar[d] \ar[dr]^{\alpha} & \calO(X) \ar[d] \\
\coprod \tau_{\leq n} \calO(U_{\alpha}) \ar[r] & \tau_{\leq n} \calO(X). }$$
It will suffice to show that the map $\alpha$ is an effective epiomorphism. 
The upper horizontal morphism is an effective epimorphism since
$\calO \in \Struct_{\calT}(\calX)$, and the right vertical morphism is an effective
epimorphism since $n \geq -1$.
\end{itemize}

To prove $(3)$, we consider a map $\alpha: \calO \rightarrow \calO'$ in $\Struct^{\loc}_{\calT}(\calX)$. 
Let $U \rightarrow X$ be an admissible morphism in $\calT$, and consider the induced diagram
$$ \xymatrix{ \calO(U) \ar[r] \ar[d] & \tau_{\leq n} \calO(U) \ar[r] \ar[d] & \calO'( U) \ar[d] \\
\calO(X) \ar[r] & \tau_{\leq n} \calO(X) \ar[r] & \calO'(X). }$$
We wish to show that the right square is a pullback. Here the outer square is a pullback
since $\alpha$ is a transformation of $\calT$-structures, and the left square is a pullback
since $\calT$ is compatible with $n$-truncations. The desired result now follows from
Lemma \ref{swimple}.
\end{proof}

The following result is often useful in verifying the hypotheses of Proposition \ref{rubb}:

\begin{proposition}\label{rabb}
Let $n \geq -1$ and let $\calT$ be a pregeometry. Suppose that every admissible morphism in
$\calT$ is $(n-1)$-truncated. Then $\calT$ is compatible with $n$-truncations.
\end{proposition}

We first need some preliminary results.

\begin{lemma}\label{atop}
Let $\calT$ be a pregeometry, and let $\calO: \calT \rightarrow \calX$ be a $\calT$-structure on
an $\infty$-topos $\calX$. If $\alpha: U \rightarrow X$ is an $n$-truncated admissible morphism 
in $\calT$, then the induced map $\calO(U) \rightarrow \calO(X)$ is again $n$-truncated.
\end{lemma}

\begin{proof}
We work by induction on $n$. If $n = -2$, then $\alpha$ is an equivalence and the result is obvious.
If $n > -2$, then it will suffice to show that the canonical map
$$ \calO(U) \rightarrow \calO(U) \times_{ \calO(X) } \calO(U) \simeq \calO( U \times_{X} U )$$
is $(n-1)$-truncated (Lemma \toposref{trunc}). In view of the inductive hypothesis, it will suffice to show that
the map $U \rightarrow U \times_{X} U$ is an $(n-1)$-truncated admissible morphism in
$\calX$. The $(n-1)$-truncatedness follows from Lemma \toposref{trunc}, and the
admissibility from Remark \ref{sagewise}.
\end{proof}

\begin{lemma}\label{btop}
Let $\calX$ be an $\infty$-topos, $n \geq -1$ an integer, and let $f: U \rightarrow X$ be an $(n-1)$-truncated morphism in $\calX$. Then the induced diagram
$$ \xymatrix{ U \ar[r] \ar[d] & \tau_{\leq n} U \ar[d] \\
X \ar[r]^{g} & \tau_{\leq n} X}$$
is a pullback square in $\calX$. 
\end{lemma}

\begin{proof}
Lemma \toposref{nicelemma} implies that the pullback functor
$g^{\ast}: \calX_{/ \tau_{\leq n} X} \rightarrow \calX_{ / X}$ induces an equivalence
when restricted to $(n-1)$-truncated objects. Consequently, there exists a pullback diagram
$$ \xymatrix{ U \ar[d] \ar[r]^{g} & V \ar[d]^{f} \\
X \ar[r] & \tau_{\leq n} X,}$$
where the morphism $f$ is $(n-1)$-truncated. To complete the proof, it will suffice to show that
this diagram exhibits $V$ as an $n$-truncation of $U$ in $\calX$. Using Lemma \toposref{trunccomp}, we see that this is equivalent to showing that $g$ exhibits $V$ as an $n$-truncation of $U$ in
$\calX_{/ \tau_{\leq n} X}$. Since $f$ is $n$-truncated, it will suffice to show that $g$ induces an
equivalence
$$ \tau_{\leq n}^{ \calX_{/ \tau_{\leq n} X} } U \rightarrow \tau_{ \leq n}^{ \calX_{/ \tau_{\leq n} X}} V,$$
which follows immediately from Lemma \toposref{slurpy} (applied in the $\infty$-topos
$\calX_{/ \tau_{\leq n} X}$).
\end{proof}

\begin{proof}[Proof of Proposition \ref{rabb}]
Let $\calX$ be an $\infty$-topos, $\calO$ a $\calT$-structure on $\calX$, and
$U \rightarrow X$ an admissible morphism in $\calT$. We wish to show that the diagram
$$ \xymatrix{ \calO(U) \ar[r] \ar[d] & \tau_{\leq n} \calO(U) \ar[d] \\
\calO(X) \ar[r] & \tau_{\leq n} \calO(X) }$$
is a pullback square. This follows immediately from Lemmas \ref{atop} and
\ref{btop}.
\end{proof}

\subsection{Geometric Envelopes}\label{geoenv}\label{app4}

Let $\calT$ be a pregeometry. Our goal in this section is to introduce a geometry
$\calG$ which is ``freely generated by $\calT$'', so that for every $\infty$-topos $\calX$ we have a canonical equivalence of $\infty$-categories $\Struct_{\calG}(\calX) \simeq \Struct_{\calT}(\calX)$.

\begin{definition}\label{skan}
Let $\calT$ be a pregeometry. For any $\infty$-category $\calC$ which admits finite limits, we let
$\Fun^{\adm}(\calT, \calC)$ denote the full subcategory of $\Fun(\calT, \calC)$ spanned by those functors $f: \calT \rightarrow \calC$ with the following properties:
\begin{itemize}
\item[$(a)$] The functor $f$ preserves finite products.
\item[$(b)$] Let $$\xymatrix{ U' \ar[r] \ar[d] & U \ar[d] \\
X' \ar[r] & X }$$ be a pullback diagram in $\calT$ such that the vertical morphisms are admissible. Then
$$ \xymatrix{ fU' \ar[r] \ar[d] & fU \ar[d] \\
fX' \ar[r] & fX }$$
is a pullback diagram in $\calC$. 
\end{itemize}

We will say that a functor
$f: \calT \rightarrow \calG$ {\it exhibits $\calG$ as a geometric envelope of $\calT$} if
the following conditions are satisfied:
\begin{itemize}
\item[$(1)$] The $\infty$-category $\calG$ is idempotent complete and admits finite limits.
\item[$(2)$] The functor $f$ belongs to $\Fun^{\adm}(\calT, \calG)$.
\item[$(3)$] For every idempotent complete $\infty$-category $\calC$ which admits finite limits, composition
with $f$ induces an equivalence of $\infty$-categories
$$ \Fun^{\lex}(\calG, \calC) \rightarrow \Fun^{\adm}(\calT, \calC).$$
Here $\Fun^{\lex}(\calG, \calC)$ denotes the full subcategory of $\Fun(\calG, \calC)$ spanned by those functors which preserve finite limits.
\end{itemize}

In this case, we regard $\calG$ as endowed with the coarsest geometry structure such that
$f$ is a transformation of pregeometries (see Remark \ref{gener}).
\end{definition}

\begin{remark}
In the situation of Definition \ref{skan}, we will abuse terminology by saying that
$\calG$ is a {\it geometric envelope of $\calT$}; in this case, the functor $f: \calT \rightarrow \calG$
is implicitly understood to be specified.
\end{remark}

Let $\calT$ be a pregeometry. The universal property demanded of a geometric envelope $\calG$ of $\calT$ ensures that $\calG$ is determined uniquely up to equivalence, provided that $\calG$ exists. The existence is a consequence of the following result:

\begin{lemma}\label{spuk}
Let $\calT$ be a pregeometry. Then there exists a geometric envelope
$f: \calT \rightarrow \calG$. Moreover, $f$ is fully faithful.
\end{lemma}

\begin{proof}
This is a special case of Proposition \toposref{cupper1}.
\end{proof}

\begin{remark}\label{ludda2}
In the situation of Lemma \ref{spuk}, the $\infty$-category $\calG$ is generated under finite limits and retracts by the essential image of $f$. In other words, if $\calG_0 \subseteq \calG$ is a full subcategory which is stable under retracts and finite limits and which contains the essential image of $f$, then $\calG_0 = \calG$.
\end{remark}

We will now show that the geometric envelope of a pregeometry $\calT$ loses no information about $\calT$-structures on $\infty$-topoi.

\begin{proposition}\label{sku}
Let $\calT$ be a pregeometry and $f: \calT \rightarrow \calG$ a functor which exhibits $\calG$ as a geometric envelope of $\calT$. Then, for every $\infty$-topos $\calX$, composition with
$f$ induces an equivalences of $\infty$-categories
$$\Struct_{\calG}(\calX) \rightarrow \Struct_{\calT}(\calX) \quad \Struct^{\loc}_{\calG}(\calX) \rightarrow \Struct_{\calT}^{\loc}(\calX).$$
\end{proposition}

\begin{proof}
We have a commutative diagram of $\infty$-categories
$$ \xymatrix{ \Struct^{\loc}_{\calG}(\calX) \ar[r] \ar[d] & \Struct_{\calT}^{\loc}(\calX) \ar[d] \\
\Struct_{\calG}(\calX) \ar[r] \ar[d] & \Struct_{\calT}(\calX) \ar[d] \\
\Fun^{\lex}(\calG, \calX) \ar[r]^{F} & \Fun^{\adm}(\calT, \calX) }$$
Here the functor $F$ is an equivalence, and the vertical arrows are inclusions of subcategories.
It will therefore suffice to show that these subcategories correspond to one another under the equivalence $F$. In other words, we must show the following:
\begin{itemize}
\item[$(1)$] Let $\calO: \calG \rightarrow \calX$ be a left exact functor such that
$\calO \circ f$ is a $\calT$-structure on $\calX$. Then $\calO$ is a $\calG$-structure on $\calX$.

\item[$(2)$] Let $\alpha: \calO \rightarrow \calO'$ be a natural transformation between 
$\calG$-structures on $\calX$, and suppose that the induced map
$\calO \circ f \rightarrow \calO' \circ f$ is a morphism of $\Struct^{\loc}_{\calT}(\calX)$. Then
$\alpha$ is a morphism of $\Struct^{\loc}_{\calG}(\calX)$.
\end{itemize}

We first prove $(1)$. Let $\calU$ denote the Grothendieck topology on $\calT$, and
$\calV$ the induced Grothendieck topology on $\calG$. We define a new Grothendieck topology $\calV'$ on $\calG$ as follows: a collection of morphisms $\{ V_{\alpha} \rightarrow Y\}$ is a covering with respect to $\calG$ if and only if the induced map $\coprod_{\alpha} \calO( V_{\alpha}) \rightarrow \calO(Y)$ is an effective epimorphism in $\calX$. We wish to prove that $\calO$ is a $\calG$-structure on $\calX$: in other words, that $\calV'$ is a refinement of $\calV$. Since $\calV$ is generated by 
$\calU$, it will suffice to show that $\calV$ contains all coverings of the form
$\{ f U_{\alpha} \rightarrow f X \}$, where $\{ U_{\alpha} \rightarrow X \}$ is a collection of morphisms
in $\calT$ which is covering with respect to the topology $\calU$. This follows from our assumption that
$\calO \circ f$ is a $\calT$-structure.

The proof of $(2)$ is similar. Let us say that a morphism $U \rightarrow X$ in $\calG$ is {\it special} if
the diagram
$$ \xymatrix{ \calO(U) \ar[r] \ar[d] & \calO'(U) \ar[d] \\
\calO(X) \ar[r] & \calO'(X) }$$
is a pullback square in $\calX$. The following assertions are all easy to verify:
\begin{itemize}
\item[$(a)$] Every equivalence in $\calG$ is special.
\item[$(b)$] Let $$\xymatrix{ U \ar[rr]^{r} \ar[dr]^{q} & & V \ar[dl]^{p} \\
& X & }$$
be a diagram in $\calG$, and suppose that $p$ is special. Then $q$ is special if and only if $r$ is special.
\item[$(c)$] Let $$ \xymatrix{ U' \ar[r] \ar[d]^{u'} & U \ar[d]^{u} \\
X' \ar[r] & X }$$
be a pullback diagram in $\calG$. If $u$ is special, then $u'$ is also special.
\end{itemize}
We wish to show that every admissible morphism in $\calG$ is special. Since the collection of
admissible morphisms in $\calG$ is generated (in the sense of Remark \ref{gener}) by morphisms of the form $f(u)$, where $u$ is an admissible morphism of $\calT$, it will suffice to show that each $f(u)$ is special. This is just a translation of our assumption that the induced map $\calO \circ f \rightarrow \calO' \circ f$ is a morphism of $\Struct^{\loc}_{\calT}(\calX)$.
\end{proof}

\begin{definition}\label{schde}
Let $\calT$ be a pregeometry. A {\it $\calT$-scheme} is a $\calT$-structured $\infty$-topos
$( \calX, \calO_{\calX})$ with the following property:
\begin{itemize}
\item[$(\ast)$] Choose a geometric envelope $f: \calT \rightarrow \calG$ and an equivalence $\calO_{\calX} \simeq \calO'_{\calX} \circ f$, where $\calO'_{\calX} \in \Struct_{\calG}(\calX)$ (the existence
and uniqueness of $\calO'$ up to equivalence follow from Proposition \ref{sku}). Then
$(\calX, \calO'_{\calX})$ is a $\calG$-scheme.
\end{itemize}
We let $\Sch(\calT)$ denote the full subcategory of $\LGeo(\calT)^{op}$ spanned by the
$\calT$-schemes. We will say that a $\calT$-scheme $(\calX, \calO_{\calX})$ is {\it affine} if the corresponding $\calG$-scheme $( \calX, \calO'_{\calX})$ is affine.
\end{definition}

We now introduce a variation on Definition \ref{skan}:

\begin{definition}\label{skaan}
Let $\calT$ be a pregeometry and $n \geq -1$ an integer. We will say that a functor
$f: \calT \rightarrow \calG$ {\it exhibits $\calG$ as an $n$-truncated geometric envelope}
if the following conditions are satisfied:
\begin{itemize}
\item[$(1)$] The $\infty$-category $\calG$ admits finite limits, and is equivalent to an
$n$-category (in other words, each of the mapping spaces $\bHom_{\calG}(X,Y)$ is
$n$-truncated).
\item[$(2)$] The functor $f$ belongs to $\Fun^{\adm}(\calT, \calG)$.
\item[$(3)$] For every $n$-category $\calC$ which admits finite limits, composition with $f$ induces an equivalence of $\infty$-categories
$$ \Fun^{\lex}(\calG, \calC) \rightarrow \Fun^{\adm}(\calT, \calC).$$
Here $\Fun^{\lex}(\calG, \calC)$ denotes the full subcategory of $\Fun(\calG, \calC)$ spanned by those functors which preserve finite limits.
\end{itemize}
In this case, we regard $\calG$ as endowed with the coarsest geometry structure such that
$f$ is a transformation of pregeometries (see Remark \ref{gener}).
\end{definition}

\begin{remark}
By convention, we will refer to a geometric envelope of a pregeometry $\calT$ (in the sense of Definition \ref{skan}) as an {\it $\infty$-truncated geometric envelope}.
\end{remark}

It is clear from the definition that an $n$-truncated geometric envelope of a pregeometry $\calT$ is uniquely determined up to equivalence if it exists. The existence follows immediately from
Lemma \ref{spuk}, Proposition \ref{snubber}, and the following elementary observation:

\begin{lemma}\label{good1}
Let $\calT$ be a pregeometry, let $f: \calT \rightarrow \calG'$ exhibit $\calG'$ as an
geometric envelope of $\calT$, and let $g: \calG' \rightarrow \calG$ exhibit $\calG$ as an
$n$-stub of the geometry $\calG'$. Then $g \circ f$ exhibits $\calG$ as an $n$-truncated
geometric envelope of $\calT$.
\end{lemma}

The role of $n$-truncated geometric envelopes is explicated by the following result:

\begin{proposition}\label{sulw}
Let $\calT$ be a pregeometry, $n \geq -1$ an integer, and let $f: \calT \rightarrow \calG$ exhibit
$\calG$ as an $n$-truncated geometric envelope of $\calT$. Then for every $\infty$-topos $\calX$, composition with $f$ induces a equivalences of $\infty$-categories
$$ \Struct_{\calG}(\calX) \rightarrow \Struct_{\calT}^{\leq n}(\calX)$$
$$ \Struct_{\calG}^{\loc}(\calX) \rightarrow \Struct_{\calT}^{\loc, \leq n}(\calX).$$
\end{proposition}

\begin{proof}
Without loss of generality, we may assume that $f$ factors as a composition
$$ \calT \stackrel{f'}{\rightarrow} \calG' \stackrel{f''}{\rightarrow} \calG,$$
where $f'$ exhibits $\calG'$ as a geometric envelope of $\calT$, and $f''$ exhibits
$\calG$ as an $n$-stub of the geometry $\calG'$. We have a commutative diagram
$$ \xymatrix{ \Struct^{\loc}_{\calG}(\calX) \ar[r]^{\phi_0} \ar[d] & \Struct_{\calG'}^{\loc, \leq n}(\calX)
\ar[r] \ar[d] & \Struct_{\calT}^{\loc, \leq n} \ar[d] \\
\Struct_{\calG}(\calX) \ar[r] & \Struct_{\calG'}^{\leq n}(\calX) \ar[r]^{\phi} & \Struct_{\calT}^{\leq n}(\calX). }$$
Proposition \ref{lubba1} implies that the horizontal arrows on the left are equivalences of $\infty$-categories. It will therefore suffice to show that the functors $\phi$ and $\phi_0$ are
equivalences. We will give the argument for $\phi$; the proof for $\phi_0$ is identical.
The functor $\phi$ fits into a commutative diagram
$$ \xymatrix{ \Struct_{\calG'}^{\leq n}(\calX) \ar[r]^{\phi} \ar[d] & \Struct^{\leq n}_{\calT}(\calX) \ar[d] \\
\Struct_{\calG'}(\calX) \ar[r] & \Struct_{\calT}(\calX) }$$
where the vertical arrows are inclusions of full subcategories, and the bottom horizontal arrow is an equivalence of $\infty$-categories by Proposition \ref{sku}. To complete the proof, it will suffice to show
that if $\calO: \calG' \rightarrow \calX$ is a $\calG'$-structure such that $\calO \circ f'$ is $n$-truncated, then $\calO$ is itself $n$-truncated. Let $\calG'_0$ denote the full subcategory of $\calG'$ spanned by those objects $U$ such that $\calO(U) \subseteq \calX$ is $n$-truncated. Since $\calO$ is left exact, 
the subcategory $\calG'_0 \subseteq \calG'$ is stable under finite limits. Our assumption that
$\calO \circ f'$ is $n$-truncated guarantees that $\calG'_0$ contains the essential image of $f'$.
It follows from Remark \ref{ludda2} that $\calG'_0 = \calG'$, so that $\calO$ is $n$-truncated as desired.
\end{proof}

The theory of pregeometries can also be described in the language of classifying $\infty$-topoi:

\begin{definition}\label{tinner}
Let $\calT$ be a pregeometry and let $-1 \leq n \leq \infty$. We will say that a functor
$\calO: \calT \rightarrow \calK$ is a {\it universal $n$-truncated $\calT$-structure} if the following conditions are satisfied:
\begin{itemize}
\item[$(1)$] The $\infty$-category $\calK$ is an $\infty$-topos, and the functor $\calO$ is an $n$-truncated $\calT$-structure on $\calK$.
\item[$(2)$] For every $\infty$-topos $\calX$, composition with $\calO$ induces an equivalence of $\infty$-categories
$$ \Struct_{\calK}(\calX) \rightarrow \Struct_{\calT}^{\leq n}(\calX).$$
\end{itemize}
In this case, we will say that $\calK$ is a {\it classifying $\infty$-topos for $n$-truncated $\calT$-structures}. In the case $n= \infty$, we will simply say that $\calO$ is 
a {\it universal $\calT$-structure} and that $\calK$ is a {\it classifying $\infty$-topos for $\calT$-structures}. The equivalence of $(2)$ then determines a factorization system on each
$\infty$-category $\Struct_{\calK}(\calX)$, so that $\calK$ is endowed with a geometric structure.
\end{definition}

Let $\calT$ be a pregeometry and let $-1 \leq n \leq \infty$. It is clear that a universal $n$-truncated $\calT$-structure is uniquely determined up to equivalence, if it exists. The existence follows from Proposition \ref{sulw} together with the corresponding result for geometries (Proposition \ref{usebil}).

\begin{remark}\label{sumbus}
Let $f: \calT \rightarrow \calT'$ be a transformation of pregeometries, and choose
universal structures $\calT \rightarrow \calK$, $\calT' \rightarrow \calK'$.
The composition $\calT \rightarrow \calT' \rightarrow \calK'$ is classified by a
geometric morphism of $\infty$-topoi $\phi^{\ast}: \calK \rightarrow \calK'$.
Unwinding the definitions, we deduce that $f$ is a Morita equivalence if and only if
$\phi^{\ast}$ is an equivalence of $\infty$-topoi which identifies the geometric structures
on $\calK$ and $\calK'$. In particular, the assertion that $\phi^{\ast}$ is an equivalence
guarantees that, for any $\infty$-topos $\calX$, the restriction map
$$ \Struct_{\calT'}(\calX) \simeq \Fun^{\ast}( \calK', \calX)
\rightarrow \Fun^{\ast}( \calK, \calX) \rightarrow \Struct_{\calT}(\calX)$$ is
an equivalence of $\infty$-categories.
\end{remark}

We conclude this section by establishing a criterion which guarantees that the collection of
$\calT$-schemes is closed under the formation of truncations.

\begin{proposition}\label{sableware}
Let $\calT$ be a pregeometry having a geometric envelope $\calG$ and let $n \geq -1$.
Assume the following:
\begin{itemize}
\item[$(a)$] The pregeometry $\calT$ is compatible with $n$-truncations.
\item[$(b)$] For every admissible morphism $f: A \rightarrow B$ in $\Ind( \calG^{op})$, if
$A$ is $n$-truncated, then $B$ is also $n$-truncated.
\end{itemize}
Let $\calX$ be an $\infty$-topos and $\calO_{\calX}: \calG \rightarrow \calX$ a $\calG$-structure on $\calX$, and choose a local transformation $\calO_{\calX} \rightarrow \calO'_{\calX}$ in
$\Struct_{\calG}^{\loc}(\calX)$ which induces an equivalence $\tau_{\leq n} \calO_{\calX}(U)
\simeq \calO'_{\calX}(U)$ for each $U$ lying in the essential image of $\calT$ (so that
$\calO'_{\calX}$ is an $n$-truncated $\calG$-structure on $\calX$). 
Then:
\begin{itemize}
\item[$(1)$] Let $f: A \rightarrow \Gamma(\calX; \calO_{\calX} )$ be a morphism in
$\Ind(\calG^{op})$ which determines an equivalence $\Spec^{\calG} A \simeq (\calX, \calO_{\calX})$
in $\LGeo(\calG)$.
Then the induced map $\tau_{\leq n} A \rightarrow \Gamma( \calX; \calO'_{\calX})$
determines an equivalence $\Spec^{\calG}( \tau_{\leq n} A) \simeq (\calX, \calO'_{\calX})$
in $\LGeo(\calG)$.
\item[$(2)$] If $(\calX, \calO_{\calX})$ is an affine $\calG$-scheme, then $(\calX, \calO'_{\calX})$ is an affine $\calG$-scheme.
\item[$(3)$] If $(\calX, \calO_{\calX})$ is a $\calG$-scheme, then $(\calX, \calO'_{\calX})$ is
a $\calG$-scheme.
\end{itemize}
\end{proposition}

\begin{proof}
The implication $(1) \Rightarrow (2)$ is obvious, and the implication $(2) \Rightarrow (3)$ follows by working locally on $\calX$. We will prove $(1)$. According to Corollary \ref{tabletime}, the spectrum $\Spec^{\calG}( \tau_{\leq n} A)$ is $n$-truncated. It will therefore suffice to show that
the canonical map
$$ \bHom_{ \LGeo(\calG) }( (\calX, \calO'_{\calX}), (\calY, \calO_{\calY}))
\rightarrow \bHom_{ \LGeo(\calG) }( \Spec^{\calG} (\tau_{\leq n} A), (\calY, \calO_{\calY}))
\simeq \bHom_{ \Ind(\calG^{op})}( \tau_{\leq n} A, \Gamma( \calY, \calO_{\calY}) )$$
is a homotopy equivalence for every $n$-truncated $\calG$-structure
$\calO_{\calY}: \calG \rightarrow \calY$. Since $\Gamma(\calY, \calO_{\calY})$ is 
$n$-truncated, we can identify the target with 
$$\bHom_{ \Ind(\calG^{op})}( A, \Gamma(\calY, \calO_{\calY}))
\simeq \bHom_{ \LGeo(\calG)}( (\calX, \calO_{\calX} ), (\calY, \calO_{\calY}) ).$$
Consequently, we just need to show that for every geometric morphism 
$\pi^{\ast}: \calX \rightarrow \calY$, the induced map
$$ \bHom_{ \Struct^{\loc}_{\calG}( \calY)}( \pi^{\ast} \calO'_{\calX}, \calO_{\calY} )
\rightarrow \bHom_{ \Struct^{\loc}_{\calG}(\calY)}( \pi^{\ast} \calO_{\calX}, \calO_{\calY} )$$
is a homotopy equivalence. This follows from Propositions \ref{sulw} and \ref{rubb}.
\end{proof}

\subsection{Smooth Affine Schemes}\label{app5}

Let $\calT$ be a pregeometry. We think of the objects of $\calT$ as being smooth geometric objects of some kind (such as algebraic varieties or manifolds), and of $\calT$-structured $\infty$-topoi 
$(\calY, \calO_{\calY}: \calT \rightarrow \calY)$ as being (possibly) singular geometric objects of the same type. These perspectives are connected as follows: if $X$ is an object of $\calT$, then
we think of $\calO_{\calY}(X) \in \calY$ as the sheaf of ``$X$-valued functions on $\calY$''.
To make this idea more precise, we would like to be able to extract from $X$
another $\calT$-structured $\infty$-topos $\Spec^{\calT} X$ such that we have natural homotopy equivalences
$$ \bHom_{\calY}( U, \calO_{\calY}(X) ) \rightarrow \bHom_{ \LGeo(\calT) }( \Spec^{\calT} X,
(\calY_{/U}, \calO_{\calY} | U ) )$$
for each object $U \in \calY$. Our goal in this section is to produce (by explicit construction)
a $\calT$-structured $\infty$-topos $\Spec^{\calT} X$ with this universal property. 

\begin{remark}
Replacing $(\calY, \calO_{\calY})$ by $(\calY_{/U}, \calO_{\calY}|U)$ in the discussion above,
we see that it suffices to verify the universal property of $\Spec^{\calT} X$ in the case 
where $U$ is final in $\calY$.
\end{remark}

\begin{remark}\label{suskine}
Let $\calT$ be a pregeometry, and choose a geometric envelope $f: \calT \rightarrow \calG$.
Let $\Spec^{\calT}$ denote the composition
$$ \calT \stackrel{f}{\rightarrow} \calG \rightarrow \Pro(\calG) \stackrel{ \Spec^{\calG} }
\Sch(\calG) \simeq \Sch(\calT).$$
It follows from Proposition \ref{sku} that, for every object $X \in \calT$, the 
$\calT$-scheme $\Spec^{\calT} X$ has the desired universal property. The reader who is content with this description of $\Spec^{\calT}$ can safely skip the remainder of this section. Our efforts will be directed to providing a more direct construction of $\Spec^{\calT} X$, which does not make reference to the geometric envelope $\calG$: instead, we will mimic the constructions of \S \ref{abspec} using
$\calT$ in place of $\calG$.
\end{remark}

For the remainder of this section, we fix a pregeometry $\calT$ and an object $X \in \calT$.
Let $\calT_{/X}^{\adm}$ denote the full subcategory of $\calT_{/X}$ spanned by the admissible
morphisms $U \rightarrow X$. We regard $\calT_{/X}^{\adm}$ as endowed with the Grothendieck topology induced by the Grothendieck topology on $\calT$. Let $L: \calP( \calT_{/X}^{\adm} ) \rightarrow
\Shv( \calT_{/X}^{\adm} )$ denote a left adjoint to the inclusion of $\Shv( \calT_{/X}^{\adm} )$ into
$\calP( \calT_{/X}^{\adm} )$. We let $\calO_{X}$ denote the composition
$$ \calT \stackrel{j}{\rightarrow} \calP( \calT) \rightarrow \calP( \calT_{/X}^{\adm} )
\stackrel{L}{\rightarrow} \Shv( \calT_{/X}^{\adm} ),$$
where $j$ denotes the Yoneda embedding. 

\begin{remark}
Suppose that the topology on $\calT$ is {\it precanonical}: that is, that every object $Y \in \calT$
represents a sheaf on $\calT$. Then the composite map
$$ \calT \stackrel{j}{\rightarrow} \calP(\calT) \rightarrow \calP( \calT_{/X}^{\adm} )$$
already factors through $\Shv( \calT_{/X}^{\adm} )$. This composition can therefore be identified with
$\calO_{X}$. 
\end{remark}

\begin{proposition}\label{psyk}
Let $\calT$ be a pregeometry containing an object $X$. Then the functor
$\calO_{X}: \calT \rightarrow \Shv( \calT_{/X}^{\adm} )$ is a $\calT$-structure
on the $\infty$-topos $\Shv( \calT_{/X}^{\adm} )$.
\end{proposition}

\begin{proof}
The Yoneda embedding $j: \calT \rightarrow \calP(\calT)$ preserves all limits which
exist in $\calT$ (Proposition \toposref{yonedaprop}), the functor
$\calP( \calT) \rightarrow \calP( \calT_{/X}^{\adm} )$ preserves small limits, and the
localization functor $L: \calP( \calT_{/X}^{\adm}) \rightarrow \Shv( \calT_{/X}^{\adm} )$
is left exact. It follows that $\calO_{X}$ preserves all finite limits which exist in $\calT$.
In particular, $\calO_{X}$ preserves finite products and pullbacks by admissible morphisms.
To complete the proof, it will suffice to show that if $\{ V_{\alpha} \rightarrow Y \}$ is
a collection of admissible morphisms which generate a covering sieve on an object
$Y \in \calT$, then the induced map $\coprod \calO_{X}(V_{\alpha}) \rightarrow \calO_{X}(Y)$
is an effective epimorphism in $\Shv( \calT_{/X}^{\adm})$. Let $U \in \calT_{/X}^{\adm}$, and let
$\eta \in \pi_0 \calO_{X}(Y)(U)$; we wish to show that, locally on $U$, the section
$\eta$ belongs to the image of $\pi_0 \calO_{X}(V_{\alpha})(U)$ for some index $\alpha$. Without loss of generality, we may suppose that $\eta$ arises from a map $U \rightarrow Y$ in $\calT$. Then
the fiber products $U_{\alpha} = V_{\alpha} \times_{Y} U$ form an admissible cover of $U$, and
each $\eta_{\alpha} = \eta | U_{\alpha} \in \pi_0 \calO_{X}(Y)(U_{\alpha})$ lifts to
$\pi_0 \calO_{X}(V_{\alpha})(U_{\alpha})$.
\end{proof}

Let $\calT$ be a pregeometry containing an object $X$. We will denote the
$\calT$-structured $\infty$-topos $( \Shv( \calT^{\adm}_{/X}), \calO_{X} )$ by
$\Spec^{\calT} X$. 

\begin{warning}
We now have two definitions for $\Spec^{\calT} X$: the first given in Remark \ref{suskine} using a geometric envelope for $\calT$, and the second via the direct construction above. We will show eventually that these definitions are (canonically) equivalent to one another; this is a consequence of Proposition \ref{und} below. In the meanwhile, we will use the second of these definitions.
\end{warning}

\begin{definition}
Let $\calT$ be a pregeometry. We will say that a $\calT$-scheme $(\calX, \calO_{\calX})$
is {\it smooth} if there exists a collection of objects $\{ U_{\alpha} \}$ of $\calX$ with the following properties:
\begin{itemize}
\item[$(a)$] The objects $\{ U_{\alpha} \}$ cover $\calX$: that is, the canonical map
$\coprod_{\alpha} U_{\alpha} \rightarrow 1_{\calX}$ is an effective epimorphism, where
$1_{\calX}$ denotes the final object of $\calX$. 
\item[$(b)$] For every index $\alpha$, the $\calT$-structured $\infty$-topos
$(\calX_{/U_{\alpha}}, \calO_{\calX} | U_{\alpha} )$ is equivalent to
$\Spec^{\calT} X_{\alpha}$ for some object $X_{\alpha} \in \calT$.
\end{itemize}
\end{definition}

By construction, the object $\calO_{X}(X)$ has a canonical point when evaluated at
$\id_{X} \in \calT^{\adm}_{/X}$; we will denote this point by $\eta_{X}$. Our goal is to prove that $\Spec^{\calT} X$ is universal among $\calT$-structured $\infty$-topoi equipped with such a point. More precisely, we will prove the following:

\begin{proposition}\label{und}
Let $\calT$ be a pregeometry containing an object $X$. Let
$( \calY, \calO_{\calY} )$ be an arbitrary $\calT$-structured $\infty$-topos, and let
$\Gamma: \calY \rightarrow \SSet$ denote the global sections functor
(that is, the functor co-represented by the final object $1_{\calY}$). Then
evaluation at the point $\eta_{X}$ induces a homotopy equivalence
$$ \theta: \bHom_{ \LGeo(\calT)}( \underline{X}, ( \calY, \calO_{\calY} ) )
\rightarrow \Gamma \calO_{\calY}(X).$$
\end{proposition}

Before giving the proof, we must establish a few preliminaries.

\begin{lemma}\label{tucca}
Let $f: \calC \rightarrow \calD$ be a functor between small $\infty$-categories, and let
$j_{\calC}: \calC \rightarrow \calP(\calC)$ and $j_{\calD}: \calD \rightarrow \calP(\calD)$ denote the 
Yoneda embeddings, and let $G$ denote the composition
$$ \calD \stackrel{j_{\calD}}{\rightarrow} \calP(\calD) \stackrel{\circ f}{\rightarrow} \calP(\calC).$$
Then there exists a canonical natural transformation 
$$ \alpha: j_{\calC} \rightarrow G \circ f $$
which exhibits $G$ as a left Kan extension of $j_{\calC}$ along $f$.
\end{lemma}

\begin{proof}
Let $\overline{\calC} = \Sing | \sCoNerve[\calC] |$ be a fibrant replacement for
the simplicial category $\sCoNerve[\calC]$, and let $\overline{\calD}$ be defined likewise.
By definition, the Yoneda embedding $j_{\calC}$ classifies the composite map
$$u: 
\sCoNerve[\calC \times \calC^{op}] \rightarrow 
\sCoNerve[\calC] \times \sCoNerve[\calC]^{op} 
\rightarrow \overline{\calC} \times \overline{\calC}^{op} \stackrel{h}{\rightarrow} \Kan,$$
where $\Kan$ denotes the simplicial category of Kan complexes and the functor
$h$ is given by the formula $(X,Y) \mapsto \bHom_{ \overline{\calC} }( Y, X)$.
Similarly, the composition $G \circ f$ classifies the composition
$$v: \sCoNerve[\calC \times \calC^{op}] \rightarrow 
\sCoNerve[\calC] \times \sCoNerve[\calC]^{op} 
\rightarrow \overline{\calD} \times \overline{\calD}^{op} \stackrel{h'}{\rightarrow} \Kan,$$
where $h'$ is given by the formula $(X,Y) \mapsto \bHom_{ \overline{\calD}}( Y, X)$. 
The evident maps $\bHom_{\overline{\calC}}(X,Y) \rightarrow \bHom_{\overline{\calD}}(fX, fY)$
determine a natural transformation $u \rightarrow v$ of simplicial functors, which gives rise to a natural transformation $\alpha: j_{\calC} \rightarrow G \circ f$. 

We claim that $\alpha$ exhibits $G$ as a left Kan extension of $j_{\calC}$ along $f$. 
Since colimits in $\calP(\calC) = \Fun( \calC^{op}, \SSet)$ are computed pointwise, it will suffice
to show that for every object $C \in \calC$, the map $\alpha$ exhibits
$e_{fC} \circ j_{\calD} \circ f$ as a left Kan extension of $e_{C} \circ j_{\calC}$ along
$f$, where $e_{C}: \calP(\calC) \rightarrow \SSet$ and $e_{fC}: \calP(\calD) \rightarrow \SSet$
are given by evaluation at $C \in \calC$ and $fC \in \calD$, respectively. 

We have a commutative diagram
$$ \xymatrix{ & \calC \ar[dr]^{f} & \\
\Delta^0 \ar[ur]^{i} \ar[rr] & & \calD, }$$
where $i$ denotes the inclusion $\{ C \} \hookrightarrow \calC$. Since the formation of left 
Kan extensions is transitive, to prove the result for $f$, it will suffice to prove it for
the functors $f$ and $f \circ i$. In other words, we may assume that $\calC \simeq \Delta^0$ consists
of a single vertex $C$. The result now follows from a simple calculation.
\end{proof}

\begin{lemma}\label{blaha}
Let $\calT$ be a pregeometry containing an object $X$. Let
$j$ denote the composition
$$ \calT^{\adm}_{/X} \rightarrow \calP( \calT^{\adm}_{/X}) \stackrel{L}{\rightarrow} \Shv( \calT^{\adm}_{/X}),$$
where $L$ denotes a left adjoint to the inclusion $\Shv( \calT^{\adm}_{/X} ) \subseteq
\calP( \calT^{\adm}_{/X} )$. 
Then there is a canonical natural transformation $\alpha: j \rightarrow \calO_{X} | \calT^{\adm}_{/X}$,
which exhibits $\calO_{X}$ as a left Kan extension of $j$ along the projection
$\calT^{\adm}_{/X} \rightarrow \calT$.
\end{lemma}

\begin{proof}
Because the localization functor $L$ preserves
small colimits (and therefore left Kan extensions), this follows immediately from Lemma \ref{tucca}.
\end{proof}

\begin{proof}[Proof of Proposition \ref{und}]
Let $\calO_0$ denote the composition
$$ \calT^{\adm}_{/X} \rightarrow \calT \stackrel{\calO_{\calY}}{\rightarrow} \calY.$$
Let $\calI_0$ denote the simplicial set
$$ ( \{X\} \times \Delta^1 ) \coprod_{ \{X \} \times \{1\} } ( \calT^{\adm}_{/X} \times \{1\} ),$$
and let $\calI$ denote the essential image of $\calI_0$ in $\calT^{\adm}_{/X} \times \Delta^1.$
Since the inclusion $\calI_0 \subseteq \calI$ is a categorical equivalence, the
induced map
$$ \Fun( \calI, \calY) \rightarrow \Fun( \calI_0, \calY)$$
is a trivial Kan fibration. 

Let $\calC$ denote the
the full subcategory of $\Fun( \calT^{\adm}_{/X} \times \Delta^1, \calY)$ spanned by those functors
$F$ which satisfy the following conditions:
\begin{itemize}
\item[$(i)$] The functor $F$ is a right Kan extension of $F | \calI$. More concretely, 
for every admissible morphism $U \rightarrow X$, the diagram
$$ \xymatrix{ F(U, 0) \ar[r] \ar[d] & F( U, 1) \ar[d] \\
F(X,0) \ar[r] & F(X,1) }$$
is a pullback diagram in $\calY$.
\item[$(ii)$] The object $F(X, 0)$ is final in $\calY$.
\end{itemize}

Using Proposition \toposref{lklk}, we deduce that the forgetful
functors
$$ \calC \rightarrow \Fun^{0}( \calI, \calY) \rightarrow \Fun^{0}( \calI_0, \calY)$$
are trivial Kan fibrations, where $\Fun^{0}( \calI, \calY)$ and $\Fun^{0}( \calI_0, \calY)$
denote the full subcategories of $\Fun( \calI, \calY)$ and $\Fun( \calI_0, \calY)$ spanned
by those functors $F$ which satisfy condition $(ii)$. Form a pullback diagram
$$ \xymatrix{ \calC_0 \ar[r] \ar[d] & Z \ar[d] \ar[r] & \{ \calO_0 \} \ar[d] \\
\calC \ar[r] & \Fun^{0}( \calI_0, \calY) \ar[r] & \Fun( \calT^{\adm}_{/X} \times \{1\}, \calY). }$$
Then $Z$ is a Kan complex, which we can identify with the space
$\Gamma \calO_{\calY}(X)$. The projection map $\calC_0 \rightarrow Z$ is a trivial Kan fibration, so that $\calC_0$ is also a Kan complex which we can identify with $\Gamma \calO_{\calY}(X)$.  

The inclusion $\calT^{\adm}_{/X} \times \{0\} \subseteq \calT^{\adm}_{/X} \times \Delta^1$
induces a functor $\psi: \calC_0 \rightarrow \Fun( \calT^{\adm}_{/X}, \calY )$. In terms of the identification
above, we can view this functor as associating to each global section $1_{\calY} \rightarrow \calO_{\calY}(X)$ the functor
$$U \mapsto \calO_{\calY}(U) \times_{ \calO_{\calY}(X) } 1_{\calY}.$$
It follows that the essential image of $\psi$ belongs to $\Fun^{(0)}( \calT^{\adm}_{/X}, \calY)$, where
$\Fun^{(0)}( \calT^{\adm}_{/X}, \calY) \subseteq \Fun( \calT^{\adm}_{/X}, \calY)$ is the full subcategory
spanned by those functors $F$ which satisfy the following conditions:
\begin{itemize}
\item[$(a)$] The functor $F$ carries $X$ to a final object of $\calY$.
\item[$(b)$] The functor $F$ preserves pullback squares
(since every pullback square in $\calT^{\adm}_{/X}$ gives rise to an admissible pullback
square in $\calT$). 
\item[$(c)$] For every covering $\{ U_{\alpha} \rightarrow V \}$ of an object
$V \in \calT_{/X}$, the induced map $\coprod F(U_{\alpha}) \rightarrow F(V)$ is an effective
epimorphism in $\calY$.
\end{itemize}

The map $\theta$ fits into a homotopy pullback diagram
$$ \xymatrix{ \bHom_{ \LGeo(\calT)}( \underline{X}, (\calY, \calO_{\calY}) ) \ar[r]^-{\theta} \ar[d] & \calC_0 \ar[d]^{\psi} \\
\Fun^{(0)}( \Shv( \calT^{\adm}_{/X}), \calY) \ar[r]^-{\theta'} & \Fun^{(0)}( \calT^{\adm}_{/X}, \calY).}$$
Here $\Fun^{(0)}( \Shv( \calT^{\adm}_{/X} ), \calY)$ denotes the full subcategory of
$\Fun( \Shv(\calT^{\adm}_{/X} ), \calY)$ spanned by those functors which preserve small colimits and finite limits, and $\theta'$ is induced by composition with the map
$$ \calT^{\adm}_{/X} \stackrel{j}{\rightarrow} \calP( \calT^{\adm}_{/X} ) 
\stackrel{L}{\rightarrow} \Shv( \calT^{\adm}_{/X} ),$$
where $j$ is the Yoneda embedding and $L$ is a left adjoint to the inclusion
$\Shv( \calT^{\adm}_{/X} ) \subseteq \calP( \calT^{\adm}_{/X} )$. Using
Propositions \toposref{natash} and \toposref{igrute}, we deduce that $\theta'$ is an equivalence
of $\infty$-categories. Consequently, to show that $\theta$ is a homotopy equivalence, it will suffice to show
that it induces a homotopy equivalence after passing to the fiber over every geometric morphism
$f^{\ast}: \Shv( \calT^{\adm}_{/X} ) \rightarrow \calY$. In other words, we must show that
the canonical map
$$ \bHom_{ \Struct^{\loc}_{\calT}( \calY )}( f^{\ast} \circ \calO_{X}, \calO_{\calY} )
\rightarrow \bHom_{ \Fun'( \calT^{\adm}_{/X}, \calY) }( i \circ f^{\ast} \circ \calO_{X}, i \circ \calO_{\calY})$$
is a homotopy equivalence, where $i: \calT^{\adm}_{/X} \rightarrow \calT$ denotes the projection
and $\Fun'( \calT^{\adm}_{/X}, \calY)$ denotes the subcategory of
$\Fun( \calT^{\adm}_{/X}, \calY)$ spanned by those morphisms
which correspond to functors $\calT^{\adm}_{/X} \times \Delta^1 \rightarrow \calY$ satisfying $(i)$.
This follows immediately from Lemmas \ref{blaha} and Proposition \ref{prespaz}.
\end{proof}

\section{Examples of Pregeometries}\label{app6}

In \S \ref{exzar} and \S \ref{exet}, we described some examples of $0$-truncated geometries $\calG$ and their relationship with classical algebraic geometry. In this section, we will see that each of these geometries arises as the $0$-truncated geometric envelope of a pregeometry $\calT$.
Consequently, we can identify the $\infty$-category $\Sch(\calG)$ with the full subcategory of
$\Sch(\calT)$ spanned by the $0$-truncated $\calT$-schemes. We can then view the
entire $\infty$-category $\Sch(\calT)$ as ``derived version'' of the theory of $\calG$-schemes.

As the simplest instance of the above paradigm, we consider a commutative ring $k$ and
let $\calG(k)$ denote the {\em discrete} geometry consisting of affine $k$-schemes of finite presentation. Then $\Ind( \calG(k)^{op} )$ can be identified with (the nerve of) the category
$\Comm_{k}$ of commutative $k$-algebras. We can identify $\calG(k)$ with the $0$-truncated geometric envelope of the full subcategory $\calT \subseteq \calG(k)$ spanned by the affine spaces over $k$. The (discrete) pregeometry $\calT$ has a geometric envelope $\calG$, so that
$\SCR_{k} = \Ind( \calG^{op} )$ is a presentable $\infty$-category with $\tau_{\leq 0} \SCR_{k} \simeq \Nerve( \Comm_{k} )$. We can therefore view objects of $\SCR_{k}$ as ``generalized'' $k$-algebras.
In fact, we can be quite a bit more precise: objects of $\SCR_{k}$ can be identified with {\em simplicial commutative $k$-algebras}. We will review the theory of simplicial commutative $k$-algebras in
\S \ref{escr}.

The $\infty$-category $\SCR_{k}$ admits various Grothendieck topologies, generalizing several of the natural topologies on the ordinary category $\Comm_{k}$ of commutative $k$-algebras. In \S \ref{derzar} and \S \ref{juet} we will consider the Zariski and \etale topologies. To each of these topologies we can associate a pregeometry $\calT$, which gives rise to the notion of a $\calT$-scheme. In the case of the Zariski topology, we obtain the theory of
derived $k$-schemes outlined in the introduction to this paper.

Using a similar method, we can produce derived versions of other classical geometric theories.
In \S \ref{secondcomp} we will outline a theory of derived complex analytic spaces (we will introduce a derived theory of rigid analytic spaces in \cite{ellipticloop}. Finally, in \S \ref{derdiff} we discuss a derived version of the theory of smooth manifolds, which is essentially equivalent to the theory described \cite{spivak}.

\subsection{Simplicial Commutative Rings}\label{escr}

To pass from classical algebraic geometry to derived algebraic geometry, we must replace the category of commutative rings with some homotopy-theoretic generalization. In this section, we will describe one such generalization: the $\infty$-category $\SCR$ of simplicial commutative rings.

\begin{definition}\label{hutip}
Let $k$ be a commutative ring. We let
$\Poly_{k}$ denote the full subcategory of $\Comm_{k}$ spanned by those commutative $k$-algebras of the form $k[x_1, \ldots, x_n]$, for $n \geq 0$. We let $\SCR_{k}$ denote the $\infty$-category
$\calP_{\Sigma}( \Nerve \Poly_{k} )$ (see \S \toposref{stable11} for an explanation of this notation, or Remark \ref{humba} for a summary).
We will refer to $\SCR_{k}$ as {\it the $\infty$-category of simplicial commutative $k$-algebras}. In the special case where $k$ is the ring $\Z$ of integers, we will denote $\SCR_{k}$ simply by $\SCR$.
\end{definition}

\begin{remark}\label{humba}
Let $k$ be a commutative ring.
In view of Proposition \toposref{smearof}, the $\infty$-category $\SCR_{k}$ can be characterized up to equivalence by the following properties:
\begin{itemize}
\item[$(1)$] The $\infty$-category $\SCR_{k}$ is presentable.
\item[$(2)$] There exists a coproduct-preserving, fully faithful functor
$\phi: \Nerve(\Poly_{k}) \rightarrow \SCR_{k}$.
\item[$(3)$] The essential image of $\phi$ consists of compact, projective objects of $\SCR_{k}$ which generate $\SCR_{k}$ under sifted colimits.
\end{itemize}
\end{remark}

Our choice of terminology is motivated by the following observation, which follows immediately from Corollary \toposref{smokerr}:

\begin{itemize}
\item[$(\ast)$] Let $\bfA$ be the ordinary category of simplicial commutative $k$-algebras, regarded
as a simplicial model category in the usual way (see Proposition \toposref{sutcoat} or \cite{homotopicalalgebra}), and let $\bfA^{\degree}$ denote the full subcategory spanned by the fibrant-cofibrant objects. Then there is a canonical equivalence of $\infty$-categories
$\Nerve( \bfA^{\degree}) \rightarrow \SCR_{k}$.
\end{itemize}

\begin{remark}\label{humple}
Using $(\ast)$, we deduce that the $\infty$-category of discrete objects of $\SCR_{k}$
is canonically equivalent with (the nerve of) the category $\Comm_{k}$ of commutative $k$-algebras. We will generally abuse notation and not distinguish between commutative $k$-algebras and the corresponding discrete objects of $\SCR_{k}$. In particular, we will view the
polynomial algebras $k[x_1, \ldots, x_n]$ as objects of $\SCR_{k}$; these objects constitute a class of compact, projective generators for $\SCR_{k}$.
\end{remark}

\begin{remark}\label{tweeze}
The forgetful functor from $\bfA$ to simplicial sets determines a functor $\theta: \SCR_{k} \rightarrow \SSet$.
This functor is also given more directly by the composition
$$ \SCR_{k} = \calP_{\Sigma}( \Nerve \Poly_{k}) \subseteq
\Fun( (\Nerve \Poly_k)^{op}, \SSet) \rightarrow \SSet,$$
where the final map is given by evaluation on the object $k[x] \in \Poly_{k}$. 
We will often abuse notation by identifying an object $A \in \SCR_{k}$ with its image under the functor
$\theta$. In particular, to every simplicial commutative $k$-algebra $A$ we can associate a collection of homotopy groups $\pi_{\ast} A$ (the additive structure on $A$ implies that the homotopy groups
$\pi_{\ast}(A,x)$ do not depend on the choice of a basepoint $x \in A$; unless otherwise specified, we take the base point to be the additive identity in $A$).

Note that the functor $\theta$ is {\em conservative}: a map $f: A \rightarrow B$ of simplicial
commutative $k$-algebras is an equivalence if and only if $\theta(f)$ is a homotopy equivalence of spaces. This follows from the observation that every object of $\Poly_{k}$ can be obtained as the coproduct of a finite number of copies of $k[x]$.
\end{remark}

\begin{remark}\label{mcguin}
The geometric realization functor from simplicial sets to (compactly generated) topological spaces preserves products, and therefore carries simplicial commutative $k$-algebras to topological commutative $k$-algebras. Combining this observation with $(\ast)$, we can extract from
every object $A \in \SCR_{k}$ an underlying topological space equipped with a commutative $k$-algebra structure, such that the addition and multiplication operations are given by continuous maps.
\end{remark}

\begin{remark}\label{latersense}
Let us analyze the structure of $\pi_{\ast} A$, where $A$ is an object of $\SCR_{k}$. We will think of
$A$ as given by a topological commutative $k$-algebra (see Remark \ref{mcguin}). The addition map $+: A \times A \rightarrow A$ preserves the base point of $A$, and therefore determines a commutative group structure on each homotopy group $\pi_{i} A$. These group structures coincide with the usual group structure for $i > 0$.

We can identify $\pi_{n} A$ with the set of homotopy classes of maps of pairs $([0,1]^{n}, \bd [0,1]^{n}) \rightarrow (A,0)$. Given a pair of elements $x \in \pi_{m} A$,
$y \in \pi_{n} A$, we can use the multiplication in $A$ to extract a map
$$( [0,1]^{m+n}, \bd [0,1]^{m+n}) \rightarrow (A,0),$$
which represents an element $xy \in \pi_{m+n} A$. This multiplication on $\pi_{\ast} A$ is additive in each variable, associative, and commutative in the graded sense: we have
we have $xy = (-1)^{mn} yx \in \pi_{m+n} A$ (the sign results from the fact that the natural map from
the sphere $S^{m+n}$ to itself given by permuting the coordinates has degree $(-1)^{mn}$).
Consequently, $\pi_{\ast} A$ has the structure of a graded commutative ring. In particular,
$\pi_0 A$ has the structure of a commutative ring, and each $\pi_{i} A$ the structure of a module over $\pi_0 A$.

Let us fix a pair of points $a,b \in A$, and analyze the map
$$ \phi: \pi_{n}(A,a) \times \pi_{n}(A,b) \rightarrow \pi_{n}(A,ab)$$
induced by the multiplication in $A$. If $n=0$, this map is simply given by the multiplication
on $\pi_0 A$. If $n > 0$, then this map is necessarily a group homomorphism. We therefore have
$$ \phi(x,y) = \phi(x,0) + \phi(0,y) = \psi(b) x + \psi(a) y,$$
where $\psi: A \rightarrow \pi_0 A$ is the map which collapses every path component of $A$ to a point.
\end{remark}

\begin{remark}\label{bull}
Let $A$ be a simplicial commutative $k$-algebra. The homotopy groups $\pi_{\ast} A$ can be identified with the homotopy groups of the mapping space
$\bHom_{ \SCR_{k}}( k[x], A)$. In particular, we have a canonical bijection
$$ \Hom_{ \h{\SCR_{k}}}( k[x], A) \rightarrow \pi_0 A$$
given by applying the functor $\pi_0$ and evaluating at the element $x$.
More generally, evaluation separately on each variable induces a homotopy equivalence
$$ \bHom_{\SCR_{k}}( k[x_1, \ldots, x_n], A) \simeq \bHom_{ \SCR_{k}}( k[x], A)^{n}$$
and a bijection $\Hom_{ \h{\SCR_{k}}}( k[x_1, \ldots, x_n], A) \simeq (\pi_0 A)^{n}$.
\end{remark}

\begin{remark}\label{snapple}
Let $f: k[x_1, \ldots, x_n] \rightarrow k[y_1, \ldots, y_m]$ be a map of polynomial rings, given by
$$ x_i \mapsto f_{i}(y_1, \ldots, y_m).$$
For any simplicial commutative ring $A$, composition with $f$ induces a map of spaces
$$ \bHom_{ \SCR_{k}}( k[y_1, \ldots, y_m], A) \rightarrow \bHom_{ \SCR_{k}}( k[x_1, \ldots, x_n], A).$$
Passing to homotopy groups at some point $\eta \in \bHom_{\SCR_{k}}( k[y_1, \ldots, y_m], A)$, we get a map
$(\pi_{\ast} A)^{m} \rightarrow (\pi_{\ast} A)^{n}$. For $\ast = 0$, this map is given by
$$(a_1, \ldots, a_m) \mapsto ( f_1( a_1, \ldots, a_m), \ldots, f_n(a_1, \ldots, a_m)).$$
For $\ast > 0$, it is given instead by the action of the Jacobian matrix
$[\frac{\bd f_i}{\bd y_j}]$ (which we regard as a matrix taking values in $\pi_0 A$ using the
morphism $\eta$). This follows from repeated application of Remark \ref{latersense}.
\end{remark}

The following result is an immediate consequence of Proposition \toposref{surottt} and its proof:

\begin{proposition}\label{spakk}
Let $j: \Nerve \Poly_{k} \rightarrow \SCR_{k}$ denote the Yoneda embedding. Let
$\calC$ be an $\infty$-category which admits small sifted colimits, and
$\Fun_{\Sigma}( \SCR_{k}, \calC)$ the full subcategory of $\Fun(\SCR_{k}, \calC)$ spanned by those functors which preserve sifted colimits. Then:
\begin{itemize}
\item[$(1)$] Composition with $j$ induces an equivalence of $\infty$-categories
$$ \Fun_{\Sigma}( \SCR_{k}, \calC) \rightarrow \Fun( \Nerve \Poly_{k}, \calC).$$
\item[$(2)$] A functor $F: \SCR_{k} \rightarrow \calC$ belongs to $\Fun_{\Sigma}( \SCR_{k}, \calC)$ if and only if $F$ is a left Kan extension of $F \circ j$ along $j$.
\item[$(3)$] Suppose $\calC$ admits finite coproducts, and let $F: \SCR_{k} \rightarrow \calC$ preserve sifted colimits. Then $F$ preserves finite coproducts if and only if $F \circ j$ preserves finite coproducts.
\end{itemize}
\end{proposition}

\begin{remark}
All of the results of this section can be generalized without essential change to the case where $k$ is a simplicial commutative ring, not assumed to be discrete. This does not really lead to any additional generality, since the universal base ring $k = \Z$ is already discrete.
\end{remark}

The $\infty$-category $\SCR_{k}$ is closely related to the $\infty$-category $(\EInfty)^{\conn}_{k/}$
of connective $E_{\infty}$-algebras over $k$ (here we identify $k$ with the corresponding discrete $E_{\infty}$-rings). To see this, we recall that full subcategory of $( \EInfty)^{\conn}_{k/}$ spanned by the
discrete objects is equivalent to the (nerve of the) category $\Comm_{k}$ of commutative $k$-algebras, via the functor $A \mapsto \pi_0 A$ (this follows immediately from Proposition \symmetricref{umberhilt}). Choosing a homotopy inverse to this equivalence and restricting to polynomial algebras over $k$, we obtain a functor
$$ \theta_0: \Nerve \Poly_{k} \rightarrow ( \EInfty)^{\conn}_{k/}.$$
Using Proposition \ref{spakk}, we deduce that $\theta_0$ is equivalent to a composition
$$ \Nerve \Poly_{k} \rightarrow \SCR_{k} \stackrel{\theta}{\rightarrow} ( \EInfty)^{\conn}_{k/},$$
where the functor $\theta$ preserves small sifted colimits.

\begin{proposition}\label{skillmon}
Let $k$ be a commutative ring, and let $\theta: \SCR_{k} \rightarrow (\EInfty)^{\conn}_{k/}$ be
the functor defined above. Then:
\begin{itemize}
\item[$(1)$] The functor $\theta$ preserves small limits and colimits.
\item[$(2)$] The functor $\theta$ is conservative.
\item[$(3)$] The functor $\theta$ admits both left and right adjoints.
\item[$(4)$] If $k$ is a $\Q$-algebra, then $\theta$ is an equivalence of $\infty$-categories.
\end{itemize}
\end{proposition}

\begin{proof}
We first prove $(1)$. To prove that $\theta$ preserves small colimits, it will suffice
to show that $\theta_0$ preserves finite coproducts (Proposition \ref{spakk}). 
Since coproducts in $( \EInfty)^{\conn}_{k/}$ are computed by relative tensor products
over $k$, this follows from the fact that every polynomial algebra $k[x_1, \ldots, x_n]$ is flat
as a $k$-module. 

To complete the proof of $(1)$, let us consider the functor
Consider the functor $\phi: ( \EInfty)^{\conn}_{k/} \rightarrow \SSet$
defined by the composition
$$ ( \EInfty)^{\conn}_{k/} \simeq \CAlg( \Mod^{\conn}_{k}) \rightarrow
\Mod^{\conn}_{k} \rightarrow (\Spectra)_{\geq 0} \stackrel{\Omega^{\infty}}{\rightarrow} \SSet.$$
Using Corollary \symmetricref{slimycamp2} and Corollary \monoidref{goop}, we deduce that
$\phi$ is conservative and preserves small limits. It will therefore suffice to show that $\psi= \phi \circ \theta$ is conservative and preserves small limits. Let $\psi': \SCR_{k} \rightarrow \SSet$ be given by evaluation on $k[x] \in \Poly_{k}$.
The functor $\psi'$ obviously preserves small limits, and is conservative by Remark \ref{tweeze}.
To complete the proof of $(1)$ and $(2)$, it will suffice to show
that $\psi$ and $\psi'$ are equivalent. The functor $\psi'$ obviously preserves small sifted colimits.
Combining Proposition \stableref{denkmal}, Corollary \monoidref{gloop}, and Corollary \symmetricref{filtfemme}, we conclude that $\psi: \SCR_{k} \rightarrow \SSet$ preserves small sifted colimits as well.
In view of Proposition \ref{spakk}, it will suffice to show that the composite functors
$\psi \circ j, \psi' \circ j: \Nerve(\Poly_{k}) \rightarrow \SSet$ are equivalent. We now simply observe that both of these compositions can be identified with the functor which associates to each polynomial
ring $k[x_1, \ldots, x_n]$ its underlying set of elements, regarded as a discrete space.

The implication $(1) \Rightarrow (3)$ follows immediately from Corollary \toposref{adjointfunctor}.
Let us prove $(4)$. Suppose that $k$ is a $\Q$-algebra. Then, for every $n \geq 0$, every
flat $k$-module $M$, and every $i > 0$, the homology group $\HH_{i}( \Sigma_n; M^{\otimes n})$
vanishes. It follows that the symmetric power $\Sym^{n}_{k}(M) \in \Mod_{k}^{\conn}$ is
discrete, so that the $E_{\infty}$-algebra $\Sym^{\ast}_{k}(k^{m})$ can be identified with
the (discrete) polynomial ring $k[x_1, \ldots, x_m]$. Using Proposition \symmetricref{baseprops},
we conclude that the essential image of $\theta_0$ consists of compact projective objects of
$(\EInfty)^{\conn}_{k/}$ which generate $(\EInfty)^{\conn}_{k}$ under colimits, so that
$\theta$ is an equivalence by Proposition \toposref{protus}.
\end{proof}

\begin{remark}\label{ugher}
Maintaining the notation of the proof of Proposition \ref{skillmon}, we observe that
the proof gives a natural identification
$$ \bHom_{ \SCR_{k}}( k[x], A) \simeq \phi( \theta(A) ).$$
In particular, the homotopy groups $\pi_{\ast} A$ are canonically isomorphic to the homotopy
groups $\pi_{\ast} \theta(A)$ of the associated $E_{\infty}$-ring $\theta(A)$.
\end{remark}

\begin{remark}
Let $k$ be a commutative ring, and let $\theta: \SCR_{k} \rightarrow ( \EInfty)^{\conn}_{k/}$ be as in Proposition \ref{skillmon}. Then $\theta$ admits a right adjoint $G$. Let $A$ be a connective
$E_{\infty}$-algebra over $k$. The underlying space of the simplicial commutative $k$-algebra
$G(A)$ can be identified with
$$ \bHom_{\SCR_{k}}( k[x], G(A) ) \simeq \bHom_{ (\EInfty)^{\conn}_{k/}}( k[x], A).$$
Note that this is generally {\em different} from the underlying space $\phi(A) \in \SSet$
(with notation as in the proof of Proposition \ref{skillmon}), because the discrete $k$-algebra
$k[x]$ generally does not agree with the symmetric algebra $\Sym^{\ast}_{k}(k) \in (\EInfty)^{\conn}_{k/}$ (though they do coincide whenever $k$ is a $\Q$-algebra). We can think of
$\bHom_{(\EInfty)^{\conn}_{k/}}( k[x], G(A) )$ as the space of ``very commutative'' points $A$, which generally differs from the space $\phi(A) \simeq \bHom_{ (\EInfty)^{\conn}_{k/}}( \Sym^{\ast}_{k}(k), A)$
of all points of $A$. The difference between these spaces can be regarded as a measure of the failure of $k[x]$ to be free as an $E_{\infty}$-algebra over $k$, or as a measure of the failure of
$\Sym^{\ast}_{k}(k)$ to be flat over $k$.

We can interpret the situation as follows. The affine line $\SSpec k[x]$ has the structure of
{\em ring scheme}: in other words, $k[x]$ is a commutative $k$-algebra object in the
category $\Poly_{k}^{op}$. Since the functor $\theta_0: \Nerve \Poly_{k} \rightarrow (\EInfty)^{\conn}_{k/}$ preserves finite coproducts, we can also view $k[x]$ as a commutative $k$-algebra object in
the opposite $\infty$-category of $(\EInfty)^{\conn}_{k/}$. In other words, the functor
$( \EInfty)^{\conn}_{k/} \rightarrow \SSet$ corepresented by $k[x]$ can naturally be lifted to
a functor taking values in a suitable $\infty$-category of ``commutative $k$-algebras in $\SSet$'':
this is the $\infty$-category $\SCR_{k}$ of simplicial commutative $k$-algebras, and the lifting
is provided by the functor $G$.

The composition $T = \theta \circ G$ has the structure of a comonad on $(\EInfty)^{\conn}_{k/}$.
Using the Barr-Beck theorem (Corollary \monoidref{usualbb}), we can identify $\SCR_{k}$ with
the $\infty$-category of $T$-comodules in $(\EInfty)^{\conn}_{k/}$. In other words, the $\infty$-category
$\SCR_{k}$ arises naturally when we attempt to correct the disparity between the
$E_{\infty}$-algebras $k[x]$ and $\Sym^{\ast}_{k}(k)$ (which is also measured by the failure of
$T$ to be the identity functor). In the $\infty$-category $\SCR_{k}$, the polynomial ring
$k[x]$ is both flat over $k$ {\em and} free (Remark \ref{bull}).

The functor $\theta$ also admits a left adjoint $F$, and the composition
$T' = \theta \circ F$ has the structure of a monad on $( \EInfty)^{\conn}_{k/}$. We can use
Corollary \monoidref{usualbb} to deduce that $\SCR_{k}$ is equivalent to the
$\infty$-category of $T'$-modules in $( \EInfty)^{\conn}_{k/}$, but this observation seems to be of a more formal (and less useful) nature.
\end{remark}

\begin{corollary}\label{tubble}
Suppose given a pushout diagram
$$ \xymatrix{ A \ar[r] \ar[d] & B \ar[d] \\
A' \ar[r] & B' }$$
in the $\infty$-category of simplicial commutative rings. Then there exists a convergent spectral
sequence
$$ E^{p,q}_{2} = \Tor_{p}^{\pi_{\ast} A}( \pi_{\ast} B, \pi_{\ast} A')_{q} \Rightarrow
\pi_{p+q} B'.$$
\end{corollary}

\begin{proof}
Applying Proposition \ref{skillmon} in the case $k = \Z$, we obtain a pushout diagram
of connective $E_{\infty}$-algebras. The desired result now follows from Remark
\ref{ugher} and Proposition \monoidref{siiwe} (since pushouts of $E_{\infty}$-rings are computed by relative tensor products).
\end{proof}

\begin{notation}\label{cumby}
Given a diagram of simplicial commutative $k$-algebras
$$ A_0 \leftarrow A \rightarrow A_1,$$
we let $A_0 \otimes_{A} A_1$ denote the pushout of $A_0$ and $A_1$ over $A$
in the $\infty$-category $\SCR_{k}$. In view of the fact that
pushouts of $E_{\infty}$-algebras are computed by relative tensor products, this notation is compatible with the functor $\theta: \SCR_{k} \rightarrow (\EInfty)^{\conn}_{k/}$ of Proposition \ref{skillmon}.
\end{notation}

\begin{warning}
Notation \ref{cumby} introduces some danger of confusion in combination with the convention of identifying commutative $k$-algebras with the associated discrete objects of $\SCR_{k}$.
Namely, if we are given a diagram
$$ A_0 \leftarrow A \rightarrow A_1$$
of {\em discrete} commutative $k$-algebras, then the relative tensor product
$A_0 \otimes_{A} A_1$ in the $\infty$-category $\SCR_{k}$ generally does not coincide with the analogous pushout in the ordinary category $\Comm_{k}$. To avoid confusion, we will sometimes denote this latter pushout by $\Tor^{A}_0( A_0, A_1)$. However, these two notions of tensor product are closely related: we have a canonical isomorphism $\pi_0( A_0 \otimes_{A} A_1) \simeq \Tor^{A}_0(A_0, A_1)$. In fact, Corollary \ref{tubble} implies the existence of canonical isomorphisms
$$ \pi_{n}( A_0 \otimes_{A} A_1) \simeq \Tor^{A}_{n}( A_0, A_1) $$
for each $n \geq 0$. It follows that $A_0 \otimes_{A} A_1$ is discrete (and therefore equivalent to
the ordinary commutative $k$-algebra $\Tor^{A}_0(A_0, A_1)$) if and only if each of the higher $\Tor$-groups $\Tor_{i}^{A}( A_0, A_1)$ vanish; this holds in particular if either $A_0$ or $A_1$ is flat over $A$.
\end{warning}

\begin{remark}\label{distrib}
Let $A$ be a simplicial commutative $k$-algebra, and let $B, C, C' \in (\SCR_{k})_{A/}$. Then
the canonical map $B \otimes_{A} (C \times C') \rightarrow (B \otimes_A C) \times (B \otimes_A C')$
is an equivalence in $\SCR_{k}$. To see this, we invoke Proposition \ref{skillmon} to reduce
to the corresponding assertion for $E_{\infty}$-algebras, which follows from the fact that
the relative tensor product functor $M \mapsto B \otimes_{A} M$ is exact (as a functor
from $A$-modules to $B$-modules).
\end{remark}

\begin{proposition}\label{umber}
Let $f: A \rightarrow B$ be a morphism of simplicial commutative $k$-algebras, and let
$a \in \pi_0 A$ such that $f(a)$ is invertible in $\pi_0 B$. The following conditions are equivalent:
\begin{itemize}
\item[$(1)$] For every object $C \in \SCR_{k}$, composition with
$f$ induces a homotopy equivalence $$\bHom_{\SCR_{k}}(B, C) \rightarrow \bHom^{0}_{\SCR_{k}}( A, C),$$ where $\bHom^{0}_{\SCR_{k}}(A,C)$ denotes the union
of those connected components of $\bHom_{\SCR_{k}}(A,C)$ spanned by those maps
$g: A \rightarrow C$ such that $g(a)$ is invertible in $\pi_0 C$.
\item[$(2)$] For every nonnegative integer $n$, the map
$$ \pi_{n} A \otimes_{ \pi_0 A} (\pi_{0} A)[\frac{1}{a}] \rightarrow \pi_{n} B$$
is an isomorphism of abelian groups.
\end{itemize}
Moreover, given $A$ and $a \in \pi_0 A$, there exists a morphism $f: A \rightarrow B$
satisfying both $(1)$ and $(2)$.
\end{proposition}

\begin{proof}
We first show that $(2) \Rightarrow (1)$. Form a pushout diagram
$$ \xymatrix{ A \ar[r]^{f} \ar[d]^{f} & B \ar[d]^{g'} \\
B \ar[r]^{g} & B'. }$$
Using Corollary \ref{tubble} (and the observation that $\pi_{\ast} B$ is flat over
$\pi_{\ast} A$, by virtue of $(2)$), we conclude that $\pi_{\ast} B' \simeq \pi_{\ast} B$, so that
the maps $g$ and $g'$ are both equivalences. In other words, the map $f$ is a monomorphism
in $\SCR_{k}^{\op}$. Thus, for every simplicial commutative $k$-algebra $C$, we can identify
$\bHom_{\SCR_{k}}(B, C)$ with a union of connected components
$\bHom^{1}_{\SCR_{k}}( A, C) \subseteq \bHom_{\SCR_{k}}(A,C)$. Since
$f(a)$ is invertible in $\pi_{\ast} B$, we must have
$\bHom^{1}_{\SCR_{k}}(A, C) \subseteq \bHom^{0}_{ \SCR_{k}}(A,C)$. To complete the proof
that $(2) \Rightarrow (1)$, it will suffice to verify the reverse inclusion. In other words, we must show that if $h: A \rightarrow C$ is a morphism such that $h(a)$ is invertible in $\pi_{\ast} C$, then
$h$ factors through $f$ (up to homotopy). For this, we form a pushout diagram
$$ \xymatrix{ A \ar[r]^{h} \ar[d]^{f} & C \ar[d]^{f'} \\
B \ar[r] & C'. }$$
It now suffices to show that $f'$ is an equivalence, which follows immediately
from Corollary \ref{tubble} (again using the flatness of $\pi_{\ast} B$ over $\pi_{\ast} A$).

We now prove the final assertion. Fix $A \in \SCR_{k}$ and $a \in \pi_0 A$.
Using Remark \ref{bull}, we can choose a map $k[x] \rightarrow A$
carrying $x$ to $a \in \pi_0 A$. We now form a pushout diagram
$$ \xymatrix{ k[x] \ar[r] \ar[d] & A \ar[d]^{f} \\
k[x,x^{-1}] \ar[r] & B. }$$
It follows immediately from Corollary \ref{tubble} that $f$ satisfies $(2)$. The first
part of the proof shows that $f$ also satisfies $(1)$.

We now prove that $(1) \Rightarrow (2)$. Let $f: A \rightarrow B$ satisfy $(1)$, and let
$f': A \rightarrow B$ satisfy both $(1)$ and $(2)$. Since $f$ and $f'$ both satisfy $(1)$,
we must have $f \simeq f'$ in $( \SCR_{k})_{A/}$, so that $f$ satisfies $(2)$ as well.
\end{proof}

We will say that a morphism $f: A \rightarrow B$ {\it exhibits $B$ as a localization of $A$
with respect to $a \in \pi_0 A$} if the equivalent conditions of Proposition \ref{umber} are satisfied.
It follows from characterization $(1)$ of Proposition \ref{umber} that $B$ is then determined up
to equivalence by $A$ and $a$. We will denote the localization $B$ by $A[ \frac{1}{a}]$. 

\begin{proposition}\label{presweet}
Let $k$ be a commutative ring, $A$ a compact object of $\SCR_{k}$, and
$a \in \pi_0 A$. Then $A[\frac{1}{a}]$ is also a compact object of $\SCR_{k}$.
\end{proposition}

\begin{proof}
Using Proposition \ref{umber}, we deduce the existence of a commutative diagram
$$ \xymatrix{ k[x] \ar[r]^{f} \ar[d] & A \ar[d] \\
k[x,x^{-1}] \ar[r] & A[ \frac{1}{a} ] }$$
such that $f(x) = a$. Using Corollary \ref{tubble} (and the flatness of the vertical maps), we deduce that this diagram is a pushout square in $\SCR_{k}$. Since the collection of compact objects
of $\SCR_{k}$ is stable under finite colimits (Corollary \toposref{tyrmyrr}), it will suffice to show
that $k[x]$ and $k[x,x^{-1}]$. The first of these statements is obvious. For the second, we consider the commutative diagram
$$ \xymatrix{ k[z] \ar[r]^{z \mapsto 1} \ar[d]^{z \mapsto xy} & k \ar[d] \\
k[x,y] \ar[r] & k[x, x^{-1}]. }$$
Because $k[z]$, $k$, and $k[x,y]$ are compact objects of $\SCR_{k}$, it will suffice to show that this diagram is a pushout square in $\SCR_{k}$. Since it is evidently a pushout square in the category of
ordinary commutative rings, it is sufficient (by virtue of the spectral sequence of Corollary \ref{tubble}) to show that the groups $\Tor_{i}^{k[z]}(k, k[x,y])$ vanish for $i > 0$. This follows from the observation
that $xy$ is a not a zero divisor in the commutative ring $k[x,y]$.
\end{proof}

We conclude this section by recording the following observation, which can be deduced immediately from characterization $(1)$ of Proposition \ref{umber} (or characterization $(2)$, together with
Corollary \ref{tubble}):  

\begin{proposition}\label{axe}
Suppose given commutative diagram
$$ \xymatrix{ A \ar[r]^{f} \ar[d]^{g} & B \ar[d] \\
A' \ar[r]^{f'} & B' }$$
of simplicial commutative $k$-algebras. Suppose that $f$ exhibits $B$ as the localization
of $A$ with respect to $a \in \pi_0 A$. Then the following conditions are equivalent:
\begin{itemize}
\item[$(a)$] The map $f'$ exhibits $B'$ as the localization of $A'$ with respect to $g(a) \in \pi_0 A'$.
\item[$(b)$] The above diagram is a pushout square.
\end{itemize}
\end{proposition}

\subsection{Derived Algebraic Geometry (Zariski topology)}\label{derzar}

Let $k$ be a commutative ring. In this section, we will introduce an $\infty$-category
$\Sch_{\Zar}^{\der}(k)$ of {\it derived $k$-schemes}, which includes the (nerve of the) usual category of $k$-schemes as a full subcategory. 

Our first step is to introduce a geometry $\calT_{\Zar}(k)$. 
Let $\Comm_{k}^{\Zar}$ denote the category of all commutative $k$-algebras of the form
$k[x_1, \ldots, x_n, \frac{1}{f(x_1, \ldots, x_n)}]$, where $f$ is a polynomial with coefficients in $k$.
We let $\calT_{\Zar}(k)$ denote the $\infty$-category $\Nerve( \Comm_{k}^{\Zar})^{op}$; we can identify
$\calT_{\Zar}(k)$ with a full subcategory of the (nerve of the) category of affine $k$-schemes: namely, those affine $k$-schemes which appear as the complement of a hypersurface in some affine space $\Aff^{n}_{k}$. If $A$ is a commutative $k$-algebra which belongs to $\Comm^{k}_{\Zar}$, then we let $\SSpec A$ denote the corresponding object of $\calT_{\Zar}(k)$ (this notation is potentially in conflict with that of \S \ref{abspec}, but will hopefully not lead to any confusion).

We regard $\calT_{\Zar}(k)$ as a pregeometry as follows:
\begin{itemize}
\item A morphism $\SSpec A \rightarrow \SSpec B$ is admissible if and only if
it induces an isomorphism of commutative rings $B[ \frac{1}{b}] \simeq A$, for some element $b \in B$.

\item A collection of admissible morphisms $\{ \SSpec B[\frac{1}{b_{\alpha}}] \rightarrow \SSpec B \}$ generates a covering sieve on $\SSpec B \in \Open_{k}$ if and only if the elements $\{ b_{\alpha} \}$ generate the unit ideal of $B$.
\end{itemize}

\begin{remark}\label{defmod}
It is possible to make several variations on the definition of $\calT_{\Zar}(k)$ without changing the resulting theory of $\calT_{\Zar}(k)$-structures. For example, we can define a pregeometry $\calT'_{\Zar}(k)$ as follows:
\begin{itemize}
\item The underlying $\infty$-category of $\calT'_{\Zar}(k)$ agrees with that of $\calT_{\Zar}(k)$.
\item The Grothendieck topology on $\calT'_{\Zar}(k)$ agrees with that of $\calT_{\Zar}(k)$.
\item A morphism $\SSpec A \rightarrow \SSpec B$ is $\calT'_{\Zar}(k)$-admissible if and only if
it is an open immersion (in the sense of classical algebraic geometry over $k$).
\end{itemize}
There is an evident transformation of pregeometries $\calT_{\Zar}(k) \rightarrow \calT'_{\Zar}(k)$. 
Proposition \ref{spunkk} implies that this transformation is a Morita equivalence. 

We can define a larger pregeometry $\calT''_{\Zar}(k)$ as follows:
\begin{itemize}
\item The underlying $\infty$-category of $\calT''_{\Zar}(k)$ is the nerve of the category of all $k$-schemes which appear as open subschemes of some affine space $\Aff^{n}_{k} \simeq \SSpec k[x_1, \ldots, x_n]$. 
\item A morphism in $\calT''_{\Zar}(k)$ is admissible if and only if it is an open immersion of $k$-schemes.
\item A collection of admissible morphisms $\{ j_{\alpha}: U_{\alpha} \rightarrow X \}$ generates a covering sieve
on $X \in \calT''_{\Zar}(k)$ if and only if every point of $X$ lies in the image of some $j_{\alpha}$.
\end{itemize}

We have a fully faithful inclusion $\calT'_{\Zar}(k) \subseteq \calT''_{\Zar}(k)$, which is a transformation of pregeometries. Proposition \ref{silver} guarantees that this inclusion is a Morita equivalence, so that the theories of $\calT'_{\Zar}(k)$-structures and $\calT''_{\Zar}(k)$-structures are equivalent. However, we note that this equivalence does {\em not} restrict to an equivalence between the theory of $\calT'_{\Zar}(k)$-schemes and the theory of $\calT''_{\Zar}(k)$-schemes (the latter class is strictly larger, and does not compare well with the scheme theory of classical algebraic geometry).
\end{remark}

\begin{remark}\label{recor}
The pair of admissible morphisms 
$$ k[x, \frac{1}{1-x}] \stackrel{\alpha}{\leftarrow} k[x] \stackrel{\beta}{\rightarrow} k[x, x^{-1}]$$
in $\Comm_{k}^{\Zar}$ generates a covering sieve on $k[x]$. In fact, this admissible covering
together with the empty covering of the zero ring $0 \in \Comm_{k}^{\Zar}$ 
generate the Grothendieck topology on $\calT_{\Zar}(k)$. To prove this, let $\calT$ be another pregeometry with the same underlying $\infty$-category as $\calT_{\Zar}(k)$ and the same admissible morphisms, such that $\alpha$ and $\beta$ generate a covering sieve on $k[x] \in \calT$, and the empty sieve is a covering of $0 \in \calT$. We will show that $\calT_{\Zar}(k) \rightarrow \calT$ is a transformation of pregeometries.

Let $R \in \Comm_{k}^{\Zar}$, and let $\{ x_{\alpha} \}_{\alpha \in A}$ be a collection of elements of $R$ which generate the unit ideal. We wish to show that the maps $S = \{ R \rightarrow R[ \frac{1}{x_{\alpha}} ] \}$ generate a $\calT$-covering sieve on $R$. Without loss of generality, we may suppose that
$A = \{ 1, \ldots, n\}$ for some nonnegative integer $n$; we work by induction on $n$. Write
$1 = x_1 y_1 + \ldots + x_n y_n$. Replacing each $x_i$ by the product $x_i y_i$, we may suppose
that $1 = x_1 + \ldots + x_n$. If $n=0$, then $R \simeq 0$ and $S$ is a covering sieve by hypothesis. If $n=1$, then $S$ contains an isomorphism and therefore generates a covering sieve. If $n=2$, we have a map $\phi: k[x] \rightarrow R$ given by $x \mapsto x_1$. Then $S$ is obtained from the admissible covering $\{ \alpha, \beta \}$ by base change along $\phi$, and therefore generates a covering sieve. Suppose finally that $n > 2$, and set $y = x_2 + \ldots + x_n$. The inductive hypothesis implies that the maps $R[\frac{1}{y}] \leftarrow R \rightarrow R[ \frac{1}{x_1}]$ generate a $\calT$-covering sieve on $R$. It will therefore suffice to show that the base change of
$S$ to $R[ \frac{1}{y} ]$ and $R[ \frac{1}{x_1}]$ generate covering sieve, which again follows from the inductive hypothesis.

Let $\calX$ be an $\infty$-topos, and let $\calO \in \Fun^{\adm}( \calT_{\Zar}(k), \calX)$.
The above analysis shows that $\calO$ is a $\calT_{\Zar}(k)$-structure on $\calX$ if and only if the following conditions are satisfied:
\begin{itemize}
\item[$(1)$] The object $\calO(0)$ is initial in $\calX$.
\item[$(2)$] The map $\calO( k[x,x^{-1}] ) \coprod \calO( k[x, \frac{1}{1-x}]) \rightarrow \calO( k[x] )$
is an effective epimorphism.
\end{itemize}
\end{remark}

Let $\calG_{\Zar}^{\der}(k)$ denote the {\em opposite} of the full subcategory of $\SCR_{k}$ spanned by the compact objects. As before, if $A \in \SCR_{k}$ is compact then we let $\SSpec A$ denote the corresponding object of $\calG_{\Zar}^{\der}(k)$. We will view $\calG_{\Zar}^{\der}(k)$ as a geometry via an analogous prescription:

\begin{itemize}
\item[$(a)$] A morphism $f: \SSpec A \rightarrow \SSpec B$ in $\calG_{\Zar}^{\der}(k)$ is {\it admissible} if and only if
there exists an element $b \in \pi_0 B$ such that $f$ carries $b$ to an invertible element in
$\pi_0 A$, and the induced map $B[ \frac{1}{b} ] \rightarrow A$ is an equivalence
in $\SCR_{k}$. 

\item[$(b)$] A collection of admissible morphisms $\{\SSpec B[\frac{1}{b_{\alpha}} \rightarrow \SSpec B \}$ 
generates a covering sieve on $B$ if and only if the elements $b_{\alpha}$ generate the unit ideal in the commutative ring $\pi_0 B$.
\end{itemize}

We can identify $\Comm_{k}^{\Zar}$ with the full subcategory of $\SCR_{k}$ spanned by
those objects $A$ such that $\pi_0 A \in \Comm_{k}^{\Zar}$ and $\pi_{i} A$ vanishes for $i > 0$.
Proposition \ref{presweet} implies that every object of $\Comm_{k}^{\Zar}$ is compact in $\SCR_{k}$. We can therefore identify $\calT_{\Zar}(k) \simeq \Nerve( \Comm_{k}^{\Zar})^{op}$ with a full subcategory of
$\calG^{\der}_{\Zar}(k)$. The main result of this section is the following:

\begin{proposition}\label{sturma}
The above identification exhibits $\calG^{\der}_{\Zar}(k)$ as a geometric envelope of $\calT_{\Zar}(k)$.
\end{proposition}

The proof of Proposition \ref{sturma} will require a few preliminaries. Let $\calC$ be
a category which admits finite products. Recall that a {\it monoid object} of $\calC$ is
an object $C \in \calC$ equipped with a unit map
$1_{\calC} \rightarrow C$ (here $1_{\calC}$ denotes a final object of $\calC$) and
a multiplication map $C \times C \rightarrow C$, which satisfy the usual unit and associativity axioms for monoids. Equivalently, a monoid object of $\calC$ is an object $C \in \calC$ together with a
monoid structure on each of the sets $\Hom_{\calC}(D, C)$, depending functorially on $D$.

Let $C \in \calC$ be a monoid object. A {\it unit subobject} of $C$ is a map
$i: C^{\times} \rightarrow C$ with the following property: for every object $D \in \calC$, composition
with $i$ induces a bijection from $\Hom_{\calC}(D,C^{\times})$ to the set of invertible elements of the monoid $\Hom_{\calC}(D, C)$. In particular, $i$ is a monomorphism (so we are justified in describing
$C^{\times}$ as a {\it subobject} of $C$), and $C^{\times}$ is uniquely determined up to canonical isomorphism.

Now suppose that $\calC$ is an $\infty$-category which admits finite products.
A {\it homotopy associative monoid object} of $\calC$ is a monoid object of the homotopy category
$\h{\calC}$. Given a homotopy associative monoid object $C \in \calC$, a {\it unit subobject} 
of $C$ in $\calC$ is a {\em monomorphism} $i: C^{\times} \rightarrow C$ which is a unit subobject of
$C$ in the homotopy category $\h{\calC}$. In other words, a unit subobject of $C$ is a final object
of $\calC^{0}_{/C}$, where $\calC^{0}_{/C}$ denotes the full subcategory of $\calC_{/C}$ spanned by those morphisms $D \rightarrow C$ which correspond to invertible elements of the monoid
$\pi_0 \bHom_{\calC}(D,C)$. From this description, it is clear that $C^{\times}$ and the map
$i: C^{\times} \rightarrow C$ are determined by $C$ up to equivalence, if they exist.

\begin{example}\label{truebad}
The affine line $\SSpec k[x]$ is a homotopy associative monoid object of the $\infty$-category
$\calT_{\Zar}(k)$, with respect to the {\em multiplicative} monoid structure determined by the maps
$$ \SSpec k[x] \leftarrow \SSpec k $$
$$ x \mapsto 1$$
$$ \SSpec k[x] \leftarrow \SSpec k[x_0,x_1] \simeq \SSpec k[x] \times \SSpec k[x]$$
$$ x \mapsto x_0 x_1.$$
Consequently, if $\calC$ is any $\infty$-category which admits finite products and $f: \calT_{\Zar}(k) \rightarrow \calC$ is a functor which preserves finite products, then $f( \SSpec k[x] ) \in \calC$
inherits the structure of a homotopy associative monoid object of $\calC$. (In fact, the object $f( \SSpec k[x] )$ inherits the structure of a monoid object of $\calC$ itself: the associativity of multiplication holds not only up to homotopy, but up to coherent homotopy. However, this is irrelevant for our immediate purposes.)
\end{example}

\begin{lemma}\label{mainpoint}
Let $\calC$ be an $\infty$-category which admits finite limits, and let
$f: \calT_{\Zar}(k) \rightarrow \calC$ be a functor which belongs to
$\Fun^{\adm}( \calT_{\Zar}(k), \calC)$. Then the induced map
$$ \alpha: f( \SSpec k[x,x^{-1}] ) \rightarrow f( \SSpec k[x] )$$ 
is a unit subobject of the homotopy associative monoid object $f( \SSpec k[x] ) \in \calC$
(see Example \ref{truebad}).
\end{lemma}

\begin{proof}
To simplify the notation, we let $X = f( \Spec k[x] ) \in \calC$ and let $X_0 = f( \Spec k[x,x^{-1}] ) \in \calC$. We have a pullback diagram in $\Open_{k}$
$$ \xymatrix{ \SSpec k[x, x^{-1}] \ar[r]^{\id} \ar[d]^{\id} & \SSpec k[x, x^{-1} ] \ar[d] \\
\SSpec k[x,x^{-1}] \ar[r] & \SSpec k[x], }$$
where the vertical arrows are admissible. Since $f$ belongs to $\Fun^{\adm}( \calT_{\Zar}(k), \calC)$, the induced diagram
$$ \xymatrix{ X_0 \ar[r]^{\id} \ar[d]^{\id} & X_0 \ar[d] \\
X_0 \ar[r] & X }$$
is a pullback square in $\calC$. This proves that $\alpha$ is a monomorphism.

We observe that the homotopy associative monoid structure on $\SSpec k[x]$ described in
Example \ref{truebad} determines a homotopy associative monoid structure on the subobject
$\SSpec k[x,x^{-1}]$. Moreover, this homotopy associative monoid structure is actually a homotopy associative {\em group} structure: the inversion map from $\SSpec k[x,x^{-1}]$ is induced by the ring involution which exchanges $x$ with $x^{-1}$. Since $f$ preserves finite products, we conclude
that $X_0$ inherits the structure of a homotopy associative group object of $\calC$, and that
$\alpha$ is compatible with the homotopy associative monoid structure. It follows that if
$p: C \rightarrow X$ is a morphism in $\calC$ which factors through $X_0$ up to homotopy, then
$p$ determines an invertible element of the monoid $\pi_0 \bHom_{\calC}(C, X)$. To complete the proof, we need to establish the converse of this result. Let us therefore assume that 
$p: C \rightarrow X$ is a morphism in $\calC$ which determines an invertible element of
$\pi_0 \bHom_{\calC}(C,X)$; we wish to show that $p$ factors (up to homotopy) through
$\alpha$.

Let $p': C \rightarrow X$ represent a multiplicative inverse to $p$ in 
$\pi_0 \bHom_{\calC}(C,X)$. We wish to show that the product map
$(p,p'): C \times C \rightarrow X \times X$ factors (up to homotopy) through
the monomorphism $\alpha \times \alpha: X_0 \times X_0 \rightarrow X \times X$.
We observe that the multiplication map
$\SSpec k[x] \times \SSpec k[x] \rightarrow \SSpec k[x]$ fits into a pullback diagram
$$ \xymatrix{ \SSpec k[x,x^{-1}] \times \SSpec k[x,x^{-1}] \ar[r] \ar[d] & \SSpec k[x] \times \SSpec k[x] \ar[d] \\
\SSpec k[x,x^{-1}] \ar[r] & \SSpec k[x] }$$
in which the horizontal morphisms are admissible. Since $f \in \Fun^{\adm}( \calT_{\Zar}(k), \calC)$, we conclude that the induced diagram
$$ \xymatrix{ X_0 \times X_0 \ar[r] \ar[d] & X \times X \ar[d] \\
X_0 \ar[r] & X }$$
is a pullback square in $\calC$, where the vertical arrows are given by multiplication.
It will therefore suffice to show that the product map $pp': C \times C \rightarrow X$ factors (up to homotopy) through $X_0$. By definition, this product map is homotopic to the composition
$$ C \times C \rightarrow 1_{\calC} \stackrel{u}{\rightarrow} X,$$
where $u: 1_{\calC} \rightarrow X$ is the unit map. It therefore suffices to show that
$u$ factors through $\alpha$. The desired factorization is an immediate consequence of the commutativity of the following diagram in $\calT_{\Zar}(k)$: 
$$ \xymatrix{ \Spec k \ar[rr]^{x \mapsto 1} \ar[dr]^{x \mapsto 1} & & \Spec k[x,x^{-1}] \ar[dl]^{x \mapsto x} \\
& \Spec k[x]. & }$$
\end{proof}

\begin{proof}[Proof of Proposition \ref{sturma}]
We must prove the following:
\begin{itemize}
\item[$(1)$] For every $\infty$-category idempotent complete $\calC$ which admits finite limits, the restriction map
$$\Fun^{\lex}(\calG_{\Zar}^{\der}(k), \calC) \rightarrow \Fun^{\adm}(\calT_{\Zar}(k), \calC)$$ is an equivalence of $\infty$-categories.
\item[$(2)$] The collections of admissible morphisms and admissible coverings in $\calG_{\Zar}^{\der}(k)$ are generated by admissible morphisms and admissible coverings in the full subcategory $\calT_{\Zar}(k) \subseteq \calG_{\Zar}^{\der}(k)$.
\end{itemize}

We begin with the proof of $(1)$. We have a commutative diagram
$$ \xymatrix{ \Fun^{\lex}(\calG_{\Zar}^{\der}(k), \calC) \ar[rr] \ar[dr]^{u} & & \Fun^{\adm}(\calT_{\Zar}(k), \calC) \ar[dl]^{v} \\
& \Fun^{\pi}( \Nerve(\Poly_{k}), \calC), & }$$
where $\Fun^{\pi}( \Nerve( \Poly_k), \calC)$ denotes the full subcategory of $\Fun( \Nerve(\Poly_{k}), \calC)$ spanned by those functors which preserve finite products. Using Propositions \ref{spakk} and \toposref{lklk}, we deduce that $u$ is an equivalence of $\infty$-categories.
It will therefore suffice to show that $v$ is an equivalence of $\infty$-categories. In view of
Proposition \toposref{lklk}, it will suffice to show the following:
\begin{itemize}
\item[$(a)$] Every functor $f_0: \Nerve( \Poly_{k}) \rightarrow \calC$ which preserves finite products
admits a right Kan extension $f: \calT_{\Zar}(k) \rightarrow \calC$.
\item[$(b)$] A functor $f: \calT_{\Zar}(k) \rightarrow \calC$ belongs to $\Fun^{\adm}( \calT_{\Zar}(k), \calC)$ if and only if $f_0 = f | \Nerve( \Poly_{k})$ preserves products, and $f$ is a right Kan extension of $f_0$.
\end{itemize}
Assertion $(a)$ and the ``if'' direction of $(b)$ follow immediately from Proposition \ref{spakk}.
To prove the reverse direction, let us suppose that $f: \calT_{\Zar}(k) \rightarrow \calC$ belongs to
$\Fun^{\adm}(\calT_{\Zar}(k), \calC)$. Let $f_0 = f | \Nerve( \Poly_{k})$, and let $f': \calT_{\Zar}(k) \rightarrow \calC$
be a right Kan extension of $f_0$ (whose existence is guaranteed by $(a)$). The
identification $f | \Nerve(\Poly_k) = f' | \Nerve(\Poly_k)$ extends to a natural transformation $\alpha: f \rightarrow f'$, which is unique up to homotopy. We wish to show that $\alpha$ is an equivalence.
Fix an arbitrary object $\SSpec R\in \calT_{\Zar}(k)$, where $R = k[x_1, \ldots, x_n, \frac{1}{p(x_1, \ldots, x_n)}]$. We have a pullback diagram
$$ \xymatrix{ \SSpec R \ar[r] \ar[d] & \SSpec k[p,p^{-1}]  \ar[d] \\
\SSpec k[x_1, \ldots, x_n] \ar[r] & \SSpec k[p] }$$
where the vertical arrows are admissible. Since $f$ and $f'$ belong to $\Fun^{\adm}(\calT_{\Zar}(k), \calC)$, 
the map $\alpha_{R}$
will be an equivalence provided that $\alpha_{\SSpec k[x_1, \ldots, x_n]}$, $\alpha_{\SSpec k[p]}$, and $\alpha_{\SSpec k[p,p^{-1}]}$ are equivalences. The first two cases are evident, and the third follows from Lemma \ref{mainpoint}.

Let us now prove $(2)$. Let $\calG$ denote a geometry whose underlying $\infty$-category coincides with $\calG_{\Zar}^{\der}(k)$, such that the inclusion $\calT_{\Zar}(k) \rightarrow \calG$ is a transformation of pregeometries. We must show:
\begin{itemize}
\item[$(2a)$] Every admissible morphism of $\calG_{\Zar}^{\der}(k)$ is an admissible morphism of $\calG$.
\item[$(2b)$] Every collection of $\calG_{\Zar}^{\der}(k)$-admissible morphisms $\{ \SSpec A[ \frac{1}{a_{\alpha}} ] \rightarrow \SSpec A \}$ which generates a $\calG_{\Zar}^{\der}(k)$-covering sieve also generates a $\calG$-covering sieve.
\end{itemize}

We first prove $(2a)$. Let $A$ be a compact object of $\SCR_{k}$, and let $a \in \pi_0 A$. We wish to show that the associated map $u: \SSpec A[ \frac{1}{a} ] \rightarrow \SSpec A$ is $\calG$-admissible.
Let $p: k[x] \rightarrow A$ be a morphism which carries $x \in \pi_0 k[x]$ to $a \in \pi_0 A$.
According to Proposition \ref{axe}, we have a pullback diagram
$$ \xymatrix{ \SSpec A[\frac{1}{a}] \ar[r]^{u} \ar[d] & \SSpec A \ar[d]^{p} \\
\SSpec k[x, x^{-1}] \ar[r]^{u'} & \SSpec k[x] }$$
in $\calG$. It will therefore suffice to show that $u'$ is $\calG$-admissible, which follows from our assumption that $\phi$ is a transformation of pregeometries.

We now prove $(2b)$. Without loss of generality, we may suppose that the covering has the form
$ \{ f_{i}: \SSpec A[ \frac{1}{a_{i}}] \rightarrow \SSpec A \}_{1 \leq i \leq n}$, where the elements $a_i \in \pi_0 A$ generate the unit ideal in $A$. We therefore have an equation of the form
$$ a_1 b_1 + \ldots + a_n b_n = 1$$
in the commutative ring $\pi_0 A$. Let $B = k[ x_1, \ldots, x_n, y_1, \ldots, y_n, \frac{1}{x_1 y_1 + \ldots + x_n y_n } ] \in \SCR_{k}$. Using Remark \ref{bull} and Proposition \ref{umber}, we deduce the existence of a morphism $B \rightarrow A$
in $\SCR_{k}$ carrying each $x_i \in \pi_0 B$ to $a_i \in \pi_0 A$, and each $y_i \in \pi_0 B$ to
$b_i \in \pi_0 A$. Using Proposition \ref{axe}, we deduce that each $f_{i}$ fits into a pullback diagram
$$ \xymatrix{ \SSpec A[ \frac{1}{a_i}] \ar[r]^{f_i} \ar[d] & \SSpec A \ar[d] \\
\SSpec B[ \frac{1}{x_i} ] \ar[r]^{g_i} & \SSpec B. }$$
It therefore suffices to show that the maps $\{ g_i: \SSpec B[ \frac{1}{x_i} ] \rightarrow \SSpec B \}$
determine a $\calG$-covering of $\SSpec B$. Since $\phi$ is a transformation of pregeometries, 
it suffices to show that the maps $g_i$ determine an $\calT_{\Zar}(k)$-covering of $\SSpec B$, which
follows from the observation that the elements $x_i$ generate the unit ideal of $B$.
\end{proof}

\begin{corollary}\label{struba}
For each $n \geq 0$, let $\SCR_{k}^{\leq n}$ denote the full subcategory of
$\SCR_{k}$ spanned by the $n$-truncated objects, and let $\calG^{\der,\leq n}_{\Zar}(k)$ denote the opposite of the full subcategory of $\SCR_{k}^{\leq n}$ spanned by the compact objects. Then
the functor $\phi: \calT_{\Zar}(k) \rightarrow \calG^{\der, \leq n}_{\Zar}(k)$ exhibits $\calG^{\der,\leq n}_{\Zar}(k)$ as an $n$-truncated
geometric envelope of $\calT_{\Zar}(k)$.
\end{corollary}

\begin{proof}
Combine Lemma \ref{good1}, Proposition \ref{sturma}, and the proof of Proposition \ref{snubber}. 
\end{proof}

\begin{remark}\label{stabbb}
Applying Corollary \ref{struba} in the case where $n=0$ and $k$ is the ring $\Z$ of integers, we recover the geometry $\calG_{\Zar}$ of Definition \ref{jin}. We may therefore identify a
$0$-truncated $\calT_{\Zar}(\Z)$-structure on an $\infty$-topos $\calX$ with a local sheaf of commutative rings on $\calX$, as explained in \S \ref{exzar}.
\end{remark}

\begin{definition}
Let $k$ be a commutative ring.
A {\it derived $k$-scheme} is a pair $(\calX, \calO_{\calX})$, where
$\calX$ is an $\infty$-topos, $\calO_{\calX}: \calT_{\Zar}(k) \rightarrow \calX$ is an
$\calT_{\Zar}(k)$-structure on $\calX$, and the pair $(\calX, \calO_{\calX})$ is an
$\calT_{\Zar}(k)$-scheme in the sense of Definition \ref{schde}. 

Fix $n \geq 0$. We will say that a derived scheme
$(\calX, \calO_{\calX})$ is {\it $n$-truncated} if the $\calT_{\Zar}(k)$-structure
$\calO_{\calX}$ is $n$-truncated (Definition \ref{bluha}). We will say that $(\calX, \calO_{\calX})$ is 
{\it $n$-localic} if the $\infty$-topos $\calX$ is $n$-localic (Definition \toposref{stuffera}).
\end{definition}

Combining Remark \ref{stabbb} with Theorem \ref{stumm}, we obtain the following relationship between classical and derived algebraic geometry:

\begin{proposition}\label{tooper}
Let $k$ be a commutative ring, and let $\Sch^{\leq 0}_{\leq 0}( \calT_{\Zar}(k) )$ denote the full subcategory of $\Sch( \calT_{\Zar}(k) )$ spanned by the
$0$-localic, $0$-truncated derived $k$-schemes. Then $\Sch^{\leq 0}_{\leq 0}( \calT_{\Zar}(k) )$
is canonically equivalent to (the nerve of) the category of $k$-schemes.
\end{proposition}

\begin{warning}\label{gumber}
In view of Proposition \ref{tooper}, we can
imagine the theory of derived schemes as generalizing classical scheme theory in two different ways:
\begin{itemize}
\item[$(a)$] The ``underlying space'' of a derived scheme $(\calX, \calO_{\calX})$ is an $\infty$-topos $\calX$, which need not be $0$-localic (and so need not come from any underlying topological space).
\item[$(b)$] The structure sheaf $\calO_{\calX}$ of a derived scheme $(\calX, \calO_{\calX})$ can be identified with a sheaf of simplicial commutative rings on $\calX$, which need not be discrete.
\end{itemize}
Generalization $(b)$ leads to a well-behaved theory, but generalization $(a)$ often leads to pathologies (even in the case of a discrete structure sheaf: see Warning \ref{jui}). For this reason, we will generally never consider derived schemes which are not $0$-localic. When this restriction needs to be lifted (for example, in the study of algebraic stacks), it is better to work with the \etale topology
described in \S \ref{juet}.
\end{warning}

\begin{remark}\label{simpletime}
Every admissible morphism in $\calT_{\Zar}(k)$ is a monomorphism. It follows from Proposition \ref{rabb} that the pregeometry $\calT_{\Zar}(k)$ is compatible with $n$-truncations, for
each $0 \leq n \leq \infty$. In particular, to every
$\calT_{\Zar}(k)$-structure $\calO_{\calX}$ on an $\infty$-topos $\calX$ we can associate
an $n$-truncated $\calT_{\Zar}(k)$-structure $\tau_{\leq n} \calO_{\calX}$; when $n=0$, we can view this as a local sheaf of commutative rings on $\calX$ (Remark \ref{stabbb}). Using Proposition \ref{sableware}, we conclude that each $(\calX, \tau_{\leq n} \calO_{\calX})$ is a derived $k$-scheme.
In particular, if $\calX$ is $0$-localic, then we can identify $(\calX, \pi_0 \calO_{\calX})$ with
an ordinary $k$-scheme (Theorem \ref{stumm}); we will refer to this as the {\it underlying ordinary
scheme of $(\calX, \calO_{\calX})$}.
\end{remark}

\begin{remark}\label{dunse}
Let $\calX$ be an $\infty$-topos. Proposition \ref{sturma} implies the existence of a fully faithful embedding $\theta: \Struct_{\calT_{\Zar}(k)}(\calX) \rightarrow \Fun^{\lex}( \calG_{\Zar}^{\der}(k), \calX) \simeq \Shv_{\SCR_{k}}(\calX)$. We will often invoke this equivalence implicitly, and identify a $\calT_{\Zar}(k)$-structure on $\calX$ with the associated sheaf of simplicial commutative $k$-algebras on $\calX$. We will say that a sheaf of simplicial commutative $k$-algebras $\calO$ on $\calX$ is
{\it local} if belongs to the essential image of this embedding.

We will also abuse notation by identifying
a $\calT_{\Zar}(k)$-structure $\calO$ on $\calX$ with the underlying object $\calO( k[x] ) \in \calX$.
Note that both these identifications are compatible with the truncation functors $\tau_{\leq n}$, for each $n \geq 0$ (see Remark \ref{downta}).
\end{remark}

\begin{remark}
The condition that a sheaf $\calO$ of sheaf of commutative $k$-algebras be local depends only on
the underlying sheaf $\pi_0 \calO$ of ordinary commutative $k$-algebras. More precisely,
let $\calX$ be an $\infty$-topos and $\calO$ an object of $\Fun^{\adm}(\calT_{\Zar}(k), \calX)$. Then:
\begin{itemize}
\item[$(1)$] For each $n \geq 0$, the composition
$$ \tau_{\leq n} \calO: \calT_{\Zar}(k) \stackrel{\calO}{\rightarrow} \calX 
\stackrel{ \tau_{\leq n}}{\rightarrow} \calX$$
belongs to $\Fun^{\adm}( \calT_{\Zar}(k), \calX)$.
\item[$(2)$] The sheaf $\calO$ belongs to $\Struct_{ \calT_{\Zar}(k)}(\calX)$ if and only if
$\tau_{\leq n} \calO$ belongs to $\Struct_{ \calT_{\Zar}(k)}(\calX)$.
\end{itemize}
Assertion $(1)$ follows immediately from Remark \ref{downta}, and $(2)$ follows from
Proposition \toposref{pi00detects}.
\end{remark}

\begin{remark}\label{jikker}
Let $A$ be a simplicial commutative $k$-algebra, and consider the affine derived scheme
$(\calX, \calO_{\calX}) = \Spec^{\calG_{\Zar}^{\der}(k)} A$. The underlying $\infty$-topos
$\calX$ can be identified with $\Shv(X)$, where $X$ is the Zariski spectrum of the ordinary commutative ring $\pi_0 A$. This follows from Proposition \ref{sableware}, but can also be deduced from the explicit construction provided by Theorem \ref{scoo}, since the $\infty$-category
of admissible $A$-algebras in $\SCR_{k}$ is equivalent to the $\infty$-category of
admissible $\pi_0 A$-algebras in $\SCR_{k}$ (this follows from Proposition \ref{umber}).
In other words, the topology of derived schemes is no more complicated than the topology of ordinary schemes, so long as we confine our attention to $0$-localic derived schemes.
\end{remark}

We conclude this section by proving a converse to Remark \ref{simpletime}:

\begin{theorem}\label{swill}
Let $\calX = \Shv(X)$ be a $0$-localic $\infty$-topos, and let $\calO_{\calX}$ be a sheaf of simplicial commutative $k$-algebras on $\calX$, viewed (via Proposition \ref{sturma}) as an object of
$\Fun^{\adm}( \calT_{ \Zar}(k), \calX)$. Then $(\calX, \calO_{\calX})$ is a derived $k$-scheme if and only if the following conditions are satisfied:
\begin{itemize}
\item[$(1)$] The truncation $(X, \pi_0 \calO_{\calX} )$ is a $k$-scheme.
\item[$(2)$] For each $i > 0$, $\pi_{i} \calO_{\calX}$ is a quasi-coherent sheaf of
$\pi_0 \calO_{\calX}$-modules.
\item[$(3)$] The structure sheaf $\calO_{\calX}$ is hypercomplete, when regarded as an object
of $\calX$ (see \S \toposref{hyperstacks}). 
\end{itemize}
\end{theorem}

\begin{proof}
First suppose that $(\calX, \calO_{\calX})$ is a derived $k$-scheme. We will prove that $(1)$, $(2)$, and $(3)$ are satisfied. Assertion $(1)$ follows immediately from Remark \ref{simpletime}. The remaining assertions are local on $X$ (for assertion $(3)$, this follows from Remark \toposref{suchlike}), so we may
assume without loss of generality that $( \calX, \calO_{\calX})$ is
an affine derived $k$-scheme, given by the spectrum of a simplicial commutative $k$-algebra
$A$. Then we can identify $X$ with the set of prime ideals in the commutative ring $\pi_0 A$, with
a basis of open sets given by $U_{f} = \{ \mathfrak{p} \subseteq \pi_0 A: f \notin \mathfrak{p} \}$. 
Using Theorem \ref{scoo} and Proposition \ref{precon}, we can identify
$\calO_{\calX}$ with the $\SCR_{k}$-valued sheaf described by the formula
$$ U_{f} \mapsto A[f^{-1}].$$
In particular, $\pi_{i} \calO_{\calX}$ is the sheafification of the presheaf of $\pi_0 A$-modules
described by the formula $U_{f} \mapsto ( \pi_i A)[f^{-1}]$, which is the quasi-coherent sheaf
associated to $\pi_{i} A$; this proves $(2)$. To prove $(3)$, choose a Postnikov tower
$$ \ldots \rightarrow \tau_{\leq 2} A \rightarrow \tau_{\leq 1} A \rightarrow \tau_{\leq 0} A,$$
for $A$, and let
$$ \ldots \rightarrow \calO^{\leq 2}_{\calX} \rightarrow \calO^{\leq 1}_{\calX} \rightarrow \calO^{\leq 0}_{\calX} $$
be the associated $\SCR_{k}$-valued sheaves on $\calX$. Using the formula above, we conclude that the canonical map $\calO_{\calX} \rightarrow \lim \{ \calO^{\leq n}_{\calX} \}$ is an equivalence.
To prove $(3)$, it will therefore suffice to show that each $\calO^{\leq n}_{\calX}$ is hypercomplete.
It now suffices to observe that $\calO^{\leq n}_{\calX}$ is $n$-truncated, by Corollary \ref{tabletime}.

We now prove the converse. Suppose that $(1)$, $(2)$, and $(3)$ are satisfied; we wish to prove that
$(\calX, \calO_{\calX})$ is a derived $k$-scheme. The assertion is local on $X$, so we may assume without loss of generality that $(X, \pi_0 \calO_{\calX}) = \SSpec R$ is an affine $k$-scheme. Applying
$(2)$, we conclude that each $\pi_i \calO_{\calX}$ is the quasi-coherent sheaf associated to an
$R$-module $M_i$. For each $n \geq 0$, let $A_{\leq n} \in \SCR_{k}$ denote the global sections
$\Gamma( \calX; \tau_{\leq n} \calO_{\calX} )$. There is a convergent spectral sequence
$$ E_{2}^{p,q} = \HH^{p}(X; \pi_{q} (\tau_{\leq n} \calO_{\calX})) \Rightarrow \pi_{q-p} A_{\leq n}.$$
Since $X$ is affine, the quasi-coherent sheaves $\pi_{i} \calO_{\calX}$ have no cohomology
in positive degrees, and the above spectral sequence degenerates to yield isomorphisms
$$ \pi_{i} A_{\leq n} \simeq \begin{cases} M_{i} & \text{if } i \leq n \\ 0 & \text{otherwise.} \end{cases}$$
In particular, $\pi_0 A_{\leq n} \simeq R$.

Fix $n \geq 0$, and let $( \calX_{n}, \calO_{\calX_{n}} )$ be the spectrum of $A_{\leq n}$.
The equivalence $A_{n} \simeq \Gamma( \calX; \tau_{\leq n} \calO_{\calX} )$ induces a map
$\phi_{n}: (\calX_{n}, \calO_{\calX_{n}}) \rightarrow (\calX, \tau_{\leq n} \calO_{\calX} )$ in $\LGeo( \calG_{\Zar}^{\der}(k))$. The above argument shows that the induced geometric morphism
$\phi_{n}^{\ast}: \calX_{n} \rightarrow \calX$ is an equivalence of $\infty$-topoi, and that
$\phi_{n}$ induces an isomorphism of quasi-coherent sheaves
$\phi_{n}^{\ast} (\pi_{i} \calO_{\calX_n}) \simeq \pi_{i} \calO_{\calX}$ for $0 \leq i \leq n$.
Since the structure sheaves on both sides are $n$-truncated, we conclude that $\phi_{n}$ is an equivalence.

Let $A \in \SCR_{k}$ denote the inverse limit of the tower
$$ \ldots \rightarrow A_{\leq 2} \rightarrow A_{\leq 1} \rightarrow A_{\leq 0},$$
so that $\pi_0 A \simeq R$. We can therefore identify the spectrum of $A$ with
$( \calX, \calO'_{\calX})$. The first part of the proof shows that $\calO'_{\calX}$ is the inverse limit
of its truncations 
$$\tau_{ \leq n} \calO'_{\calX} \simeq \phi_{n}^{\ast} \calO_{ \calX_{n}} \simeq
\tau_{ \leq n} \calO_{\calX}.$$
Passing to the inverse limit, we obtain a map
$$\psi: \calO_{\calX} \rightarrow \lim \{ \tau_{\leq n} \calO_{\calX} \} \simeq \calO'_{\calX}.$$
By construction, $\psi$ induces an isomorphism on all (sheaves of) homotopy groups, and is
therefore $\infty$-connective. The sheaf $\calO'_{\calX}$ is hypercomplete (since it is an inverse limit of truncated objects of $\calX$), and the sheaf $\calO_{\calX}$ is hypercomplete by assumption $(3)$.
It follows that $\psi$ is an equivalence, so that $(\calX, \calO_{\calX}) \simeq \Spec^{\calG^{\der}_{\Zar}} A$ is an affine derived $k$-scheme as desired.
\end{proof}

\begin{remark}
It is not difficult to formulate an analogue of Theorem \ref{swill} in the case where the
underlying $\infty$-topos $\calX$ is not $0$-localic. We have refrained from doing so, because we do not wish to discuss the appropriate generalization of the theory of quasi-coherent sheaves. The proof in general involves no additional difficulties: as in the argument given above, it immediately reduces to the case of affine (and therefore $0$-localic) schemes. We leave the details to the reader.
\end{remark}

\subsection{Derived Algebraic Geometry (\Etale topology)}\label{juet}

In this section, we present a variant on the theory of derived schemes, using the \etale topology in place of the Zariski topology.

\begin{definition}\label{juy}
Let $k$ be a commutative ring. We let
$\Comm_{k}^{\smooth}$ denote the full subcategory of $\Comm_{k}$ spanned by those commutative $k$-algebras $A$ for which there exists an \etale map $k[x_1, \ldots, x_n] \rightarrow A$.
We define a pregeometry $\calT_{\mathet}(k)$ as follows:
\begin{itemize}
\item[$(1)$] The underlying $\infty$-category of $\calT_{\mathet}(k)$ is
$\Nerve( \Comm_{k}^{\smooth})^{op}$.
\item[$(2)$] A morphism in $\calT_{\mathet}(k)$ is admissible if and only if the corresponding
map of $k$-algebras $A \rightarrow B$ is \etale.
\item[$(3)$] A collection of admissible morphisms $\{ A \rightarrow A_{\alpha} \}$ in
$\calT_{\mathet}(k)$ generates a covering sieve on $A$ if and only if there exists a finite
set of indices $\{ \alpha_{i} \}_{1 \leq i \leq n}$ such that the induced map
$A \rightarrow \prod_{1 \leq i \leq n} A_{\alpha_i}$ is faithfully flat.
\end{itemize}
\end{definition}

\begin{remark}
As in Remark \ref{defmod}, it is possible to make several variations on Definition
\ref{juy} which give rise to Morita-equivalent pregeometries. For example, 
Proposition \ref{silver} shows that we could enlarge $\calT_{\mathet}(k)$ to include
{\em all} smooth affine $k$-schemes, or even all smooth $k$-schemes.
\end{remark}

We now describe a geometric envelope for the pregeometry $\calT_{\mathet}(k)$.
First, we need to introduce a few definitions.

\begin{definition}\label{juh}
Let $f: A \rightarrow B$ be a morphism of simplicial commutative rings. We will say that $f$ is {\it \etale} if the following conditions are satisfied:
\begin{itemize}
\item[$(1)$] The underlying map $\pi_0 A \rightarrow \pi_0 B$ is an \etale map of ordinary commutative rings.
\item[$(2)$] For each $i > 0$, the induced map $\pi_{i} A \otimes_{ \pi_0 A} \pi_0 B \rightarrow \pi_{i} B$ is an isomorphism of abelian groups.
\end{itemize}
\end{definition}

\begin{remark}\label{than}
Let $f: A \rightarrow B$ be a map of simplicial commutative rings, and suppose that $A$ is {\em discrete}. Then $f$ is \etale if and only if $B$ is discrete, and $f$ is \etale in the sense of classical commutative algebra.
\end{remark}

\begin{remark}
Let $f: A \rightarrow B$ be a map of simplicial commutative rings. Then $f$ is \etale in the sense of
Definition \ref{juh} if and only if the induced map between the underlying $E_{\infty}$-rings 
(see \S \ref{escr}) is \etale, in the sense of Definition \deformationref{defet}.
\end{remark}

\begin{remark}
The collection of \etale morphisms between simplicial commutative rings contains all equivalences
and is closed under composition and the formation of retracts.
\end{remark}

\begin{notation}
If $A$ is a simplicial commutative $k$-algebra, we let $(\SCR_{k})_{A/}^{\mathet}$ denote the
full subcategory of $(\SCR_{k})_{A/}$ spanned by the \etale morphisms $f: A \rightarrow B$.
\end{notation}

\begin{lemma}\label{thanh}
Suppose given a pushout diagram
$$ \xymatrix{ A \ar[r] \ar[d]^{f} & A' \ar[d]^{f'} \\
B \ar[r] & B' }$$
of simplicial commutative rings. If $f$ is \etale, then $f'$ is also \etale.
\end{lemma}

\begin{proof}
Since every \etale morphism of ordinary commutative rings is flat, the spectral sequence
of Corollary \ref{tubble} degenerates and yields an isomorphism
$\pi_{\ast} B' \simeq (\pi_{\ast} A') \otimes_{ \pi_{\ast} A} (\pi_{\ast} B)$. The result now follows from the analogous assertion for ordinary commutative rings.
\end{proof}

\begin{proposition}\label{swimm}
Let $k$ be a commutative ring, and let $\phi: A \rightarrow B$ be a morphism of simplicial commutative $k$-algebras. The following conditions are equivalent:
\begin{itemize}
\item[$(1)$] The map $\phi$ is \etale.
\item[$(2)$] There exists a pushout diagram
$$ \xymatrix{ k[x_1, \ldots, x_n] \ar[r] \ar[d]^{\phi_0} & A \ar[d]^{\phi} \\
k[y_1, \ldots, y_n, \Delta^{-1}] \ar[r] & B }$$
in $\SCR_{k}$, where $\phi_0(x_i) = f_i(y_1, \ldots, y_n)$ and
$\Delta$ denotes the determinant of the Jacobian matrix
$$[ \frac{ \bd f_i}{\bd y_j}]_{1 \leq i,j \leq n}$$ (in particular, $\phi_0$ is a map
of smooth discrete commutative $k$-algebras).
\end{itemize}
\end{proposition}

\begin{proof}
The implication $(2) \Rightarrow (1)$ follows from Lemma \ref{thanh} and Remark \ref{than}.
For the converse, suppose that $\phi$ is \etale. The structure theorem for \etale morphisms of ordinary commutative rings implies the existence of an isomorphism
$$ \pi_0 B \simeq (\pi_0 A)[ y_1, \ldots, y_m]/(f_1, \ldots, f_m),$$
such that the determinant of the Jacobian matrix $[ \frac{ \bd f_i}{\bd y_j}]_{1 \leq i,j \leq m}$
is invertible in $\pi_0 B$. Let $\{ a_i \}_{ 1 \leq i \leq k}$ be the nonzero coefficients appearing
in the polynomials $f_i$. Using Remark \ref{bull}, we can choose a commutative diagram
$$ \xymatrix{ k[x_1, \ldots, x_k] \ar[r]^-{g_0} \ar[d] & A \ar[d]^{\phi} \\
k[x_1, \ldots, x_k, y_1, \ldots, y_m] \ar[r]^-{g_1} & B}$$
where $g_0$ carries each $x_i$ to $a_i \in \pi_0 A$. 
Choose, for each $1 \leq i \leq m$, a polynomial $\overline{f}_i \in k[y_1, \ldots, y_m, x_1, \ldots, x_k]$
lifting $f_i$, so that $g_1( \overline{f}_i) = 0 \in \pi_0 B$. 
Let $\Delta \in k[y_1, \ldots, y_m, x_1, \ldots, x_k]$ be the determinant of the Jacobian
matrix $[ \frac{ \bd \overline{f}_i}{\bd y_j}]_{1 \leq i,j \leq n}$.
Using Remark \ref{bull} and Proposition \ref{umber}, we deduce the existence of a commutative diagram
$$ \xymatrix{ k[x_1, \ldots, x_k, z_1, \ldots, z_m] \ar[d]^{h} \ar[r]^-{\epsilon} & k[x_1, \ldots, x_k] \ar[r] & A \ar[d]^{\phi} \\
k[x_1, \ldots, x_k, y_1, \ldots, y_m,\Delta^{-1}] \ar[rr]^-{g_1} & & B }$$
where $h(z_i) = \overline{f}_i$ and $\epsilon(z_i) = 0$ for $1 \leq i \leq m$.
We claim that the outer square appearing in this diagram is a pushout. To see this, form a
pushout diagram
$$ \xymatrix{ k[x_1, \ldots, x_k, z_1, \ldots, z_m] \ar[r] \ar[d]^{\phi_0} & A \ar[d] \\
k[x_1, \ldots, x_k, y_1, \ldots, y_m,\Delta^{-1}] \ar[r] & B' }$$
so that we have a canonical map $\psi: B' \rightarrow B$; we wish to show that $\psi$ is an
equivalence. Since $\phi_0$ is flat, the spectral sequence of Corollary \ref{tubble}
implies the existence of a canonical isomorphism
$$\pi_{i} B' \simeq (\pi_{i} A) \otimes_{ k[x_1, \ldots, x_k, z_1, \ldots, z_m] } 
k[x_1, \ldots, x_k, y_1, \ldots, y_m, \Delta^{-1}]
\simeq (\pi_i A) \otimes_{ \pi_0 A} (\pi_0 B).$$
Combining this with our assumption that $\phi$ is \etale, we conclude that
$\psi$ induces an isomorphism $\pi_{i} B' \simeq \pi_{i} B$ for each $i \geq 0$, as desired.
\end{proof}

\begin{corollary}\label{sweetup}
Let $k$ be a commutative ring, $A$ a compact object of $\SCR_{k}$, and
$f: A \rightarrow B$ an \etale map. Then $B$ is a compact object of $\SCR_{k}$. 
\end{corollary}

\begin{proof}
Corollary \toposref{tyrmyrr} asserts that the collection of compact objects of
$\SCR_{k}$ is stable under finite colimits. Using Proposition \ref{swimm}, we may
assume without loss of generality that $A = k[y_1, \ldots, y_m]$ and that $R$ is a discrete
$A$-algebra. The classical structure theory for \etale morphisms implies the existence of
an isomorphism $R \simeq A[x_1, \ldots, x_n, \Delta^{-1}]/(f_1, \ldots, f_n)$, where
$\Delta$ is the determinant of the Jacobian matrix $[ \frac{ \bd f_i}{\bd x_j}]_{1 \leq i,j \leq n}$.
We have a pushout diagram of ordinary commutative rings
$$ \xymatrix{ A[ z_1, \ldots, z_n] \ar[r]^{f} \ar[d]^{g} & A \ar[d] \\
A[x_1, \ldots, x_n, \Delta^{-1}] \ar[r] & R, }$$
where $f(z_i) = 0$ and $g(z_i) = f_i$ for $1 \leq i \leq n$. The map $f$ is \etale and therefore flat, so the spectral sequence of Corollary \ref{tubble} implies that the above diagram is also a pushout
square in $\SCR_{k}$. It will therefore suffice to show that $A$, $A[z_1, \ldots, z_n]$, and
$A[x_1, \ldots, x_n, \Delta^{-1}]$ are compact objects ofq $\SCR_{k}$. In the first two cases, this is clear; in the third, it follows from Proposition \ref{presweet}.
\end{proof}

\begin{proposition}\label{sweetdown}
Let $k$ be a commutative ring, $f: A \rightarrow B$ an \etale morphism in $\SCR_{k}$, and let
$R$ another object in $\SCR_{k}$. 
$$ \xymatrix{ \bHom_{\SCR_{k}}( B, R) \ar[r] \ar[d]^{\phi} & \Hom_{\Comm_{k}}( \pi_0 B, \pi_0 R) \ar[d] \\
\bHom_{\SCR_{k}}(A, R) \ar[r] & \Hom_{\Comm_{k}}( \pi_0 A, \pi_0 R) }$$
is a homotopy pullback square.
\end{proposition}

\begin{proof}
Using Proposition \ref{swimm}, we may suppose that
$A = k[x_1, \ldots, x_n]$ and $B = k[y_1, \ldots, y_n, \Delta^{-1}]$, with 
$f(x_i) = f_i( y_1, \ldots, y_n)$ and $\Delta$ the determinant of the Jacobian matrix
$[ \frac{ \bd f_i}{\bd y_j} ]_{1 \leq i,j \leq n}$.

Fix a morphism $\eta: A \rightarrow R$. We wish to show that
the homotopy fiber $F$ of $\phi$ over the point $\eta$ is homotopy equivalent to the
discrete set $\Hom_{\Comm_{A}}( B, \pi_0 R)$. Let $X = \bHom_{\SCR_{k}}( k[t], R)$ be the underlying space of $R$. Using Remark \ref{bull} and Proposition \ref{umber}, we conclude that
$F$ fits into a homotopy fiber sequence
$$ F \rightarrow (X^{n})' \stackrel{\psi}{\rightarrow} X^n$$
where $(X^{n})'$ is the union of those connected components of $X^{n}$ corresponding
to $n$-tuples $(r_1, \ldots, r_n) \in (\pi_0 R)^{n}$ where the polynomial $\Delta$ does not vanish,
and $\psi$ is induced by the collection polynomials $\{ f_i \}_{1 \leq i \leq n}$. 
In particular, for any base point $\widetilde{\eta}$ in $F$ and $i > 0$, we have a long exact sequence
$$ \ldots \pi_{i}(F,\widetilde{\eta}) \rightarrow (\pi_{i} R)^{n} \stackrel{g_i}{\rightarrow} (\pi_{i} R)^{n} \rightarrow \ldots.$$
Using Remark \ref{snapple} and the invertibility of $\Delta$ in $\pi_0 R$, we conclude that $g_i$ is an isomorphism for $i > 0$. This proves that $F$ is homotopy equivalent to a discrete space: namely, 
the fiber of the map $\pi_0 (X^{n})' \stackrel{\pi_0 \psi}{\rightarrow} \pi_0 X^{n}$.
This fiber can be identified with $\Hom_{\Comm_{A}}( B, \pi_0 R)$ by construction.
\end{proof}

From this, we can deduce the following analogue of Theorem \deformationref{turncoat} (note that the proof in this context is considerably easier):

\begin{corollary}\label{sweepdown}
Let $f: A \rightarrow B$ be a morphism of simplicial commutative $k$-algebras which
induces an isomorphism $\pi_0 A \simeq \pi_0 B$, and let
$F: (\SCR_{k})_{A/} \rightarrow (\SCR_{k})_{B/}$ be the functor
described by the formula $A' \mapsto A' \otimes_{A} B$. Then
$F$ induces an equivalence of $\infty$-categories 
$F^{\mathet}: (\SCR_{k})_{A/}^{\mathet} \rightarrow (\SCR_{k})_{B/}^{\mathet}$.
\end{corollary}

\begin{proof}
Proposition \ref{sweetdown} implies that $F^{\mathet}$ is fully faithful. To prove that $F^{\mathet}$ is fully faithful, we consider
an arbitrary \etale map $B \rightarrow B'$. Proposition \ref{swimm} implies the existence of a pushout diagram 
$$ \xymatrix{ k[x_1, \ldots, x_n] \ar[r]^{\eta} \ar[d] & B \ar[d] \\
k[y_1, \ldots, y_n, \Delta^{-1}] \ar[r] & B'.}$$
Using Remark \ref{bull}, we can lift $\eta$ to a map $\widetilde{\eta}: k[x_1, \ldots, x_n] \rightarrow A$.
Form another pushout diagram
$$ \xymatrix{ k[x_1, \ldots, x_n] \ar[r]^{\widetilde{\eta}} \ar[d] & A \ar[d] \\
k[y_1, \ldots, y_n, \Delta^{-1}] \ar[r] & A'.}$$
Then $F^{\mathet}(A') \simeq B'$.
\end{proof}

\begin{definition}
Let $k$ be a commutative ring. We define a geometry $\calG_{\mathet}^{\der}(k)$ as follows:
\begin{itemize}
\item[$(1)$] The underlying $\infty$-category of $\calG_{\mathet}^{\der}(k)$ coincides with
$\calG_{\Zar}^{\der}(k)$, the (opposite of) the $\infty$-category of {\em compact} objects of
$\SCR_{k}$.
\item[$(2)$] A morphism in $\calG_{\Zar}^{\der}(k)$ is admissible if and only if the corresponding map of simplicial commutative $k$-algebras is \etale.
\item[$(3)$] A collection of admissible morphisms $\{ A \rightarrow A_{\alpha} \}$ in $\calG_{\Zar}^{\der}(k)$ generates a covering sieve on $A$ if and only if there exists a finite set of indices
$\{ \alpha_{i} \}_{1 \leq i \leq n}$ such that the induced map
$\pi_0 A \rightarrow \prod_{1 \leq i \leq n} \pi_0 A_{\alpha_i}$ is faithfully flat.
\end{itemize}
\end{definition}

\begin{remark}\label{swimble}
The proof of Proposition \ref{snubber} shows that the truncation functor
$\pi_0: \SCR_{k} \rightarrow \Nerve( \Comm_{k} )$ induces a transformation of
geometries $\calG_{\mathet}^{\der}(k) \rightarrow \calG_{\mathet}(k)$, which exhibits
$\calG_{\mathet}(k)$ as a $0$-stub of $\calG_{\mathet}^{\der}(k)$.
\end{remark}

If $A$ is a smooth commutative $k$-algebra, then $A$ is compact when regarded as a discrete object
of $\SCR_{k}$. By a mild abuse of notation, we may use this observation to identify
$\calT_{\mathet}(k)$ with a full subcategory of $\calG_{\mathet}^{\der}(k)$. Our main result can then be stated as follows:

\begin{proposition}\label{sturman}
Let $k$ be a commutative ring. The inclusion $\calT_{\mathet}(k) \subseteq \calG_{\mathet}^{\der}(k)$ exhibits $\calG_{\mathet}^{\der}(k)$ as a geometric envelope of $\calT_{\mathet}(k)$.
\end{proposition}

We will postpone the proof of Proposition \ref{sturman} for a moment, and develop some consequences.

\begin{corollary}\label{supp}
The inclusion $\calT_{\mathet}(k) \subseteq \calG_{\mathet}(k)$ exhibits
$\calG_{\mathet}(k)$ as a $0$-truncated geometric envelope of $\calT_{\mathet}(k)$.
\end{corollary}

\begin{proof}
Combine Proposition \ref{sturman}, Lemma \ref{good1}, and Remark \ref{swimble}.
\end{proof}

\begin{remark}\label{dunsee}
Let $\calX$ be an $\infty$-topos. Proposition \ref{sturman} implies the existence of a fully faithful embedding $\theta: \Struct_{\calT_{\mathet}(k)}(\calX) \rightarrow \Fun^{\lex}( \calG_{\mathet}^{\der}(k), \calX) \simeq \Shv_{\SCR_{k}}(\calX)$. We will often invoke this equivalence implicitly, and identify a $\calT_{\mathet}(k)$-structure on $\calX$ with the associated sheaf of simplicial commutative $k$-algebras on $\calX$. We will say that a sheaf of simplicial commutative $k$-algebras $\calO$ on $\calX$ is {\it strictly Henselian} if belongs to the essential image of this embedding.
\end{remark}

\begin{remark}
The condition that a sheaf $\calO$ of sheaf of commutative $k$-algebras be strictly Henselian depends only on the underlying sheaf $\pi_0 \calO$ of ordinary commutative $k$-algebras. More precisely,
let $\calX$ be an $\infty$-topos and $\calO$ an object of $\Fun^{\adm}(\calT_{\mathet}(k), \calX)$. Then:
\begin{itemize}
\item[$(1)$] For each $n \geq 0$, the composition
$$ \tau_{\leq n} \calO: \calT_{\mathet}(k) \stackrel{\calO}{\rightarrow} \calX 
\stackrel{ \tau_{\leq n}}{\rightarrow} \calX$$
belongs to $\Fun^{\adm}( \calT_{\mathet}(k), \calX)$.
\item[$(2)$] The sheaf $\calO$ belongs to $\Struct_{ \calT_{\mathet}(k)}(\calX)$ if and only if
$\tau_{\leq n} \calO$ belongs to $\Struct_{ \calT_{\mathet}(k)}(\calX)$.
\end{itemize}
Assertion $(1)$ follows immediately from Remark \ref{downta} below, and $(2)$ follows from
Proposition \toposref{pi00detects}.
\end{remark}

\begin{definition}
Let $k$ be a commutative ring. A {\it derived Deligne-Mumford stack} over $k$ is
a $\calT_{\mathet}(k)$-scheme.
\end{definition}

The following result shows that the theory of derived Deligne-Mumford stacks really does generalize the classical theory of Deligne-Mumford stacks:

\begin{proposition}\label{jak}
Let $\Sch^{\leq 0}_{\leq 1}( \calT_{\mathet}(k) )$ denote the full subcategory of
$\Sch( \calT_{\mathet}(k) )$ spanned by those $\calT_{\mathet}(k)$-schemes which
are $0$-truncated and $1$-localic. Then $\Sch^{\leq 0}_{\leq 1}( \calT_{\mathet}(k) )$
is canonically equivalent to the $\infty$-category of Deligne-Mumford stacks over $k$, as
defined in Definition \ref{spud}.
\end{proposition}

\begin{proof}
Combine Corollary \ref{supp} with Theorem \ref{sup}.
\end{proof}

\begin{remark}
Let $k$ be a commutative ring.
There is an evident transformation of geometries $\calG_{\Zar}^{\der}(k) \rightarrow
\calG_{\mathet}^{\der}(k)$ (which is the identity functor on the underlying $\infty$-categories), which induces a relative spectrum functor 
$$\Spec^{\calG_{\mathet}^{\der}(k)}_{ \calG_{\Zar}^{\der}(k)}: \Sch(\calG_{\Zar}^{\der}(k))
\rightarrow \Sch( \calG_{\mathet}^{\der}(k)).$$
This functor associates a derived Deligne-Mumford stack over $k$ to every derived $k$-scheme. This
functor is {\em not} fully faithful in general. However, it is fully faithful when restricted to $0$-localic
derived $k$-schemes (see Warning \ref{gumber}).
\end{remark}

\begin{proof}[Proof of Proposition \ref{sturman}]
As in the proof of Proposition \ref{sturma}, we must prove the following:
\begin{itemize}
\item[$(1)$] For every idempotent complete $\infty$-category $\calC$ which admits finite limits, the restriction map
$$\Fun^{\lex}(\calG_{\mathet}^{\der}(k), \calC) \rightarrow \Fun^{\adm}(\calT_{\mathet}(k), \calC)$$ is an equivalence of $\infty$-categories.
\item[$(2)$] The collections of admissible morphisms and admissible coverings in $\calG_{\mathet}^{\der}(k)$ are generated by admissible morphisms and admissible coverings in the full subcategory $\calT_{\mathet}(k) \subseteq \calG_{\mathet}^{\der}(k)$.
\end{itemize}

We begin with the proof of $(1)$. We have a commutative diagram
$$ \xymatrix{ \Fun^{\lex}(\calG_{\mathet}^{\der}(k), \calC) \ar[rr] \ar[dr]^{u} & & \Fun^{\adm}(\calT_{\mathet}(k), \calC) \ar[dl]^{v} \\
& \Fun^{\pi}( \Nerve(\Poly_{k}), \calC), & }$$
where $\Fun^{\pi}( \Nerve( \Poly_{k} ), \calC)$ denotes the full subcategory of $\Fun( \Nerve( \Poly_{k} ), \calC)$ spanned by those functors which preserve finite products. Using Propositions \ref{spakk} and \toposref{lklk}, we deduce that $u$ is an equivalence of $\infty$-categories.
It will therefore suffice to show that $v$ is an equivalence of $\infty$-categories. In view of
Proposition \toposref{lklk}, it will suffice to show the following:
\begin{itemize}
\item[$(a)$] Every functor $f_0: \Nerve( \Poly_{k} ) \rightarrow \calC$ which preserves finite products
admits a right Kan extension $f: \calT_{\Zar}(k) \rightarrow \calC$.
\item[$(b)$] A functor $f: \calT_{\mathet}(k) \rightarrow \calC$ belongs to $\Fun^{\adm}( \calT_{\mathet}(k), \calC)$ if and only if $f_0 = f | \Nerve( \Poly_{k} )$ preserves products, and $f$ is a right Kan extension of $f_0$.
\end{itemize}
Assertion $(a)$ and the ``if'' direction of $(b)$ follow immediately from Proposition \ref{spakk}.
To prove the reverse direction, let us suppose that $f: \calT_{\mathet}(k) \rightarrow \calC$ belongs to
$\Fun^{\adm}(\calT_{\mathet}(k), \calC)$. Let $f_0 = f | \Nerve( \Poly_{k} )$, and let $f': \calT_{\mathet}(k) \rightarrow \calC$
be a right Kan extension of $f_0$ (whose existence is guaranteed by $(a)$). The
identification $f | \Nerve( \Poly_{k} ) = f' | \Nerve( \Poly_{k} )$ extends to a natural transformation $\alpha: f \rightarrow f'$, which is unique up to homotopy. We wish to show that $\alpha$ is an equivalence.

Fix an object of $\calT_{\mathet}(k)$, corresponding to a commutative $k$-algebra $A$.
Then there exists an \etale map $k[z_1, \ldots, z_m] \rightarrow A$. Using the structure theory for
\etale morphisms (see Proposition \ref{swimm}), we deduce the existence of a pushout diagram
$$ \xymatrix{  k[x_1, \ldots, x_n] \ar[r] \ar[d]^{\phi} & k[z_1, \ldots, z_m] \ar[d] \\
k[y_1, \ldots, y_n, \Delta^{-1}] \ar[r] & A }$$
where $\phi(x_i) = f_i(y_1, \ldots, y_n)$ and $\Delta$ is the determinant of the Jacobian matrix
$[ \frac{ \bd f_i }{\bd y_j} ]_{1 \leq i,j \leq n}$. We wish to show that $\alpha(A)$ is an equivalence.
Since $f, f': \calT_{\mathet}(k)$ both preserve pullbacks by \etale morphisms, it will suffice to show that
$\alpha( k[z_1, \ldots, z_m])$, $\alpha(k[x_1, \ldots, x_n])$, and $\alpha( k[y_1, \ldots, y_n], \Delta^{-1})$
are equivalences. In the first two cases, this is clear (since $f | \Nerve( \Poly_{k} ) = f' | \Nerve( \Poly_{k} )$), and in the third case it follows from the proof of Proposition \ref{sturma}.

Let us now prove $(2)$. Let $\calG$ denote a geometry whose underlying $\infty$-category coincides with $\calG_{\mathet}^{\der}(k)$, such that the inclusion $\calT_{\mathet}(k) \rightarrow \calG$ is a transformation of pregeometries. We must show:
\begin{itemize}
\item[$(2a)$] Every admissible morphism of $\calG_{\mathet}^{\der}(k)$ is an admissible morphism of $\calG$.
\item[$(2b)$] Every collection of $\calG_{\mathet}^{\der}(k)$-admissible morphisms $\{ \psi_{\alpha}: A \rightarrow A_{\alpha} \}$ which generates a $\calG_{\mathet}^{\der}(k)$-covering sieve also generates a $\calG$-covering sieve.
\end{itemize}

Assertion $(2a)$ follows immediately from Proposition \ref{swimm}. Let us prove $(2b)$.
Without loss of generality, we may assume that the set of indices $\alpha$ is finite.
Using Proposition \ref{swimm}, we conclude that each $\psi_{\alpha}$ fits into a pushout diagram
$$ \xymatrix{ k[x_1, \ldots, x_{n_{\alpha}}] \ar[r] \ar[d] & A \ar[d]^{\psi_{\alpha}} \\
R_{\alpha} \ar[r] & A_{\alpha}, }$$
where the vertical maps are \etale. Let $B$ be the tensor product of the polynomial algebras
$\{ k[x_1, \ldots, x_{n_{\alpha}} \}$, and let 
$$B_{\alpha} = B \otimes_{ k[x_1, \ldots, x_{n_{\alpha}}]} R_{\alpha}.$$
Each of the maps $\SSpec B_{\alpha} \rightarrow \SSpec B$ is \etale, and therefore has open image $U_{\alpha} \subseteq \SSpec B$. Let $V_{\alpha} \subseteq \SSpec \pi_0 A$ be the inverse
image of $U_{\alpha}$, so that $\bigcup V_{\alpha} = \SSpec \pi_0 A$. It follows that
there exists a collection of elements $b_1, \ldots, b_n \in B$ whose images generate the unit
ideal in $\pi_0 A$, such that each of the open sets $W_{b_i} = \{ \mathfrak{p} \subseteq B: b_i \notin \mathfrak{p} \}$ is contained in $U = \bigcup U_{\alpha}$. 

For $1 \leq i \leq n$, let $a_i \in \pi_0 A$ denote the image of $b_i$. The proof of Proposition \ref{sturma} shows that the collection of maps $\{ A \rightarrow A[ \frac{1}{a_i} ] \}$ generates
a covering sieve on $A$ with respect to the geometry $\calG$. It will therefore suffice to prove
$(2b)$ after replacing $A$ by $A[ \frac{1}{a_i}]$; in other words, we may assume that
the image of $\SSpec \pi_0 A$ in $\SSpec B$ is contained in $W_{b} \subseteq U$ for
some element $b \in B$. The desired result now follows from the observation that the maps
$\{ B[ \frac{1}{b}] \rightarrow B_{\alpha}[\frac{1}{b}] \}$ generate a $\calT_{\mathet}(k)$-covering sieve on $B[ \frac{1}{b} ] \in \calT_{\mathet}(k)$.
\end{proof}

To make a more detailed study of the theory of derived Deligne-Mumford stacks, it is essential to observe the following:

\begin{proposition}\label{precon}
Let $k$ be a commutative ring. Then the \etale topology on $\SCR_{k}$ is
{\it precanonical}. In other words, for every $A \in \SCR_{k}$, the corresponding corepresentable functor 
$$ \Pro( \calG_{\mathet}^{\der} )^{op} \simeq \SCR_{k} \rightarrow \SSet$$
belongs to $\Shv( \Pro(\calG_{\mathet}^{\der}(k)) )$ (see Definition \ref{sputt}).
\end{proposition}

Before giving the proof of Proposition \ref{precon}, we introduce a bit of terminology.

\begin{notation}\label{swik}
Let $R^{\bigdot}: \Nerve( \cDelta_{+}) \rightarrow \EInfty$ be an augmented cosimplicial $E_{\infty}$-ring. For each $n \geq 0$, we let $\cDelta_{+}^{< n}$ denote the full subcategory of
$\cDelta_{+}$ spanned by the objects $[m]$ with $m < n$. The {\it $n$th comatching object}
of $R^{\bigdot}$ is defined to be a colimit of the diagram
$$ \Nerve( \cDelta_{+}^{< n}) \times_{ \Nerve( \cDelta_{+}) }
\Nerve( \cDelta_{+})_{/ [n] } \rightarrow \Nerve( \cDelta_{+}) \stackrel{R^{\bigdot}}{\rightarrow}
\EInfty.$$
If we denote this object by $M^n$, then we have a canonical map of
$E_{\infty}$-rings $e_{n}: M^{n} \rightarrow R^{n}$.

We will say that $R^{\bigdot}$ is a {\it flat hypercovering} if, for every $n \geq 0$, the map
$e_{n}$ is faithfully flat: that is, the underlying map of ordinary commutative rings
$\pi_0 M^{n} \rightarrow \pi_0 R^{n}$ is faithfully flat, and the maps
$$ \pi_i M^{n} \otimes_{ \pi_0 M^{n} } (\pi_0 R^{n}) \rightarrow \pi_i R^{n}$$
are isomorphisms for every integer $i$. In this case, we will also say that
the underlying cosimplicial object is a {\it flat hypercovering} of $R^{-1} \in \EInfty$.

Let $R^{\bigdot}: \Nerve(\cDelta_{+}) \rightarrow \SCR_{k}$ be a cosimplicial object of the
$\infty$-category of simplicial commutative $k$-algebras. We will say that
$R^{\bigdot}$ is a {\it flat hypercovering} (or that the underlying cosimplicial object is
a {\it flat hypercovering} of $R^{-1} \in \SCR_{k}$) if the composition
$$ \Nerve( \cDelta_{+} ) \rightarrow \SCR_{k} \stackrel{\theta}{\rightarrow}
(\EInfty)^{\conn}_{k/} \rightarrow \EInfty$$
is a flat hypercover, where the functor $\theta$ is defined as in Proposition \ref{skillmon}.
\end{notation}

\begin{example}\label{koble}
Let $f: R^{-1} \rightarrow R^{0}$ be a faithfully flat map of $E_{\infty}$-rings. 
Let $R^{\bigdot}: \Nerve( \cDelta_{+} ) \rightarrow \SCR_{k}$ be the \Cech nerve of $f$,
regarded as a morphism in $( \EInfty)^{op}$. More informally, $R^{\bigdot}$ is the cosimplicial $E_{\infty}$-ring described by the formula
$$ R^{n} = R^{0} \otimes_{ R^{-1} } \ldots \otimes_{ R^{-1} } R^{0}.$$
Then $R^{\bigdot}$ is a flat hypercovering:
in fact, the map $M^{n} \rightarrow R^{n}$ appearing in Notation \ref{swik} is an equivalence
for $n > 0$, and can be identified with $f$ for $n=0$. The same reasoning can be applied
if $f$ is instead a map of simplicial commutative $k$-algebras.
\end{example}

Our interest in the class of flat hypercoverings stems from the following result, whose proof we will defer until \cite{spectral}:

\begin{proposition}\label{cober}
Let $R^{\bigdot}: \Nerve( \cDelta) \rightarrow (\EInfty)_{R/}$ be a flat hypercovering of
an $E_{\infty}$-algebra $R$. Then the induced map $R \rightarrow \varprojlim R^{\bigdot}$ is an equivalence of $E_{\infty}$-rings.
\end{proposition}

\begin{corollary}\label{koblee}
Let $R^{\bigdot}: \Nerve( \cDelta) \rightarrow (\SCR_{k})_{R/}$ be a flat hypercovering of a simplicial commutative $k$-algebra $R$. Then the induced map $R \rightarrow \varprojlim R^{\bigdot}$ is an equivalence of simplicial commutative $k$-algebras.
\end{corollary}

\begin{proof}
Combine Propositions \ref{cober} and \ref{skillmon}.
\end{proof}

\begin{proof}[Proof of Proposition \ref{precon}]
Fix an object $R \in \SCR_{k}$, let $\calC$ denote the full subcategory of
$( \SCR_{k} )_{R/}$ spanned by the \etale $R$-algebras, let
$\calC^{(0)} \subseteq \calC$ be a covering sieve on $R$, and let
$\phi_0$ denote the composition
$$ \calC^{(0)} \subseteq \calC \rightarrow \SCR_{k}.$$
$R \rightarrow \varprojlim \phi_0$ is an equivalence in $\SCR_{k}$.
We first treat the case where $\calC^{(0)}$ is the sieve generated by a finite
collection of \etale morphisms $\{ R \rightarrow R_{i} \}_{i \in S}$
such that the map $R \rightarrow \prod_{i \in S} R_i$ is faithfully flat.
Let $\cDelta^{S}$ denote the category whose objects are finite, nonempty linearly
ordered sets $[n]$ equipped with a map $c: [n] \rightarrow S$, and $\cDelta^{S, \leq 0}$ the full subcategory spanned by those objects with $n=0$ (which we can identify with the discrete
set of objects of $S$). The collection of objects $\{ R_i \}_{i \in S}$ determines a functor
$\psi_0: \Nerve( \cDelta^{S, \leq 0}) \rightarrow (\SCR_{k})_{R/}$. Let $\psi: \Nerve( \cDelta^{S}) \rightarrow (\SCR_{k})_{R/}$ be a left Kan extension
of $\psi_0$, so that $\psi$ can be described by the formula
$$ \psi( c: [n] \rightarrow S ) \mapsto R_{c(0)} \otimes_{R} R_{c(1)} \otimes_{R} \ldots
\otimes_{R} R_{c(n)}.$$
Note that $\psi$ factors through $\calC^{(0)}$. Using Corollary \toposref{hollowtt}, we deduce that the map $\psi^{op}: \Nerve( \cDelta^{S} )^{op} \rightarrow (\calC^{(0)})^{op}$ is cofinal.
It will therefore suffice to show that the canonical map
$R \rightarrow \varprojlim( \phi_0 \circ \psi)$ is a homotopy equivalence.
Let $R^{\bigdot}$ denote the cosimplicial object $\Nerve( \cDelta) \rightarrow ( \SCR_{k})_{R/}$ be the functor obtained
from $\psi$ by right Kan extension along the forgetful functor $\cDelta^{S} \rightarrow \cDelta$.
Then $R^{0} \simeq \prod_{i \in S} R_{s}$. Using the distributive law of Remark \ref{distrib}, we conclude
that the canonical map $R^{0} \otimes_{R} R^{0} \otimes_{R} \ldots \otimes_{R} R^{0} \rightarrow R^{n}$ is an equivalence for each $n$. In particular, $R^{\bigdot}$ is a flat hypercovering of
$R$ (Example \ref{koble}). We can identify $\varprojlim( \phi_0 \circ \psi)$ with the limit of the cosimplicial object
$\phi_0 \circ \overline{\psi}: \Nerve(\cDelta) \rightarrow \SCR_{k}$. The desired result now follows from
Proposition \ref{koblee}.

We now treat the case of an arbitrary covering sieve $\calC^{(0)} \subseteq \calC$ on $R$.
Choose a finite collection of \etale morphisms $\{ \alpha_i: R \rightarrow R_i \}_{1 \leq i \leq n}$
which belong to $\calC^{(0)}$ such that the induced map $R \rightarrow \prod R_{i}$ is faithfully flat, and let $\calC^{(1)} \subseteq \calC$ be the sieve generated by the morphisms $\alpha_{i}$.
We have a commutative diagram
$$ \xymatrix{ \phi(R) \ar[rr]^{f} \ar[dr]^{f'} & & \varprojlim \phi| \calC^{(0)} \ar[dl]^{f''} \\
& \varprojlim \phi|\calC^{(1)} & }$$
We wish to show that $f$ is a homotopy equivalence. The first part of the proof shows that $f'$ is a homotopy equivalence, so it will suffice to show that $f''$ is a homotopy equivalence.
In view of Lemma \toposref{basekann}, it will suffice to show that $\phi | \calC^{(0)}$ is a
right Kan extension of $\phi | \calC^{(1)}$. Unwinding the definitions, this reduces to the following assertion: if $R \rightarrow R'$ is an \etale morphism belonging to the sieve $\calC^{(0)}$, and
$\calC' \subseteq ( \SCR_{k})_{/R'}$ is the sieve given by the inverse image of
$\calC^{(1)}$, then the map $\phi(R') \rightarrow \varprojlim \phi| \calC'$ is a homotopy
equivalence. This follows from the first part of the proof, since the sieve
$\calC'$ is generated by the morphisms $\{ R' \rightarrow R' \otimes_{R} R_{i} \}_{1 \leq i \leq n}$.
\end{proof}

We conclude this section with a few remarks about the relationship between
$\calG_{\mathet}^{\der}(k)$ and its zero stub $\calG_{\mathet}(k)$.

\begin{proposition}\label{stungun}
Let $k$ be a commutative ring. The pregeometry $\calT_{\mathet}(k)$ is compatible with
$n$-truncations for each $n \geq 0$.
\end{proposition}

\begin{proof}
If $n > 0$, then this follows from Proposition \ref{rabb}. Let us therefore suppose that $n=0$.
Let $\calX$ be an $\infty$-topos, let $\calO_0 \in \Struct_{ \calT_{\mathet}(k)}(\calX)$, and let
$U \rightarrow X$ be an admissible morphism in $\calT_{\mathet}(k)$. We wish to show that the diagram
$$ \xymatrix{ \calO_0(U) \ar[r] \ar[d] & \tau_{\leq 0} \calO_0(U) \ar[d] \\
\calO_0(X) \ar[r] & \tau_{\leq 0} \calO_0(X) }$$
is a pullback square in $\calX$. In view of Proposition \ref{sturman}, we may assume that
$\calO_0 = \calO | \calT_{\mathet}(k)$, where $\calO$ is a $\calG^{\der}_{\mathet}(k)$-structure on $\calX$. Let $\calG^{\der}_{\disc}(k)$ denote the discrete geometry underlying $\calG^{\der}_{\mathet}(k)$. It will evidently suffice to prove the following assertion:

\begin{itemize}
\item[$(\ast)$] Let $\calO: \calG^{\der}_{\disc}(k) \rightarrow \calX$ be a
$\calG^{\der}_{\disc}(k)$-structure on an $\infty$-topos $\calX$. Then the diagram
$$ \xymatrix{ \calO(U) \ar[r] \ar[d] & \tau_{\leq 0} \calO(U) \ar[d] \\
\calO(X) \ar[r] & \tau_{\leq 0} \calO(X) }$$
is a pullback square in $\calX$.
\end{itemize}

Note that if $\pi^{\ast}: \calY \rightarrow \calX$ is a geometric morphism of $\infty$-topoi
and $\calO'$ is a $\calG^{\der}_{\disc}(k)$-structure on $\calY$ satisfying $(\ast)$, then
$\pi^{\ast} \calO'$ also satisfies $(\ast)$. We may suppose that $\calX$ is a left exact localization of the
presheaf $\infty$-category $\calP(\calC)$, for some small $\infty$-category $\calC$, so that we have an adjunction
$$ \Adjoint{\pi^{\ast}}{\calP(\calC)}{\calX.}{\pi_{\ast}}$$
The counit map $\pi^{\ast} \pi_{\ast} \calO \rightarrow \calO$ is an equivalence. It will therefore suffice to show that $\calO' = \pi_{\ast} \calO$ satisfies $(\ast)$. In other words, we may reduce to the case where
$\calX = \calP(\calC)$. In particular, $\calX$ has enough points (given by the evaluation functors
$\calP(\calC) \rightarrow \SSet$ corresponding to objects of $\calC$), and we may reduce to the case where $\calX = \SSet$.

We can identify the object $\calO \in \Struct_{ \calG^{\der}_{\disc}(k)}(\SSet) \simeq \SCR_{k}$ with 
a simplicial commutative $k$-algebra $R$. Similarly, we can identify the map $U \rightarrow X$
with an \etale map $A \rightarrow B$ in $\SCR_{k}$. We wish to show that the diagram
$$ \xymatrix{ \bHom_{\SCR_{k}}( B, R) \ar[d]^{\phi} \ar[r] & \pi_0 \bHom_{\SCR_{k}}( B, R) \ar[d] \\
\bHom_{\SCR_{k}}(A,R) \ar[r] & \pi_0 \bHom_{\SCR_{k}}( A, R). }$$
is a homotopy pullback square in $\SSet$. Unwinding the definitions, we must show:
\begin{itemize}
\item[$(\ast')$] For every point $\eta \in \bHom_{ \SCR_{k} }(A, R)$, the homotopy fiber $F$ of
$\phi$ over the point $\eta$ is homotopy equivalent to the discrete space $\pi_0 F$, and
the action of $\pi_1( \bHom_{\SCR_{k}}(A,R), \eta)$ on $\pi_0 F$ is trivial.
\end{itemize}
This follows from the existence of the homotopy pullback diagram
$$ \xymatrix{ \bHom_{\SCR_{k}}( B, R) \ar[d]^{\phi} \ar[r] & \Hom_{\Comm_{k}}( \pi_0 B, \pi_0 R) \ar[d] \\
\bHom_{\SCR_{k}}(A,R) \ar[r] & \Hom_{\Comm_{k}}( \pi_0 A, \pi_0 R), }$$
(see Proposition \ref{sweetdown}).
\end{proof}

\begin{remark}\label{downta}
Let $k$ be a commutative ring and $\calX$ an $\infty$-topos. Propositions \ref{sturma} and
\ref{sturman} imply that the restriction functors
\begin{eqnarray*} 
\Fun^{\lex}( \calG^{\der}_{\mathet}(k), \calX) & \rightarrow & 
\Fun^{\adm}( \calT_{\mathet}(k), \calX) \\
& \rightarrow & \Fun^{\adm}( \calT_{\Zar}(k), \calX) \\
& \rightarrow & \Fun^{\pi}( \Nerve( \Poly_{k}^{op}), \calX) \end{eqnarray*}
are equivalences of $\infty$-categories (here $\Fun^{\pi}( \Nerve(\Poly_{k}^{op}), \calX)$ denotes the full subcategory of $\Fun( \Nerve( \Poly_{k}^{op}), \calX)$ spanned by those functors which preserve finite products). Remarks \ref{switchera} and \ref{switcheru} allow us to identify the $\infty$-category $\Fun^{\lex}( \calG^{\der}_{\mathet}(k), \calX)$ with the $\infty$-category $\Shv_{ \SCR_{k}}(\calX)$
of $\SCR_{k}$-valued sheaves $\calX$. In particular, for each $n \geq 0$, there
we have a truncation functor $\tau_{\leq n}: \Fun^{\lex}( \calG^{\der}_{\mathet}(k), \calX)
\rightarrow \Fun^{\lex}( \calG^{\der}_{\mathet}(k), \calX)$. This induces truncation functors
$$ \tau_{\leq n}: \Fun^{\adm}( \calT_{\mathet}(k), \calX) \rightarrow \Fun^{\adm}( \calT_{\mathet}(k), \calX)$$
$$ \tau_{\leq n}: \Fun^{\adm}( \calT_{\Zar}(k), \calX) \rightarrow \Fun^{\adm}( \calT_{\Zar}(k), \calX)$$
$$ \tau_{\leq n}: \Fun^{\pi}( \Nerve( \Poly_{k}^{op}), \calX) \rightarrow \Fun^{\pi}( \Nerve( \Poly_{k}^{op}), \calX).$$
We claim that each of these truncation functors are simply given by composition with
the truncation functor $\tau_{\leq n}^{\calX}$ on $\calX$. Unwinding the definitions, this amounts to the following assertion:
\begin{itemize}
\item[$(\ast)$] Let $\calO: \calG_{\mathet}^{\der}(k) \rightarrow \calX$ be a left exact functor, and
$\calO'$ its $n$-truncation in $\Fun^{\lex}( \calG_{\mathet}^{\der}(k), \calX)$. Then, for every
$A \in \calT_{\mathet}(k)$, the induced map $\calO(A) \rightarrow \calO'(A)$ exhibits
$\calO'(A)$ as an $n$-truncation of $\calO(A)$ in $\calX$.
\end{itemize}
Note that if $\pi^{\ast}: \calY \rightarrow \calX$ is a geometric morphism and
$\calO \in \Fun^{\lex} ( \calG_{\mathet}^{\der}(k), \calY)$ satisfies $(\ast)$, then
$\pi^{\ast} \calO$ also satisfies $(\ast)$ (because the induced map
$\Fun^{\lex}( \calG_{\mathet}^{\der}(k), \calY) \rightarrow \Fun^{\lex}( \calG_{\mathet}^{\der}(k), \calX)$
commutes with $n$-truncation, by Proposition \toposref{compattrunc}). 

Without loss of generality, we may suppose that $\calX$ arises as a left-exact localization of a presheaf $\infty$-category $\calP(\calC)$. Let $\pi^{\ast}: \calP(\calC) \rightarrow \calX$ be the localization functor, and $\pi_{\ast}: \calX \rightarrow \calP(\calC)$ its right adjoint. Then, for each
$\calO \in \Fun^{\lex}( \calG_{\mathet}^{\der}(k), \calX)$, the counit map
$\pi^{\ast} \pi_{\ast} \calO \rightarrow \calO$ is an equivalence. In view of the above remark, it will suffice to prove that $( \calP(\calC), \pi_{\ast} \calO)$ satisfies $(\ast)$. In particular, we may assume that $\calX$ has enough points (given by evaluation at objects of $\calC$), and can therefore reduce to the case $\calX = \SSet$. In this case, we can identify $\calO$ with a simplicial commutative $k$-algebra $R$, and assertion $(\ast)$ can be reformulated as follows:

\begin{itemize}
\item[$(\ast')$] Let $R$ be a simplicial commutative $k$-algebra, and let
$A \in \calT_{\mathet}(k)$. Then the map
$$ \bHom_{ \SCR_{k} }(A,R) \rightarrow \bHom_{ \SCR_{k}}( A, \tau_{\leq n} R)$$
exhibits $\bHom_{\SCR_{k}}( A, \tau_{\leq n} R)$ as an $n$-truncation of
the mapping space $\bHom_{\SCR_{k}}(A,R)$. 
\end{itemize}

If $A$ is a polynomial ring over $k$, this follows from Remark \toposref{parei}. In the general case
we may assume that there exists an \etale map $k[x_1, \ldots, x_n] \rightarrow A$, and the
result follows from Proposition \ref{sweetdown}.
\end{remark}

Combining Propositions \ref{stungun} and \ref{sableware}, we obtain the following result:

\begin{corollary}\label{jik}
Let $(\calX, \calO_{\calX})$ be a derived Deligne-Mumford stack, and let $n \geq 0$. Then
$(\calX, \tau_{\leq n} \calO_{\calX})$ is a derived Deligne-Mumford stack.
\end{corollary}

\begin{remark}
Let $(\calX, \calO_{\calX})$ be a $1$-localic derived Deligne-Mumford stack over $k$.
Using Corollary \ref{jik} and Proposition \ref{jak}, we can identify the  $0$-truncation $(\calX, \pi_0 \calO_{\calX})$ with an ordinary Deligne-Mumford stack over $k$. We will refer to this ordinary
Deligne-Mumford stack as the {\it underlying ordinary Deligne-Mumford stack} of
$(\calX, \calO_{\calX})$. 
\end{remark}

We conclude this section by proving an analogue of Theorem \ref{swill}:

\begin{theorem}\label{swillboard}
Let $\calX$ be an $\infty$-topos and $\calO_{\calX}$ a sheaf of simplicial commutative $k$-algebras on $\calX$, viewed (via Proposition \ref{sturman}) as an object of
$\Fun^{\adm}( \calT_{ \mathet}(k), \calX)$. Then $(\calX, \calO_{\calX})$ is a derived Deligne-Mumford stack over $k$ if and only if the following conditions are satisfied:
\begin{itemize}
\item[$(1)$] Let $\phi_{\ast}: \calX \rightarrow \calX'$ be geometric morphism of $\infty$-topoi,
where $\calX'$ is $1$-localic and $\phi_{\ast}$ is an equivalence on discrete objects
(so that $\phi_{\ast}$ exhibits $\calX'$ as the $1$-localic reflection of $\calX$). Then
$( \calX', \phi_{\ast} \pi_0 \calO_{\calX} )$ is a derived Deligne-Mumford stack over $k$
(which is $1$-localic and $0$-truncated, and can therefore be identified with an
ordinary Deligne-Mumford stack $X$ over $k$ by Proposition \ref{jak}).  

\item[$(2)$] For each $i > 0$, $\pi_{i} \calO_{\calX}$ is a quasi-coherent sheaf on
$X$. 

\item[$(3)$] The structure sheaf $\calO_{\calX}$ is hypercomplete, when regarded as an object
of $\calX$ (see \S \toposref{hyperstacks}). 
\end{itemize}
\end{theorem}

\begin{proof}
The proof follows the same lines as that of Theorem \ref{swill}. Suppose first that
$(\calX, \calO_{\calX})$ is a derived Deligne-Mumford stack over $k$. 
We will prove that $(1)$, $(2)$, and $(3)$ are satisfied.
Corollary \ref{jik} implies that $(\calX, \pi_0 \calO_{\calX})$ is a derived Deligne-Mumford stack over $k$. Since $(\calX, \pi_0 \calO_{\calX})$ is $0$-truncated, Corollary \ref{supp} allows us to identify 
$\pi_0 \calO_{\calX}$ with a $\calG_{\mathet}(k)$-structure on $\calX$. Let
$\phi_{\ast}: \calX \rightarrow \calX'$ be the $1$-localic reflection of $\calX$; then
Theorem \ref{top4} implies that $(\calX', \phi_{\ast} (\pi_0 \calO_{\calX'}))$ is
again a derived Deligne-Mumford stack over $k$, and that the map
$(\calX, \pi_0 \calO_{\calX}) \rightarrow (\calX', \phi_{\ast} (\pi_0 \calO_{\calX}) )$
is \etale. This proves $(1)$.

We now prove $(3)$. By virtue of Remark \toposref{suchlike}, we can work locally on $\calX$ and we may therefore suppose that $(\calX, \calO_{\calX})$ is the affine $\calT_{\mathet}(k)$-scheme associated to a simplicial commutative $k$-algebra $A$. 
Choose a Postnikov tower
$$ \ldots \rightarrow \tau_{\leq 2} A \rightarrow \tau_{\leq 1} A \rightarrow \tau_{\leq 0} A,$$
for $A$. Corollary \ref{sweepdown} implies that for each $n$, the
$\infty$-category $(\SCR_{k})_{\tau_{\leq n} A/}^{\mathet}$ is equivalent to
$(\SCR_{k})_{A/}^{\mathet}$. Consequently, Theorem \ref{scoo} implies that we can identify
$\Spec^{\calG^{\der}_{\mathet}(k)}(\tau_{\leq n} A)$ with $(\calX, \calO_{\calX}^{\leq n})$, for
some sheaf $\calO_{\calX}^{\leq n}$ of simplicial commutative $k$-algebras on $\calX$.
Moreover, if we identify $\calX$ with the $\infty$-category $\Shv( ((\SCR_{k})_{A/}^{\mathet})^{op} )$,
Proposition \ref{precon} shows that $\calO_{\calX}^{\leq n}$ can be described by the formula
$A' \mapsto \tau_{\leq n} A'$, where $A'$ ranges over the $\infty$-category of
\etale $A$-algebras. Similarly, $\calO_{\calX}$ can be identified with the sheaf given by the forgetful functor $(\SCR_{k})_{A/}^{\mathet} \rightarrow \SCR_{k}$. It follows that the canonical map
$\calO_{\calX} \rightarrow \varprojlim \calO_{\calX}^{\leq n}$ is an equivalence. Thus
$\calO_{\calX}$ is an inverse limit of truncated objects of $\calX$, and therefore hypercomplete.

To prove $(2)$, we consider a collection of objects $\{ U_{\alpha} \in \calX \}$ such that
$\coprod U_{\alpha} \rightarrow 1_{\calX}$ is an effective epimorphism, and each
of the $\calT_{\mathet}(k)$-schemes $( \calX_{/U_{\alpha}}, \calO_{\calX} | U_{\alpha})$
is affine, equivalent to $\Spec^{ \calG^{\der}_{\mathet}(k)} A_{\alpha}$ for some
simplicial commutative $k$-algebra $A_{\alpha}$. The composite geometric morphisms
$$ \calX_{/U_{\alpha}} \rightarrow \calX \rightarrow \calX'$$
are \etale and cover $\calX'$. Since assertion $(2)$ is local on $\calX'$, it is sufficient to show
that the restriction of each $\pi_{i} \calO_{\calX}$ to $\calX_{/U_{\alpha}}$ is a quasi-coherent
sheaf on ordinary Deligne-Mumford stack given by $( \calX_{/U_{\alpha}}, \pi_0 (\calO_{\calX} | U_{\alpha} ))$ (in other words, the affine scheme $\SSpec \pi_0 A_{\alpha})$. This follows immediately
from Theorem \ref{scoo}: the restriction of $\pi_{i} \calO_{\calX}$ is the quasi-coherent sheaf
associated to $\pi_{i} A_{\alpha}$, viewed as a module over the commutative ring $\pi_0 A_{\alpha}$.

We now prove the converse. Suppose that $(1)$, $(2)$, and $(3)$ are satisfied; we wish to prove that
$(\calX, \calO_{\calX})$ is a derived Deligne-Mumford stack over $k$. The assertion is local on $\calX'$. The \etale geometric morphism $\calX \rightarrow \calX'$ determines an equivalence
$\calX \simeq \calX'_{/U}$, for some $2$-connective object $U$ in $\calX'$. Passing to a cover
of $\calX'$, we may assume without loss of generality that $U$ admits a global section
$s: 1_{\calX} \rightarrow U$; since $U$ is $1$-connective, this map is an effective epimorphism.
This section determines a geometric morphism of $\infty$-topoi $s_{\ast}: \calX' \rightarrow \calX$.
In view of Proposition \ref{sizem}, it will suffice to show that $(\calX', s^{\ast} \calO_{\calX})$ is
a derived Deligne-Mumford stack over $k$. Replacing $\calX$ by $\calX'$, we are reduced to the case where $\calX$ is $1$-localic and $(\calX, \pi_0 \calO_{\calX})$ is a derived Deligne-Mumford stack over $k$. Passing to a cover of $\calX$ again if necessary, we may suppose that
$(\calX, \pi_0 \calO_{\calX})$ is the spectrum of a (discrete) $k$-algebra $R$.

Applying $(2)$, we conclude that each $\pi_i \calO_{\calX}$ is the quasi-coherent sheaf associated to an $R$-module $M_i$. We then have isomorphisms
$$ \HH^{n}( \calX; \pi_i \calO_{\calX} ) \simeq \begin{cases} M_i & \text{if } n = 0 \\
0 & \text{otherwise.} \end{cases}$$
(see \S \toposref{chmdim} for a discussion of the cohomology of an $\infty$-topos, and Remark
\toposref{compfood} for a comparison with the usual theory of sheaf cohomology.)
For each $n \geq 0$, let $A_{\leq n} \in \SCR_{k}$ denote the global sections
$\Gamma( \calX; \tau_{\leq n} \calO_{\calX} )$. There is a convergent spectral sequence
$$ E_{2}^{p,q} = \HH^{p}(\calX; \pi_{q} (\tau_{\leq n} \calO_{\calX})) \Rightarrow \pi_{q-p} A_{\leq n}.$$
It follows that this spectral sequence degenerates to yield isomorphisms
$$ \pi_{i} A_{\leq n} \simeq \begin{cases} M_{i} & \text{if } i \leq n \\ 0 & \text{otherwise.} \end{cases}$$
In particular, $\pi_0 A_{\leq n} \simeq R$.

Fix $n \geq 0$, and let $( \calX_{n}, \calO_{\calX_{n}} )$ be the spectrum of $A_{\leq n}$.
The equivalence $A_{n} \simeq \Gamma( \calX; \tau_{\leq n} \calO_{\calX} )$ induces a map
$\phi_{n}: (\calX_{n}, \calO_{\calX_{n}}) \rightarrow (\calX, \tau_{\leq n} \calO_{\calX} )$ in $\LGeo( \calG_{\mathet}^{\der}(k))$. Since $\pi_0 A_{n} \simeq R$, the analysis of 
$\Spec^{ \calG_{\mathet}^{\der}(k)} A_{n}$ given in the first part of the proof shows that
$\phi_{n}^{\ast}: \calX_{n} \rightarrow \calX$ is an equivalence of $\infty$-topoi, and that
$\phi_{n}$ induces an isomorphism of quasi-coherent sheaves
$\phi_{n}^{\ast} (\pi_{i} \calO_{\calX_n}) \simeq \pi_{i} \calO_{\calX}$ for $0 \leq i \leq n$.
Since the structure sheaves on both sides are $n$-truncated, we conclude that $\phi_{n}$ is an equivalence.

Let $A \in \SCR_{k}$ denote the inverse limit of the tower
$$ \ldots \rightarrow A_{\leq 2} \rightarrow A_{\leq 1} \rightarrow A_{\leq 0},$$
so that $\pi_0 A \simeq R$. We can therefore identify the spectrum of $A$ with
$( \calX, \calO'_{\calX})$. The first part of the proof shows that $\calO'_{\calX}$ is the inverse limit
of its truncations 
$$\tau_{ \leq n} \calO'_{\calX} \simeq \phi_{n}^{\ast} \calO_{ \calX_{n}} \simeq
\tau_{ \leq n} \calO_{\calX}.$$
Passing to the inverse limit, we obtain a map
$$\psi: \calO_{\calX} \rightarrow \lim \{ \tau_{\leq n} \calO_{\calX} \} \simeq \calO'_{\calX}.$$
By construction, $\psi$ induces an isomorphism on all (sheaves of) homotopy groups, and is
therefore $\infty$-connective. The sheaf $\calO'_{\calX}$ is hypercomplete (since it is an inverse limit of truncated objects of $\calX$), and the sheaf $\calO_{\calX}$ is hypercomplete by assumption $(3)$.
It follows that $\psi$ is an equivalence, so that $(\calX, \calO_{\calX}) \simeq \Spec^{\calG^{\der}_{\mathet}} A$ is an affine derived $k$-scheme as desired.
\end{proof}

\subsection{Derived Complex Analytic Geometry}\label{secondcomp}

In \S \ref{juet}, we introduced the $\infty$-category of derived Deligne-Mumford stacks over a commutative ring $k$. In essence, we began with the ordinary category of smooth affine $k$-schemes, and used our theory of pregeometries to ``extrapolate'' to a theory which includes
singular objects. In this section, we will pursue an analogous strategy for describing a derived version of complex analytic geometry. 

\begin{definition}
We define a pregeometry $\calT_{\Stein}$ as follows:
\begin{itemize}
\item[$(1)$] The underlying $\infty$-category of $\calT_{\Stein}$ is
$\Nerve(\calC)$, where $\calC$ is the category of finite dimensional {\em Stein manifolds}: that is, complex analytic manifolds $X$ which admit closed immersions $X \hookrightarrow \C^{n}$.
\item[$(2)$] A morphism $U \rightarrow X$ in $\calT_{\Stein}$ is admissible if and only if 
it is a local homeomorphism.
\item[$(3)$] A collection of admissible morphisms $\{ U_{\alpha} \rightarrow X \}$ in
$\calT_{\Stein}$ generates a covering sieve on $X$ if and only if, for every point
$x \in X$, some inverse image $U_{\alpha} \times_{X} \{x\}$ is nonempty.
\end{itemize}
\end{definition}

\begin{remark}\label{spuut}
Let $\calT^{0}_{\Stein}$ be the pregeometry defined in the same way $\calT_{\Stein}$, except that the class of admissible morphisms in $\calT^{0}_{\Stein}$ is the class of {\em open immersions} of
Stein manifolds. Then the identity transformation $\calT^{0}_{\Stein} \rightarrow \calT_{\Stein}$ is a Morita equivalence; this follows immediately from Proposition \ref{spunkk}.
\end{remark}

\begin{remark}
Let $\calT^{1}_{\Stein}$ be the pregeometry defined in the same way as $\calT_{\Stein}$, but using
the class of {\em all} complex manifolds. Proposition \ref{silver} implies that the inclusion
$\calT_{\Stein} \subseteq \calT^{1}_{\Stein}$ is a Morita equivalence.
\end{remark}

\begin{proposition}
The pregeometry $\calT_{\Stein}$ is compatible with $n$-truncations.
\end{proposition}

\begin{proof}
This follows immediately from Remark \ref{spuut} and Proposition \ref{rabb}.
\end{proof}

\begin{remark}\label{slup}
For every smooth $\C$-algebra $A$, the set $\Hom_{\C}(A, \C)$ of $\C$-points of $\SSpec A$
has the structure of a complex analytic manifold $(\SSpec A)^{\an}$. This construction determines a transformation of geometries $\calT_{\mathet}(\C) \rightarrow \calT_{\Stein}$. In particular, to every
$\calT_{\Stein}$-structured $\infty$-topos $(\calX, \calO_{\calX})$, we can associate an underlying
sheaf $\calO^{\alg}_{\calX}$ of simplicial commutative $\C$-algebras on $\calX$. For each
$n \geq 0$, we let $\pi_{n} \calO_{\calX}$ denote the $n$th homotopy group of this sheaf. Then
$\pi_0 \calO_{\calX}$ is a commutative $\C$-algebra object of the underlying topos
$\h{(\tau_{\leq 0} \calX)}$, and each $\pi_{i} \calO_{\calX}$ has the structure of a module over
$\pi_0 \calO_{\calX}$.
\end{remark}

\begin{notation}
For each $0 \leq n \leq \infty$, let $\calG^{\leq n}_{\Stein}$ denote an $n$-truncated geometric envelope of $\calT_{\Stein}$.
\end{notation}

We now state a complex-analytic version of Theorem \ref{swill}:

\begin{proposition}\label{swun}
Let $( \calX, \calO_{\calX} )$ be a $\calT_{\Stein}$-structured $\infty$-topos, where
$\calX$ is $0$-localic. Let $n$ be a nonnegative integer. The following conditions are equivalent:
\begin{itemize}
\item[$(1)$] The $\calT_{\Stein}$-structure $\calO_{\calX}$ is $n$-truncated, and the associated
object of $\LGeo( \calG^{\leq n}_{\Stein})$ is a $\calG^{\leq n}_{\Stein}$-scheme which is locally of finite presentation.
\item[$(2)$] There exists a complex analytic space $(X, \calO_X)$ such that:
\begin{itemize}
\item[$(a)$] The $\infty$-topos $\calX$ is equivalent to $\Shv(X)$. 
\item[$(b)$] There is an equivalence $\calO_X \simeq \pi_0 \calO_{\calX}$ (of sheaves of
commutative $\C$-algebras on $X$).
\item[$(c)$] For $0 < i \leq n$, $\pi_{i} \calO_{\calX}$ is a coherent sheaf of $\calO_X$-modules on $X$.
\item[$(d)$] For $i > n$, the sheaf $\pi_{i} \calO_{\calX}$ vanishes.
\end{itemize}
\end{itemize}
\end{proposition}

The proof requires some ideas which we have not yet introduced, and will be given
in \cite{spectral}.

\begin{corollary}
Let $(\calX, \calO_{\calX})$ be a $\calG^{\leq n}_{\Stein}$-scheme which is locally of finite presentation, and let $0 \leq m \leq n$. Then $(\calX, \tau_{\leq m} \calO_{\calX})$ is a $\calG^{\leq m}_{\Stein}$-scheme which is locally of finite presentation.
\end{corollary}

\begin{proof}
The question is local on $\calX$, so we may assume without loss of generality that $(\calX, \calO_{\calX})$ is affine. In this case, $\calX$ is $0$-localic so that the desired result is an immediate consequence of Proposition \ref{swun}.
\end{proof}

\begin{definition}\label{oops}
A {\it derived complex analytic space} is a $\calT_{\Stein}$-structured $\infty$-topos
$(\calX, \calO_{\calX})$ such that, for every integer $n \geq 0$, the truncation
$(\calX, \tau_{\leq n} \calO_{\calX} )$ is a $\calG^{\leq n}_{\Stein}$-scheme which is locally of finite presentation. We let $\CompAn^{\der} \subseteq \LGeo( \calT_{\Stein} )^{op}$ denote the full subcategory spanned by the derived complex analytic spaces.
\end{definition}

\begin{remark}\label{noop}
In view of Proposition \ref{swun}, we can reformulate Definition \ref{oops} as follows:
if $\calX$ is a $0$-localic $\infty$-topos, then a $\calT_{\Stein}$-structure
$\calO_{\calX}: \calT_{\Stein} \rightarrow \calX$ determines a
derived complex analytic space $(\calX, \calO_{\calX})$ if and only if the following conditions are satisfied:
\begin{itemize}
\item[$(a)$] The $\infty$-topos $\calX$ has enough points, so that $\calX = \Shv(X)$ for some topological space $X$.
\item[$(b)$] The pair $(X, \pi_0 \calO_{\calX})$ is a complex analytic space.
\item[$(c)$] For each $i > 0$, $\pi_i \calO_{\calX}$ is a coherent sheaf on $X$.
\end{itemize}
If $\calX$ is not assumed to be $0$-localic, then we can still apply this criterion locally on $\calX$.
\end{remark}

In view of Proposition \ref{swun}, we can associate to every $0$-localic derived complex analytic space $(\calX, \calO_{\calX})$ an underlying complex analytic space $(X, \calO_{X})$, where
$X$ is the topological space of points of the ($0$-localic) $\infty$-topos $\calX$, and
$\calO_{X} = \pi_0 \calO_{\calX}$. This construction determines a functor
$$ \theta: \CompAn^{\der}_{\leq 0} \rightarrow \CompAn,$$
where $\CompAn^{\der}_{\leq 0}$ denotes the full subcategory of
$\CompAn^{\der}$ spanned by the $0$-localic objects, and $\CompAn$ the (nerve of the) ordinary category of complex analytic spaces. The following result shows that our theory of derived complex analytic geometry really does generalize the classical theory of complex analytic geometry:

\begin{proposition}\label{swune}
Let $\CompAn^{\der, \leq 0}_{\leq 0}$ denote the full subcategory of
$\CompAn^{\der}$ spanned by those objects which are $0$-localic and $0$-truncated.
Then the functor $\theta$ induces an equivalence of $\infty$-categories
$$ \CompAn^{\der, \leq 0}_{\leq 0} \rightarrow \CompAn.$$
\end{proposition}

Once again, the proof requires some ideas which have not yet been introduced, and will be deferred to \cite{spectral}.

\begin{remark}\label{jill}
Let $(\calX, \calO_{\calX})$ be a derived complex analytic space. In view of Theorem \ref{top4}, there exists an \etale map $(\calX, \tau_{\leq 0} \calO_{\calX}) \rightarrow (\calY, \calO_{\calY})$, where
$(\calY, \calO_{\calY})$ is a $0$-truncated, $1$-localic derived complex analytic space. If $\calY$ is $0$-localic, then we can identify it with an ordinary complex analytic space (Proposition \ref{swune}). 
In general, we can view $(\calY, \calO_{\calY})$ as a ringed topos which is locally equivalent
to the underlying topos of a complex analytic space; in other words, we can think of
$(\calY, \calO_{\calY})$ as a {\it complex analytic orbifold}.
\end{remark}

\begin{remark}
The transformation of pregeometries $\calT_{\mathet}(\C) \rightarrow \calT_{\Stein}$ of
Remark \ref{slup} determines a relative spectrum functor
$\Sch( \calT_{\mathet}(\C) ) \rightarrow \Sch( \calT_{\Stein} )$. This functor carries
derived Deligne-Mumford stacks which are locally of finite presentation over $\C$
to derived complex analytic spaces (generally not $0$-localic; see Remark \ref{jill}); we will refer to this as the {\it analytification} functor.
\end{remark}

\begin{warning}
The category of complex analytic spaces can be identified with a full subcategory
the $\LRingSpace_{\C}$ of locally ringed spaces over $\C$. The $\infty$-categorical analogue of this statement is {\em false}. The structure sheaf $\calO_{\calX}$ of a derived complex analytic space
$(\calX, \calO_{\calX})$ has more structure than simply a sheaf of (simplicial) commutative $\C$-algebras. Roughly speaking, this structure allows us to compose sections of $\calO_{\calX}$ with arbitrary complex analytic functions $\phi$ defined on open sets $U \subseteq \C^{n}$; this structure
descends to the algebraic structure sheaf $\calO_{\calX}^{\alg} = \calO_{\calX} | \calT_{\Zar}(\C)$ only if $\phi$ is a rational function. That this algebraic structure sheaf determines $(\calX, \calO_{\calX})$ in the $0$-truncated case is somewhat remarkable, and depends on strong finiteness properties enjoyed by the class of Stein algebras (which do not hold in the derived setting).
\end{warning}

\begin{remark}
With some effort, the ideas presented in this section can be carried over to the setting of rigid analytic geometry. We will return to this subject in \cite{ellipticloop}.
\end{remark}

\subsection{Derived Differential Geometry}\label{derdiff}

Let $\Diff$ denote the category whose objects are smooth submanifolds of some Euclidean space
$\R^{n}$, and whose morphisms are smooth maps.
We regard the $\infty$-category $\calT_{\Diff} = \Nerve( \Diff)$ as a pregeometry as follows:

\begin{itemize}
\item[$(a)$] A morphism $f: U \rightarrow X$ in $\Diff$ is admissible if it identifies
$U$ with an open submanifold of $X$.
\item[$(b)$] A collection of admissible morphisms $\{ U_{\alpha} \rightarrow X \}$ generates a covering sieve on $X$ if and only if, for every point $x \in X$, some preimage $U_{\alpha} \times_{X} \{x\}$ is
nonempty.
\end{itemize}

\begin{remark}
The exact definition of the category $\Diff$ is not very important. For example, we would obtain
a Morita equivalent pregeometry if we replace $\Diff$ by the category of {\em all} smooth manifolds (Proposition \ref{spunkk}), or if we were to allow {\em all} local homeomorphisms as
admissible morphisms in $\calT_{\Diff}$ (Proposition \ref{silver}).
\end{remark}

\begin{example}\label{otte}
Let $M$ be an object of $\Diff$. We define a $\calT_{\Diff}$-structure $\calO_M$ on
$\Shv(M)$ by the formula $\calO_{M}(X)(U) = \Hom( U, X)$; here
$U$ ranges over the open subsets of $M$ and the set of morphisms is computed in the category
of all smooth manifolds. Then $( \Shv(M), \calO_{M} )$ is a smooth $\calT_{\Diff}$-scheme: in fact, it can
be identified with the absolute spectrum $\Spec^{\calT_{\Diff}}(M)$ (Proposition \ref{und}). 

The same definition makes sense for {\em any} smooth manifold $M$. Via this construction, we can identify the category of smooth manifolds with the $\infty$-category of $0$-localic, smooth
$\calT_{\Diff}$-schemes.
\end{example}

In view of Definition \ref{otte}, we can regard $\Sch( \calT^{\Diff})$ as an {\em enlargement} of the category of smooth manifolds. In addition to ordinary smooth manifolds, it contains orbifolds and their higher-categorical cousins (these can be identified with the smooth $\calT_{\Diff}$-schemes), and many other objects which are relevant to their study; for example, some of the basic constructs of {\it synthetic differential geometry} find their home here. For a more detailed exposition of the theory of derived smooth manifolds, we refer the reader to \cite{spivak}.


\end{document}